\documentclass[a4paper,11pt]{amsart}
\usepackage{amsfonts,amsmath,amsthm,amssymb,stmaryrd}
\newtheorem{theorem}{Theorem}[section]
\newtheorem{lemma}[theorem]{Lemma}
\newtheorem{Proposition}[theorem]{Proposition}
\newtheorem{Remark}[theorem]{Remark}
\numberwithin{equation}{section}

\usepackage[left=1 in, right=1 in,top=1 in, bottom=1 in]{geometry}

\def\Lbrack{\left \llbracket}
\def\Rbrack{\right \rrbracket}

\DeclareMathOperator{\diverge}{div}
\DeclareMathOperator{\supp}{supp}

\providecommand{\norm}[1]{\left\Vert#1\right\Vert}
\def\ls{\lesssim}
\def\weak{\rightharpoonup}
\def\wstar{\overset{*}{\rightharpoonup}}
\def\dt{\partial_t}
\def\H{{}_0H^1}
\def\Hd{(\H)^\ast}
\def\na{\nabla_\ast}
\def\bna{\bar{\nabla}_\ast}
\def\rj{\Lbrack \rho \Rbrack}

\title[The viscous surface-internal wave problem]{The viscous surface-internal wave problem:\\ global well-posedness and decay}

\author{Yanjin Wang}
\address{
School of Mathematical Sciences\\
Xiamen University\\
Fujian 361005, China}
\email[Y. J. Wang]{yanjin$\_$wang@xmu.edu.cn}
\thanks{Y. J. Wang was partially supported by National Natural Science Foundation of China-NSAF (No. 10976026)}

\author{Ian Tice}
\address{
Division of Applied Mathematics\\
Brown University \\
Providence, RI 02912, USA
}
\email[I. Tice]{tice@dam.brown.edu}
\thanks{I. Tice was supported by an NSF Postdoctoral Research Fellowship.}

\author{Chanwoo Kim}
\address{
DPMMS\\
University of Cambridge\\
Cambridge CB3 0WA, UK
}
\email[C. Kim]{C.W.Kim@dpmms.cam.ac.uk}
\thanks{C. Kim was supported in part by NSF Grant FRG DMS 07-57227.}


\begin{document}

\begin{abstract}
We consider the free boundary problem for two layers of immiscible, viscous, incompressible fluid in a uniform gravitational field, lying above a general rigid bottom in a three-dimensional horizontally periodic setting.  We establish the global well-posedness of the problem  both with and without surface tension.  We prove that without surface tension the solution decays to the equilibrium state at an almost exponential rate;  with surface tension, we show that the solution decays at an exponential rate.  Our results include the case in which a heavier fluid lies above a lighter one, provided that the surface tension at the free internal interface is above a critical value, which we identify. This means that sufficiently large surface tension stabilizes the Rayleigh-Taylor instability in the nonlinear setting.  As a part of our analysis, we establish elliptic estimates for the two-phase stationary Stokes problem.
\end{abstract}

\maketitle

%\tableofcontents

%%%%%%%%%%%%%%%%%%%%%%%%%%%%%%%%%%%%%%%%%%%%%%%
\section{Introduction}
%%%%%%%%%%%%%%%%%%%%%%%%%%%%%%%%%%%%%%%%%%%%%%%

%%%%%%%%%%%%%%%%%%%%%%%%%%%%%%%%%%%%%%%%%%%%%%%
\subsection{Formulation of the problem in Eulerian coordinates}
%%%%%%%%%%%%%%%%%%%%%%%%%%%%%%%%%%%%%%%%%%%%%%%

In this paper we study the viscous surface-internal wave problem, which concerns the dynamics of two layers of  distinct, immiscible, viscous, incompressible fluid lying above a general rigid bottom and below an atmosphere of constant pressure.  We assume that a uniform gravitational field points in the direction of the rigid bottom.  This is a free boundary problem since both the upper surface, where the upper fluid meets the atmosphere, and the internal interface, where the upper and lower fluids meet, are free to evolve in time with the motion of the fluids.  In our analysis we consider the dynamics both with and without the effect of surface tension on the free surfaces.  

We will assume that the two fluids occupy the moving domain $\Omega(t)$, which we take to be three-dimensional and horizontally periodic.  One fluid $(+)$ fills the upper domain
\begin{equation}\label{omega_plus}
\Omega_+(t)=\{y\in   \mathrm{T}^2\times \mathbb{R}\mid \eta_-(t,y_1,y_2)<y_3< 1 +\eta_+(t,y_1,y_2)\},
\end{equation}
and the other fluid $(-)$ fills the lower domain
\begin{equation}
\Omega_-(t)=\{y\in  \mathrm{T}^2\times \mathbb{R}\mid  -b(y_1,y_2)<y_3<\eta_-(t,y_1,y_2)\}.
\end{equation}
Here we have written the periodic horizontal cross-section as $\mathrm{T}^2=(2\pi L_1\mathbb{T}) \times (2\pi L_2\mathbb{T})$, where $\mathbb{T} = \mathbb{R}/\mathbb{Z}$ is the usual 1--torus and $L_1,L_2>0$ are fixed constants.  We assume that the function $0<b\in C^\infty(\mathrm{T}^2)$ is fixed but that the surface functions $\eta_\pm$ are unknowns in the problem.  The surface $\Gamma_+(t) = \{y_3= 1  + \eta_+(t,y_1,y_2)\}$ is the moving upper boundary of $\Omega_+(t)$, $\Gamma_-(t) = \{y_3=\eta_-(t,y_1,y_2)\}$ is the moving internal interface between the two fluids, and $\Sigma_b = \{y_3=-b(y_1,y_2)\}$ is the fixed lower boundary of $\Omega_-(t)$.  

The fluids are described by their velocity and pressure functions, which are given for each $t\ge0$ by $v_\pm (t,\cdot):\Omega_\pm (t)\rightarrow \mathbb{R}^3$ and $\bar{p}_\pm (t,\cdot):\Omega_\pm (t)\rightarrow \mathbb{R}$, respectively. The fluid motion is governed by the pair of the incompressible Navier-Stokes equations:
\begin{equation}\label{e1}
\left\{\begin{array}{ll}
\rho_\pm \partial_t v_\pm  + \rho_\pm  v_\pm
\cdot\nabla v_\pm +\nabla \bar{p}_\pm =\mu_\pm \Delta v_\pm
-g\rho_\pm e_3
\\
\diverge{v_\pm} =0.
\end{array}\right.\quad\text{in }\Omega_\pm (t).
\end{equation}
Here the positive constants $\rho_\pm $ and $\mu_\pm $ denote the densities and viscosities of the respective fluids, $g>0$ is the acceleration of gravity, $e_3=(0,0,1)^T$  and $-g \rho_\pm e_3$ is the gravitational force.  Throughout the paper we will write $\rj := \rho_+ - \rho_-$.

We now prescribe the boundary conditions. We assume that the upper fluid is in contact with an atmosphere of constant pressure $p_{e}$. The standard assumptions at the free internal interface (cf. \cite{WL}) are that the velocity is continuous across the interface and the jump in the normal stress is proportional to the mean curvature of the surface multiplied by the normal to the surface. These give  rise to the conditions at the free boundaries:
\begin{equation} \label{e2}
\left\{\begin{array}{ll}
(\bar{p}_+I-\mu_+\mathbb{D}(v_+))n_+=p_{e}n_+-\sigma_+ H_+n_+
\quad &\text{on }\Gamma_+(t)\\
v_+=v_-,\
(\bar{p}_+I-\mu_+\mathbb{D}(v_+))n_-=(\bar{p}_-I-\mu_-\mathbb{D}(v_-))n_-+
\sigma_- H_-n_- \quad &\text{on
}\Gamma_-(t),
\end{array}\right.
\end{equation}
where we have written $I$ for the $3\times3$ identity matrix, $(\mathbb{D} v_\pm)_{ij} = \partial_j v_\pm^i + \partial_i v_\pm^j$ for twice the velocity deformation tensor, and $n_\pm$ for the unit normal of $\Gamma_\pm (t)$ (pointing up), which is given by
\begin{equation}
n_\pm=\frac{(-\partial_{y_1}\eta_\pm,-\partial_{y_2}\eta_\pm,1)}{\sqrt{1+|\nabla_\ast\eta_\pm|^2}}.
\end{equation}
In this paper,  we let $\nabla_\ast$ denote the horizontal gradient, $\diverge_\ast$ denote the horizontal divergence  and $\Delta_\ast$ denote the horizontal Laplace operator. The quantity $S(\bar{p}_\pm,v_\pm) = \bar{p}_\pm I-\mu_\pm\mathbb{D} v_\pm$ is the viscous stress tensor.  We take $\sigma_\pm\ge0$ to be the constant
coefficient of surface tension, and  $H_\pm$ to be twice the mean curvature of the surface $\Gamma_\pm(t)$, which is given by the formula
\begin{equation}
H_\pm=\diverge_\ast\left(\frac{\nabla_\ast\eta_\pm}{\sqrt{1+|\nabla_\ast\eta_\pm|^2}}\right).
\end{equation}
We enforce the no-slip condition at the fixed lower boundary:
\begin{equation} \label{e3}
v_-=0 \quad\text{on }\Sigma_b.
\end{equation}

The free surfaces are advected with the fluids according to the kinematic boundary condition
\begin{equation} \label{e4}
\partial_t\eta_\pm  =v_{3,\pm}-v_{1,\pm}\partial_{y_1}\eta_\pm  -v_{2,\pm}\partial_{y_2}\eta_\pm \quad\text{on }
\Gamma_\pm(t).
\end{equation}
Notice that on $\Gamma_-(t)$, the continuity of velocity implies that $v_+ = v_-$, which means that it is the common value of $v_\pm$ that advects the surface in \eqref{e4}.

Note that the constant $1$ appearing above in the definition of $\Omega_+(t)$ and $\Gamma_+(t)$ is the equilibrium height of the upper fluid, i.e. the height of a solution with $v_\pm =0$, $\eta_\pm =0$, etc.  It is not a loss of generality for us to assume this value is unity.  Indeed, if we were to replace the $1$ in $\Omega_+(t)$ and $\Gamma_+(t)$ by an arbitrary equilibrium height $L_3 >0$, then a standard scaling argument would allow us to rescale the coordinates and unknowns in such a way that $L_3 >0$ would be replaced by $1$.  Of course, in doing so we would have to multiply $L_1,$ $L_2,$ $\mu_\pm$, $\sigma_\pm$, $g$ and $b$ by positive constants and to rescale $b$ as well.  In our above formulation of the problem we assume that this procedure has already been done.

To complete the statement of the problem, we must specify initial conditions. We suppose that the initial surfaces $\Gamma_\pm(0)$ are given by the graphs of the functions $\eta_\pm(0)=\eta_{0,\pm}$, which yield the open sets $\Omega_\pm(0)$ on which we specify the initial data for the velocity, $v_\pm(0)=v_{0,\pm}:\ \Omega_\pm(0) \rightarrow \mathbb{R}^3$. We will assume that $1+\eta_{0,+}>\eta_{0,-}>-b$ on $\mathrm{T}^2$ and that $\eta_{0,\pm}, v_{0,\pm}$ satisfy certain compatibility conditions, which we will present in detail later. We assume further that the initial surface functions satisfy the ``zero-average'' condition
\begin{equation}\label{zero0}
\int_{\mathrm{T}^2}\eta_{0,\pm}=0.
\end{equation}
For sufficiently regular solutions to the problem, the condition \eqref{zero0} persists in time, i.e.
\begin{equation}
\label{zerot}\int_{\mathrm{T}^2}\eta_{\pm}(t)=0 \text{ for } t\ge 0.
\end{equation}
Indeed, from the equations $\eqref{e1}_2$, \eqref{e3}, and \eqref{e4}   we obtain
\begin{equation}
\frac{d}{dt} \int_{\mathrm{T}^2} \eta_-=\int_{\mathrm{T}^2}\partial_t\eta_-=\int_{\Gamma_-(t)}v_-\cdot
n_-=\int_{\Omega_-(t)} \diverge{v_-}=0,
\end{equation}
and
\begin{equation}
\frac{d}{dt} \int_{\mathrm{T}^2} \eta_+ = \int_{\mathrm{T}^2}\partial_t\eta_+ = \int_{\Gamma_+(t)}v_+ \cdot
n_+  = \int_{\Omega_+(t)} \diverge{v_+} + \int_{\Gamma_-(t)} v_- \cdot n_- = 0.
\end{equation}

If it happens that the initial surfaces do not satisfy the zero average condition \eqref{zero0}, then it is possible to shift the data and the coordinate system so that \eqref{zero0} is satisfied, provided that the extra condition $\inf_{\mathrm{T}^2} b + (\eta_{0,-}) >0$ is satisfied, where
\begin{equation}
 (\eta_{0,-}) := \frac{1}{|\mathrm{T}^2|} \int_{\mathrm{T}^2} \eta_{0,-}.
\end{equation}
In this case we may shift
\begin{equation}
\begin{split}
 &y_3 \mapsto y_3 - (\eta_{0,-}), b \mapsto b+ (\eta_{0,-}), L_3 =1 \mapsto L_3 = 1 + (\eta_{0,+}) - (\eta_{0,-}), \\
& \eta_- \mapsto \eta_0 - (\eta_{0,-}),  \eta_+ \mapsto \eta_+ - (\eta_{0,+})
\end{split}
\end{equation}
to obtain a new solution with the initial surfaces satisfying \eqref{zero0}. The shifted initial data will still satisfy the compatibility conditions, while $b + (\eta_{0,-}) \ge \inf_{\mathrm{T}^2} b + (\eta_{0,-}) >0$ and $L_3 = 1 + (\eta_{0,+}) - (\eta_{0,-}) >0$ because of the assumption that $1 + \eta_{0,+} > \eta_{0,-}$. This new $L_3>0$ may not be unity, but we may then again rescale the problem as discussed above to set $L_3 =1$ without disrupting the zero average condition.

To simply the equations we introduce the indicator function $\chi$ and denote
\begin{eqnarray}\label{uni-notation}
\left\{\begin{array}{ll}
v=v_+{\chi_{\Omega_+(t)}}+v_-{\chi_{\Omega_-(t)}},
\quad
\bar{p}=\bar{p}_+{\chi_{\Omega_+}(t)}+\bar{p}_-{\chi_{\Omega_-(t)}},
\\\rho=\rho_+{\chi_{\Omega_+(t)}}+\rho_-{\chi_{\Omega_-(t)}},\quad
\mu=\mu_+{\chi_{\Omega_+(t)}}+\mu_-{\chi_{\Omega_-(t)}}
,\end{array}\right.
\end{eqnarray}
on $\Omega(t) = \Omega_+(t) \cup \Omega_-(t)$.  We similarly denote quantities on $\Gamma(t): = \Gamma_+(t) \cup \Gamma_-(t)$, etc.  We define the modified pressure through $\tilde{p}=\bar{p}+\rho gy_3-p_e-\rho_+g$. Then the modified equations are
\begin{equation}\label{form1}
\left\{\begin{array}{ll}
\rho\partial_t v + \rho v\cdot\nabla v+\nabla \tilde{p}=\mu\Delta v\quad&\text{in }\Omega(t)
\\ \diverge{v}=0\quad&\text{in }\Omega(t)
\\(\tilde{p}_+I-\mu_+\mathbb{D}(v_+))n_+  =\rho_+g\eta_+ n_+-\sigma_+ H_+n_+
\quad &\text{on }\Gamma_+(t)\\
v_+=v_-\quad&\text{on }\Gamma_-(t)
\\
(\tilde{p}_+I-\mu_+\mathbb{D}(v_+))n_--(\tilde{p}_-I-\mu_-\mathbb{D}(v_-))n_-=\rj g\eta_-
n_-+ \sigma_- H_-n_- \quad &\text{on }\Gamma_-(t)
\\\partial_t\eta =v_3 -v_1\partial_{y_1}\eta -v_2\partial_{y_2}\eta \quad &\text{on }\Gamma (t)
\\(v,\eta)\mid_{t=0}=(v_0,\eta_0).
\end{array}\right.
\end{equation}
In this paper we shall always unify the notations by means of \eqref{uni-notation} to suppress the subscript ``$\pm$'' unless clarification is needed.

%%%%%%%%%%%%%%%%%%%%%%%%%%%%%%%%%%%%%%%%%%%%%%%
\subsection{Reformulation via a flattening coordinate transformation}
%%%%%%%%%%%%%%%%%%%%%%%%%%%%%%%%%%%%%%%%%%%%%%%

%%%%%%%%%%%%%%%%%%%%%%%%%%%%%%%%%%%%%%%%%%%%%%%
\subsubsection{Special flattening coordinate transformation}
%%%%%%%%%%%%%%%%%%%%%%%%%%%%%%%%%%%%%%%%%%%%%%%
The movement of the free boundaries $\Gamma_\pm(t)$ and the subsequent change of the domains $\Omega_\pm(t)$  create numerous mathematical difficulties. To circumvent these, as usual, we will transform the free boundary problem under consideration to a problem with a fixed domain and fixed boundary. Since we are interested in the asymptotic decay of solutions, we will use the equilibrium domain. We will not use the usual Lagrangian coordinate transformation as in \cite{So, B1}, but rather utilize a special flattening coordinate transformation motivated by the one introduced in \cite{B2}.

To this end, we define the fixed domain
\begin{equation}
\Omega = \Omega_+\cup\Omega_-\text{ with }\Omega_+:=\{0<x_3<1\}, \Omega_-:=\{-b<x_3<0\}
\end{equation}
for which we have written the coordinates as $x\in \Omega$. We shall write $\Sigma_+:=\{x_3= 1\}$ for the upper boundary, $\Sigma_-:=\{x_3=0\}$ for the internal interface and $\Sigma_b:=\{x_3=-b(x_1,x_2)\}$ for the lower boundary.  Throughout the paper we will write $\Sigma = \Sigma_+ \cup \Sigma_-$.   We think of $\eta_\pm$ as a function on $\Sigma_\pm$ according to $\eta_+: (\mathrm{T}^2\times\{1\}) \times \mathbb{R}^{+} \rightarrow\mathbb{R}$ and $\eta_-:(\mathrm{T}^2\times\{0\}) \times \mathbb{R}^{+} \rightarrow \mathbb{R}$, respectively, where $\mathbb{R}^{+} = (0,\infty)$. We will transform the free boundary problem in $\Omega(t)$ to one in the fixed domain $\Omega $ by using the unknown free surface functions $\eta_\pm$. For this we define
\begin{equation}
\bar{\eta}_+:=\mathcal{P}_+\eta_+=\text{ Poisson extension of }\eta_+ \text{ into }\mathrm{T}^2 \times \{x_3\le 1\}
\end{equation}
and
\begin{equation}
\bar{\eta}_-:=\mathcal{P}_-\eta_-=\text{ specialized Poisson extension of }\eta_-\text{ into }\mathrm{T}^2 \times \mathbb{R},
\end{equation}
where $\mathcal{P}_\pm$ are defined by \eqref{P+def} and \eqref{P-def}.  We now encounter the first key problem of how to define an appropriate coordinate transformation to flatten the two free surfaces together.  Since we only need to transform the third component of the spatial coordinate and keep the other two fixed, we can flatten the domain by a linear combination of the three boundary functions, as introduced by Beale in \cite{B2}. However, this would result in the discontinuity of the Jacobian matrix of the coordinate transformation, which would then lead to severe technical difficulties in our proof of the local well-posedness of the problem. We overcome this difficulty by flattening  the coordinate domain via the following special coordinate transformation, writing $\tilde{b}=1+ x_3/b$:
\begin{equation}\label{cotr}
\left\{\begin{array}{lll}
\Omega_+\ni x\mapsto(x_1,x_2,
x_3^2(\bar{\eta}_+-(1+1/b)\bar{\eta}_-)+x_3+\tilde{b}\bar{\eta}_-)=\Theta_+(t,x)=(y_1,y_2,y_3)\in\Omega_+(t),
\\\Omega_-\ni x\mapsto(x_1,x_2,x_3+\tilde{b}\bar{\eta}_-)=\Theta_-(t,x)=(y_1,y_2,y_3)\in\Omega_-(t).
\end{array}\right.
\end{equation}
Note that $\Theta(\Sigma_+,t)=\Gamma_+(t),\ \Theta (\Sigma_-,t)=\Gamma_-(t)$ and $\Theta_-(\cdot,t) \mid_{\Sigma_b} = Id \mid_{\Sigma_b}$. In order to write down the equations in the new coordinate system,  we compute
\begin{equation}
\begin{array}{ll} \nabla\Theta =\left(\begin{array}{ccc}1&0&0\\0&1&0\\A &B &J \end{array}\right)
\text{ and }\mathcal{A} := \left(\nabla\Theta
^{-1}\right)^T=\left(\begin{array}{ccc}1&0&-A  K \\0&1&-B  K \\0&0&K
\end{array}\right)\end{array}.
\end{equation}
Here the components in the matrix are
\begin{equation}\label{ABJ_def}
\left\{\begin{array}{lll}
A_+=x_3^2\left(\partial_1\bar{\eta}_+-(1+1/b)\partial_1\bar{\eta}_-+(\partial_1b\bar{\eta}_-)/b^2\right)+\partial_1\bar{\eta}_-\tilde{b}-(x_3\bar{\eta}_-\partial_1b)/b^2,
\\ B_+=x_3^2\left(\partial_2\bar{\eta}_+-(1+1/b)\partial_2\bar{\eta}_-+(\partial_2b\bar{\eta}_-)/b^2\right)+\partial_2\bar{\eta}_-\tilde{b}-(x_3\bar{\eta}_-\partial_2b)/b^2,
\\ J_+=2x_3 (\bar{\eta}_+-(1+1/b)\bar{\eta}_-)+x_3^2(\partial_3\bar{\eta}_+-(1+1/b)\partial_3\bar{\eta}_-)+1+\bar{\eta}_-/b+\partial_3\bar{\eta}_- \tilde{b},
\\ A_-=\partial_1\bar{\eta}_-\tilde{b}-(x_3\bar{\eta}_-\partial_1b)/b^2,
\
B_-=\partial_2\bar{\eta}_-\tilde{b}-(x_3\bar{\eta}_-\partial_2b)/b^2,
\\  J_-=1+\bar{\eta}_-/b+\partial_3\bar{\eta}_- \tilde{b},\  K =J ^{-1}.
\end{array}\right.
\end{equation}
Notice that $J={\rm det}\, \nabla\Theta $ is the Jacobian of the coordinate transformation. We also denote
\begin{equation}\label{WNT_def}
 \begin{split}
    W & = \partial_t\Theta^3K \\
    \mathcal{N} &= (-\partial_1\eta ,-\partial_2\eta ,1) \\
    \mathcal{T}^i &= e_i + \partial_i\eta e_3 \text{ for }i=1,2\\
    \theta_{ij}&=\partial_j\Theta_i\text{ for }i,j=1,2,3.
 \end{split}
\end{equation}
It is straightforward to check that, because of how we have defined $\bar{\eta}_\pm$,  $\mathcal{A}$, etc., are all continuous across the internal interface $\Sigma_-$. This is crucial for our whole analysis. Also, we may remark that one can replace $x_3^2$ in the coordinate transformation \eqref{cotr} by some smooth function $h(x_3)$ with $h(1)=1$ and $h(0)=h'(0)=\cdots=0$, to let $\Theta$ be more regular across the interface $\Sigma_-$.

Note that if $\eta $ is sufficiently small (in an appropriate Sobolev space), then the mapping $\Theta $ is a  $C^1$
diffeomorphism.  This allows us to transform the problem to one in the fixed spatial domain $\Omega$ for each $t\ge 0$.

%%%%%%%%%%%%%%%%%%%%%%%%%%%%%%%%%%%%%%%%%%%%%%%
\subsubsection{``Geometric'' reformulation without surface tension}
%%%%%%%%%%%%%%%%%%%%%%%%%%%%%%%%%%%%%%%%%%%%%%%

Without  surface tension, i.e. $\sigma_\pm=0$, we  define the velocity $u$ and the pressure $p$ on $\Omega$ directly by the composition with $\Theta$ as $u(t,x)=v\circ\Theta(t,x)$ and $p=\tilde{p}\circ \Theta(t,x)$. Hence in the new coordinates, the system \eqref{form1} without surface tension becomes
\begin{equation}\label{nosurface}
\left\{\begin{array}{lll}
\rho\partial_t u-\rho W\partial_3u
+ \rho u\cdot
\nabla_{\mathcal{A}}u-\mu\Delta_{\mathcal{A}}u+\nabla_{\mathcal{A}}p=0\quad&\text{in
}\Omega
\\ \diverge_{\mathcal{A}} u=0&\text{in }\Omega
\\  S_{\mathcal{A}}(p_+,u_+) \mathcal{ N}_+=\rho_+g\eta_+ \mathcal{N}_+&\text{on }\Sigma_+
\\ \Lbrack u\Rbrack=0,\quad\Lbrack S_{\mathcal{A}}(p,u)\Rbrack \mathcal{N}_-=\rj g\eta_- \mathcal{N}_-&\text{on }\Sigma_-
\\ \partial_t\eta=u\cdot \mathcal{N}&\text{on }\Sigma
\\ u=0 &\text{on }\Sigma_b
\\(u,\eta )\mid_{t=0}=(u_0,\eta_0).
\end{array}\right.
\end{equation}
Here we have written the differential operators $\nabla_\mathcal{A}, \diverge_\mathcal{A}$  and $\Delta_\mathcal{A}$ with their actions given by $(\nabla_\mathcal{A}f)_i:= \mathcal{A}_{ij}\partial_jf$, $\diverge_\mathcal{A}X := \mathcal{A}_{ij}\partial_jX_i$, and $\Delta_\mathcal{A} f = \diverge_\mathcal{A}\nabla_\mathcal{A}f$ for
appropriate $f$ and $X$.   We write $S_\mathcal{A}(p, u):=pI - \mu \mathbb{D}_\mathcal{A}u$ for the stress tensor, where $(\mathbb{D}_\mathcal{A}u)_{ij} = \mathcal{A}_{ik} \partial_k u^j + \mathcal{A}_{jk} \partial_k u^i$ is the
symmetric $\mathcal{A}$--gradient. Note that if we extend $\diverge_\mathcal{A}$ to act on symmetric tensors in the natural way, then $\diverge_\mathcal{A}S_\mathcal{A}(p,u)=-\mu\Delta_{\mathcal{A}}u+\nabla_{\mathcal{A}}p$ for those $u$ satisfying $\diverge_\mathcal{A} u =0$. We have also denoted the interfacial jump as $\Lbrack f \Rbrack
= f_+ \mid_{x_3=0} - f_-\mid_{x_3=0}$.  We refer to \eqref{nosurface} as a ``geometric'' reformulation of the problem since the differential operators involve $\mathcal{A}$, which encodes the geometry of the domains $\Omega_\pm(t)$.

The form \eqref{nosurface} is more faithful to the geometry of the free boundary problem, but it is inconvenient for many of our a priori estimates. It is unsuitable to apply the differential operators since the underlying linear operators have non-constant coefficients. To get around this problem, we will  also analyze the PDE in a different formulation, which looks like a perturbation of the linearized problem:
\begin{equation}\label{nosurface2}
\left\{\begin{array}{lll}
\rho\partial_t u-\mu\Delta u+\nabla p=G^1\quad&\text{in }\Omega
\\ \diverge{u}=G^2&\text{in }\Omega
\\ ( p_+I-\mu_+\mathbb{D}(u_+)) e_3= \rho_+g\eta_+ e_3+G^3_+&\text{on }\Sigma_+
\\ \Lbrack u\Rbrack=0,\quad \Lbrack pI-\mu\mathbb{D}(u)\Rbrack e_3=\rj g\eta_- e_3-G^3_- &\text{on }\Sigma_-
\\ \partial_t\eta-u_3=G^4&\text{on }\Sigma
\\ u_-=0 &\text{on }\Sigma_b
\\(u,\eta )\mid_{t=0}=(u_0,\eta_0),
\end{array}\right.
\end{equation}
where
\begin{eqnarray} &&
G^1=\rho W\partial_3u-\rho
u\cdot\nabla_\mathcal{A}u+\mu(\Delta_\mathcal{A}-\Delta)u-(\nabla_\mathcal{A}-\nabla)p,
\\&& G^2=(\diverge-\diverge_\mathcal{A})u,
\\ &&G^3_+=(p-\rho_+ g \eta_+)(e_3-\mathcal{N}_+)+\mu_+(\mathbb{D}u e_3-\mathbb{D}_{\mathcal{A}_+}u_+ \mathcal{N}_+),
\\ &&G^3_-=-(\Lbrack p \Rbrack-\rj  g \eta_-)(e_3-\mathcal{N}_-)-\Lbrack \mu(\mathbb{D}u e_3-\mathbb{D}_\mathcal{A}u \mathcal{N})\Rbrack,
\\ &&G^4=-u_1\partial_1\eta-u_2\partial_2\eta.
\end{eqnarray}
One may also refer to \cite{GT_inf,GT_per,H,S} for the the explicit expressions of $G^i$ in components.

%%%%%%%%%%%%%%%%%%%%%%%%%%%%%%%%%%%%%%%%%%%%%%%
\subsubsection{Reformulation with surface tension}
%%%%%%%%%%%%%%%%%%%%%%%%%%%%%%%%%%%%%%%%%%%%%%%

With surface tension, i.e. $\sigma_\pm>0$, $\eta$ will be in a higher regularity class than $u$. This allows us to use the clever idea introduced by Beale in \cite{B2} to transform the velocity field in a manner that preserves the divergence-free condition. We define the pressure $p$ on $\Omega$ by the composition $p(t,x)=\tilde{p} \circ \Theta(t,x)$, but we define the velocity $u$ on $\Omega$ according to $v_i\circ \Theta(t,x)=K \theta_{ij} u_j(t,x)$ (equivalently, $u_i(t,x)=J \mathcal{A}_{ji} v_j \circ \Theta(t,x)$). The advantages of this transform are twofold.  First, $u$ has divergence zero in $\Omega$ if and only if $v$ has the same property in $\Omega(t)$.  Second,
the right-hand side of $\eqref{form1}_6$ is replaced simply by $u_3$.

In this case, we shall only need to write down the equations for the new unknowns $(u,p)$  in the perturbation form. In the new coordinates, the momentum equations $\eqref{form1}_1$ become
\begin{equation}
\rho\partial_t u-\mu\Delta u+\nabla p=f\quad \text{ in } \Omega,
\end{equation}
where $f=(f^1,f^2,f^3)$ is given by
\begin{equation}\label{f}
\begin{split}
f^i= &-\rho J \mathcal{A}_{ji}[ \partial_t(K \theta_{jk})u_k-W\partial_3(K \theta_{jk})u_k+K \partial_k(K\theta_{jl})u_ku_l] + \rho [W \partial_3 u_i -K u_k \partial_k u_i ]  \\&
+\mu J\mathcal{A}_{ni}\mathcal{A}_{jk}\partial_k(\mathcal{A}_{jl})\partial_l(K\theta_{nm})u_m
+\mu   \mathcal{A}_{jk}\partial_k(\mathcal{A}_{jl})   \partial_lu_i
 +\mu J\mathcal{A}_{ni}\mathcal{A}_{jk} \mathcal{A}_{jl}\partial_k\partial_l(K\theta_{nm})u_m
\\& +\mu J\mathcal{A}_{ni}\mathcal{A}_{jk}
\mathcal{A}_{jl}\partial_l(K\theta_{nm})\partial_k u_m +\mu
J\mathcal{A}_{ni}\mathcal{A}_{jk}
\mathcal{A}_{jl}\partial_k(K\theta_{nm})\partial_l u_m
\\& +\mu   (\mathcal{A}_{jk} \mathcal{A}_{jl}-\delta_{kl})    \partial_k \partial_l u_i
 +(\delta_{ki}-J\mathcal{A}_{ji}\mathcal{A}_{jk} )\partial_k p,\quad i=1,2,3.
\end{split}
\end{equation}
The boundary condition at $\Sigma_+$ becomes
\begin{equation}
(p_+ I - \mu_+\mathbb{D}(u_+)) e_3= (\rho_+g\eta_+  -\sigma_+\Delta_\ast\eta_+)e_3+g_+ \quad\text{on }\Sigma_+
\end{equation}
where $g_+=(g_+^1,g_+^2,g_+^3)$ is given by (here $\mathcal{N}$ and  $\mathcal{T}^i$ are as defined in \eqref{WNT_def})
\begin{equation}\label{g_+^i}
 \begin{split}
g_+^i&=\mu J\mathcal{A}_{mk} \partial_k(K\theta_{jl}) u_l \mathcal{N}_j \mathcal{T}_m^i + \mu(\mathcal{A}_{mk}\theta_{jl}-\delta_{ik}\delta_{jl})\partial_k u_l \mathcal{ N}_j \mathcal{T}_m^i
\\&+ \mu J\mathcal{A}_{jk} \partial_k (K\theta_{ml}) u_l \mathcal{N}_j \mathcal{T}_m^i
 + \mu (\mathcal{A}_{jk}\theta_{ml}-\delta_{jk}\delta_{il}) \partial_k u_l \mathcal{ N}_j\mathcal{T}_m^i -\mu\partial_i u_1 \partial_1 \eta - \mu\partial_i u_2 \partial_2\eta
\\& + \partial_3 u_l \mathcal{N}_l \partial_i \eta
 -\mu \partial_1 u_i \partial_1\eta -\mu \partial_2 u_i \partial_2 \eta  + \mu \partial_k u_3 \mathcal{ N}_k \partial_i \eta \text{ for } i=1,2,
 \end{split}
\end{equation}
\begin{equation}\label{g_+^3}
 \begin{split}
g_+^3 = &2\mu  (\mathcal{A}_{3k} K\theta_{3l}-\delta_{3k}\delta_{3l})\partial_ku_l +2\mu \mathcal{A}_{3k} \partial_k (K\theta_{3l}) u_l
\\& +2 \mu (\mathcal{A}_{ik} \partial_k(K\theta_{jl}) u_l + \mathcal{A}_{ik} K \theta_{jl}\partial_k u_l  )
 (\mathcal{ N}_j\mathcal{ N}_i-\delta_{j3}\delta_{i3})|\mathcal{N}|^{-2}
\\& + 2 \mu  (\mathcal{A}_{3k}\partial_k(K\theta_{3l} ) u_l + \mathcal{A}_{3k} K \theta_{3l} \partial_k u_l
)(|\mathcal{N}|^{-2}-1)  - \sigma_+((1+|\nabla_\ast\eta|^2)^{-1/2} -1) \Delta_\ast\eta
 \\& + \sigma_+(1+|\nabla_\ast\eta|^2)^{-3/2}((\partial_1\eta)^2 \partial_1^2 \eta + 2 \partial_1 \eta \partial_2 \eta \partial_1 \partial_2 \eta +(\partial_2\eta)^2 \partial_2^2 \eta).
 \end{split}
\end{equation}
Similarly, the second jump condition on $\Sigma_-$ becomes
\begin{equation}
\Lbrack p -\mu\mathbb{D}(u)\Rbrack e_3=(\rj g\eta_- +\sigma_-\Delta_\ast\eta_-)e_3-g_-\quad\text{on } \Sigma_-,
\end{equation}
where $g_-=(g_-^1,g_-^2,g_-^3)$ is given by
\begin{equation}\label{g_-^i}
 \begin{split}
-g_-^i = &\Lbrack \mu\Rbrack J\mathcal{A}_{mk}\partial_k(K\theta_{jl})u_l\mathcal{
N}_j \mathcal{T}_m^i + (\mathcal{A}_{mk}\theta_{jl}-\delta_{ik}\delta_{jl}) \Lbrack \mu
\partial_k u_l\Rbrack\mathcal{ N}_j\mathcal{T}_m^i
\\& + \Lbrack \mu\Rbrack J\mathcal{A}_{jk}\partial_k(K\theta_{ml})u_l\mathcal{ N}_j\mathcal{T}_m^i
 + (\mathcal{A}_{jk}\theta_{ml}-\delta_{jk}\delta_{il})\Lbrack\mu\partial_k
u_l \Rbrack\mathcal{ N}_j \mathcal{T}_m^i
\\& - \Lbrack\mu\Rbrack\partial_i u_1 \partial_1\eta
- \Lbrack\mu\Rbrack\partial_i u_2 \partial_2\eta + \Lbrack\mu\partial_3u_l\Rbrack\mathcal{
N}_l\partial_i\eta
\\& -\Lbrack\mu\Rbrack \partial_1 u_i \partial_1\eta -\Lbrack\mu \Rbrack\partial_2 u_i \partial_2\eta + \Lbrack\mu\partial_k u_3 \Rbrack \mathcal{ N}_k\partial_i\eta, \text{ for } i=1,2,
 \end{split}
\end{equation}
\begin{equation}\label{g_-^3}
 \begin{split}
-g^3 = &2  (\mathcal{A}_{3k} K\theta_{3l}-\delta_{3k}\delta_{3l})\Lbrack\mu\partial_ku_l\Rbrack
+ 2\Lbrack\mu  \Rbrack \mathcal{A}_{3k}\partial_k(K\theta_{3l})u_l
\\ & + 2 (\Lbrack\mu\Rbrack\mathcal{A}_{ik}\partial_k(K\theta_{jl})u_l
+ \mathcal{A}_{ik}K\theta_{jl}\Lbrack\mu\partial_k u_l\Rbrack  ) (\mathcal{ N}_j\mathcal{ N}_i-\delta_{j3}\delta_{i3})|\mathcal{N}|^{-2}
\\& + 2 (\Lbrack\mu\Rbrack\mathcal{A}_{3k}\partial_k(K\theta_{3l}) u_l
+ \mathcal{A}_{3k}K\theta_{3l}\Lbrack\mu\partial_k u_l\Rbrack )(|\mathcal{N}|^{-2}-1)
\\& +\sigma_- ((1+|\nabla_\ast\eta|^2)^{-1/2} -1) \Delta_\ast\eta
 \\ & -\sigma_- (1+|\nabla_\ast\eta|^2)^{-3/2}((\partial_1\eta)^2 \partial_1^2 \eta + 2 \partial_1 \eta \partial_2 \eta \partial_1\partial_2\eta + (\partial_2\eta)^2 \partial_2^2\eta).
 \end{split}
\end{equation}
Note that  the coordinate transformation \eqref{cotr} guarantees that $\Lbrack u\Rbrack=0$ on $\Sigma_-$.

Combining these, we see that in the new coordinates  the system \eqref{form1} with  surface tension becomes
 \begin{equation}\label{surface}
\left\{\begin{array}{lll}
\rho\partial_t u-\mu\Delta u+\nabla p=f\quad&\text{in }\Omega
\\ \diverge{u}=0&\text{in }\Omega
\\ ( p_+I-\mu_+\mathbb{D}(u_+)) e_3= (\rho_+g\eta_+  -\sigma_+\Delta_\ast \eta_+)e_3+g_+&\text{on }\Sigma_+
\\\Lbrack u\Rbrack=0,\quad \Lbrack pI-\mu\mathbb{D}(u)\Rbrack e_3=(\rj g\eta_- +\sigma_- \Delta_\ast \eta_-)e_3-g_-&\text{on }\Sigma_-
\\ \partial_t\eta=u_3 &\text{on }\Sigma
\\ u_-=0 &\text{on }\Sigma_b
\\(u,\eta )\mid_{t=0}=(u_0,\eta_0),
\end{array}\right.
\end{equation}
where $f$ is given by \eqref{f} and $g_\pm$ is given by \eqref{g_+^i}, \eqref{g_+^3}, \eqref{g_-^i}, \eqref{g_-^3}. One may also refer to \cite{B,NTY} for the the  expressions of $f$ and $g$.

%%%%%%%%%%%%%%%%%%%%%%%%%%%%%%%%%%%%%%%%%%%%%%%
\subsection{Some previous work}
%%%%%%%%%%%%%%%%%%%%%%%%%%%%%%%%%%%%%%%%%%%%%%%

Free boundary problems in fluid mechanics have been studied by many authors in many different contexts. Here we will mention only the work most relevant to our present setting. For a more thorough review of the literature, we refer to the review paper of Shibata and Shimizu \cite{ShSh} and the references therein.

The one-phase problem of describing the motion of an isolated mass of viscous fluid bounded by a free boundary, i.e. a drop in a vacuum, was studied by Solonnikov in a series of papers; for instance, \cite{So} considers the case without surface tension, while \cite{So_2,So_3} concern the problem with surface tension.  Solonnikov's technique for proving local existence in these papers did not rely on energy methods, but rather on H\"older estimates in the case without surface tension and Fourier-Laplace transform methods in the case with surface tension.  A local existence theory based on the energy method was developed by Coutand and Shkoller for the problem with surface tension \cite{CS}.

Next, we recall some works on the two-phase problem of describing the motion of two viscous fluids separated by a closed free interface, i.e. a drop in another fluid medium. For the case of the fluids filling the whole space, Denisova \cite{D} proved local-in-time unique solvability in Sobolev-Slobodetskdii spaces with or without surface tension; Denisova and Solonnikov \cite{DS} proved the local-in-time unique solvability in H\"older spaces with surface tension. For the case of the fluids filling a bounded domain, Tanaka \cite{T1} proved the local solvability of the problem with general data in the Sobolev-Slobodetskdii spaces and Tanaka \cite{T2} proved the global solvability for data close  to equilibrium  with surface tension; Denisova \cite{D2} proved the global solvability in H\"older spaces without surface tension.

The one-phase problem of describing  the motion of a layer of viscous fluid lying above a fixed bottom, i.e. the viscous surface wave problem, has attracted the attention of many mathematicians since the pioneering work of Beale \cite{B1}. For the case without surface tension, Beale \cite{B1} proved the local well-posedness in the Sobolev spaces and Sylvester \cite{S} studied the global well-posedness by using Beale's method. In the case of horizontally periodic domains with a completely flat fixed lower boundary,  Hataya \cite{H} obtained the global existence of small solutions with an algebraic decay rate in time. Recently, Guo and Tice \cite{GT_lwp,GT_inf,GT_per} developed a new two-tier energy method to prove global well-posedness and decay of this problem.  They proved that if the free boundary is horizontally infinite, then the solution decays to equilibrium at an algebraic rate; on the other hand, if the free boundary is horizontally periodic, then the solution decays at an almost exponential rate. For the case with surface tension, Beale \cite{B2} proved global well-posedness of the problem, while Allain \cite{A} obtained a local existence theorem in two dimension using a different method. Bae \cite{B} showed the global solvability in Sobolev spaces via the energy method. Beale and Nishida \cite{BN} showed that the solution obtained in \cite{B2} decays in time with an optimal algebraic decay rate. For the periodic motion, Nishida, Teramoto and Yoshihara \cite{NTY}  showed the global existence of solutions with an exponential decay rate in the case of a domain with a flat fixed lower boundary. Tani \cite{Ta} and Tani and Tanaka \cite{TT} also discussed the solvability of the problem with or without surface tension by using Solonnikov's method.

There are very few results on the two-phase free boundary problem of describing the motion of two layers of viscous fluids, i.e. the viscous surface-internal wave problem and viscous internal wave problem. Hataya \cite{H2} proved an existence result for a periodic free interface problem with surface tension as a perturbation around the plane Couette flow of two fluids; he showed the local  existence of small smooth solution for any physical constants, and the existence of exponentially decaying small solution if the viscosities of the two fluids are sufficiently large and their difference is small. Pr\"uss and Simonett \cite{PS} proved the local well-posedness of a free interface problem with surface tension in which two layers of viscous fluids fill  the whole space and are separated by a horizontal interface. For the same problem in our setting but without horizontally periodic assumption and taking into account the effect of surface tension, Xu and Zhang \cite{XZ} proved  the local solvability for general  data and global solvability for data near the equilibrium state by using Tani and Tanaka's argument. The purpose of this paper is to investigate in the energy spaces the well-posedness and decay of solutions to the viscous surface-internal wave problem under the horizontally periodic assumption, both with and without surface tension and without any constraints on the viscosities.

%%%%%%%%%%%%%%%%%%%%%%%%%%%%%%%%%%%%%%%%%%%%%%%
\section{Main results}
%%%%%%%%%%%%%%%%%%%%%%%%%%%%%%%%%%%%%%%%%%%%%%%

%%%%%%%%%%%%%%%%%%%%%%%%%%%%%%%%%%%%%%%%%%%%%%%
\subsection{Notation}
%%%%%%%%%%%%%%%%%%%%%%%%%%%%%%%%%%%%%%%%%%%%%%%

In this paper, for any given domain, we write $H^k$ for the usual $L^2$-based Sobolev space of order $k\ge0$. We define the piecewise Sobolev spaces $\ddot{H}^k(\Omega), k\ge 0$ by
\begin{equation}
\ddot{H}^k(\Omega) = \{u\chi_{\Omega_\pm}\in H^k(\Omega_\pm)\}\text{ with norm } \|u\|_{k}^2 :=\|u\|_{\ddot{H}^k(\Omega)}^2 := \|u_+\|_{H^k(\Omega_+)}^2 + \|u_-\|_{H^k(\Omega_-)}^2.
\end{equation}
When $k=0$, $\ddot{H}^0(\Omega)={H}^0(\Omega)=L^2(\Omega)$. We let the Sobolev spaces $H^s( \Sigma_\pm)$ for $s\in \mathbb{R}$ be equivalent to $H^s( \mathrm{T}^2)$, with norm $\|\cdot\|_{H^s(\Sigma_\pm)}$; we  write
\begin{equation}
\|\eta\|_{s}^2 := \|\eta\|_{{H}^s(\Sigma)}^2 :=\|\eta_+\|_{H^s(\Sigma_+)}^2 + \|\eta_-\|_{H^s(\Sigma_-)}^2.
\end{equation}
We do not distinguish the notation of norms on the domain or on the boundary. Basically, when we write  $\|\partial_t^ju\|_{k}$ and $\|\partial_t^jp\|_{k}$ we always mean that the space is $\ddot{H}^k(\Omega)$, and when we write $\|\partial_t^j\eta\|_{k}$ we always mean that the space is $ {H}^k(\Sigma)$. When there is potential for confusion, typically when trace estimates are employed, we will clarify as needed. %and since $\|\cdot\|_{H^k}$ and $\|\cdot\|_{\ddot{H}^k}$ are equal we do %not distinguish them and unify them by $\|\cdot\|_{{H}^k}$.
We let the spaces ${}_0H^1(\Omega)$ and ${}^0H^1(\Omega)$ be as defined later in \eqref{0H}. It is easy to see that $u\in {}_0H^1(\Omega)$ if and only if $u\in \ddot{H}^1(\Omega)$ with $\Lbrack u\Rbrack=0$ on $\Sigma_-$ and
$u=0$ on $\Sigma_b$.  We will not need negative index spaces on $\Omega$ except for $({}_0H^1(\Omega))^\ast$.  In our analysis, we will occasionally abuse notation by writing  $\|\cdot\|_{-1}$ for the norm in $({}_0H^1(\Omega))^\ast$.  Here it is not the case that $({}_0H^1(\Omega))^\ast = H^{-1}$ because of boundary conditions;  we employ this abuse of notation in order to have indexed sums of norms include terms like $\|\cdot\|_{4N-2j+1}$ for $j=2N+1$.  Sometimes we use $\|\cdot\|_{L^pX}$ to denote the norm of the space $L^p(0,T;X)$. We will sometimes extend the above abuse of notation by writing $\|\cdot\|_{L^p H^{k}}$ for $k=-1$ in a sum of norms; when we do this we actually mean $\|\cdot\|_{L^p ({}_0H^1(\Omega))^\ast}$.

We will employ the Einstein convention of summing over  repeated indices.  Throughout the paper $C>0$ will denote a generic constant that can depend on the parameters of the problem, $N$, and $\Omega$, but does not depend on the data, etc.  We refer to such constants as ``universal.''  Such constants are allowed to change from line to line. When a constant depends on a quantity $z$ we will write $C = C(z)$ to indicate this.  We will employ the notation $a \lesssim b$ to mean that $a \le C b$ for a universal constant $C>0$. To indicate some constants in some places so that they can be referred to later, we will denote them in particular by $C_1,C_2$, etc.

%%%%%%%%%%%%%%%%%%%%%%%%%%%%%%%%%%%%%%%%%%%%%%%
\subsection{Main theorem without surface tension}
%%%%%%%%%%%%%%%%%%%%%%%%%%%%%%%%%%%%%%%%%%%%%%%

We will show the well-posedness and decay of the system \eqref{nosurface} by using the mathematical framework of a two-tier energy method, as developed by Guo and Tice \cite{GT_lwp,GT_inf,GT_per}. Since the problem is considered in a domain with boundary, we should impose some compatibility conditions for the initial data $(u_0,\eta_0)$. We will work in a high-regularity context, essentially with regularity up to $2N$ time derivatives for $N \ge 3$ an integer. This requires us to use $u_0$ and $\eta_0$ to construct the initial data $\partial_t^j  u(0)$ and $\partial_t^j \eta(0)$ for $j=1,\dotsc,2N$ and $\partial_t^j  p(0)$ for $j = 0,\dotsc, 2N-1$. These other data must then satisfy various conditions (essentially what one gets by applying $\partial_t^j $ to \eqref{nosurface} and then setting $t=0$), which in turn require $u_0$ and $\eta_0$ to satisfy the $(2N)^{th}$ compatibility conditions which will be described in \eqref{2Ncompatibility}.

To state our result, we define the energies and dissipations. These definitions rely on the linear energy identity of the homogeneous form of \eqref{nosurface2}:
\begin{equation}\label{lin_en_ev}
\frac{d}{dt}\left(\frac{1}{2}\int_\Omega \rho |u|^2 +\frac{1}{2}\int_{\Sigma_+}
\rho_+g|\eta_+|^2+\frac{1}{2}\int_{\Sigma_-} -\rj g|\eta_-|^2\right)+\frac{1}{2}\int_{\Omega}\mu|\mathbb{D}
u|^2=0.
\end{equation}
According to this energy identity and the structure of the equations, as in \cite{GT_per}, we define the energies
and dissipations as follows. For any integer $N \ge 3$ we write the high-order energy as
\begin{equation}
 \mathcal{E}_{2N} =
 \sum_{j=0}^{2N} \left( \|\partial_t^j u\|_{4N-2j}^2 + \|\partial_t^j \eta\|_{4N-2j}^2 \right)
 + \sum_{j=0}^{2N-1} \|\partial_t^j p\|_{4N-2j-1}^2
\end{equation}
and the corresponding dissipation as
\begin{multline}
  \mathcal{D}_{2N} = \sum_{j=0}^{2N} \|\partial_t^j u\|_{4N-2j+1}^2 + \sum_{j=0}^{2N-1} \|\partial_t^j p\|_{4N-2j}^2 \\
+  \| \eta\|_{4N-1/2}^2 + \|\partial_t \eta\|_{4N-1/2}^2 + \sum_{j=2}^{2N+1} \|\partial_t^j \eta\|_{4N-2j+5/2}^2.
\end{multline}
We define the low-order energy as
\begin{equation}
 \mathcal{E}_{N+2} = \sum_{j=0}^{N+2} \left( \|\partial_t^j u\|_{2(N+2)-2j}^2
 + \|\partial_t^j \eta\|_{2(N+2)-2j} ^2\right)
 + \sum_{j=0}^{N+1} \|\partial_t^j p\|_{2(N+2)-2j-1}^2.
\end{equation}
We also  define
\begin{equation}
\mathcal{F}_{2N}=   \|\eta\|_{4N+1/2}^2.\label{F_2N}.
\end{equation}
Finally, we define the total energy functional as
\begin{equation}
\mathcal{G}_{2N}(t) = \sup_{0 \le r \le t} \mathcal{E}_{2N}(r) +
\int_0^t \mathcal{D}_{2N}(r) dr + \sup_{0 \le r \le t} (1+r)^{4N-8}
\mathcal{E}_{N+2}(r) + \sup_{0 \le r \le t} \frac{
\mathcal{F}_{2N}(r)}{(1+r)}.
\end{equation}

The main results for the case without surface tension are stated as follows.

\begin{theorem}\label{th0}
Let $N\ge 3$ be an integer. Assume that $u_0\in {}_0H^1(\Omega) \cap \ddot{H}^{4N}(\Omega)$ and $\eta_0\in {H}^{4N+1/2}(\Sigma)$ satisfy the $(2N)^{th}$ order compatibility conditions \eqref{2Ncompatibility}. There exist  $\delta_0,T_0 >0$ so that if
\begin{equation}
0 < T \le  T_0 \min\left\{1, \frac{1}{\|\eta_0\|_{4N+1/2}} \right\}
\end{equation}
and $\|u_0\|_{4N}^2+\|\eta_0\|_{4N}^2\le \delta_0$, then there exists a unique solution $(u,p,\eta)$ to the problem \eqref{nosurface} on $[0,T]$. The solution satisfies the estimates
\begin{multline} \label{localestimate}
  \sup_{0\le t\le T}\mathcal{E}_{2N}(t)
  +\int_0^T\mathcal{D}_{2N}(t)dt+\int_0^T\left(\|\rho\partial_t^{2N+1}u(t)\|_{-1}^2+\|\partial_t^{2N}p(t)\|_0^2\right)dt\\
\le C_1( \|u_0\|_{4N}^2 + \|\eta_0\|_{4N}^2 + T \|\eta_0\|_{4N+1/2}^2)
\end{multline}
and
\begin{equation}\label{localestimate2}
\sup_{0\le t\le T}\mathcal{F}_{2N}(t) \le C_1(\|u_0\|_{4N}^2 + (1+T)\|\eta_0\|_{4N+1/2}^2).
\end{equation}

If additionally $\rj <0$ and $\eta_0$ satisfies the zero-average condition \eqref{zero0}, then  there exists a $ \kappa=\kappa(N)>0 $ so that if $ \mathcal{E}_{2N}(0) + \mathcal{F}_{2N}(0) < \kappa$,  then there exists a unique solution $(u,p,\eta)$ to the problem \eqref{nosurface} on $[0,\infty)$ that satisfies the estimate
\begin{equation}\label{Gbound}
 \mathcal{G}_{2N}(\infty) \le C_2\left( \mathcal{E}_{2N}(0) + \mathcal{F}_{2N}(0) \right) <
 C_2  \kappa.
\end{equation}
\end{theorem}

\begin{Remark}
Note that the bound \eqref{Gbound}  on $\mathcal{G}_{2N}$ implies the decay estimate
\begin{equation}\label{decayn+2}
\mathcal{E}_{N+2}(t) \le C_2 \kappa (1+t)^{-4N+8}.
\end{equation}
Since $N$ may be taken to be arbitrarily large, this decay result can be regarded as an ``almost exponential'' decay rate.
\end{Remark}

\begin{Remark}
The requirement that $\rj <0$ is essential for our global result. Indeed, if $\rj >0$, i.e. a heavier fluid lying above a lighter one, then the well-known (linear) Rayleigh-Taylor instability occurs (see \cite{C,R,Tay}). In a forthcoming paper, we will show the nonlinear Rayleigh-Taylor instability for the system \eqref{nosurface}. It is also interesting that our energy method fails  if  $\rj =0$. In this case, we lose the control of the $\eta_-$ part of the energy. In fact, in this  case there is no effect of the gravitational force on $\Sigma_-$ and $\eta_-=0$  is not the only equilibrium solution.
\end{Remark}

\begin{Remark}
The surface $\eta$ is sufficiently small to guarantee that the mapping $\Theta(\cdot,t)$, defined in \eqref{cotr}, is a diffeomorphism for each $t\ge 0$ when restricted to $\Omega_\pm$.  As such, we may change coordinates to $y \in \Omega(t)$ to produce a global-in-time, decaying solution to \eqref{form1} in the case without surface tension.
\end{Remark}

We will prove Theorem \ref{th0} in Section \ref{0surface}.  The proof is based on the mathematical framework of a two-tier energy method recently developed by Guo and Tice \cite{GT_inf,GT_per} for the one-phase problem.  The necessity of the two-tier framework in an analysis of the decay properties of solutions to \eqref{nosurface} is explained in detail in the introductions of \cite{GT_inf,GT_per}.  For the sake of completeness, here we will briefly sketch the two-tier method by referring to  the energy evolution equation \eqref{lin_en_ev}.  In this equation, the energy (say $\mathcal{E}$ denotes the terms in the time derivative) involves both $u$ and $\eta$, while the dissipation (say $\mathcal{D}$ for the $\mathbb{D} u$ term) only involves $u$.  We might hope to use the structure of the equations \eqref{nosurface} to prove that  $C \mathcal{E} \le \mathcal{D}$ for some constant $C>0$; this could then be used with \eqref{lin_en_ev} to deduce the differential inequality $d \mathcal{E}/dt + C \mathcal{E} \le 0$, which would then imply the exponential decay $\mathcal{E}(t) \le \mathcal{E}(0) e^{-Ct}$.   It turns out that the equations \eqref{nosurface} do not allow us to prove  that $C \mathcal{E} \le \mathcal{D}$, and so this method cannot work.  Moreover, we run into the same problem with any higher regularity version of $\mathcal{E}$ and $\mathcal{D}$; in particular, we cannot show that $C \mathcal{E}_{2N} \le \mathcal{D}_{2N}$ for any $N$.

Instead, we pursue an alternate strategy that employs two tiers of energy: the high regularity $\mathcal{E}_{2N}$ and $\mathcal{D}_{2N}$, and the lower regularity $\mathcal{E}_{N+2}$ and $\mathcal{D}_{N+2}$.  Instead of the differential inequality mentioned above, we aim to establish the pair of inequalities
\begin{equation}\label{tt_e_1}
 \mathcal{E}_{2N}(t) +  \int_0^t \mathcal{D}_{2N}(s)ds \lesssim \mathcal{E}_{2N}(0) \text{ and }   \frac{d}{dt} \mathcal{E}_{N+2} + C \mathcal{D}_{N+2} \le 0
\end{equation}
and the  interpolation inequality
\begin{equation}\label{tt_e_2}
 \mathcal{E}_{N+2} \lesssim (\mathcal{E}_{2N})^{\theta/(1+\theta)} (\mathcal{D}_{N+2})^{1/(1+\theta)} \text{ for some } \theta >0.
\end{equation}
With these three estimates in hand we may deduce that
\begin{equation}\label{tt_e_3}
 \frac{d}{dt} \mathcal{E}_{N+2} + C \mathcal{E}_{N+2}^{1+\theta} \le 0
\text{ and that }
\mathcal{E}_{N+2}(t) \le \frac{\mathcal{E}_{2N}(0)}{(1+t)^{1/\theta}},
\end{equation}
which establishes that solutions decay in time at an algebraic rate depending on the interpolation power $\theta$ appearing in \eqref{tt_e_2}.  The two-tier energy method allows us to complete this program by coupling the boundedness of high-order norms to the decay of the low-order norms.  Indeed, we will first use the decay of $\mathcal{E}_{N+2}$ to establish the first inequality in \eqref{tt_e_1}, and then we use the boundedness of $\mathcal{E}_{2N}$ to establish the second inequality in \eqref{tt_e_1}.  Finally, we show that \eqref{tt_e_2} holds with $\theta$ depending on $N$ so that the algebraic decay rate in \eqref{tt_e_3} increases with $N$.

In order to close the two-tier energy method, we must also have control of $\mathcal{F}_{2N}$.  Unfortunately, the only way to estimate it is through the kinematic transport equation for $\eta$, and this yields estimates whose right-hand sides can a priori grow exponentially in time (see \eqref{l0}) unless $u$ decays rapidly.  Even if $u$ does decay rapidly (a fixed algebraic rate is sufficient), the estimates from the transport equation can still grow linearly in time (see \eqref{Fg}).  This growth is potentially disastrous in closing the high-order, global-in-time estimates.  To manage the growth, we must identify a special decaying term that always appears in products with the highest derivatives of $\eta$.  If the special term decays quickly enough, then we can hope to balance the growth and close the high-order estimates.  Due to the growth of $\mathcal{F}_{2N}$, we believe that it is not possible to construct global-in-time solutions without also deriving a decay result.

In addition to the difficulties implicit in the two-tier energy method, there are several new  difficulties that arise when we try to adapt the method to our two-phase problem.  We shall summarize these now.

First, in order to borrow many of the estimates from \cite{GT_per}, we must have that our problem \eqref{nosurface} has the same essential structure as the problem studied in \cite{GT_per}.  This requires using the special flattening coordinate transformation \eqref{cotr}.  For this to be useful, we must have that the transformation is sufficiently regular across the interface, which then requires us to develop an appropriate specialization of the Poisson extension of $\eta_-$.

Second, to prove the local well-posedness, we use the iteration scheme as in \cite{GT_lwp} after first solving the linearized problems, namely the time-dependent $\mathcal{A}$--Stokes equations and transport equations. The details of this procedure would be the same as in \cite{GT_lwp} if we already possessed the elliptic regularity theory for the two-phase Stokes and Poisson problem. However, up to our best knowledge, there is no published work on this problem, and as such we must develop it here. The essential point is to prove the elliptic regularity theory for the two-phase Stokes problem \eqref{cS}, i.e. Theorem \ref{cStheorem}. The key observation is that the interface $\Sigma_-$ between the domains $\Omega_\pm$ is flat, which means that the definition of Sobolev spaces on $\Sigma_-$ only involves  horizontal derivatives. Hence we can first control the regularity of $u$ on $\Sigma_-$ by localizing near $\Sigma_-$ (since the lower boundary is not flat) and deriving estimates for the horizontal derivatives of solutions.   Then the crucial idea is to apply the well-known elliptic regularity theory for the classical one-phase Stokes problems to the domain $\Omega_\pm$ respectively, which then leads to the elliptic regularity theory for the two-phase Stokes problem \eqref{cS}.

Third, to prove that the local solution is actually global we will derive a priori estimates as in \cite{GT_per}. It turns out that we can again follow the scheme used in \cite{GT_per} if we can show that the full energy $\mathcal{E}_{n}$ and dissipation $\mathcal{D}_{n}$ with $n=2N$ or $n=N+2$  are comparable to their ``horizontal'' counterparts $\bar{\mathcal{E}}_{n}$ and $\bar{\mathcal{D}}_{n}$ (see \eqref{hp}). It is easy to show that $\mathcal{E}_{n}$  is comparable to $\bar{\mathcal{E}}_{n}$ by applying the two-phase Stokes elliptic theory to the equations \eqref{nosurface2} in perturbation form.  However, due to the lack of  $\eta$ terms in $\bar{\mathcal{D}}_{n}$, we can not compare the dissipations by using this technique. To overcome this difficulty, Guo and Tice \cite{GT_inf,GT_per} instead directly use the structure of the equations as well as the equation for the vorticity (from which the pressure and $\eta$ can effectively be eliminated) to derive various estimates for $u$ and $\nabla p$ without reference to $\eta$. Then a bootstrapping procedure is employed to obtain estimates for $\eta$ and $p$ (not just its gradient).  However, if we tried to use such a procedure for our two-phase problem, the difference of the two viscosities $\mu_\pm$ would prevent us from obtaining a ``good'' (properly eliminating $\eta$) boundary condition for the vorticity, and then we would be unable to complete the argument above.  To get around this difficulty, we again employ the idea we use in proving Theorem \ref{cStheorem}.  More precisely, since
$\bar{\mathcal{D}}_{n}$ controls horizontal derivatives, we can use it to gain the regularity of $u$ on $\Sigma_\pm$.  The idea is then to apply the well-known elliptic regularity theory for the Dirichlet problems of one-phase Stokes equations to the domain $\Omega_\pm$ respectively in order to deduce the desired estimates of $u$.  Then we get the
desired estimates of $p$ and $\eta$  by using the equations directly. We should note that this procedure could also  be applied to the one-phase problem, and thus it provides an alternative to the dissipation comparison analysis used in \cite{GT_inf,GT_per}.

%%%%%%%%%%%%%%%%%%%%%%%%%%%%%%%%%%%%%%%%%%%%%%%
\subsection{Main theorem with surface tension}
%%%%%%%%%%%%%%%%%%%%%%%%%%%%%%%%%%%%%%%%%%%%%%%

Our use of the energy method to prove the well-posedness and decay of the system \eqref{surface} is motivated by the works of Bae \cite{B}, Coutand and Shkoller \cite{CS}, and Jin and Padula \cite{JP}. The presence of surface tension gives rise to a gain of regularity for $\eta$, which (roughly speaking) causes the highest norm of $\eta$, namely \eqref{F_2N}, to be included in the dissipation.  This suggest that with surface tension  we do not need to balance the decay of low-order energy and the growth of the highest norm of $\eta$.  Indeed, we expect it to not grow in time.  This in turn also allows us to avoid establishing the well-posedness in the high regularity context we used in the case without surface tension.  Rather, we use the $H^2$ framework as in \cite{B,CS,JP}, say, $u_0 \in \ddot{H}^2(\Omega)$. Given the initial data $(u_0,\eta_0)$, we only need to construct the initial data $\partial_t u(0)$, $\partial_t \eta(0)$ and $p(0)$, and then we only require a compatibility condition for $(u_0,\eta_0)$. Letting $g(0)=g(u_0,\eta_0)$ and $\Pi_\ast$ be the horizontal projection defined by $\Pi_\ast v = (v_1,v_2,0)$, we may state the compatibility condition as
\begin{equation}\label{compatibility1ge}
\Pi_\ast\left(g_+(0)+\mu_+ \mathbb{D}u_{0,+} e_3\right)=0,\ \Pi_\ast\left(g_-(0)-\llbracket \mu\mathbb{D}u_0\rrbracket  e_3\right) = 0.
\end{equation}

To state our result, we define the energies and dissipations, the definitions of which rely on the linear energy identity associated to the homogeneous version of \eqref{surface}:
\begin{multline}
\frac{d}{dt}\left(\frac{1}{2}\int_\Omega \rho |u|^2 +\frac{1}{2}\int_{\Sigma_+}
\rho_+g|\eta_+|^2+\sigma_+|\nabla_\ast\eta_+|^2+\frac{1}{2}\int_{\Sigma_-}
-\rj g|\eta_-|^2+\sigma_-|\nabla_\ast\eta_-|^2\right) \\
+\frac{1}{2}\int_{\Omega}\mu|\mathbb{D} u|^2=0.
\end{multline}
According to this energy identity and the structure of the equations, we define the instantaneous energy  as
 \begin{equation}\label{E}
\mathcal{E}(t)=\|u\|_{ 2 }^2+\|\partial_tu\|_{  0 }^2 + \|p\|_{ 1 }^2  +\|\eta\|_{  3 }^2
  + \|\partial_t\eta\|_{  3/2 }^2 + \|\partial_t^2 \eta\|_{  -1/2 }^2
\end{equation}
and the dissipation rate as
\begin{multline} \label{D}
\mathcal{D}(t) = \|u\|_{3 }^2+\|\partial_tu\|_{ 1 }^2+\|\nabla \partial_t u_+\|_{H^{-1/2}(\Sigma_+) }^2 +
\|\llbracket \mu \nabla \partial_t u\rrbracket \|_{ H^{-1/2}(\Sigma_-) }^2
+ \|p\|_{2}^2 + \|\partial_tp\|_{0}^2 \\
+ \|\partial_t p_+\|_{H^{-1/2}(\Sigma_+) }^2 + \|\llbracket \partial_t p \rrbracket \|_{H^{-1/2}(\Sigma_-) }^2 + \|\eta\|_{7/2 }^2 + \|\partial_t\eta\|_{ 5/2}^2 + \|\partial_t^2\eta\|_{1/2 }^2 .
\end{multline}

We define the critical surface tension value according to
\begin{equation}\label{crit_st_def}
 \sigma_c := \rj g \max\{L_1^2,L_2^2\}.
\end{equation}
The importance of this value will be seen in our main result for the case with  surface tension, which we state now.

\begin{theorem}\label{surth}
Assume that $\sigma_\pm >0$ and that either $\rj \le 0$ or $\rj > 0$ and  $\sigma_- > \sigma_c$, where $\sigma_c$ is the critical surface tension defined by \eqref{crit_st_def}.  Let $u_0\in  {}_0H_\sigma^1(\Omega)\cap \ddot{H}^{2}(\Omega)$ and $\eta_0\in  {H}^{3}(\Sigma)$ satisfy the  compatibility condition \eqref{compatibility1ge} and the zero-average condition \eqref{zero0}.   There exists a $\delta_0>0$ so that if $\|u_0\|_2^2 + \|\eta_0\|_3^2 \le \delta_0$, then there exists a unique strong solution $(u,p,\eta)$ to the problem \eqref{surface} on $[0,\infty)$ satisfying the estimate
\begin{equation} \label{gles}
  \sup_{t\ge0} \mathcal{E}(t)+ \int_0^\infty \mathcal{D}(t)dt \lesssim  \mathcal{E}(0).
\end{equation}
Moreover, there exists $\lambda>0$ so that
\begin{equation}\label{decayes}
  \mathcal{E}(t) \lesssim \mathcal{E}(0)\exp(-\lambda t)  \text{ for all } t\ge 0.
\end{equation}
\end{theorem}

\begin{Remark}
Unlike in the case without surface tension, in the case with surface tension we have that the dissipation is stronger than the energy, i.e. $\mathcal{E} \le \mathcal{D}$.  This leads to the improved exponential decay rate for solutions.  The boundary terms appearing in $\mathcal{D}$ are crucial in the problem with surface tension.  Without them, we would be unable to close the nonlinear estimates.
\end{Remark}

%\begin{Remark}
%Theorem \ref{surth} handles the case in which $\rj =0$.  In this case the effect of surface tension at the %internal interface guarantees that $\eta_- = 0$  is the only equilibrium steady-state solution.  Comparing %with the result of Theorem \ref{th0}, we see that either a purely gravitational force ($\rj <0$, %$\sigma_-%=0$) or surface tension force lead to the global well-posedness of the viscous surface-internal wave %problem.
%\end{Remark}

\begin{Remark}
We believe that our method could be adapted to prove a version of Theorem \ref{surth} for the one-phase problem with surface tension.  This would give a proof based on the energy method of the result of Nishida, Teramoto, and Yoshihara \cite{NTY}.  The result would be somewhat more general than that of \cite{NTY} since it would allow for non-flat fixed lower boundaries, $\Sigma_b$.
\end{Remark}

\begin{Remark}
In the case $\rj >0$, Theorem \ref{surth} shows that sufficiently large surface tension at the internal interface, namely $\sigma_- > \sigma_c$, stabilizes the Rayleigh-Taylor instability in the nonlinear setting.   If $\rj  >0$ but $0< \sigma_- \le \sigma_c$, then our proof can still produce a local-in-time solution to \eqref{surface}.

The number $\sigma_c$ is critical in the sense that a linear growing mode solution can be constructed when $\sigma_- <\sigma_c$ (see \cite{C,GT2}).  In a forthcoming paper, we will show the nonlinear Rayleigh-Taylor instability for the system \eqref{surface} when $\sigma_- <\sigma_c$.
\end{Remark}

\begin{Remark}
The surface $\eta$ is sufficiently small to guarantee that the mapping $\Theta(\cdot,t)$, defined in \eqref{cotr}, is a diffeomorphism for each $t\ge 0$ when restricted to $\Omega_\pm$.  As such, we may change coordinates to $y \in \Omega(t)$ to produce a global-in-time, decaying solution to \eqref{form1} in the case with surface tension.
\end{Remark}

We prove Theorem \ref{surth} in Section \ref{surfacetension}.  We now present a sketch of the main techniques we use in the proof.  The extra control of $\eta$ in the dissipation that is provided by the surface tension allows us to prove the theorem in a standard perturbative manner without a two-tier method.  That is, we combine analysis of the linearized problem with estimates of the nonlinear terms to produce a solution via a contraction mapping argument.

The majority of our effort is devoted to the linearized problem \eqref{surfaceLP}.  We make use of the condition $\partial_t\eta=u_3$ on $\Sigma$ to transform the problem for $(u,p,\eta)$ to the equivalent linear problem \eqref{sLP0} for only $(u,p)$. This allows us to use a Galerkin method to prove the existence of a unique weak solution.  Because of the surface tension terms on the interfaces we cannot directly apply the two-phase elliptic regularity theory to improve regularity in the dissipation as we did in our analysis of the energy in the problem without surface tension.  Instead, we use a variant of the trick we used in the proof of Theorem \ref{cStheorem}.  More precisely, we localize away from $\Sigma_b$ and then use horizontal difference quotients in the definition of weak solution to gain control of horizontal derivatives on the interfaces $\Sigma_\pm$.  Then we use classical elliptic regularity theory for the one-phase Stokes problem with Dirichlet boundary conditions to complete the estimate for $u$ and to subsequently obtain estimates for $p$ and $\eta$ by using the structure of the equations.  On the other hand, the energy provides control of $\eta$, which allows us to estimate the surface terms on the interfaces.  We then directly apply Theorem \ref{cStheorem} to obtain estimates for $u$ and $p$ in terms of the energy.

With the linear theory in hand, we turn to the details of the contraction mapping argument.  First we develop estimates for the nonlinear terms.  Then we define a function space $\mathfrak{X}$, determined by the energy, dissipation, and initial data, in which the contraction argument may be framed.  We show that the linear theory allows us to define a solution operator $\mathcal{L}$ that is contractive in $\mathfrak{X}$, which then yields a unique fixed point that is the strong solution of the problem  \eqref{surface}.

%%%%%%%%%%%%%%%%%%%%%%%%%%%%%%%%%%%%%%%%%%%%%%%
\subsection{Plan of paper}
%%%%%%%%%%%%%%%%%%%%%%%%%%%%%%%%%%%%%%%%%%%%%%%

Section \ref{section_elliptic} contains the two-phase elliptic estimates that will be used to gain regularity in the time-dependent problems.  We study both the constant coefficient problems and their ``geometric,'' $\mathcal{A}$-dependent analogues.  In Section \ref{0surface} we present our analysis of the problem without surface tension.  First we study the relevant linearized problem, and then we consider the issue of local well-posedness.  After that we develop the a priori estimates needed to extend to global well-posedness and to prove the decay of the solutions.  Section \ref{surfacetension} concerns the problem with surface tension.  We develop the well-posedness theory for the linear problem.  Then we prove the global  well-posedness of the nonlinear problem through the use of the contraction mapping principle in an appropriate function space.  The decay of these solutions follows from a differential inequality involving the  energy and dissipation.  At the end of the paper we present Appendix  \ref{section_appendix}, where we record various analytic tools that are useful throughout the paper.

%%%%%%%%%%%%%%%%%%%%%%%%%%%%%%%%%%%%%%%%%%%%%%%
\section{Two-phase elliptic problems}\label{section_elliptic}
%%%%%%%%%%%%%%%%%%%%%%%%%%%%%%%%%%%%%%%%%%%%%%%

%%%%%%%%%%%%%%%%%%%%%%%%%%%%%%%%%%%%%%%%%%%%%%%2
\subsection{Two-phase Stokes problem}\label{stokesconst}
%%%%%%%%%%%%%%%%%%%%%%%%%%%%%%%%%%%%%%%%%%%%%%%

In Section \ref{stokesconst}, the domain under consideration can be a bit more general than we need for our analysis of \eqref{nosurface} and \eqref{surface}. Indeed, we assume that $G=G_+\cup G_-$ is a double-layer of horizontally periodic slabs with the upper boundary $\Gamma_+$,  the internal interface $\Gamma_-$ and the lower boundary  $\Gamma_b$. We only assume that $\Gamma_-$ is flat and that   $\Gamma_+,\ \Gamma_b$ are sufficiently regular but may not be flat.  We write $\Gamma=\Gamma_+ \cup \Gamma_-$.

Consider the stationary Stokes problem
\begin{equation}\label{cS}\left\{\begin{array}{ll}
-\mu \Delta u +\nabla {p} =F^1 \ &\text{ in }G
\\ \diverge{u} =F^2   \ &\text{ in}\ G\\
(p_+I-\mu_+\mathbb{D}(u_+))\nu=F^3_+  \ &\text{ on }\Gamma_+
\\
 \Lbrack u\Rbrack=0,\quad \Lbrack(pI-\mu\mathbb{D}(u))e_3\Rbrack=-F^3_-  \ &\text{ on }\Gamma_-
\\  u_-=0 &\text{ on }\Gamma_b.\end{array}\right.
\end{equation}
Here, as in the rest of the paper, $\nu$ denotes the outward pointing unit normal on the corresponding boundary. Our purpose is to show the existence and regularity of the solutions to \eqref{cS} and the result is stated as the following theorem.

\begin{theorem}\label{cStheorem}
Let $r\ge 2$.  If $F^1\in \ddot{H}^{r-2}(G),\ F^2\in \ddot{H}^{r-1}(G),\ F^3\in  {H}^{r-3/2}(\Gamma)$, then the problem \eqref{cS} admits a unique strong solution $(u,p)\in ({}_0H^1(G)\cap \ddot{H}^r(G))\times \ddot{H}^{r-1}(G)$.
Moreover,
\begin{equation} \label{cSresult}
\|u\|_{r}+\| p\|_{r-1} \lesssim \|F^1\|_{r-2}+\|F^2\|_{r-1}+\|F^3\|_{r-3/2}.
\end{equation}
\end{theorem}

The rest of this subsection is devoted to proving Theorem \ref{cStheorem}. First, we adjust the divergence of $u$ to reduce the problem to be a divergence-free problem. Then we obtain the existence of weak solutions for this problem and then improve the regularity of the solutions to show that the weak solution is indeed strong. Finally, we combine the previous two steps to deduce Theorem \ref{cStheorem}. We should point out here that the flatness of the internal interface $\Gamma_-$ is crucial to our analysis.

The adjustment of the divergence of $u$ is achieved by the solvability of the problem $\diverge{v}=p$. It is well-known that for $p\in L^2(G)$, there exists a $v\in {}_0H^1(G)$ so that $\diverge{v}=p$ in $G$ and $\|v\|_{1} \lesssim \|p\|_{0}$ (see for instance \cite{B1}). However, we are interested in solutions to \eqref{cS} with higher but piecewise regularity, so we need an improved version of this result.
\begin{lemma}\label{div}
If $p\in \ddot{H}^{r-1}(G),\ r\ge 1$, then there exists $v\in {}_0H^1(G)\cap \ddot{H}^{r}(G)$ so that $\diverge{v}=p$ in $G$ and
\begin{equation}\label{dives}
\|v\|_{r}\lesssim\|p\|_{r-1}.
\end{equation}
\end{lemma}
\begin{proof}
We construct $v$ as follows.  First, we let $v^{(1)}=\nabla\phi$ with $\phi$ is a solution to the problem
\begin{equation}\label{diffraction}
\left\{\begin{array}{lll}
-\Delta \phi =p    \ &\text{ in}\ G,
 \\\phi_+=0&\text{ on }\Gamma_+,
 \\\Lbrack \phi\Rbrack=0,\quad \Lbrack \partial_3\phi \Rbrack=0  \ &\text{ on }\Gamma_-,
 \\\nabla\phi_-\cdot\nu=0 &\text{ on }\Gamma_b.
\end{array}\right.
\end{equation}
The existence of the unique weak solution $\phi\in H^1(G)$ to \eqref{diffraction} is standard.  The fact that $\phi$ is a strong solution with  $\phi \in \cap \ddot{H}^{r+1}(\Omega)$ and
\begin{equation}\label{1es}
\|v^{(1)}\|_{r} =\|\nabla\phi\|_{r} \lesssim \|p\|_{r-1}
\end{equation}
is not standard, but may be deduced from a modification of the argument we will use in Lemma \ref{cS0l2} for the Stokes problem. For the sake of brevity we omit the details in the present case. Note that $\diverge v^{(1)}=p$ in $G$ and $v_-^{(1)}\cdot\nu=0$ on $\Gamma_b$.

Next, we find $v^{(2)}$ so that $\diverge v^{(2)}=0$ in $G$ and $v^{(2)}=-v^{(1)}$ on $\Gamma_b$. Such $v^{(2)}$ can be found since $v_-^{(1)}\cdot\nu=0$ on $\Gamma_b$. Indeed, we   define $v^{(2)}=\nabla\times  d $, where  $d$ is the vector found in \cite[pp. 24--26]{L}. Moreover, $d\in H^{r+1}(G)$ and $\|d\|_{r+1}\lesssim \|v^{(1)}\|_{H^{r-1/2}(\Gamma_b)}$. This implies $v^{(2)}\in H^r(G)$ and
 \begin{equation}\label{2es}
\|v ^{(2)}\|_{r}=\|\nabla\times d\|_{r} \lesssim \|v^{(1)}\|_{H^{r-1/2}(\Gamma_b)}.
\end{equation}

Consequently, setting $v =v^{(1)}+v^{(2)}$, we have $v\in {}_0H^1(G) \cap \ddot{H}^{r}(G)$ and $\diverge{v}=p$ in $G$. Moreover, by \eqref{1es}--\eqref{2es} and the trace theorem  we get \eqref{dives}.
\end{proof}

We can now  adjust the divergence of $u$ by using Lemma \ref{div}. Since $F^2\in \ddot{H}^{r-1}(G),\ r\ge 2$, by Lemma \ref{div} there is $\bar{u}\in {}_0H^1(G)\cap \ddot{H}^{r}(G)$ so that $\diverge{\bar{u}}=F^2$ in $G$ and
\begin{equation} \label{ubares}
\|\bar{u}\|_{r}\lesssim\|F^2\|_{r-1}.
\end{equation}
Hence,  defining $w=u-\bar{u}$, we may switch the problem \eqref{cS} to  the following problem for $w$
\begin{equation}\label{cS0}
\left\{\begin{array}{ll}
-\mu \Delta w +\nabla {p} =G^1 \ &\text{ in }G
\\ \diverge{w} =0   \ &\text{ in}\ G\\
(p_+I-\mu_+\mathbb{D}(w_+))\nu=G^3_+  \ &\text{ on }\Gamma_+
\\
 \Lbrack w\Rbrack=0,\quad \Lbrack(pI-\mu\mathbb{D}(w))e_3\Rbrack=-G^3_-  \ &\text{ on }\Gamma_-
\\  w_-=0 &\text{ on }\Gamma_b,
\end{array}\right.
\end{equation}
where $G^1,G^3$ are given by
\begin{equation}
G^1=F^1+\mu\Delta \bar{u}, \  G^3_+=F_+^3+  \mu_+\mathbb{
D}(\bar{u}_+)e_3 \text{ and  } G^3_-=F^3_--\Lbrack
\mu\mathbb{D}(\bar{u})e_3\Rbrack.
\end{equation}
By \eqref{ubares} and the usual trace theory, we have
\begin{equation} \label{Ges}
\|G^1\|_{r-2} + \|G^3\|_{r-3/2} \lesssim \|F^1\|_{r-2} + \|F^2\|_{r-1} + \|F^3\|_{r-3/2}.
\end{equation}

Next, we  turn to the strong solvability of the problem \eqref{cS0}. Before doing this, we shall first obtain a weak
solution.  Let ${}_0H_\sigma^1$ be the space defined in \eqref{0H}.  Supposing that $G^1\in  ({}_0H^1(G))^*$ and $G^3 \in H^{-1/2}(\Gamma)$, we shall say $(w,p) \in {}_0H_\sigma^1(G) \times L^2(G)$ is a weak solution to \eqref{cS0} provided that
\begin{equation}
\label{cS0d2}(\frac{\mu}{2}\mathbb{D}w, \mathbb{D}v) - (p,\diverge{v}) = \langle G^1, v\rangle_\ast -
\langle G^3, v\rangle_{-1/2} \text{ for all } v\in  {}_0H^1(G).
\end{equation}
Here $\langle \cdot, \cdot\rangle_\ast$ is the dual pairing between $({}_0H^1(G))^*$ and ${}_0H^1(G)$, $\langle \cdot, \cdot\rangle_{-1/2}$ is the dual pairing between $H^{-1/2}(\Gamma)$ and $ H^{1/2}(\Gamma)$ and $(\cdot, \cdot)$ is the $L^2$ inner product on $\Omega$.

\begin{lemma}\label{cS0l1}
If $G^1\in  ({}_0H^1(G))^*$ and $G^3\in H^{-1/2}(\Gamma)$, then there exists a unique weak solution $(w,p)\in {}_0H_\sigma^1(G)\times L^2(G)$ to \eqref{cS0}. Moreover,
\begin{equation}\label{cS0l1es}
\|w\|_{1} + \|p\|_{0} \lesssim \|G^1\|_{({}_0H^1(G))^*} + \|G^3\|_{-1/2}.
\end{equation}
\end{lemma}
\begin{proof}
The vanishing of $\diverge{w}$ allows us to restrict the test function to $v\in {}_0H_\sigma^1(G)$ so that the pressure term in \eqref{cS0d2} vanishes. Hence, we can first obtain a pressureless weak solution $w$ to \eqref{cS0} in the sense
\begin{equation} \label{cS0d1}
(\frac{\mu}{2}\mathbb{D}w, \mathbb{D}v)=\langle G^1, v\rangle_\ast-\langle G^3, v\rangle_{-1/2} \text{ for all}\ v\in  {}_0H_\sigma^1(G).
\end{equation}
The trace theorem implies that the right-hand side of \eqref{cS0d1} can be regarded as a linear functional on ${}_0H_\sigma^1(G)$, and Korn's inequality, Lemma \ref{korn}, shows that the form on the left-hand side   is bilinear, continuous and coercive on ${}_0H_\sigma^1(G)$. Hence the Lax-Milgram theorem provides a unique $w\in {}_0H_\sigma^1(\Omega)$ satisfying \eqref{cS0d1} and $\|w\|_1\lesssim \|G^1\|_{({}_0H^1(G))^*} + \|G^3\|_{-1/2}$.   With $w$ in hand, we then introduce the pressure $p$ as a Lagrange multiplier. For this, we define $\Lambda \in ({}_0H^1(G))^*$ so that $\Lambda (v)$ equals the difference between the left  and right hand sides of \eqref{cS0d1}. Then $\Lambda=0$ on ${}_0H_\sigma^1(G)$, and hence according to Proposition \ref{Pressure} with $\eta=0$  there exists a unique $p \in L^2(G)$ so that $(p ,\diverge{v})=\Lambda(v)$ for all $v\in{}_0H^1(G)$.  This yields \eqref{cS0d2} and \eqref{cS0l1es}.
\end{proof}

We now prove higher regularity for the weak solution.

\begin{lemma}\label{cS0l2}
Let  $(w,p)\in {}_0H_\sigma^1(G)\times L^2(G)$  be the weak solution to \eqref{cS0} obtained in Lemma \ref{cS0l1}. If $r \ge 1$ and $G^1\in \ddot{H}^{r-2}(\Omega)$, $G^3\in {H}^{r-3/2}(\Gamma)$, then $w\in \ddot{H}^r(G)$,  $p\in \ddot{H}^{r-1}(G)$ and they satisfy
\begin{equation}\label{cS0l2es}
\|w\|_{r}+\| p\|_{r-1} \lesssim\|G^1\|_{ r-2 }+\|G^3\|_{ r-3/2 }.
\end{equation}
\end{lemma}
\begin{proof}
We will prove our lemma by induction. For $r=1$, it is already proved in Lemma \ref{cS0l1}. Now assume that the lemma holds up to $r=m\ge 1$.  We will prove that it holds for $r=m+1$. Since our upper boundary $\Gamma_+$ and the lower boundary $\Gamma_b$ may not be flat, we are not free to take the horizontal difference quotient in \eqref{cS0}. The standard regularity-improving technique would be to resort to a flattening boundary argument in order to allow the local application of horizontal difference quotients (i.e. horizontal derivatives). Such an  argument would be lengthy, so we abandon it in favor of an alternate approach  that makes full use of the fact that the interface $\Gamma_-$ is flat.  We will localize the equations near $\Gamma_-$ and then apply difference quotients in the localization.

We may assume without loss of generality that $\Gamma_+\subset\{x_3\ge 3\}$, $\Gamma_-=\{x_3=0\}$ and $\Gamma_b\subset\{x_3\le -3\}$.  Indeed, if this is not so, then we can translate $G$ and then rescale so as to make these assumptions true; these operations can then be undone to arrive at the desired estimates in the original domain.  We define the cut-off function $\chi\in C_c^\infty(\mathbb{R})$ with property that
\begin{equation}
\chi=1\text{ on }\{-1\le x_3\le 1\}\text{ and }\chi=0\text{ on }\{|x_3|\ge 3/2 \}.
\end{equation}
Now we define $v=\chi w$, $q=\chi p$, $\tilde{G}:=\{-2\le x_3\le 2\}$, $\tilde{\Gamma}_+:=\{x_3=2\}$, $\tilde{\Gamma}_- :=\Gamma_-$, and $\tilde{\Gamma}_b:=\{x_3=-2\}$.  Then it is easy to see that $(v,  q)$ is the solution to the  problem
\begin{equation}\label{cutoffeq}
\left\{\begin{array}{ll}
-\mu \Delta v  +\nabla q  =\tilde{F}^1\ &\text{ in } \tilde{G}
\\   \diverge{v} =\tilde{F}^2  \ &\text{ in }\tilde{G}
\\(q_+I-\mu_+\mathbb{D}(v_+))e_3=0  \ &\text{ on } \tilde{\Gamma}_+
\\\Lbrack v\Rbrack=0,\quad \Lbrack(qI-\mu\mathbb{D}(v))e_3\Rbrack = -G^3_- \ &\text{ on } \tilde{\Gamma}_-
\\v_-=0 &\text{ on } \tilde{\Gamma}_b,
\end{array}\right.
\end{equation}
where $\tilde{F}^1,\tilde{F}^2$ are given by
\begin{equation}\label{cutofff}
\tilde{F}^1 = G^1-\mu (\partial_3^2\chi w + 2\partial_3 \chi\partial_3 w) +\partial_3\chi p e_3,
\quad \tilde{F}^2 = \partial_3\chi w^3.
\end{equation}

Since now  $G^1\in \ddot{H}^{m-1}(G),\ G^3\in H^{m-1/2}(\Gamma)$, by the induction assumption at the level $m$, we have $w\in \ddot{H}^m(G) ,\ p\in \ddot{H}^{m-1}(G)$ and satisfy
\begin{equation} \label{mestt}
\|w\|_{m}+\| p\|_{m-1} \lesssim \|G^1\|_{ m-2 } + \|G^3\|_{ m-3/2 }.
\end{equation}
This implies in particular that $v\in \ddot{H}^{m }(\tilde{G})$, $q\in \ddot{H}^{m-1}(\tilde{G})$ satisfy
\begin{equation} \label{vmes}
\|v\|_{m}+\| q\|_{m-1} \lesssim \|G^1\|_{ m-2 } + \|G^3\|_{ m-3/2 }
\end{equation}
and that
\begin{equation} \label{esF}
\|\tilde{F}^1\|_{m-1}+\|\tilde{F}^2\|_{m}+\|G^3_-\|_{m-1/2} \lesssim \|G^1\|_{ m-1 } + \|w\|_{ m } +
\|p\|_{m-1} + \|G^3_-\|_{m-1/2} \lesssim \mathcal{Z},
\end{equation}
where we have compactly written
\begin{equation}
\mathcal{Z}:=\|G^1\|_{m-1} +\|G^3 \|_{m-1/2}.
\end{equation}

Now we want to adjust the divergence of $v$. Since $\tilde{F}^2\in \ddot{H}^m(\tilde{G})$, by Lemma \ref{div} there exists $\bar{v}\in {}_0H^1(\tilde{G})\cap \ddot{H}^{m+1}(\tilde{G})$ so that $\diverge{\bar{v}}=\tilde{F}^2$ in $\tilde{G}$ and
\begin{equation}\label{esF2}
\|\bar{v}\|_{m+1} \lesssim \|\tilde{F}^2\|_{m} \lesssim \mathcal{Z}.
\end{equation}
Hence,  defining $\tilde{v}=v-\bar{v}$, we may switch the problem \eqref{cutoffeq}  to  the following problem for $\tilde{v}$:
\begin{equation}\label{cutoffeq0}
\left\{\begin{array}{ll}
-\mu \Delta \tilde{v} +\nabla q =\tilde{G}^1 \ &\text{ in }\tilde{G}
\\ \diverge{\tilde{v}} =0   \ &\text{ in}\ \tilde{G}\\
(q_+I-\mu_+\mathbb{D}(\tilde{v}_+))\nu=\tilde{G}^3_+  \ &\text{ on
}\tilde{\Gamma}_+
\\
 \Lbrack \tilde{v}\Rbrack=0,\quad \Lbrack(qI-\mu\mathbb{D}(\tilde{v}))e_3\Rbrack=-\tilde{G}^3_-  \ &\text{ on }\tilde{\Gamma}_-
\\ \tilde{v}_-=0 &\text{ on }\tilde{\Gamma}_b,
\end{array}\right.
\end{equation}
where $\tilde{G}^1,\tilde{G}^3$ are given by
\begin{eqnarray}
\tilde{G}^1=\tilde{F}^1+\mu\Delta \bar{v},  \tilde{G}^3_+= \mu_+\mathbb{D}(\bar{v}_+)e_3 \text{ and
}\tilde{G}^3_-=G^3_-- \Lbrack \mu\mathbb{D}(\bar{v})e_3\Rbrack.
\end{eqnarray}
By \eqref{esF}, \eqref{esF2} and the trace theorem, we have
\begin{equation}  \label{Ges222}
\|\tilde{G}^1\|_{m-1}+\|\tilde{G}^3\|_{m-1/2}  \lesssim \|\tilde{F}^1\|_{m-1} + \|\bar{v}\|_{m+1} + \|G_-^3\|_{m-1/2} \lesssim \mathcal{Z}.
\end{equation}

Notice that \eqref{vmes} and \eqref{esF2} imply that we already have  $\tilde{v}\in {}_0H_\sigma^1(\tilde{G}) \cap \ddot{H}^{m}(\tilde{G})$ and
\begin{equation} \label{tvmes}
\|\tilde{v}\|_{m}+\| q\|_{m-1} \lesssim \mathcal{Z}.
\end{equation}
Let us now write $D^{k}$ to mean any differential operator $\partial^\alpha$ with $\alpha \in \mathbb{N}^2$ and $| \alpha | = k$.  The estimates \eqref{Ges222} and \eqref{tvmes}, combined with the fact that all the boundaries $\tilde{\Gamma}_+$, $\tilde{\Gamma}_-,$ $\tilde{\Gamma}_b$ are flat allow  us to see that $(D^{m-1} \tilde{v}, D^{m-1}q) \in  {}_0H_\sigma^1(\tilde{G}) \times L^2(\tilde{G})$ is the unique weak solution of the problem resulting from applying $D^{m-1}$ to \eqref{cutoffeq0}, in the sense of \eqref{cS0d2}.  That is,
\begin{equation} \label{cS0l81}
(\frac{\mu}{2}\mathbb{D}D^{m-1}\tilde{v}, \mathbb{D}v)-(D^{m-1}q,\diverge{v}) =( D^{m-1}\tilde{G}^1, v)-(
D^{m-1}\tilde{G}^3, v) \text{ for all } v\in {}_0H^1(\tilde{G}).
\end{equation}
Note that  $D^{m-1}\tilde{G}^1\in L^2(\tilde{G})$ and $ D^{m-1}\tilde{G}^3\in {H}^{ {1}/{2}}(\tilde{\Gamma})$. Take the test function  $v=D_{-h}D_hD^{m-1}\tilde{v}$ in \eqref{cS0l81}, where $D_h$ is the standard difference  quotient in any horizontal direction $h\in \mathbb{R}^2$; then we have
\begin{equation} \label{Dmu}
(\frac{\mu}{2}\mathbb{D}D^{m-1}\tilde{v}, \mathbb{D}D_{-h}D_hD^{m-1}\tilde{v})  =( D^{m-1}\tilde{G}^1, D_{-h}D_h D^{m-1} \tilde{v}) - (D^{m-1}\tilde{G}^3, D_{-h} D_h D^{m-1}\tilde{v}).
\end{equation}
Here the pressure term vanishes since $ D^{m-1}\tilde{v}\in {}_0H_\sigma^1(\tilde{G})$.  We shall now estimate both sides of \eqref{Dmu}. First, by the properties  of difference  quotient and Korn's inequality, we obtain
\begin{equation}\label{cS0l82}
(\frac{\mu}{2}\mathbb{D}D^{m-1}\tilde{v}, \mathbb{D} D_{-h} D_h D^{m-1}\tilde{v}) = (\frac{\mu}{2}\mathbb{D}D_hD^{m-1}\tilde{v}, \mathbb{D} D_h D^{m-1} \tilde{v})\ge C\| D_h \nabla D^{m-1} u\|_{0}^2.
\end{equation}
We again combine the properties  of difference  quotients, the trace theorem and Poincar\'e's inequality to  obtain
\begin{multline}\label{cS0l83}
( D^{m-1}\tilde{G}^1, D_{-h}D_hD^{m-1}\tilde{v})+( D^{m-1}\tilde{G}^3, D_{-h}D_hD^{m-1}\tilde{v})
\\
= (D^{m-1}\tilde{G}^1, D_{-h}D_hD^{m-1}\tilde{v})+( D_h D^{m-1} \tilde{G}^3, D_h D^{m-1}\tilde{v})
\\
\lesssim \| D^{m-1}\tilde{G}^1\|_{L^2(\tilde{G})}\|D_h D^m \tilde{v}\|_{L^2( \tilde{G})} + \|D_h D^{m-1} \tilde{G}^3 \|_{H^{-1/2}(\tilde{\Gamma})} \|D_hD^{m-1}\tilde{v}\|_{H^{1/2}(\tilde{\Gamma})}
\\
\lesssim \left(\|D^{m-1}\tilde{G}^1\|_{0}+\|D^{m-1}\tilde{G}^3\|_{1/2}\right)\|D_h\nabla D^{m-1}u\|_{0}.
\end{multline}
By \eqref{cS0l82}--\eqref{cS0l83} and \eqref{Ges222}, we have
\begin{equation}\label{cS0l84}
\|D_h\nabla D^{m-1}\tilde{v}\|_{0} \lesssim \|D^{m-1}\tilde{G}^1\|_{0} + \|D^{m-1}\tilde{G}^3\|_{1/2}
\lesssim \mathcal{Z}.
\end{equation}
By the properties  of difference quotients and Poincar\'e's inequality, we  deduce from \eqref{cS0l84} that
\begin{equation} \label{tvmes2}
\| \na^{m}\tilde{v}\|_{1}\lesssim  \mathcal{Z}.
\end{equation}

It is crucial to observe that \eqref{tvmes} and \eqref{tvmes2} have yielded  higher regularity of  $\tilde{v}$ on $\Gamma_-$ due to its flatness. Indeed, by the definition of Sobolev norm $\|\cdot\|_{H^s(\Gamma_-)}:=\|\cdot\|_{H^s(  \mathrm{T}^2)}$, the Sobolev interpolation inequality and the trace theorem, we obtain from \eqref{tvmes} and \eqref{tvmes2} that
\begin{equation}\label{cS0bes}
\| \tilde{v} \|_{H^{m+1/2}(\Gamma_-)}^2  \lesssim \| \tilde{v}\|_{L^2(\Gamma_-)} + \|\na^{m} \tilde{v}\|_{H^{ 1/2} (\Gamma_-)} \lesssim \|\tilde{v}\|_{H^1(\Omega)} + \|\na^{m} \tilde{v}\|_{H^1(\Omega)} \lesssim \mathcal{Z}.
\end{equation}
The key point is to notice that  we can obtain the required regularity of $w$ on $\Gamma_-$ in order to allow us to to apply the classical regularity theory of the one-phase Stokes problems. More precisely, since $v=\tilde{v} +\bar{v}$ and $w=v$ on $\Gamma_-$, we deduce from \eqref{cS0bes}, \eqref{esF2} and the trace theorem that
\begin{equation}\label{bde}
\| w \|_{H^{m+1/2}(\Gamma_-)}^2 = \| v \|_{H^{m+1/2}(\Gamma_-)}^2\le \| \tilde{v} \|_{H^{m+1/2}(\Gamma_-)}^2 +
\|\bar{v} \|_{H^{m+1/2}(\Gamma_-)}^2  \lesssim \mathcal{Z}.
\end{equation}
With the estimate \eqref{bde} in hand, we first apply Lemma \ref{cS1phaselemma1} to the following Stokes problem
\begin{equation}\label{cS0111}
\left\{\begin{array}{ll}
-\mu_+ \Delta w_+ +\nabla p_+ =G^1_+ \ &\text{ in }G_+
\\ \diverge{w_+} =0   \ &\text{ in}\ G_+\\
(p_+I-\mu_+\mathbb{D}(w_+))\nu=G^3_+  \ &\text{ on }\Gamma_+
\\ w_+=w_+=v \ &\text{ on }\Gamma_-
\end{array}\right.
\end{equation}
to find that indeed $(w_+,p_+)\in H^{m+1}(G_+)\times H^{m}(G_+)$ and satisfy the   estimate
\begin{multline}\label{finales1}
\|w_+\|_{H^{m+1}(G_+)}+\|p_+\|_{H^m(G_+)}
\\
\lesssim \|G^1_+\|_{H^{m-1}(G_+)} + \|G_+^3\|_{H^{m-1/2}(\Gamma_+)} + \|w_+\|_{H^{m+1/2}(\Gamma_-)} \lesssim \mathcal{Z}.
\end{multline}
Similarly, we apply Lemma \ref{cS1phaselemma2} to the  Stokes problem
\begin{equation}\label{cS0112}
\left\{\begin{array}{ll}
-\mu_- \Delta w_- +\nabla p_- =G^1_- \ &\text{ in }G_-
\\ \diverge{w_-} =0   \ &\text{ in}\ G_-
\\ w_-=w_-=v \ &\text{ on }\Gamma_-
\\ w_-=0 \ &\text{ on }\Gamma_b\end{array}\right.
\end{equation}
to find that indeed $(w_-,p_-)\in H^{m+1}(G_-)\times H^{m}(G_-)$ and satisfy the estimate
\begin{equation}\label{finales2}
\|w_-\|_{H^{m+1}(G_-)} + \|\nabla p_-\|_{H^{m-1}(G_-)}
\lesssim \|G^1_-\|_{H^{m-1}(G_-)} + \|w_-\|_{H^{m+1/2}(\Gamma_-)} \lesssim \mathcal{Z}.
\end{equation}

Consequently, combining \eqref{finales1}--\eqref{finales2} with \eqref{mestt}, we conclude that $w\in \ddot{H}^{m+1}(G) ,\ p\in \ddot{H}^{m}(G)$ and satisfy
\begin{equation}
\|w\|_{m+1}+\| p\|_{m} \lesssim  \mathcal{Z}:=\|G^1\|_{m-1} +\|G^3\|_{m-1/2}.
\end{equation}
This implies that our lemma holds for $r=m+1$, and hence by  induction the proof is completed.
\end{proof}

We now present the
\begin{proof}[Proof of Theorem \ref{cStheorem}]
We set $u=\bar{u}+w$ with $\bar{u}$ found in \eqref{ubares} and $(w,p)$ found in Lemma \ref{cS0l2}.  Then $(u,p)$ is the unique strong solution to the problem \eqref{cS} and satisfies the estimate \eqref{cSresult}.
\end{proof}

%%%%%%%%%%%%%%%%%%%%%%%%%%%%%%%%%%%%%%%%%%%%%%%
\subsection{Two-phase $\mathcal{A}$--Stokes problem}\label{sec_two_A_stokes}
%%%%%%%%%%%%%%%%%%%%%%%%%%%%%%%%%%%%%%%%%%%%%%%

In Section \ref{sec_two_A_stokes}, in order to derive the regularity of the solutions to the time-dependent problem, we consider the following two-phase $\mathcal{A}$--Stokes problem:
 \begin{equation}\label{cSA}
\left\{\begin{array}{lll}
 -\mu\Delta_\mathcal{A} u+\nabla_\mathcal{A} {p}=F^1\ &\text{ in}\ \Omega
 \\ \diverge_\mathcal{A}{u} =F^2  \ &\text{ in}\ \Omega
\\S_{\mathcal{A}_+}(p_+,u_+)\mathcal{N}_+= F^3_+  \ &\text{ on }\Sigma_+,
\\ \Lbrack u\Rbrack=0,\quad \Lbrack S_\mathcal{A}(p,u)\Rbrack \mathcal{N}_-=-F^3_-  \ &\text{ on }\Sigma_-,
\\  u_-=0 &\text{ on }\Sigma_b.
\end{array}\right.
\end{equation}
Here we view $\eta$ as given and let $\mathcal{A},  \mathcal{N},$ etc, be determined in terms of $\eta$ as in \eqref{ABJ_def}.  We shall use the regularity of \eqref{cSA} for each fixed $t$ in the context of the time-dependent problem, and hence we will temporarily ignore the time dependence of $\eta,$ $\mathcal{A},$ $\mathcal{N}$, etc, to view \eqref{cSA} as a stationary problem.

Recall that the two-phase Stokes problem with constant coefficients differential operators is only solved in Section
\ref{stokesconst} in the case that the interface is flat. This prevents us from establishing the strong solvability of \eqref{cSA} as in Lemma 3.2 of \cite{GT_lwp} by transforming \eqref{cSA} back to a Stokes problem  with constant coefficients in a general domain. Instead, we will directly solve \eqref{cSA} by viewing it as a perturbation of the Stokes problem with flat interface (and flat upper boundary) that was considered in Section \ref{stokesconst}. For this technique to be viable, we must impose a smallness condition on $\eta$ in our results.   Rewriting \eqref{cSA} in this form yields
\begin{equation}\label{cSperturbeeq}
\left\{\begin{array}{lll}
-\mu \Delta u+\nabla {p} =F^1+G^1\ &\text{in}\ \Omega
\\  \diverge{u}  =F^2+G^2    \ &\text{in}\ \Omega
\\  (p_+I-\mu_+\mathbb{D}(u_+))e_3=F_+^3+G_+^3 \ &\text{on }\Sigma_+,
\\
 \Lbrack u\Rbrack=0,\quad \Lbrack(pI-\mu\mathbb{D}(u))e_3\Rbrack=-F_-^3-G_-^3  \ &\text{on }\Sigma_-,
\\  u_-=0 &\text{on }\Sigma_b,
\end{array}\right.
\end{equation}
where $G^1,G^2,G^3$ are given by
\begin{equation}\label{cSperturbedata00000}
\left\{\begin{array}{lll}
 G^1:=\mu(\Delta_\mathcal{A}-\Delta)u-(\nabla_\mathcal{A}-\nabla)p,\quad G^2:=-(\diverge_\mathcal{A}-\diverge)u,
 \\ G^3_+:=p_+(\mathcal{N}_+-e_3)-\mu_+(\mathbb{D}_{\mathcal{A}_+}(u_+)\mathcal{N}_+-\mathbb{D}(u_+)e_3),
 \\- G^3_-:=\Lbrack p\Rbrack(\mathcal{N}_--e_3)+\Lbrack\mu\left(\mathbb{D}_{\mathcal{A}}(u)\mathcal{N}_--\mathbb{D}(u)e_3\right)\Rbrack.
\end{array}  \right.
\end{equation}
Note that $G^1$, $G^2$, and $G^3$ are linear in $u$ and $p$.

We first give some estimates of these perturbation functions.

\begin{lemma}\label{cSperturbedata}
Let $k\ge 4$ be an integer and suppose that $\|\eta\|_{k+1/2}\le 1$, then for $r=2,\dots,k$,
\begin{equation}\label{cSperturbedataes1}
   \|G^1\|_{r-2}+\|G^2\|_{r-1}
  +\|G^3\|_{r-3/2}  \lesssim\|\eta\|_{k-1/2}  \left(\|u\|_{r}+\|p\|_{ r-1}\right),
\end{equation}
 and for $r=k+1$,
\begin{equation}\label{cSperturbedataes2}
\|G^1\|_{k-1}+\|G^2\|_{k} +\|G^3\|_{k-1/2}  \lesssim\|\eta\|_{k-1/2}  \left(\|u\|_{k+1}+\|p\|_{ k}\right)
+  \|\eta\|_{k+1/2}\|u\|_{7/2}.
\end{equation}
\end{lemma}

\begin{proof}
These perturbations have the same structure as those of \cite{GT_inf,GT_per}.   Hence, the straightforward estimates performed there may be easily modified to yield the estimates \eqref{cSperturbedataes1}--\eqref{cSperturbedataes2}.
\end{proof}

With these estimates in hand, we can now prove the existence of a unique solution to \eqref{cSA}.

\begin{theorem}\label{cSAt2}
Let $k\ge 4$ be an integer and suppose that $\|\eta\|_{k+1/2}\le 1$. Assume that $F^1\in \ddot{H}^{r-2}(\Omega),\ F^2\in \ddot{H}^{r-1}(\Omega),\ F^3\in  {H}^{r-3/2}(\Sigma),\ r\ge 2$. There exists $\varepsilon_0>0$ so that
if $\|\eta\|_{k-1/2}\le \varepsilon_0$, then there exists a unique strong solution $(u,p)\in ({}_0H^1(\Omega)\cap
\ddot{H}^r(\Omega))\times \ddot{H}^{r-1}(\Omega)$ to the problem \eqref{cSA}. Moreover, for $r=2,\dots,k$,
\begin{equation}\label{cSAt21}
\|u\|_{r}+\| p\|_{r-1} \lesssim \|F^1\|_{r-2}+\|F^2\|_{r-1}+\|F^3\|_{r-3/2},
\end{equation}
and for $r=k+1$,
\begin{equation}\label{cSAt22}
\|u\|_{k+1}+\| p\|_{k} \lesssim \|F^1\|_{k-1}+\|F^2\|_{k}+\|F^3\|_{k-1/2} + \|\eta\|_{k+1/2}( \|F^1\|_{2}
+\|F^2\|_{3} + \|F^3\|_{5/2}).
\end{equation}
\end{theorem}

\begin{proof}
We will solve \eqref{cSA} by the method of successive approximations. Let $ (u^{0} , p^{0})=(0,0)$.  To emphasize the dependence of the $G^i$  in \eqref{cSperturbedata00000} on $u$ and $p$, we will write $G^1 = G^1(u,p)$, etc.  For each $m\ge0$ we define $(u^{m+1},p^{m+1})$ as the solution of the problem
\begin{equation}\label{iterationm}
\left\{\begin{array}{lll}
-\mu \Delta u^{m+1}+\nabla {p^{m+1}} =F^1+G^1(u^m,p^m) \ &\text{
in}\ \Omega
 \\  \diverge{u^{m+1}}  =F^2+G^2(u^m)    \ &\text{ in}\ \Omega
 \\  (p_+^{m+1}I-\mu_+\mathbb{D}(u_+^{m+1}))e_3=F_+^3+G^3_+(u^m,p^m)  \ &\text{ on }\Sigma_+,
\\
 \Lbrack u^{m+1}\Rbrack=0,\quad \Lbrack(p^{m+1}I-\mu\mathbb{D}(u^{m+1}))e_3\Rbrack=-F_-^3-G^3_-(u^m,p^m)  \ &\text{ on }\Sigma_+,
\\  u_-^{m+1}=0 &\text{ on }\Sigma_b,
\end{array}\right.
\end{equation}
provided that $(u^{m},p^{m})$ are given and satisfy, for $r=2,\dots,k$,
\begin{equation}\label{cSAt21m}
\|u^m\|_{r}+\| p^m\|_{r-1} \lesssim \|F^1\|_{r-2}+\|F^2\|_{r-1}+\|F^3\|_{r-3/2},
\end{equation}
and for $r=k+1$,
\begin{equation}\label{cSAt22m}
\|u^m\|_{k+1}+\| p^m\|_{k} \lesssim \|F^1\|_{k-1}+\|F^2\|_{k}+\|F^3\|_{k-1/2} +\|\eta\|_{k+1/2}( \|F^1\|_{2}
+ \|F^2\|_{3}+\|F^3\|_{5/2}).
\end{equation}

Due to the estimates \eqref{cSAt21m}--\eqref{cSAt22m}, by Theorem \ref{cStheorem}, there exists a unique  $(u^{m+1},p^{m+1}) \in ({}_0H^1(\Omega)\cap \ddot{H}^r(\Omega)) \times \ddot{H}^{r-1}(\Omega)$ solving the problem \eqref{iterationm}.  Moreover, by Theorem \ref{cStheorem}, Lemma \ref{cSperturbedata} and \eqref{cSAt21m}--\eqref{cSAt22m} we have that for $r=2,\dots,k$,
\begin{multline}\label{cSAt23}
\|u^{m+1}\|_{r}+\| p^{m+1}\|_{ r-1 }
\\
\lesssim
\|F^1+G^1(u^m,p^m)\|_{r-2} + \|F^2+G^2(u^m)\|_{r-1} + \|F^3+G^3(u^m,p^m)\|_{r-3/2}
\\
\lesssim \|F^1\|_{r-2} + \|F^2\|_{r-1} + \|F^3\|_{r-3/2} + \|\eta\|_{k-1/2}  \left(\|u^m\|_{r}+\|p^m\|_{r-1}\right)
\end{multline}
and for $r=k+1$,
\begin{multline}\label{cSAt24}
\|u^{m+1}\|_{k+1}+\| p^{m+1}\|_{k}
\\
\lesssim \|F^1+G^1(u^m,p^m)\|_{k-1} + \|F^2+G^2(u^m)\|_{k} + \|F^3+G^3(u^m,p^m)\|_{k-1/2}
\\
\lesssim \|F^1\|_{k-1}+\|F^2\|_{k}+\|F^3\|_{k-1/2} +\|\eta\|_{k-1/2} \left(\|u^m\|_{k+1}+\|p^m\|_{k}\right)
+ \|\eta\|_{k+1/2}\|u^m\|_{7/2}.
\end{multline}
In the estimates \eqref{cSAt23} and \eqref{cSAt24} the constants do not depend on $m$.  This implies that the estimates \eqref{cSAt21m}--\eqref{cSAt22m} hold for all $m\ge0$ since $\|\eta\|_{k+1/2}\le 1$.

In order to pass to the limit in \eqref{iterationm}, we need to prove the convergence of the whole sequence of  successive approximation $\{(u^m,p^m)\}_{m=1}^\infty$. For this, we define
\begin{equation}\label{difference}
U^{m+1}=u^{m+1}-u^{m},\quad P^{m+1}=p^{m+1}-p^{m},\ m=1,2,\dots.
\end{equation}
Then $(U^{m+1}, P^{m+1})$ is the solution of the problem
\begin{equation}\label{iteration2}
\left\{\begin{array}{lll}
-\mu \Delta U^{m+1}+\nabla {P^{m+1}} = G^1(U^m,P^m) \ &\text{ in }G
\\ \diverge{U^{m+1}}  =G^2(U^m)   \ &\text{ in }G
\\  (P_+^{m+1}I-\mu_+\mathbb{D}(U_+^{m+1}))e_3= G^3_+(U_+^m,P_+^m) \ &\text{ on }\Sigma_+
\\ \Lbrack U^{m+1}\Rbrack=0,\quad \Lbrack(P^{m+1}I-\mu\mathbb{D}(U^{m+1}))e_3\Rbrack=- G^3_-(U^m,P^m)  \ &\text{ on }\Sigma_-
\\  U_-^{m+1}=0 &\text{ on }\Sigma_b.
\end{array}\right.
\end{equation}
Then applying Theorem  \ref{cStheorem} to the problem \eqref{iteration2} and employing Lemma \ref{cSperturbedata} to estimate the nonlinear forcing terms, we obtain for $r=2,\dots,k$,
\begin{multline}\label{differencees}
\|U^{m+1}\|_{r}+\| P^{m+1}\|_{ r-1 } \lesssim \|G^1(U^m,P^m)\|_{ r-2 }+\|G^2(U^m)\|_{ r-1 } + \|G^3(U^m,P^m)\|_{ r-3/2 }
\\
\lesssim\|\eta\|_{k-1/2}  \left(\|U^m\|_{r}+\|P^m\|_{ r-1 }\right).
\end{multline}
By assuming that $\|\eta\|_{k-1/2}\le \varepsilon_0$ is sufficiently small, we find that $\{(u^m,p^m)\}_{m=1}^\infty$ is a Cauchy sequence in $\ddot{H}^r(\Omega)\times \ddot{H}^{r-1}(\Omega)$ for $r=2,\dots,k$. Hence $(u^m, p^m) \rightarrow (u,p)$ so that $(u,p)$ solves \eqref{cSperturbeeq}, which is equivalent to \eqref{cSA}. The estimates \eqref{cSAt21}--\eqref{cSAt22} follow from \eqref{cSAt21m}--\eqref{cSAt22m} by weak lower semi-continuity.
\end{proof}

%%%%%%%%%%%%%%%%%%%%%%%%%%%%%%%%%%%%%%%%%%%%%%%
\subsection{Two-phase  Poisson problem}
%%%%%%%%%%%%%%%%%%%%%%%%%%%%%%%%%%%%%%%%%%%%%%%

We now consider the two-phase scalar elliptic problem
\begin{equation}\label{poisson}
\left\{\begin{array}{lll}
 \rho^{-1}\Delta p  =f  ^1  \ &\text{ in}\ \Omega
 \\p=f^2   &\text{ on }\Sigma_+
\\
 \Lbrack p\Rbrack=f^3 &\text{ on }\Sigma_-
 \\ \Lbrack \rho^{-1}\partial_3 p \Rbrack =f^4  \ &\text{ on }\Sigma_-
\\\rho^{-1}_-  \nabla p\cdot\nu=f^5 &\text{ on }\Sigma_b.
\end{array}\right.
\end{equation}

We first consider the weak formulation of  \eqref{poisson}. We suppose that $f^1\in ({}^0H^1(\Omega))^\ast,$ $f^2\in  {H}^{1/2}(\Sigma_+),$ $f^3\in {H}^{ 1/2}(\Sigma_-),$  $f^4\in  {H}^{ -1/2}(\Sigma_-),$ and $f^5 \in  {H}^{-1/2}(\Sigma_b)$. Let $\bar{p}\in \ddot{H}^1(\Omega)$ be so that $\bar{p}=f^2$ on $\Sigma_+$, $\Lbrack
\bar{p} \Rbrack = f^3$ on $\Sigma_-$ and $ \bar{p}=0$ near $\Sigma_b$.  The existence of such a $\bar{p}$  is standard, and it may be selected so that $\| \bar{p}\|_{1} \lesssim \|f^2\|_{1/2} + \|f^3\|_{1/2}$.  We switch  the unknown to $q=p-\bar{p}$ and then define a weak formulation of \eqref{poisson} as
\begin{equation} \label{poiweak}
(\rho^{-1}\nabla q,\nabla\varphi) =-(\rho^{-1}\nabla \bar{p},\nabla\varphi)-\langle f^1,\varphi\rangle_\ast + \langle
f^4,\varphi \rangle_{-1/2,-}+\langle f^5,\varphi\rangle_{-1/2,\,b}, \forall \varphi\in {}^0H^1(\Omega).
\end{equation}
Here $\langle \cdot, \cdot \rangle_\ast$ is the dual paring between $({}^0H^{1}(\Omega))^\ast$ and ${}^0H^{1}(\Omega)$,  $\langle \cdot, \cdot\rangle_{-1/2,-}$ is the dual paring between $H^{-1/2}(\Sigma_-)$ and $ H^{1/2}(\Sigma_-)$, $\langle \cdot, \cdot\rangle_{-1/2,\,b}$ is the dual paring between $H^{-1/2}(\Sigma_b)$ and $ H^{1/2}(\Sigma_b)$. Then it is standard to show the unique existence of solution $q\in {}^0{H}^1(\Omega)$ of \eqref{poiweak} so that $p\in \ddot{H}^1(\Omega)$ satisfies
\begin{equation}
\|p\|_1^2\lesssim \|f^1\|_{({}^0H^1(\Omega))^\ast}^2 + \|f^2\|_{H^{1/2}(\Sigma_+)}^2 + \|f^3\|_{H^{1/2}(\Sigma_-)}^2
+ \|f^4\|_{H^{-{1/2}}(\Sigma_-)}^2 + \|f^5\|_{H^{-{1/2}}(\Sigma_b)}^2.
\end{equation}

In the case with surface tension, to determine the initial pressure we will consider the action of $f^1\in
({}^0H^1(\Omega))^\ast$  given in a more specific fashion by
\begin{equation}
\langle f^1,\varphi\rangle_\ast=(\rho^{-1} G,\nabla\varphi),\ \forall \varphi\in {}^0H^1(\Omega)
\end{equation}
for $G\in L^2(\Omega)$ with $\|\rho^{-1} G\|_0 = \|f^1\|_{({}^0H^1(\Omega))^\ast}$. Then \eqref{poiweak}  may be
written as
\begin{equation}\label{poiweakt}
(\rho^{-1}(\nabla p+G),\nabla\varphi)=  \langle f^4,\varphi \rangle_{-1/2,-}+\langle f^5,\varphi\rangle_{-1/2,\,b},\ \forall \varphi\in {}^0H^1(\Omega).
\end{equation}
In this case we should say $p\in \ddot{H}^1(\Omega)$ is a weak solution to the problem
\begin{equation}\label{poisson0}
\left\{\begin{array}{lll}
 \diverge(\rho^{-1}(\nabla p+G))  =0 \ &\text{ in}\ \Omega
 \\p=f^2   &\text{ on }\Sigma_+
\\
 \Lbrack p\Rbrack=f^3 &\text{ on }\Sigma_-
 \\ \Lbrack \rho^{-1}(\nabla_3 p +G)\Rbrack =f^4  \ &\text{ on }\Sigma_-
\\\rho^{-1}_-  (\nabla p+G)\cdot\nu=f^5 &\text{ on }\Sigma_b.
\end{array}\right.
\end{equation}

Now we record a result on the existence and regularity of solutions to \eqref{poisson}.

\begin{theorem}\label{Poissonth}
Let $r \ge 2$.  If $f^1\in \ddot{H}^{r-2}(\Omega),$  $f^2\in  {H}^{r-1/2}(\Sigma_+),$  $f^3\in  {H}^{r-1/2}(\Sigma_-),$   $f^4\in  {H}^{r-3/2}(\Sigma_-),$ and  $f^5\in  {H}^{r-3/2}(\Sigma_b),$ then the problem \eqref{poisson} admits a unique strong solution $p \in  \ddot{H}^r(\Omega)$. Moreover,
\begin{equation}\label{poiresult}
\| p\|_{r } \lesssim \|f^1\|_{r-2}+\|f^2\|_{r-1/2}+\|f^3\|_{r-1/2}+\|f^4\|_{r-3/2}+\|f^5\|_{r-3/2}.
\end{equation}
\end{theorem}
\begin{proof}
We may argue as in the proof of Lemma \ref{cS0l2}, using difference quotients and the fact that the interface is flat, to deduce the desired estimates.  We omit further details.
\end{proof}

%%%%%%%%%%%%%%%%%%%%%%%%%%%%%%%%%%%%%%%%%%%%%%%
\subsection{Two-phase $\mathcal{A}$--Poisson problem}
%%%%%%%%%%%%%%%%%%%%%%%%%%%%%%%%%%%%%%%%%%%%%%%

Now we consider the scalar two-phase $\mathcal{A}$--Poisson problem
\begin{equation}\label{Apoisson}
\left\{\begin{array}{lll}
\rho^{-1} \Delta_{\mathcal{A} } p  =f  ^1  \ &\text{ in}\ \Omega
\\p=f^2 &\text{ on }\Sigma_+
\\
 \Lbrack  p\Rbrack=f^3  &\text{ on }\Sigma_-
 \\  \Lbrack \rho^{-1}\nabla_\mathcal{A} p  \Rbrack \cdot \mathcal{N}_{-} =f^4  \ &\text{ on }\Sigma_-
\\  \rho^{-1}_-\nabla_\mathcal{A} p\cdot\nu=f^5 &\text{ on }\Sigma_b.
\end{array}\right.
\end{equation}

We first consider the weak formulation of \eqref{Apoisson}. For this, we define a time-dependent inner-product on ${}^0H^1(\Omega)$ according to
\begin{equation}
(u,v)_{\widetilde{\mathcal{H}}^1(t)}:=  \int_\Omega \rho^{-1} J(t) | \nabla_{\mathcal{A}(t)} f|^2
\end{equation}
and the  norm by $\|u\|_{\widetilde{\mathcal{H}}^1(t)} := \sqrt{(u,u)_{\widetilde{\mathcal{H}}^1(t)}}$. Then we write $\widetilde{\mathcal{H}}^1(t) := \{\|u\|_{\widetilde{\mathcal{H}}^1(t)} < \infty \}$.  As in Lemma \ref{equal}, under a smallness assumption on $\eta$, $\widetilde{\mathcal{H}}^1(t)$ has the same topology as $H^1$.

For the weak formulation we suppose that $f^1 \in ({}^0H^1(\Omega))^\ast,$ $f^2\in  {H}^{ 1/2}(\Sigma_+),$ $f^3\in {H}^{1/2}(\Sigma_-)$,  $f^4\in  {H}^{ -1/2}(\Sigma_-),$  and $f^5\in  {H}^{-1/2}(\Sigma_b)$.  Let $\bar{p}\in
\ddot{H}^1(\Omega)$ be so that $\bar{p}=f^2$ on $\Sigma_+$, $\Lbrack \bar{p}\Rbrack=f^3$ on $\Sigma_-$ and $\bar{p}=0$ near $\Sigma_b$. The existence of such a $\bar{p}$  is standard, and it may be selected so that $\| \bar{p}\|_{1} \lesssim \|f^2\|_{1/2} + \|f^3\|_{1/2}$.  We switch the unknown to $q=p-\bar{p}$ and then define a weak formulation of \eqref{Apoisson} as
\begin{equation}\label{Apoissonweak}
 (q,\varphi)_{\widetilde{\mathcal{H}}^1(t)} = -(\bar{p},\varphi)_{\widetilde{\mathcal{H}}^1(t)}- \langle f^1,\varphi\rangle_{*}
  + \langle f^4,\varphi\rangle_{-1/2,-}+ \langle f^5,\varphi\rangle_{-1/2,\,b},\ \forall\varphi \in
  {^0}H^1(\Omega).
\end{equation}
Here $\langle \cdot, \cdot\rangle_\ast$ is the dual paring between $(\widetilde{\mathcal{H}}^{1}(t))^\ast$ and
$\widetilde{\mathcal{H}}^{1}(t)$, $\langle \cdot, \cdot\rangle_{-1/2,-}$  is the dual paring between
$H^{-1/2}(\Sigma_-)$ and $ H^{1/2}(\Sigma_-)$, and $\langle \cdot, \cdot \rangle_{-1/2,\,b}$  is the dual paring between $H^{-1/2}(\Sigma_b)$ and $ H^{1/2}(\Sigma_b)$. The existence and uniqueness of a solution  to \eqref{Apoissonweak} follows from standard arguments and
\begin{equation}
\|p\|_1^2 \lesssim \|f^1\|_{({}^0H^1(\Omega))^\ast}^2 + \|f^2\|_{H^{1/2}(\Sigma_+)}^2 + \|f^3\|_{H^{1/2}(\Sigma_-)}^2 + \|f^4\|_{H^{-{1/2}}(\Sigma_-)}^2 + \|f^5\|_{H^{-{1/2}}(\Sigma_b)}^2.
\end{equation}

In the case without surface tension, to determine the initial pressure we will consider the action of $f^1\in
({}^0H^1(\Omega))^\ast$  given in a more specific fashion by
\begin{equation}
 \langle f^1,\varphi\rangle_* = ( g_0,\varphi )_{\mathcal{H}^0} + ( \rho^{-1}G, \nabla_\mathcal{A} \varphi )_{\mathcal{H}^0}
\end{equation}
for $g_0, G  \in L^2(\Omega)$ with $\|g_0\|_{0}^2 + \|{G}\|_{0}^2 = \|f^1\|_{({}^0 H^1(\Omega))*}^2$. Then \eqref{Apoissonweak} may be rewritten as
\begin{equation}\label{Apoissonweak2}
(\rho^{-1}(\nabla_\mathcal{A} p+G),\nabla_\mathcal{A}
\varphi)_{\mathcal{H}^0} = - ( g_0,\varphi )_{\mathcal{H}^0}
+\langle f^4,\varphi\rangle_{-1/2}+ \langle
f^5,\varphi\rangle_{-1/2} \text{ for all } \varphi \in
{^0}H^1(\Omega).
\end{equation}
We then  say that $p\in \ddot{H}^1(\Omega)$ is a weak solution to the problem
 \begin{equation}\label{Apoisson2}
\left\{\begin{array}{ll}
 \diverge_{\mathcal{A} }\left(\rho^{-1}(\nabla_{\mathcal{A} } p +G) \right)=g_0 \ &\text{ in}\ \Omega
 \\p_+ =f^2 &\text{ on }\Sigma_+
\\
 \Lbrack p \Rbrack=f^3\ &\text{ on }\Sigma_-
  \\ \Lbrack\rho^{-1}(\nabla_{\mathcal{A} } p +G)  \Rbrack\cdot  \mathcal{N}_{-} =f^4  \ &\text{ on }\Sigma_-
\\  \rho_-^{-1}(\nabla_{\mathcal{A}_{ -}} p_- +G_-)\cdot\nu=f^5&\text{ on }\Sigma_b.
\end{array}\right.
\end{equation}

Now we record the analog of Theorem \ref{cSAt2} for the problem \eqref{Apoisson}.

\begin{theorem}\label{Apoissont1}
Let $k\ge 4$ be an integer and suppose that $\eta\in H^{k+1/2}$. Suppose that $f^1\in \ddot{H}^{r-2}(\Omega),$  $f^2\in  {H}^{r-1/2}(\Sigma_+),$  $f^3\in  {H}^{r-1/2}(\Sigma_-),$  $f^4\in  {H}^{r-3/2}(\Sigma_-),$ and $f^5\in  {H}^{r-3/2}(\Sigma_b)$ for $r \ge 2$.  Then there exists $\varepsilon_0>0$ so that if $\|\eta\|_{k-1/2} \le \varepsilon_0$, then there exists a unique strong solution $p\in  \ddot{H}^r(\Omega)$ solving the problem \eqref{Apoisson}. Moreover, for $r=2,\dots,k$, we have
\begin{equation} \label{Apoissont1es}
\| p\|_{r} \lesssim \|f^1\|_{  r-2 } +\|f^2\|_{r-1/2 } + \|f^3\|_{ r-1/2 }+\|f^4\|_{ r-3/2 }
 + \|f^5\|_{r-3/2}.
\end{equation}
\end{theorem}
\begin{proof}
Notice that for $\eta$ small, the problem \eqref{Apoisson} can be viewed as a perturbation of the Poisson problem \eqref{poisson}.  As such, we may argue as in the proof of Theorem \ref{cSAt2} to deduce the desired estimates from Theorem \ref{Poissonth}.  We omit further details.
\end{proof}

%%%%%%%%%%%%%%%%%%%%%%%%%%%%%%%%%%%%%%%%%%%%%%%
\section{Case without surface tension: Proof of Theorem \ref{th0}}\label{0surface}
%%%%%%%%%%%%%%%%%%%%%%%%%%%%%%%%%%%%%%%%%%%%%%%

We will construct a local-in-time solution to \eqref{nosurface} through an iteration scheme that works as follows.  First, we view $\eta$ as given and use it to solve for $(u,p)$ in a time-dependent $\mathcal{A}-$Stokes problem.  Then we use $(u,p)$ to solve for $\eta$ via the kinematic transport equation.  Ultimately we will show that this iteration admits a fixed point, which then corresponds to our desired solution.  With the local theory in hand, we then turn to the development of a priori estimates that allow us to produce global-in-time solutions that decay to equilibrium.

Note that the local theory in Sections \ref{linears}--\ref{section_local_well_posed} allows any choice of $\rj$, but in the global theory of Section \ref{sec_nost_gwp} we assume that $\rj < 0$.

%%%%%%%%%%%%%%%%%%%%%%%%%%%%%%%%%%%%%%%%%%%%%%%
\subsection{The linearized  problems}\label{linears}
%%%%%%%%%%%%%%%%%%%%%%%%%%%%%%%%%%%%%%%%%%%%%%%

%%%%%%%%%%%%%%%%%%%%%%%%%%%%%%%%%%%%%%%%%%%%%%%
\subsubsection{The time-dependent $\mathcal{A}$--Stokes problem}\label{sec_tdas}
%%%%%%%%%%%%%%%%%%%%%%%%%%%%%%%%%%%%%%%%%%%%%%%

For given $\eta$ (and hence $\mathcal{A}$, etc.) we consider the linearized time-dependent problem for $(u,p)$:
 \begin{equation}\label{LP}
\left\{\begin{array}{lll}
\rho\partial_t u-\mu\Delta_\mathcal{A} u+\nabla_\mathcal{A} {p}=F^1 &\text{in}\ \Omega\\ \diverge_\mathcal{A } u  =0  \ &\text{in}\ \Omega\\ S_\mathcal{A_+}(p_+,u_+)\mathcal{N}_+=F_+^3  \ &\text{on }\Sigma_+
\\
 \Lbrack u\Rbrack=0,\quad \Lbrack S_\mathcal{A}(p,u)\Rbrack \mathcal{N}_-=-F_-^3  \ &\text{on }\Sigma_-
\\  u_-=0 &\text{on }\Sigma_b&
\end{array}\right.
\end{equation}
with the initial condition $u(0)=u_0.$  In our analysis of this problem we will employ the time-dependent functional framework developed in Section \ref{time_dep_fnal} of the appendix.  In particular, we will use the spaces $\mathcal{H}^1_T$, $\mathcal{X}_T$, and $\mathcal{Y}(t)$ defined there.

Note that the first equation in \eqref{LP} can be rewritten as
\begin{equation}\label{momentum}
\rho\partial_t u+\diverge_{\mathcal{A}} S_{\mathcal{A}}({p},u) =F^1\quad  \text{in} \Omega.
\end{equation}
Motivated by the identity that results from formally multiplying the above by $Jv$ for a smooth vector $v$ with $v \vert_{\Sigma_b=0}$, integrating over $\Omega$ by parts and then in time from $0$ to $T$, we may define the weak solution of \eqref{LP} as follows. Suppose that
\begin{equation}
F^1\in (\mathcal{H}_T^1)^\ast, \, F^3\in L^2(0,T;H^{-1/2}(\Sigma)), \text{
and } u_0\in \mathcal{Y}(0).
\end{equation}
We shall say $(u,p)$ is a weak solution of \eqref{LP} if
\begin{equation}\label{Lpws}
\left\{\begin{array}{lll}
u\in \mathcal{X}_T,\ \rho\partial_t u\in (\mathcal{H}_T^1)^\ast,\ p\in \mathcal{H}_T^0;\\
\langle\rho\partial_t u,v\rangle_\ast+ \frac{1}{2}(  u,
v)_{\mathcal{H}^1_T}-(p,\diverge_\mathcal{A}{v})_{\mathcal{H}^0_T}=\langle F^1, v\rangle_{\ast}-\langle F^3,
v\rangle_{-1/2},\ \forall\,v\in  \mathcal{H}^1_T;
\\u(0)=u_0.
\end{array}\right.
\end{equation}
Here $\langle \cdot, \cdot\rangle_{\ast}$ is the dual paring between $(\mathcal{H}_T^1)^\ast$ and $\mathcal{H}_T^1$, and $\langle \cdot, \cdot\rangle_{-1/2} $ is the dual paring between $L^2(0,T;H^{-1/2}(\Sigma))$ and $L^2(0,T;H^{1/2}(\Sigma))$. It is easy to see that the weak solution of \eqref{LP} in the sense of \eqref{Lpws} is unique. However, our aim is to construct solutions to \eqref{LP} with high regularity, so we will not construct weak solutions.

First, to construct  strong solutions to \eqref{LP}, we make the stronger assumptions that
\begin{equation}\label{SSdata}
\left\{
\begin{array}{lll}F^1\in L^2(0,T;\ddot{H}^1(\Omega)),\ \partial_t F^1\in L^2(0,T;({}_0H^{1}(\Omega))^*),\\
F^3\in L^2(0,T; {H}^{3/2}(\Sigma)),\ \partial_tF^3\in L^2(0,T;
{H}^{-1/2}(\Sigma)),
\\\text{
and }u_0\in \mathcal{X}(0)\cap\ddot{H}^2(\Omega).
\end{array}\right.
\end{equation}
Recall that we abuse notation by writing $L^2 H^{-1}= L^2 ({}_0H^{1}(\Omega))^*$ for the space containing $\partial_t  F^1$ in \eqref{SSdata}.  Note that the inclusions in \eqref{SSdata} imply that (see, for instance Lemmas 2.4 and A.2 of \cite{GT_lwp}) $F^1\in C([0,T];L^2(\Omega)),\ F^3\in C([0,T];H^{1/2}(\Sigma))$; in particular
\begin{equation}\label{SSdata2}
F^1(0)\in L^2(\Omega),\ F^3(0)\in  H^{1/2}(\Sigma).
\end{equation}

\begin{theorem}\label{H2SS}
Suppose that $F^1,F^3,u_0$ satisfy \eqref{SSdata}--\eqref{SSdata2}, and that $u_0,F^3(0)$ satisfy the compatibility conditions
\begin{equation}\label{SScompatibility}
\Pi_{0,+}\left(F_+^3(0)+\mu_+\mathbb{D}_{\mathcal{A}_{0,+}}(u_{0,+})\mathcal{N}_{0,+}\right)=0, \Pi_{0,-}\left(F_-^3(0)-\Lbrack\mu \mathbb{D}_{\mathcal{A}_0}(u_{0})\Rbrack \mathcal{N}_{0,-}\right)=0,
\end{equation}
where $\Pi_{0,\pm}$ are the orthogonal projections onto the tangent space of the surface $\{x_3=\eta_{0,\pm}\}$ (and then $\Pi_{0,\pm}^\perp=I-\Pi_{0,\pm}$)  defined according to
\begin{equation}\label{projection}
\Pi_{0,\pm}\,v=v-(v\cdot \mathcal{N}_{0,\pm})\mathcal{N}_{0,\pm}|\mathcal{N}_{0,\pm}|^{-2}.
\end{equation}
Further assume that
\begin{equation}\label{Kdef}
\mathcal{K}(\eta) := \sup_{0\le t\le T} \left(\|\eta\|_{9/2}^2 + \|\partial_t\eta\|_{7/2}^2
 + \|\partial_t^2\eta\|_{5/2}^2\right)
\end{equation}
is sufficiently small. Then there exists a unique strong solution $(u,p)$ of \eqref{LP} so that
\begin{equation}\label{nnnn}
\begin{array}{lll}
u\in \mathcal{X}_T\cap C([0,T];  {}_0H^1(\Omega)\cap\ddot{H}^2(\Omega))\cap L^2(0,T; \ddot{H}^3(\Omega)),
\\  \partial_t u\in C([0,T]; L^2 (\Omega))\cap L^2(0,T;H^1(\Omega)),\ \rho\partial_t^2u\in (\mathcal{H}_T^1)^\ast,
\\    p\in  C([0,T]; \ddot{H}^1 (\Omega))\cap L^2(0,T; \ddot{H}^2(\Omega)),\  \partial_tp\in  L^2(0,T; L^2(\Omega)).
 \end{array}
\end{equation}
The solution satisfies the estimate
\begin{multline}\label{SSregularity}
 \|u \|_{L^\infty {H}^2}^2 + \|u\|_{L^2  {H}^3}^2  +\|\partial_tu \|_{L^\infty L^2 }^2
 +  \|\partial_t u\|_{ L^2 H^1}^2 + \|\rho\partial_t^2u\|_{ (\mathcal{H}_T^1)^\ast }^2
\\
+ \| p \|_{ L^\infty {H}^1  }^2  +\| p\|_{L^2  {H}^2 }^2
+\|\partial_tp\|_{ L^2 L^2 }^2
\\
\lesssim (1+\mathcal{K}(\eta))\exp(C(1+\mathcal{K}(\eta))T)\left(\|u_0\|_{2}^2 + \|F^1(0)\|_{ 0}^2
+ \| F^3(0)\|_{1/2}^2 \right.
\\
+\left.\|F^1\|_{ L^2  {H}^{1} }^2 +\|\partial_tF^1\|_{ L^2 H^{-1}}^2
+ \|F^3\|_{L^2  {H}^{3/2}}^2 + \|\partial_tF^3\|_{L^2 H^{-1/2} }^2\right).
\end{multline}
The initial pressure, $p(0)\in \ddot{H}^1(\Omega)$, is determined in terms of $u_0,$ $F^1(0),$ $F^3(0)$ as the weak solution to
\begin{equation}\label{pressure00}
\left\{\begin{array}{ll}
 \diverge_{\mathcal{A}_0}\left(\rho^{-1}(\nabla_{\mathcal{A}_0} p(0)-F^1(0)) \right)=- \diverge_{\mathcal{A}_0}(R(0)u_0) \ &\text{ in}\ \Omega,
 \\p_+(0)=( F^3_+(0)\cdot \mathcal{N}_{0,+}+\mu_+\mathbb{D}_{\mathcal{A}_{0,+}}(u_{0,+})\mathcal{N}_{0,+}\cdot \mathcal{N}_{0,+}) |\mathcal{N}_{0,+}|^{-2} &\text{ on }\Sigma_+,
\\
 \Lbrack p(0)\Rbrack=(-F^3_-(0)\cdot \mathcal{N}_{0,-}+\Lbrack \mu \mathbb{D}_{\mathcal{A}_{0 }}(u_{0 })\Rbrack \mathcal{N}_{0,-}\cdot \mathcal{N}_{0,-}) |\mathcal{N}_{0,-}|^{-2}\ &\text{ on }\Sigma_-,
  \\ \Lbrack\rho^{-1}(\nabla_{\mathcal{A}_0} p(0)-F^1(0))  \Rbrack\cdot  \mathcal{N}_{0,-} =\Lbrack \rho^{-1}\mu \Delta_ {\mathcal{A}_{0 }}u_{0 }\Rbrack\cdot \mathcal{N}_{0,-}  \ &\text{ on }\Sigma_-,
\\  \rho_-^{-1}(\nabla_{\mathcal{A}_{0,-}} p_-(0)-F^1_-(0))\cdot\nu=\rho_-^{-1}\mu_-\Delta_ {\mathcal{A}_{0,-}}u_{0,-}\cdot\nu&\text{ on }\Sigma_b,
\end{array}\right.
\end{equation}
in the sense of \eqref{Apoisson2}, where $R:=\partial_t M M^{-1}$ with the matrix $M$ is defined by \eqref{M_def}.
Define the differential operator $D_t$ according to $D_tu=\partial_tu-Ru$. Then $D_tu(0)$ satisfies
\begin{equation}\label{partialtu0}
D_tu(0)=\rho^{-1}(\Delta_{\mathcal{A}_0}u_0-\nabla_{\mathcal{A}_0}p(0)+F^1(0))-R(0)u_0\in \mathcal{Y}(0).
\end{equation}

 Moreover, the pair $(D_tu,\partial_tp)$ satisfy
 \begin{equation}\label{LPt}\left\{\begin{array}{lll}
\rho\partial_t (D_tu)-\mu\Delta_\mathcal{A} (D_tu)+\nabla_\mathcal{A} (\partial_t{p})=D_tF^1 +G^1\ &\text{in}\ \Omega\\ \diverge_\mathcal{A}(D_tu) =0  \ &\text{in}\ \Omega\\ S_\mathcal{A_+}(\partial_tp_+,D_tu_+ )\mathcal{N}_+=\partial_tF_+^3 +G_+^3 \ &\text{on }\Sigma_+
\\
 \Lbrack D_t u \Rbrack=0 ,\quad \Lbrack S_\mathcal{A}(\partial_t p,D_tu)\Rbrack \mathcal{N}_-=-\partial_tF_-^3 +G_-^3 \ &\text{on }\Sigma_-
\\  D_t u_-=0 &\text{on }\Sigma_b&
\end{array}\right.
\end{equation}
in the weak sense of \eqref{Lpws}, where $G^1,$ $G^3$ are defined by
\begin{eqnarray}
&&G^1=-\mu(R+\partial_tJK)\Delta_\mathcal{A}u-\rho\partial_tRu+(\partial_tJK+R-R^T)\nabla_\mathcal{A}p
\nonumber\\&&\qquad\quad+ \mu \diverge_{\mathcal{A}}(\mathbb{D}_\mathcal{A}(Ru)+R\mathbb{D}_\mathcal{A}u+\mathbb{D}_{\partial_t\mathcal{A}}u),\nonumber
\\  &&G^3_+=\mu_+\mathbb{D}_{\mathcal{A}_+}(R_+u_+)\mathcal{N}_++(\mu_+\mathbb{D}_{\mathcal{A}_+} -p_+I) \partial_t\mathcal{N}_++\mu_+ \mathbb{D}_{\partial_t\mathcal{A}_+}u_+\mathcal{N}_+,
 \\&& G_-^3= \Lbrack\mu \mathbb{D}_{\mathcal{A} }(R u )\Rbrack \mathcal{N}_- +\Lbrack\mu\mathbb{D}_{\mathcal{A}}u-pI\Rbrack\partial_t\mathcal{N}_-+\Lbrack\mu \mathbb{D}_{\partial_t\mathcal{A}}u\Rbrack \mathcal{N}_-.\nonumber
\end{eqnarray}
Here the inclusions \eqref{nnnn} guarantee that $G^1,G^3$ satisfy the same conclusions as $F^1,F^3$ listed in \eqref{SSdata}, and \eqref{pressure00} guarantees that the initial data $D_tu(0) \in \mathcal{Y}(0)$.
\end{theorem}

\begin{proof}
Since the variational formulation \eqref{Lpws} of the weak solution to \eqref{LP} has the same structure as that of the one-phase problem in  \cite{GT_lwp}, we can employ the method used in the proof of Theorem 4.3 of \cite{GT_lwp} with some minor modifications for our two-phase case.

More precisely, we first solve a pressureless problem by the Galerkin method. To this end we first construct $\{\psi^j(t)\}_{j=1}^\infty$ that is a countable basis of $\ddot{H}^2(\Omega)\cap \mathcal{X}(t)$ for each $t\in [0,T]$.  Let $\{w^j\}_{j=1}^\infty$ be a basis of $\ddot{H}^2(\Omega)\cap  {}_0H^1_\sigma(\Omega)$ (this space is clearly separable, so such a basis exists).  Then  we set $\psi^j(t):=M(t)w^j$, where $M(t)$ is the matrix defined by \eqref{M_def}.  Then Proposition \ref{M} implies that $\{\psi^j(t)\}_{j=1}^\infty$ is our desired time-dependent basis of $\ddot{H}^2(\Omega)\cap \mathcal{X}(t)$ for each $t\in [0,T]$.  Moreover, $\psi^j(t)$ is differentiable in time and $\partial_t\psi^j(t)=R(t)\psi^j(t)$.

For any integer $m\ge 1$ we define the finite dimensional space $\mathcal{X}_m(t):={\rm span}\{\psi^1(t),\dots,\psi^m(t)\}$ and write the orthogonal projection of $ \ddot{H}^2(\Omega)\cap \mathcal{X}(t)$ onto $\mathcal{X}_m(t)$ by $\mathcal{P}_t^m$. For each $m\ge 1$, we look for an approximate solution of the form
\begin{equation}\label{Galerkin11}
u^m(t):=\sum_{j=1}^m a^m_j(t) \psi^j(t),
\end{equation}
where the coefficients $a^m_j$ will be chosen so that
\begin{multline}\label{Galerkin12}
(\rho\partial_tu^m , \psi)_{\mathcal{H}^0}+\frac{1}{2}(u^m, \psi)_{\mathcal{H}^1} = ( F^1, \psi)_{\mathcal{H}^0}
- (F_+^3- \Pi_{0,+} (F_+^3(0)+\mu_+\mathbb{D}_{\mathcal{A}_{0,+}} {(\mathcal{P}_0^mu_0)}_+ \mathcal{N}_{0,+}),
\psi)_{\Sigma_+} \\
-(F_-^3-\Pi_{0,-}(F_-^3(0)-\Lbrack \mu \mathbb{D}_{\mathcal{A}_{0}} {(\mathcal{P}_0^mu_0)} \Rbrack
\mathcal{N}_{0,-}), \psi)_{\Sigma_-},
\end{multline}
for any $\psi\in \mathcal{X}_m(t)$, and
\begin{equation} \label{Galerkin21}
u^m(0)= \mathcal{P}_0^m u_0 \in \mathcal{X}_m(0).
\end{equation}
Note that in the last term in \eqref{Galerkin12} we have introduced the projection $\Pi_{0,-}$ to compensate for the fact that $u^m(0)$ need not satisfy the compatibility conditions \eqref{SScompatibility}.  One can readily deduce from \eqref{Galerkin12}--\eqref{Galerkin21} an equivalent system of ODEs for $a_j^m$, and the classical theory of ODEs guarantees the existence of its unique solution, which in turn provides the solution $u^m$ to \eqref{Galerkin12}--\eqref{Galerkin21}. Since $F^1,$ $F^3$ satisfy \eqref{SSdata}, we have $a_j^m\in C^{1,1}([0,T])$.

If we restrict the test functions in \eqref{Galerkin12} to be $\psi(t)=b^m_j(t)\psi^j$ for $b^m_j\in C^{0,1}([0,T])$, then we may differentiate the resulting equation with respect to $t$  to arrive at an identity involving $\partial_t\psi$. Subtracting  from this the equation \eqref{Galerkin12} with test function $\partial_t\psi-R\psi\in \mathcal{X}_m(t)$, we then have
\begin{multline}\label{Galerkin3}
  \langle \rho\partial_t^2u^m , \psi \rangle_{\ast}+\frac{1}{2}(\partial_tu^m, \psi)_{\mathcal{H}^1}
= \langle \partial_tF^1, \psi\rangle_{\ast}-\langle \partial_t F^3,\psi\rangle_{\Sigma}
\\
 + (F^1,(\partial_tJK+R)\psi)_{\mathcal{H}^0}-( F^3,  R\psi)_{\Sigma}
 - (\rho\partial_tu^m,(\partial_tJK+R)\psi)_{\mathcal{H}^0}-\frac{1}{2}(u^m,R\psi)_{\mathcal{H}^1}
\\
-\frac{1}{2}\int_\Omega \mu(\mathbb{D}_{\partial_t\mathcal{A}}u^m : \mathbb{D}_\mathcal{A}\psi
  + \mathbb{D}_\mathcal{A} u^m : \mathbb{D}_{\partial_t\mathcal{A}}\psi + \partial_t J K \mathbb{D}_\mathcal{A} u^m \mathbb{D}_\mathcal{A}\psi)J.
\end{multline}

By taking the test function $\psi=u^m$ in \eqref{Galerkin12}, and then using $\psi=D_t u^m$ in both \eqref{Galerkin3} and \eqref{Galerkin12}, and then applying the same computational arguments used in Theorem 4.3 of \cite{GT_lwp}, we obtain  the estimate
\begin{equation}\label{esu}
\sup_{0\le t\le T}\left\{\|u^m\|_{\mathcal{H}^1}^2+\|\partial_tu^m\|_{\mathcal{H}^0}^2\right\}
   + \|\partial_tu^m\|_{\mathcal{H}^1_T}^2 \lesssim \mathcal{Z},
\end{equation}
where here we have written $\mathcal{Z}$ for the right-hand side of \eqref{SSregularity}. By the uniform estimates \eqref{esu}, we have, up to the extraction of a subsequence,
\begin{equation}\label{climit}
u^m \wstar u\text{ weakly-}\ast\text{ in }L^\infty H^1,
\partial_t u^m \weak \partial_t u \text{ weakly-}\ast \text{ in } L^\infty L^2\text{ and weakly in }L^2H^1,
\end{equation}
and
\begin{equation}\label{es100}
\|u\|_{L^\infty H^1}^2+\|\partial_tu\|_{L^\infty L^2}^2 + \|\partial_tu\|_{L^2H^1}^2\lesssim \mathcal{Z}.
\end{equation}
On the other hand, since $u^m(0)\rightarrow u_0$ in $\ddot{H}^2(\Omega)\cap \mathcal{X}(0)$ and $u_0,F^3(0)$ satisfy the compatibility condition \eqref{SScompatibility}, we have
\begin{equation}\label{climit0}
\begin{array}{lll}
\|\Pi_{0,+}(F^3_+(0) + \mu_+\mathbb{D}_{\mathcal{A}_{0,+}} u_+^m(0) \mathcal{N}_{0,+}) \|_{H^{1/2}(\Sigma_+)} \rightarrow 0,
 \\ \|\Pi_{0,-}(F^3_-(0)-\Lbrack \mu\mathbb{D}_{\mathcal{A}_{0}}{u^m(0)}\Rbrack \mathcal{N}_{0,-})\|_{H^{1/2}(\Sigma_-)} \rightarrow 0.
\end{array}
\end{equation}
Hence, passing to the limit in \eqref{Galerkin12}, by \eqref{climit}, \eqref{climit0}, we deduce that for a.e. $t$,
\begin{equation}\label{pless}
( \rho\partial_tu  , \psi)_{\mathcal{H}^0}+\frac{1}{2}(u , \psi)_{\mathcal{H}^1} =( F^1, \psi)_{\mathcal{H}^0}
-( F^3, \psi)_{\Sigma} \text{ for any }\psi\in \mathcal{X}(t).
\end{equation}
Here $(\cdot,\cdot)_\Sigma$ is the $L^2$ inner product on $\Sigma$.

Now that we have obtained a pressureless solution $u$, we introduce the pressure. Define the functional $\Lambda_t\in
(\mathcal{H}^1(t))^\ast$ so that $\Lambda_t(v)$ equals the difference between the left and right hand sides of  \eqref{pless}, with $\psi$ replaced by $v\in \mathcal{H}^1(t)$. Then $\Lambda_t=0$ on $\mathcal{X}(t)$, so by Proposition \ref{Pressure} there exists a unique $p(t)\in \mathcal{H}^0(t)$ so that $(p(t),\diverge_\mathcal{A }{v})_{\mathcal{H}^0} =\Lambda_t(v)$ for all $v\in \mathcal{H}^1(t)$. This is equivalent to  \begin{equation}\label{weaksolution}
( \rho\partial_tu  , v)_{\mathcal{H}^0}+\frac{1}{2}(u , v)_{\mathcal{H}^1}-(p,\diverge_\mathcal{A}{v})_{\mathcal{H}^0} =( F^1, v)_{\mathcal{H}^0} - ( F^3 , v)_{\Sigma} \text{ for any }v\in
\mathcal{H}^1(t),
\end{equation}
which implies in particular that $(u,p)$ is  the unique weak solution to \eqref{LP} in the sense of \eqref{Lpws}.

Observe that for a.e. $t\in[0,T]$, $(u(t),p(t))$ is the unique weak solution to the elliptic problem \eqref{cSA}, with $F^1$ replaced by $F^1(t)-\rho\partial_t u(t)$, $F^2=0$, and $F^3$ replaced by $F^3(t)$. We then apply Theorem \ref{cSAt2} to find for $r=2,3$,
\begin{equation}\label{esofu}
\|u(t)\|_{r}^2+\|p(t)\|_{ r-1 }^2 \lesssim \|\partial_tu(t)\|_{  r-2 }^2+\|F^1(t)\|_{ r-2
}^2+\|F^3(t)\|_{  r-3/2 }^2.
\end{equation}
When $r=2$ we take the essential supremum of \eqref{esofu} over $t\in[0,T]$, and when $r=3$ we integrate over $[0,T]$; we find that $(u,p)$ is a strong solution to \eqref{LP} and
\begin{equation}\label{es11}
\|u\|_{L^\infty  {H}^2}^2
   + \|u\|_{L^2 {H}^3}^2+\|p\|_{L^\infty  {H}^1}^2
   + \|p\|_{L^2 {H}^2}^2\lesssim \mathcal{Z}.
\end{equation}

Now we compute $p(0)$ and $\partial_tu(0)$.  By the estimates that we already have, we find that $u\in C([0,T];{}_0{H}^1(\Omega) \cap \ddot{H}^2(\Omega))$. On the other hand, integrating \eqref{Galerkin3} in time from $0$ to $T$ and passing to the limit as $m\rightarrow\infty$, we know that $\rho \partial_t^2u^m\rightarrow \rho\partial_t^2u$ in
$(\mathcal{H}^1_T)^\ast$ and that $\|\rho\partial_t^2u\|_{(\mathcal{H}^1_T)^\ast}\lesssim \mathcal{Z}$. It is more natural to regard $\rho\partial_t^2u \in (\mathcal{X}_T)^\ast$ since the action of $\rho\partial_t^2 u$ is defined with test functions in $ \mathcal{X}_T$. However, since $\mathcal{X}_T \subset  \mathcal{H}^1_T$, the usual theory of Hilbert spaces provides a unique operator $E: ( \mathcal{X}_T)^\ast\to (\mathcal{H}^1_T)^\ast$ with the property that $Ef\vert_{\mathcal{X}_T} = f$ and $\|Ef\|_{( \mathcal{H}^1_T)^\ast}= \|f\|_{(\mathcal{X}_T)^\ast}$ for all $f \in ( \mathcal{X}_T)^*$.  Using this $E$, we regard $\rho\partial_t^2u \in ( \mathcal{X}_T)^\ast$ as an element of $( \mathcal{H}^1_T)^\ast$ in a natural way. This also implies that  $\partial_tu\in C([0,T];L^2(\Omega))$. Then using the first equation in \eqref{LP} we have $\nabla_{\mathcal{A}}p\in C([0,T];L^2(\Omega))$; using the third equation in \eqref{LP} and the trace theorem, we have $ p_+\in C([0,T]; H^{1/2}(\Sigma_+))$ and then Poincar\'e's inequality (Lemma \ref{poincare}) implies that $p_+\in C([0,T]; H^1(\Omega_+))$. Hence, by the trace theorem  we have $ p_+\in C([0,T]; H^{1/2}(\Sigma_-))$. Then using the fourth equation in \eqref{LP}, we have $ p_-\in C([0,T];
H^{1/2}(\Sigma_-))$, which implies $ p_-\in C([0,T]; H^1(\Omega_-))$. Hence $p\in C([0,T]; \ddot{H}^1(\Omega))$. These time continuity results allow us to evaluate \eqref{LP} at $t=0$. We first derive the equation for $p(0)$, i.e. \eqref{pressure00}. First, the Dirichlet condition for $p_+(0)$ on $\Sigma_+$ and the Dirichlet jump condition for $\Lbrack p(0)\Rbrack$ on $\Sigma_-$ are easily deduced from  $S_{\mathcal{A}_{0,+}}(p_+(0),u_{0,+}) \mathcal{N}_{0,+} = F_+^3(0)$ on $\Sigma_+$  and $\Lbrack S_{\mathcal{A}_0}(p_0,u_0)\Rbrack \mathcal{N}_{0,-}=-F_-^3(0)$  on $\Sigma_-$, respectively. Next, to deduce the PDE for $p(0)$ in $\Omega$, the Neumann jump condition  on $\Sigma_-$ and the Neumann condition on $\Sigma_b$, we divide \eqref{LP} by $\rho$ and then multiply the resulting identity
by $\nabla_{\mathcal{A}}\varphi$ for any $\varphi \in C^\infty(\Omega)$ with $\varphi =0$ on $\Sigma_+$; since $\partial_t u - Ru\in \mathcal{X}(t)$, we obtain
\begin{equation}\label{lll}
(Ru+\rho^{-1}(\nabla_{\mathcal{A}}p-\mu\Delta_{\mathcal{A}}u-F^1),\nabla_{\mathcal{A}}\varphi)_{\mathcal{H}^0}=0
\text{ for all such }\varphi.
\end{equation}
Evaluating \eqref{lll} at $t=0$, integrating by parts over $\Omega$ and then employing a density argument, we deduce that
\begin{multline}
(\rho^{-1}(\nabla_{\mathcal{A}_0}p(0)-F^1(0)) \nabla_{\mathcal{A}_0}\varphi)_{\mathcal{H}^0} = (\diverge_{\mathcal{A}_0}(R(0)u_0),\varphi) \\
+ \langle\Lbrack\rho^{-1}\mu \Delta_{\mathcal{A}_{0}}u_{0}\Rbrack \cdot \mathcal{N}_{0,-}|\mathcal{N}_{0,-}|^{-1},\varphi\rangle
+\langle \rho_-^{-1}\mu_-\Delta_{\mathcal{A}_{0,-}}u_{0,-}\cdot\nu,\varphi\rangle.
\end{multline}
This implies that $p(0)$ is the unique weak solution to \eqref{pressure00} in the sense of \eqref{Apoisson2} and then
$p(0)\in \ddot{H}^1(\Omega)$. This allows us to define $\partial_tu(0)$ as in \eqref{partialtu0} so that $D_tu(0)\in
\mathcal{Y}(0)$.

It remains to derive \eqref{LPt}, which is the PDE satisfied by $D_t u$. Integrating \eqref{Galerkin3} in time from $0$ to $T$,  sending  $m\rightarrow\infty$ and  subtracting the resulting identity from the equation \eqref{weaksolution} with  test function $v=R\psi$ for any $\psi\in \mathcal{X}_T$, we find
\begin{multline} \label{weakt}
\langle \rho\partial_t^2u , \psi \rangle_{\ast}+\frac{1}{2}(\partial_tu, \psi)_{\mathcal{H}^1_T}
= \langle \partial_tF^1, \psi\rangle_{\ast}-\langle \partial_t F^3,\psi\rangle_{\Sigma}
\\
+(\partial_tJK(F^1-\rho\partial_tu),\psi)_{\mathcal{H}^0_T} -(p,\diverge_{\mathcal{A}}(R\psi))_{\mathcal{H}^0_T}
\\
-\frac{1}{2}\int_\Omega \mu(\mathbb{D}_{\partial_t\mathcal{A}}u:\mathbb{D}_\mathcal{A}\psi
+\mathbb{D}_\mathcal{A}u : \mathbb{D}_{\partial_t\mathcal{A}} \psi + \partial_tJK\mathbb{D}_\mathcal{A}u :\mathbb{D}_\mathcal{A}\psi)J, \text{ for any }\psi\in \mathcal{X}_T.
\end{multline}
We define the functional $\Lambda\in (\mathcal{H}^1_T)^\ast$ so that $\Lambda(v)$ equals the difference between the left and right hand sides of  \eqref{weakt}, with $\psi$ replaced by $v\in \mathcal{H}^1_T$. Then $\Lambda=0$ on $\mathcal{X}_T$, by Proposition \ref{Pressure} there exists a unique $q \in \mathcal{H}^0_T$ so that $(q ,\diverge_\mathcal{A }{v})_{\mathcal{H}^0_T} =\Lambda(v)$ for all $v\in \mathcal{H}^1_T$. A straightforward computation  shows that $q=\partial_tp$ and then
\begin{multline}\label{weakt1}
 \langle \rho\partial_t^2u , v \rangle_{\ast}+\frac{1}{2}(\partial_tu, v)_{\mathcal{H}^1_T}
 - (\partial_tp,\diverge_{\mathcal{A}}{v})_{\mathcal{H}^0_T}
 = \langle \partial_tF^1, v\rangle_{\ast}-\langle \partial_t F^3,v\rangle_{\Sigma}
 \\
+ (\partial_tJK(F^1-\rho\partial_tu),v)_{\mathcal{H}^0_T} -(p,\diverge_{\mathcal{A}}(Rv))_{\mathcal{H}^0_T}
\\
-\frac{1}{2}\int_\Omega\mu(\mathbb{D}_{\partial_t\mathcal{A}}u:\mathbb{D}_\mathcal{A}v
  +\mathbb{D}_\mathcal{A}u:\mathbb{D}_{\partial_t\mathcal{A}}v+\partial_tJK\mathbb{D}_\mathcal{A}u:\mathbb{D}_\mathcal{A}v)J,\text{ for any }v\in \mathcal{H}_T^1,
\end{multline}
and the bound for $\partial_tp$ in \eqref{SSregularity} holds.  We replace the last two terms in \eqref{weakt1} by
\begin{multline} \label{i1}
 -(p,\diverge_{\mathcal{A}}(Rv))_{\mathcal{H}^0_T}
 -\frac{1}{2}\int_\Omega\mu(\mathbb{D}_{\partial_t\mathcal{A}}u:\mathbb{D}_\mathcal{A}v
  +\mathbb{D}_\mathcal{A}u:\mathbb{D}_{\partial_t\mathcal{A}}v+\partial_tJK\mathbb{D}_\mathcal{A}u:\mathbb{D}_\mathcal{A}v)J
  \\
 =(R^T\nabla_\mathcal{A}p+\mu\diverge_{\mathcal{A}}(R\mathbb{D}_\mathcal{A}u+\mathbb{D}_{\partial_t\mathcal{A}}u),v)_{\mathcal{H}_T^0}
 \\
 -\langle p_+\partial_t\mathcal{N}_++\mu_+\mathbb{D}_{\mathcal{A}_+}u_+\partial_t\mathcal{N}_++\mu_+ \mathbb{D}_{\partial_t\mathcal{A}_+}u_+\mathcal{N}_+,v_+\rangle
\\
-\langle \Lbrack p\Rbrack \partial_t \mathcal{N}_- + \Lbrack \mu \mathbb{D}_{\mathcal{A}} u \Rbrack \partial_t \mathcal{N}_- + \Lbrack\mu \mathbb{D}_{\partial_t\mathcal{A}}u\Rbrack \mathcal{N}_-,v_-\rangle.
\end{multline}
We also replace the first two terms in \eqref{weakt1} via
\begin{equation}\label{i2}
\langle \rho\partial_t^2u , v \rangle_{\ast}=\langle \rho\partial_tD_tu , v \rangle_{\ast}
 +\langle \rho R\partial_t u , v \rangle_{ \mathcal{H}_T^0}+\langle \rho \partial_tR u , v \rangle_{ \mathcal{H}_T^0},
\end{equation}
\begin{multline}\label{i3}
\frac{1}{2}(\partial_tu, v)_{\mathcal{H}^1_T} = \frac{1}{2}(D_tu, v)_{\mathcal{H}^1_T} - (\mu\diverge_{\mathcal{A}}(\mathbb{D}_{\mathcal{A}}(Ru)),v)_{\mathcal{H}^0_T}
 \\
-(\mu_+  \mathbb{D}_{\mathcal{A}_+}(R_+u_+)\mathcal{N}_+,v_+)_{\Sigma_+}
  -(\Lbrack \mu \mathbb{D}_{\mathcal{A} }(R u )\Rbrack \mathcal{N}_-,v_-)_{\Sigma_-}.
\end{multline}
We now plug  $\eqref{i1}$--$\eqref{i3}$ into \eqref{weakt1} and then replace the $\rho\partial_tu$ term by using the first equation in \eqref{LP}; since $D_tu=\partial_tu-Ru\in \mathcal{X}_T$, we deduce that $(D_tu,\partial_tp)$ is the weak solution of \eqref{LPt} with the initial condition $D_tu(0)\in \mathcal{Y}(0)$ given by \eqref{partialtu0}.
\end{proof}

Now we investigate the higher regularity of the strong solution obtained in Theorem \ref{H2SS}. First, we need to require the  stronger assumptions on   $\eta$. To this end, we define
\begin{equation} \label{etanorm}
\begin{split}
&\mathfrak{D}(\eta):
=\|\eta\|_{L^2H^{4N+1/2}}^2 +\|\partial_t\eta\|_{L^2H^{4N-1/2}}^2 + \sum_{j=2}^{2N+1}\|\partial_t^j\eta\|_{L^2
H^{4N-2j+5/2}}^2 ,
\\
&\mathfrak{E}(\eta): =\sum_{j=0}^{2N}\|\partial_t^j\eta\|_{L^\infty H^{4N-2j}}^2, \quad
\mathfrak{R}(\eta) := \mathfrak{E}(\eta)+\mathfrak{D}(\eta),
\\
& \mathfrak{E}_0(\eta) :=\|\eta(0)\|_{H^{4N }}^2+\|\partial_t \eta (0)\|_{H^{4N-1}}^2
+\sum_{j=2}^{2N} \|\partial_t^j\eta(0)\|_{H^{4N-2j+3/2}}^2.
\end{split}
\end{equation}
Note that in all these norms, the temporal interval is assumed to be $[0,T]$.  Throughout the rest of Section \ref{sec_tdas} we will assume that $\mathfrak{R}(\eta),\mathfrak{E}_0(\eta)\le 1$, which implies that $\mathcal{Q}(\mathfrak{R}(\eta))\lesssim 1+\mathfrak{R}(\eta)$ and $\mathcal{Q}(\mathfrak{E}_0(\eta))\lesssim 1+\mathfrak{E}_0(\eta)$ for any polynomial $\mathcal{Q}$. Note that $\mathcal{K}(\eta)\le \mathfrak{E} (\eta)\le \mathfrak{R}(\eta)$, where $\mathcal{K}(\eta)$ is defined by \eqref{Kdef}, and also that $\|\eta(0)\|_{H^{4N-1/2}}^2\le \mathfrak{E}_0(\eta)$.

In order to define the forcing terms and initial data for the problems that result  from temporally differentiating \eqref{LP} several times, we define some mappings. Given $F^1,F^3,v,q$ we define the mappings for forcing terms
\begin{equation}\label{gdef}
\begin{array}{ll}
\mathfrak{G}^1(v,q)=-\mu(R+\partial_tJK)\Delta_\mathcal{A}v-\rho\partial_tRv+(\partial_tJK+R-R^T)\nabla_\mathcal{A}q
 \\ \qquad\qquad\ \ + \mu\diverge_{\mathcal{A}}(\mathbb{D}_\mathcal{A}(Rv)+R\mathbb{D}_\mathcal{A}v+\mathbb{D}_{\partial_t\mathcal{A}}v) &\text{ on }\Omega,
\\  \mathfrak{G}^3_+(v,q)=\mu_+\mathbb{D}_{\mathcal{A}_+}(R_+v_+)\mathcal{N}_++(\mu_+\mathbb{D}_{\mathcal{A}_+}v_+ -q_+I) \partial_t\mathcal{N}_++\mu_+ \mathbb{D}_{\partial_t\mathcal{A}_+}v_+\mathcal{N}_+ &\text{ on }\Sigma_+,
 \\  \mathfrak{G}_-^3(v,q)= \Lbrack\mu \mathbb{D}_{\mathcal{A} }(R v )\Rbrack \mathcal{N}_- +\Lbrack\mu\mathbb{D}_{\mathcal{A}}v-qI\Rbrack\partial_t\mathcal{N}_-+\Lbrack\mu \mathbb{D}_{\partial_t\mathcal{A}}v\Rbrack \mathcal{N}_- &\text{ on }\Sigma_-,
\end{array}
\end{equation}
and the mappings for initial data
\begin{equation}\label{fdef}
\begin{array}{ll}
\mathfrak{G}^0(F^1,v,q)=\rho^{-1} (\mu\Delta_{\mathcal{A}}v-\nabla_{\mathcal{A}}q +F^1) -Rv &\text{ on }\Omega,
\\\mathfrak{f}^1(F^1,v)= \diverge_{\mathcal{A}} (\rho^{-1}F^1-R v)&\text{ on }\Omega,
\\\mathfrak{f}^2(F^3,v)=(F^3_+\cdot \mathcal{N}_{ +}+\mu_+\mathbb{D}_{\mathcal{A}_{ +}}v_+\mathcal{N}_{ +}\cdot \mathcal{N}_{+})\mathcal{N}_{ +}|\mathcal{N}_{ +}|^{-2}&\text{ on }\Sigma_+,
\\\mathfrak{f}^3(F^3,v)=(-F^3_-\cdot \mathcal{N}_{ -}+\Lbrack \mu \mathbb{D}_{\mathcal{A} }v \Rbrack \mathcal{N}_{ -}\cdot \mathcal{N}_{ -})\mathcal{N}_{ -}|\mathcal{N}_{ -}|^{-2}&\text{ on }\Sigma_-,
\\\mathfrak{f}^4(F^3,v)=\Lbrack\rho^{-1}( F^1+\mu \Delta_ {\mathcal{A} }v)\Rbrack\cdot \mathcal{N}_{-}&\text{ on }\Sigma_-,
\\ \mathfrak{f}^5(F^1,v)=\rho_-^{-1}(F^1_-+\mu_- \Delta_ {\mathcal{A}_- }v_-)\cdot\nu&\text{ on }\Sigma_b.
\end{array}
\end{equation}
In the above definitions we assume that $\mathcal{A}, R, \mathcal{N}$, etc. are evaluated at the same time $t$ as $F^1,F^3,v,q$. We first define  the forcing terms, assuming that $F^1,F^3,u,p$ are sufficiently regular. We write $F^{1,0}=F^1,\ F^{3,0}=F^3$ and define the forcing terms iteratively for $j=1,\dots,2N$,
\begin{equation}\label{Fjdef}
\begin{array}{lll}
F^{1,j}&:=D_t F^{1,j-1}+\mathfrak{G}^1(D_t^{j-1}u,\partial_t^{j-1}p)&\text{ on }\Omega, \\
F^{3,j}&:=D_t F^{3,j-1}+\mathfrak{G}^3(D_t^{j-1}u,\partial_t^{j-1}p)&\text{ on
}\Sigma.
\end{array}
\end{equation}
In order to estimate these forcing terms, we  define the quantities
\begin{multline}\label{Fjnorm}
\mathfrak{F}(F^1,F^3) :=\sum_{j=0}^{2N}\|\partial_t^jF^1\|_{L^2H^{4N-2j-1}}^2+\|\partial_t^jF^3\|_{L^2H^{4N-2j-1/2}}^2
\\
+\sum_{j=0}^{2N-1}\|\partial_t^jF^1\|_{L^\infty H^{4N-2j-2}}^2 + \|\partial_t^jF^3\|_{L^\infty H^{4N-2j-3/2}}^2,
\end{multline}
\begin{equation}\label{F0norm}
\mathfrak{F}_0(F^1,F^3) :=\sum_{j=0}^{2N-1}\|\partial_t^jF^1(0)\|_{H^{4N-2j-2}}^2 + \|\partial_t^jF^3(0)\|_{ H^{4N-2j-3/2}}^2.
\end{equation}

We now turn to the construction of the initial data.  To begin, we assume that $D_t^0 u(0):=u_0\in \ddot{H}^{4N}(\Omega)$, $\eta_0\in H^{4N+1/2},$ $\mathfrak{F}_0(F^1,F^3)<\infty$ and that $\|\eta_0\|_{4N-1/2}^2 \le
\mathfrak{E}_0(\eta)\le 1$ is sufficiently small for the hypotheses of Theorems \ref{cSAt2} and \ref{Apoissont1} to hold when $k=4N$.  We will iteratively construct $D_t^j u(0)$ for $j=0,\dotsc,2N$ and $\partial_t^j p(0)$ for $j=0,\dotsc,2N-1$ by solving various PDEs with forcing terms given by the terms in \eqref{fdef}.  In order to estimate the resulting data, we need estimates for the forcing terms \eqref{fdef} in terms of $F^i$,$v$ and $q$.  Such estimates may be found in  Lemmas 4.4--4.6 of \cite{GT_lwp}  since the terms in \eqref{Fjdef} have the same structure as those estimated in \cite{GT_lwp}.  For  the sake of brevity we will not record versions of these results here, and we will take them for granted in the following discussion.

To begin the iterative construction, we first solve for all but the highest order data.  For $j=0$ we write $F^{1,0}(0) = F^{1}(0) \in \ddot{H}^{4N-2}$, $F^{3,0}(0) = F^3(0) \in \ddot{H}^{4N-3/2}$, and $D_t^0 u(0) = u_0 \in \ddot{H}^{4N}.$  Suppose now that $F^{1,\ell} \in \ddot{H}^{4N - 2\ell-2}$, $F^{3,\ell} \in \ddot{H}^{4N-2\ell-3/2}$, and $D_t^\ell u(0) \in \ddot{H}^{4N-2\ell}$ are given for $0\le \ell \le j \in [0,2N-2]$; we will define $\dt^j p(0) \in \ddot{H}^{4N-2j-1}$ as well as $D_t^{j+1}u(0) \in \ddot{H}^{4N-2j-2}$, $F^{1,j+1}(0) \in \ddot{H}^{4N-2j-4}$, and $F^{3,j+1}(0) \in \ddot{H}^{4N-2j-7/2}$, which allows us to define all but the highest order data via iteration.  We define $ \partial_t^{j}p(0)$ as the strong solution to \eqref{Apoisson} with
\begin{equation}
 \begin{split}
 f^1&= \mathfrak{f}^1(F^{1,j}(0),D_t^{j}u(0)), \, f^2=\mathfrak{f}^2(F^{3,j}(0),D_t^{j}u(0)), \,  f^3=\mathfrak{f}^3(F^{3,j}(0),D_t^{j}u(0)), \\
 f^4&= \mathfrak{f}^4(F^{3,j}(0),D_t^{j}u(0)) \text{ and } f^5=\mathfrak{f}^5(F^{1,j}(0),D_t^{j}u(0)).
\end{split}
\end{equation}
Then we define $D_t^{j+1}u(0)=\mathfrak{G}^0(F^{1,j}(0),D_t^{j}u(0),\partial_t^{j}p(0))$.

By construction, the initial data $D_t^j u(0)$ and $\dt^j p(0)$ are determined in terms of $u_0$ as well as  $\dt^\ell F^1(0)$ and $\dt^\ell F^3(0)$ for $\ell = 0,\dotsc,2N-1$.  In order to use these in Theorem \ref{H2SS} and to construct $D_t^{2N} u(0)$ and $\dt^{2N-1}p(0)$, we must enforce compatibility conditions for $j=0,\dotsc,2N-1$.  For such $j$, we say that the $j^{th}$ compatibility condition is satisfied if
\begin{equation}\label{jcompatibility}
 \begin{cases}
 D_t^j u(0) \in \ddot{H}^2(\Omega) \cap \mathcal{X}(0)\\
 \Pi_{0,+}(F^{3,j}_+(0)+\mu_+\mathbb{D}_{\mathcal{A}_{0,+}}D_t^ju_+(0)\mathcal{N}_{0,+})=0 &\text{on }\Sigma_+ \\
 \Pi_{0,-}(F^{3,j}_-(0)-\Lbrack\mu\mathbb{D}_{\mathcal{A}_{0}}D_t^ju(0)\Rbrack \mathcal{N}_{0,-})=0 &\text{on }\Sigma_-.
\end{cases}
\end{equation}
Note that the construction of $D_t^j u(0)$ and $\dt^j p(0)$ ensures that $D_t^j u(0) \in \ddot{H}^2(\Omega)$ and also  that $\diverge_{\mathcal{A}_0}(D_t^j u(0))=0$, so the condition $D_t^j u(0) \in \mathcal{X}(0) \cap \ddot{H}^2(\Omega)$ may be reduced to the conditions
\begin{equation}
D_t^j u(0) \vert_{\Sigma_b} =0 \text{ and } \left. \Lbrack D_t^j u(0)\Rbrack \right\vert_{\Sigma_-} =0.
\end{equation}

It remains only to define $\dt^{2N-1} p(0)\in H^1$ and $D_t^{2N} u(0) \in H^0$.   According to the $j=2N-1$ compatibility condition \eqref{jcompatibility}, $\diverge_{\mathcal{A}_0} D_t^{2N-1} u(0)=0$, which allows us to  define $\dt^{2N-1} p(0) \in \ddot{H}^1$ as a solution to \eqref{Apoisson} in the weak sense of \eqref{Apoisson2}.  Then we define
\begin{equation}
D_t^{2N} u(0) = \mathfrak{G}^0( F^{1,2N-1}(0), D_t^{2N-1} u(0), \dt^{2N-1} p(0)) \in L^2(\Omega).
\end{equation}
In fact, the construction of $\dt^{2N-1} p(0)$ guarantees that $D_t^{2N} u(0) \in \mathcal{Y}(0)$.  This construction, together with the estimates from \cite{GT_lwp} mentioned above, ensure that
\begin{multline}\label{initial}
\sum_{j=0}^{2N}\left(\|\partial_t^ju(0)\|_{4N-2j}^2+\|D_t^ju(0)\|_{4N-2j}^2\right)
+ \sum_{j=0}^{2N-1} \|\partial_t^jp(0)\|_{4N-2j-1}^2
\\
\lesssim(1+\mathfrak{E}_0(\eta)) (\|u_0\|_{4N}^2 + \mathfrak{F}_0(F^1,F^3)),
\end{multline}
where $\mathfrak{F}_0$ is defined by \eqref{F0norm}.

To state our result on higher regularity of solutions to \eqref{LP},
we also define the quantities
\begin{equation}
 \begin{split}
  &\mathfrak{D}(u,p):=\sum_{j=0}^{2N+1}\|\partial_t^j u\|_{L^2H^{4N-2j+1}}^2 +\sum_{j=0}^{2N} \|\partial_t^j p\|_{L^2H^{4N-2j}}^2, \\
  &\mathfrak{E}(u,p) :=\sum_{j=0}^{2N }\|\partial_t^j u\|_{L^\infty H^{4N-2j}}^2 + \sum_{j=0}^{2N-1 }\|\partial_t^j p\|_{L^\infty H^{4N-2j-1}}^2, \\
  &\mathfrak{R}(u,p) :=\mathfrak{E}(u,p)+\mathfrak{D}(u,p) \\
  &\mathfrak{E}_0(u,p) :=\sum_{j=0}^{2N }\|\partial_t^j u(0)\|_{  4N-2j }^2 + \sum_{j=0}^{2N-1}\|\partial_t^j p(0)\|_{ 4N-2j-1 }^2.
 \end{split}
\end{equation}

We now present our higher regularity result.

\begin{theorem}\label{HkSS}
Suppose that $u_0\in \ddot{H}^{4N}(\Omega)$, $\eta_0\in  {H}^{4N+1/2}(\Sigma),$ $\mathfrak{F}(F^1,F^3)<\infty$, and that $\mathfrak{R}(\eta)\le 1$ is sufficiently small so that $\mathcal{K}(\eta)$, defined by \eqref{Kdef}, satisfies the hypotheses of Theorem \ref{H2SS} and Theorem \ref{Apoissont1}. Let $D_t^j u(0)\in \ddot{H}^{4N-2j}(\Omega)$ and $\partial_t^jp(0) \in \ddot{H}^{4N-2j-1}(\Omega)$ for $j=0,\dots,2N-1$ along with $D_t^{2N} u(0)\in \mathcal{Y}(0)$ all be determined as above in terms of $u_0$ and $\partial_t^j F^1(0),\partial_t^jF^3(0)$ for $j=0,\dots,2N-1$. Suppose that for $j=0,\dots,2N-1$, the initial data satisfy the  $j^{th}$ compatibility condition \eqref{jcompatibility}. Then there exists a unique strong solution $(u,p)$ to \eqref{LP} so that
\begin{equation}\label{Hkreg}
\begin{array}{lll}
\partial_t^ju\in C([0,T];   {}_0H^1(\Omega)\cap\ddot{H}^{4N-2j}(\Omega))\cap L^2(0,T; \ddot{H}^{4N-2j+1}(\Omega))\text{ for }j=0,\dots,2N,
\\  \partial_t^jp \in C([0,T];    \ddot{H}^{4N-2j-1}(\Omega))\cap L^2(0,T;  \ddot{H}^{4N-2j}(\Omega))\text{ for }j=0,\dots,2N-1,
\\ \rho\partial_t^{2N+1}u\in (\mathcal{H}_T^1)^\ast, \text{ and } \partial_t^{2N}p\in  L^2(0,T; L^2(\Omega)).
 \end{array}
\end{equation}
The solution satisfies the estimate
\begin{equation}\label{Hkest}
 \begin{split}
  \mathfrak{E}(u,p)+\mathfrak{D}(u,p) &\lesssim \left(1+\mathfrak{E}_0(\eta)+\mathfrak{R}(\eta)\right)\exp(C(1+\mathfrak{E}(\eta))T) \\
&  \quad\times\left(\|u_0\|_{4N}^2
+\mathfrak{F}_0(F^1,F^3)+\mathfrak{F}(F^1,F^3)\right),
 \end{split}
\end{equation}
where $\mathfrak{F}_0$ is defined by \eqref{F0norm}.  Moreover, the pair $(D_t^ju,\partial_t^jp)$ satisfy
 \begin{equation}\label{LPkt}
\left\{\begin{array}{lll}
\rho \partial_t (D_t^ju )-\mu \Delta_\mathcal{A } (D_t^ju )+\nabla_\mathcal{A } (\partial_t^j{p} )=F ^{1,j}  &\text{in}\ \Omega
 \\ \diverge_\mathcal{A }(D_t^ju ) =0  \ &\text{in}\ \Omega
 \\ S_\mathcal{A_+}(\partial_t^jp_+,D_t^ju_+ )\mathcal{N}_+= F_+^{3,j} \ &\text{on }\Sigma_+
\\
 \Lbrack D_t u \Rbrack=0 ,\quad \Lbrack S_\mathcal{A}(\partial_t^j p,D_t^ju)\Rbrack \mathcal{N}_-= -F_-^{3,j} \ &\text{on }\Sigma_-
\\  D_t^j u_-=0 &\text{on }\Sigma_b
\end{array}\right.
\end{equation}
in the strong sense with initial data $(D_t^j u(0), \partial_t^j p(0))$ for $j = 0,\dots, 2N -1$, and in the weak sense of \eqref{Lpws} with initial data $D_t^{2N} u(0) \in \mathcal{Y}(0)$ for $j = 2N$.
\end{theorem}

\begin{proof}
The proof of Theorem 4.7 of \cite{GT_lwp} works in our present case as well.  We will only provide a brief sketch of the idea of the proof.  For full details we refer to \cite{GT_lwp}.   The estimates in Lemma 4.4 of \cite{GT_lwp} provide control of the forcing terms $F^{1,j},F^{3,j}$ in terms of $F^1,F^3,\eta,u,p$.  The estimate \eqref{initial} gives control of the initial data.  These and the  $j^{th}$ compatibility condition \eqref{jcompatibility} for $j=0,\dotsc,2N-1$ then allow us to iteratively apply Theorem \ref{H2SS} and Theorem \ref{cSAt2} to find that $(D_t^ju,\partial_t^jp)$ solve \eqref{LPkt} and satisfy the estimates \eqref{Hkest}.
\end{proof}

%%%%%%%%%%%%%%%%%%%%%%%%%%%%%%%%%%%%%%%%%%%%%%%
\subsubsection{Transport problem}
%%%%%%%%%%%%%%%%%%%%%%%%%%%%%%%%%%%%%%%%%%%%%%%

We  now need to solve for $\eta$, given $u$.  We do so via the the transport problem
\begin{equation}\label{teq}
\left\{\begin{array}{lll}
\partial_t\eta+u_1\partial_1\eta+u_2\partial_2\eta=u_3\ \text{ in } \mathrm{T}^2
 \\\eta(0)=\eta_0.
\end{array}\right.
\end{equation}
We solve for $\eta_\pm$ on $\Sigma_\pm$  using \eqref{teq} at the same time since they are the same type of equation.

To state the result, we define
\begin{equation}
\mathfrak{R}_{2N}(u) = \sum_{j=0}^{2N}(\|\partial_t^j u\|_{L^2H^{4N-2j+1}}^2
+ \|\partial_t^j u\|_{L^\infty H^{4N-2j }}^2).
\end{equation}
We assume that $u$ satisfies $\mathfrak{R}_{2N}(u)\le 1$ and achieves the initial data $\partial_t^ju(0)$ for $j=0,\dots, 2N$. Given $\eta_0$, we define the initial data $\partial_t^j\eta(0)$ iteratively by using \eqref{teq}:
\begin{equation}\label{eta_data}
 \partial_t^{j+1}\eta(0):=\sum_{l=0}^{j}C_j^l\partial_t^{j-l}u(0)\cdot\partial_t^l
\mathcal{N}(0) \text{ for } j=0,\dots,2N-1,
\end{equation}
where $C_j^l>0$ are constants from applying the Leibniz rule.  We also define the quantities
\begin{equation}
\mathcal{F}(\eta):=\|\eta\|_{L^\infty H^{4N+1/2}}^2\text{ and }\mathcal{F}_0(\eta):=\|\eta_0\|_{4N+1/2}^2,
\end{equation}
\begin{equation}\label{E0def}
 \mathcal{E}_0(u,\eta)=\|u_0\|_{4N}^2+\|\eta_0\|_{4N}^2.
\end{equation}

\begin{theorem}\label{transport}
Let $u$ be as above and suppose that $\eta_0$ satisfies $\mathfrak{E}_0(\eta)\le 1$ and $\mathcal{F}_0(\eta)<\infty$.  Then the problem \eqref{teq} admits a unique solution $\eta$ that satisfies $\mathfrak{R}(\eta) + \mathcal{F}(\eta) < \infty$ and achieves the initial data $\partial_t^j\eta(0)$ for $j=0,\dots,2N$. Moreover, there exists $0<T_0\le 1$, depending on $N$, so that if $0<T\le \min\{T_0,1/\mathcal{F}_0\}$, then we have the estimates
\begin{eqnarray}
&&\mathcal{F}(\eta)\lesssim \mathcal{F}_0(\eta)+T\mathfrak{R}_{2N}(u),
\\&& \mathfrak{E}(\eta)\lesssim  \mathcal{E}_0(u,\eta) + T\mathfrak{R}_{2N}(u),
\\&& \mathfrak{D}(\eta)\lesssim  \mathcal{E}_0(u,\eta) + T\mathcal{F}_0(\eta)+\mathfrak{R}_{2N}(u).
\end{eqnarray}
\end{theorem}

\begin{proof}
This is Theorem 5.4 of \cite{GT_lwp}.
\end{proof}

%%%%%%%%%%%%%%%%%%%%%%%%%%%%%%%%%%%%%%%%%%%%%%%
\subsection{Local well-posedness}\label{section_local_well_posed}
%%%%%%%%%%%%%%%%%%%%%%%%%%%%%%%%%%%%%%%%%%%%%%%

In Section \ref{section_local_well_posed} we will  apply the linear theory for the problems \eqref{LP}  and \eqref{teq} established in Section \ref{linears} to solve the nonlinear problem \eqref{nosurface}. In this situation, the forcing terms $F^1,F^3$ in \eqref{LP}  are given in terms of $u,\eta$ by
\begin{equation}
F^1(u,\eta)=  \rho W\partial_3 u - \rho u \cdot \nabla_{\mathcal{A}} u, \,
 F^3_+(\eta)= \rho_+g\eta_+ \mathcal{N}_+, \text{ and } F^3_-(\eta)=-\rj g\eta_- \mathcal{N}_-.
\end{equation}

Due to this dependence, we must estimate $F^1,F^3$ in terms of $u,\eta$ in the following lemma.

\begin{lemma}
Let $\mathfrak{F}$ be given by \eqref{Fjnorm}.  Suppose that $\mathfrak{R}(\eta)\le 1$ and $\mathfrak{R}_{2N} (u)<\infty$, then we have
\begin{equation}
  \mathfrak{F}(F^1(u,\eta),F^3(\eta)) \lesssim [1+T+\mathfrak{R}(\eta)] \mathfrak{E}(\eta) + \mathfrak{R}(\eta)
  \mathfrak{R}_{2N}(u) + (\mathfrak{R}_{2N}(u))^2.
\end{equation}
\end{lemma}
\begin{proof}
 This is Lemma 5.1 of \cite{GT_lwp}.
\end{proof}

Now we turn to the issue of initial data. Since for the full nonlinear problem the functions $\eta,u,p$ are unknown  and their evolutions are coupled to each other, we must revise the construction of initial data presented in Section \ref{linears} to include this coupling, assuming only that $u_0,\eta_0$ are given. This will also reveal the compatibility conditions that must be satisfied by $u_0$ and $\eta_0$ in order to solve the nonlinear problem \eqref{nosurface}.

Assume that $u_0,\eta_0$ satisfy $\mathcal{F}_0(\eta) <\infty$ and that $\|\eta_0\|_{4N-1/2}\le \mathcal{E}_0(u,\eta)\le 1$ is sufficiently small for the hypothesis of Theorem \ref{Apoissont1} to hold when $k=4N$. As in Section \ref{linears}, we may iteratively construct the high-order initial data starting from the low-order data, but now we must also define $\partial_t^j \eta(0)$ in each step as in \eqref{eta_data}.  The details of the construction may be carried out exactly as in Section 5.2 of \cite{GT_lwp} because the structure of the forcing terms is the same as there, and hence we may use the estimates in Lemma 5.2 of \cite{GT_lwp} in the present case.  We record the construction of the data, along with corresponding estimates in the following Lemma.

%Suppose that for $j\in[0,2N-1]$ and that $\partial_t^l u(0)$ are known for $0\le l\le j$, $\partial_t^l\eta(0)$  are known for $0\le l\le j+1$, and $\partial_t^lp(0)$ are known for $0\le l\le j-1$. We first define $ \partial_t^{j}p(0)$ as the solution to \eqref{Apoisson} with $f^1=\mathfrak{f}^1(F^{1,j}(0),D_t^{j}u(0))$,  $f^2=\mathfrak{f}^2(F^{3,j}(0),D_t^{j}u(0))$,  $f^3=\mathfrak{f}^3(F^{3,j}(0),D_t^{j}u(0))$,  $f^4=\mathfrak{f}^4(F^{3,j}(0),D_t^{j}u(0))$ and  $f^5=\mathfrak{f}^5(F^{1,j}(0),D_t^{j}u(0))$, in the strong sense for $j\le 2N-2$ and in the weak sense of \eqref{Apoisson2} for $j=2N-1$. Then we define $D_t^{j+1}u(0)=\mathfrak{G}^0(F^{1,j}(0),D_t^{j}u(0),\partial_t^{j}p(0))$, and last we define $\partial_t^{j+2}\eta(0)=\sum_{l=0}^{j+1}C_{j+1}^l\partial_t^l \mathcal{N}(0)\cdot\partial_t^{j+1-l}u(0)$. This construction ensures that  $D_t^ju(0)\in \ddot{H}^2(\Omega)$ with $\diverge_{\mathcal{A}_0}(D_t^ju(0))=0$ in $\Omega$ for $j\le 2N-1$, $D_t^{2N}u(0)\in \mathcal{Y}(0)$ and that the   estimates stated in the following lemma hold.

\begin{lemma}\label{datale}
Suppose that $u_0,\eta_0$ satisfy $\mathcal{F}_0(\eta)<\infty$  and that $\mathcal{E}_0(u,\eta)\le 1$ is sufficiently small for the hypothesis of Theorem \ref{Apoissont1} to hold when $k=4N$. Then there exist initial data $\partial_t^ju(0),\partial_t^j\eta(0)$ for $j=0,\dots,2N$ and $\partial_t^jp(0)$ for $j=0,\dots,2N-1$ that solve the appropriate PDEs at time  $t=0$ and that obey the estimates
\begin{equation}\label{data0}
\mathcal{E}_0(u,\eta)\le \mathfrak{E}_0(u,p)+\mathfrak{E}_0(\eta)\lesssim\mathcal{E}_0(u,\eta).
\end{equation}
\end{lemma}
\begin{proof}
 This may be proved as in Proposition 5.3 of \cite{GT_lwp}.
\end{proof}

Now we describe the compatibility conditions between $\eta_0$ and $u_0$ that are required to solve the problem \eqref{nosurface}.   Let $\partial_t^jp(0),F^{1,j}(0),F^{3,j}(0)$ for $j=0,\dots,2N-1$ and $\partial_t^ju(0),\partial_t^j\eta(0)$  for $j=0,\dots,2N$   be constructed in terms of $u_0,\eta_0$ as above. We say that $u_0,\eta_0$ satisfy the  $(2N)^{th}$ order compatibility conditions if for $j=0,\dots,2N-1$
 \begin{equation}\label{2Ncompatibility}
\left\{\begin{array}{ll}
 D_t^ju(0)\in \mathcal{X}(0)\cap \ddot{H}^2(\Omega)
 \\ \Pi_{0,+}(F^{3,j}_+(0)+\mu_+\mathbb{D}_{\mathcal{A}_{0,+}}D_t^ju_+(0)\mathcal{N}_{0,+})=0&\text{ on }\Sigma_+
 \\\Pi_{0,-}(F^{3,j}_-(0)-\llbracket\mu\mathbb{D}_{\mathcal{A}_{0}}D_t^ju(0)\rrbracket \mathcal{N}_{0,-})=0&\text{ on }\Sigma_-.
 \end{array}\right.
\end{equation}

 \begin{Remark}
%Due to the construction of the initial data $D_t^ju(0), D_t^j\eta(0)$ and $D_t^jp(0)$, the first condition in the $(2N)^{th}$ compatibility condition \eqref{2Ncompatibility} may be reduced to the condition $D_t^ju(0)|_{\Sigma_b}=0$ and $\Lbrack D_t^ju(0)\Rbrack=0$ on $\Sigma_-$.
If $u_0,\eta_0$ satisfy \eqref{2Ncompatibility}, then the $j^{th}$ order compatibility condition \eqref{jcompatibility} is satisfied for $j=0,\dotsc,2N-1$.
\end{Remark}

The local well-posedness of the problem \eqref{nosurface}, stated in Theorem \ref{th0}, follows directly from the following theorem by changing notations.

\begin{theorem}\label{Localth}
Assume that $u_0,\eta_0$ satisfy $\mathcal{E}_0(u,\eta),\mathcal{F}_0(\eta)<\infty$ and satisfy the $(2N)^{th}$ order compatibility conditions \eqref{2Ncompatibility}. Then there exists $T_0>0$ and $\delta_0$ such that if $\mathcal{E}_0(u,\eta)\le\delta_0$ and $0 < T \le T_0 \min\{1,1/\mathcal{F}_0(\eta)\}$, then there is a unique solution $(u,p,\eta)$ to the problem \eqref{nosurface} on $[0,T]$ so that
\begin{equation} \label{Localbound}
\mathfrak{R}(\eta)+\mathfrak{R}(u,p) \le C(\mathcal{E}_0(u,\eta) + T \mathcal{F}_0(\eta)) \text{ and }\mathcal{F}(\eta)\le C(\mathcal{F}_0(\eta) +\mathcal{E}_0(u,\eta) + T \mathcal{F}_0(\eta))
\end{equation}
for a universal constant $C>0$.  Moreover, $\eta$ is such that the mapping $\Theta(\cdot,t)$, defined by \eqref{cotr}, is a $C^{4N-2}$ diffeomorphism for each $t\in [0,T]$ when restricted to $\Omega_\pm$.
\end{theorem}

\begin{proof}
This is Theorem 6.3 of \cite{GT_lwp} with minor modifications, so we only sketch the proof. To begin with, suppose that $\mathcal{E}_0(u,\eta)\le \delta_0\le 1$ is sufficiently small so that the hypothesis of Lemma \ref{datale} is satisfied, and hence
 \begin{equation}
\mathfrak{E}_0(u,p)+\mathfrak{E}_0(\eta)\lesssim\mathcal{E}_0(u,\eta)\lesssim\delta_0.
\end{equation}
Then by Lemma 5.2 of \cite{GT_lwp}, we have
\begin{eqnarray}
\mathfrak{F}_0(F^1,F^3)  \lesssim \mathfrak{E}_0(u,p) +  \mathfrak{E}_0(\eta) \lesssim \mathcal{E}_0(u,\eta) \lesssim \delta_0.
\end{eqnarray}

We will use an iteration method to construct a sequence of approximate solutions as follows. First, we use Lemma 5.5 of \cite{GT_lwp} to extend the initial data $\partial_t^ju(0)$ to be a time-dependent function $u^0$ satisfying $\partial_t^ju^0(0)=\partial_t^ju(0)$ and the following estimate
\begin{equation}
\mathfrak{R}_{2N}(u^0)\lesssim\mathcal{E}_0(u,0)\lesssim\delta_0.
\end{equation}
Then we use Theorem \ref{transport} to  define $\eta^0$ as the solution to \eqref{teq} with this $u^0$ replacing $u$, which satisfies $\partial_t^j\eta^0(0)=\partial_t^j\eta(0)$ as well as for some small $T$ the following estimates
\begin{eqnarray}
 &&\mathcal{F}(\eta^0)\lesssim \mathcal{F}_0(\eta)+T\mathfrak{R}_{2N}(u^0),
 \\&& \mathfrak{E}(\eta^0)\lesssim  \mathcal{E}_0(u,\eta) + T\mathfrak{R}_{2N}(u^0),
 \\&& \mathfrak{D}(\eta^0)\lesssim  \mathcal{E}_0(u,\eta) + T\mathcal{F}_0(\eta)+\mathfrak{R}_{2N}(u^0).
\end{eqnarray}
Next, as in Theorem 6.1 of \cite{GT_lwp}, we may use Theorem  \ref{HkSS} and Theorem \ref{transport} to iteratively construct an infinite sequence $\{(u^m,p^m,\eta^m)\}_{m=1}^\infty$ satisfying
\begin{equation}\label{LPm}
\left\{\begin{array}{lll} \rho\partial_t u^{m+1}-\mu \Delta_{\mathcal{A}^m}u^{m+1}+\nabla_{\mathcal{A}^m} {p}^{m+1}  =\rho W^m\partial_3 u^m - \rho u^m \cdot \nabla_{\mathcal{A}^m} u^m &\text{in}\ \Omega \\
\diverge_{\mathcal{A}^m}{u^{m+1}} =0  \ &\text{in}\ \Omega \\
S_{\mathcal{A}_+^m}(p_+^{m+1},u_+^{m+1})\mathcal{N}_+^m=\rho_+g\eta_+
\mathcal{N}_+^m  \ &\text{on }\Sigma_+
\\
 \Lbrack u^{m+1}\Rbrack=0,\quad \Lbrack S_{\mathcal{A}^m}(p^{m+1},u^{m+1})\Rbrack \mathcal{N}_-^m=\rj g\eta_-^m \mathcal{N}_-^m   \ &\text{on }\Sigma_-
\\  u_-^{m+1}=0 &\text{on }\Sigma_b
\end{array}\right.
\end{equation}
where $W^m=W(\eta^m), \mathcal{A}^m=\mathcal{A}(\eta^m), \mathcal{N}^m=\mathcal{N}(\eta^m)$, and
\begin{equation}\label{teqm}
\partial_t\eta^{m+1}=u^{m+1}\cdot \mathcal{N}^{m+1} \text{ on }\Sigma,
\end{equation}
with $\partial_t^j u^m(0)=\partial_t^ju(0)$ and $\partial_t^j \eta^m(0) = \partial_t^j \eta(0)$ for $j=0,\dots, 2N$, while $\partial_t^j p^m(0) = \partial_t^j p(0)$ for $j=0, \dots, 2N-1$. Moreover, the following uniform estimates hold
\begin{equation} \label{mest}
\mathfrak{R}(\eta^m) + \mathfrak{R}(u^m,p^m) \le C(\mathcal{E}_0(u,\eta) + T \mathcal{F}_0(\eta)) \text{ and } \mathcal{F} (\eta^m) \le C(\mathcal{F}_0(\eta) + \mathcal{E}_0(u,\eta) + T \mathcal{F}_0(\eta)).
\end{equation}

Now in order to pass to the limit in \eqref{LPm} as well as \eqref{teqm}, we argue as in Theorems 6.2 and 6.3 of \cite{GT_lwp} to show that the sequence $\{(u^m,p^m,\eta^m)\}_{m=1}^\infty$ is contractive in the following sense
\begin{equation}\label{contract}
\begin{array}{ll}
\displaystyle\mathfrak{N}(u^{m+2}-u^{m+1},p^{m+2}-p^{m+1},T)\le \frac{1}{2}\mathfrak{N}(u^{m+1}-u^{m},p^{m+1}-p^{m},T)
\\\mathfrak{M}(\eta^{m+1}-\eta^{m},T)\le 2\mathfrak{N}(u^{m+1}-u^{m},p^{m+1}-p^{m},T)
\end{array}
\end{equation}
for some small $T$, where the quantities $\mathfrak{N},\mathfrak{M}$ are defined by
\begin{equation}\begin{array}{ll}
\mathfrak{N}(v,q,T)=\|v\|_{L^\infty H^2}^2
+\|v\|_{L^2H^3}^2+\|\partial_tv\|_{L^2H^0}^2+\|\partial_tv\|_{L^2H^1}^2+\|q\|_{L^\infty
H^1}^2+\|q\|_{L^2H^2}^2,
\\\mathfrak{M}(\zeta,T)=\|\zeta\|_{L^\infty H^{5/2}}^2+\|\partial_t\zeta\|_{L^2H^{3/2}}^2+\|\partial_t^2\zeta\|_{L^2H^{1/2}}^2.
\end{array}
\end{equation}
The estimates \eqref{contract} imply that the sequence $\{(u^m,p^m)\}_{m=1}^\infty$ contracts in the space whose square-norm is given by by $\mathfrak{N}$.  Then $\{\eta^m\}_{m=1}^\infty$ is Cauchy in the space defined by $\mathfrak{M}$.  On the other hand, the uniform estimates \eqref{mest} give weak convergence of the sequence in higher regularity spaces.  We may then combine the strong and weak convergence results with an interpolation argument (for details see Theorem 6.3 of \cite{GT_lwp}) to deduce that the sequence converges in a high enough regularity space (but not as high as in the bounds \eqref{mest}) for us  to pass to the limit in the equations \eqref{LPm}, \eqref{teqm}  for each $t\in[0,T]$, to find that the limit $(u,p,\eta)$ is the unique strong solution to the problem \eqref{nosurface} on $[0,T]$. Moreover, the bound \eqref{Localbound} follows from \eqref{mest} by weak lower semicontinuity. The fact that the mapping $\Theta(\cdot,t)$ is a $C^{4N-2}$ diffeomorphism for each $t \in [0,T]$ when restricted to $\Omega_\pm$ follows from directly from  the smallness of $\mathfrak{R}(\eta)$ and the Sobolev embedding theorem.
\end{proof}

%%%%%%%%%%%%%%%%%%%%%%%%%%%%%%%%%%%%%%%%%%%%%%%
\subsection{Global well-posedness and decay}\label{sec_nost_gwp}
%%%%%%%%%%%%%%%%%%%%%%%%%%%%%%%%%%%%%%%%%%%%%%%

We now turn to the global-in-time and decay results of Theorem \ref{th0}.  As such, we assume throughout Section \ref{sec_nost_gwp} that $\rj <0$ and that $\eta_0$ satisfies the zero-average condition \eqref{zero0}.

%%%%%%%%%%%%%%%%%%%%%%%%%%%%%%%%%%%%%%%%%%%%%%%
\subsubsection{Notation and definitions}\label{cutt}
%%%%%%%%%%%%%%%%%%%%%%%%%%%%%%%%%%%%%%%%%%%%%%%

We shall first borrow from \cite{GT_inf,GT_per} for our notational convention for derivatives.  When using space-time differential multi-indices, we will write $\mathbb{N}^{1+m} = \{ \alpha = (\alpha_0,\alpha_1,\dotsc,\alpha_m) \}$ to emphasize that the $0-$index term is related to temporal derivatives.  For just spatial derivatives we write $\mathbb{N}^m$.  For $\alpha \in \mathbb{N}^{1+m}$ we write $\partial^\alpha = \partial_t^{\alpha_0} \partial_1^{\alpha_1}\cdots \partial_m^{\alpha_m}.$ We define the parabolic counting of such multi-indices by writing $|\alpha| = 2 \alpha_0 + \alpha_1 + \cdots + \alpha_m.$  We will also write $\na f$ for the horizontal gradient of $f$, i.e. $\na f = \partial_1 f e_1 + \partial_2 f e_2$, while $\nabla f$ will denote the usual full gradient.

For a given norm $\|\cdot\|$ and an integer $k\ge 0$, we introduce the following notation for sums of spatial
derivatives:
\begin{equation}
 \|\na^k f\|^2 := \sum_{\substack{\alpha \in \mathbb{N}^2 \\  |\alpha|\le k} } \|\partial^\alpha  f\|^2 \text{ and } \|\nabla^k f\|^2 := \sum_{\substack{\alpha \in \mathbb{N}^{3} \\
 |\alpha|\le k} } \|\partial^\alpha  f\|^2.
\end{equation}
The convention we adopt in this notation is that $\na$ refers to only ``horizontal'' spatial derivatives, while $\nabla$ refers to full spatial derivatives.   For space-time derivatives we add bars to our notation:
\begin{equation}
 \| \bna^k f\|^2 := \sum_{\substack{\alpha \in \mathbb{N}^{1+2} \\  |\alpha|\le k} } \|\partial^\alpha  f\|^2 \text{ and }
\| \bar{\nabla}^k f\|^2 := \sum_{\substack{\alpha \in
\mathbb{N}^{1+3}
\\   |\alpha|\le k} } \|\partial^\alpha  f\|^2.
\end{equation}
We allow for composition of derivatives in this counting scheme in a natural way; for example, $\norm{\na \na^k f} = \norm{\na^k \na f} = \norm{\na^{k+1} f}$.

We will consider energies and dissipates at both the $N+2$ and $2N$ levels.  To define both at once we consider a generic integer $n\ge 3$.  We define the energy as
\begin{equation}
 \mathcal{E}_{n} =
 \sum_{j=0}^{n} \left( \|\partial_t^j u\|_{2n-2j}^2 + \|\partial_t^j \eta\|_{2n-2j}^2 \right)
 + \sum_{j=0}^{n-1} \|\partial_t^j p\|_{2n-2j-1}^2
\end{equation}
and the corresponding dissipation as
\begin{equation}
  \mathcal{D}_{n} = \sum_{j=0}^{n} \|\partial_t^j u\|_{2n-2j+1}^2 + \sum_{j=0}^{n-1} \|\partial_t^j p\|_{2n-2j}^2 +  \| \eta\|_{2n-1/2}^2 + \|\partial_t \eta\|_{2n-1/2}^2 +
\sum_{j=2}^{n+1} \|\partial_t^j \eta\|_{2n-2j+5/2}^2.
\end{equation}

Now we define the ``horizontal'' energies and dissipations with localization.  To compactify notation, we define
\begin{equation}
 \tilde{\rho} = 
\begin{cases}
 \rho_+ &\text{on } \Sigma_+ \\
 -\Lbrack \rho \Rbrack &\text{on } \Sigma_-
\end{cases}
\end{equation}
and note that since $\rj <0$, we have that  $\tilde{\rho}>0$ on $\Sigma_\pm$.  We define the energy involving only temporal derivatives by
\begin{equation}
 \bar{\mathcal{E}}^0_{n} =
 \sum_{j=0}^{n} \left( \|  \sqrt{\rho J} \partial_t^j u\|_{0}^2 + g \| \sqrt{\tilde{\rho}} \partial_t^j \eta\|_{0}^2
 \right) \text{ and }  \bar{\mathcal{D}}^0 =\sum_{j=0}^{n} \|\sqrt{\mu} \mathbb{D}\partial_t^j
 u\|_{0}^2.
\end{equation}
\begin{Remark}\label{temporal_remark}
Using Lemma \ref{Poi} and the usual Sobolev embeddings, one may readily show that if $\|\eta\|_{5/2}$ is small enough, then $\|J-1\|_{L^\infty} \le 1/2$ and
\begin{equation}
\frac{1}{2} \sum_{j=0}^{n} \| \sqrt{\rho} \dt^j u\|_{0}^2 \le \sum_{j=0}^{n} \| \sqrt{\rho J} \dt^j u\|_{0}^2 \le \frac{3}{2} \sum_{j=0}^{n} \| \sqrt{\rho} \dt^j u\|_{0}^2.
\end{equation}
\end{Remark}

To define the energies and dissipations localized away from lower boundary $\Sigma_b$, we first construct a cut-off function $\chi$.  Let $b_- = \inf_{\mathrm{T}^2} b >0$ and then choose $\chi \in C^\infty_c(\mathbb{R})$ so that
\begin{equation}\label{chi_properties}
\supp(\chi) \subset [-3b_-/4, 2] \text{ and }\chi(x_3) = 1 \text{ for }x_3 \in [-b_-/2, 3/2].
\end{equation}
We then define
\begin{equation}
 \bar{\mathcal{E}}^+_{n} =   \| \sqrt{\rho} \bna^{2n-1}  (\chi  u)\|_{0}^2 +\| \sqrt{\rho} \na \bna^{2n-1}  (\chi  u)\|_{0}^2 + g \|\sqrt{\tilde{\rho}} \bna^{2n-1}  \eta\|_{0}^2 + g\| \sqrt{\tilde{\rho}}\na \bna^{2n-1}  \eta\|_{0}^2
\end{equation}
and
\begin{equation}
 \bar{\mathcal{D}}^+_{n} =  \| \sqrt{\mu} \bna^{2n-1} \mathbb{D} (\chi  u)\|_{0}^2
+ \| \sqrt{\mu} \na \bna^{2n-1} \mathbb{D} (\chi  u)\|_{0}^2.
\end{equation}
We also define
\begin{equation}\label{hp}
\bar{\mathcal{E}}_{n} =  \bar{\mathcal{E}}_{n}^0 + \bar{\mathcal{E}}^+_{n} \text{ and } \bar{\mathcal{D}}_{n} =
\bar{\mathcal{D}}_{n}^0  + \bar{\mathcal{D}}^+_{n}.
\end{equation}

\begin{Remark}
Note that we only consider the energy evolution of localized terms (and their horizontal space-time derivatives) away from $\Sigma_b$.  The method employed in \cite{GT_per} involves another energy and dissipation pair for terms localized near $\Sigma_b$.   Here, our modification of the method of \cite{GT_per} frees us from the need to introduce such a lower localization.
\end{Remark}

We also define a specialized energy norm by
\begin{equation}
  \mathcal{K} := \|\nabla u\|_{L^\infty}^2 + \|\nabla^2 u\|_{L^\infty}^2
+ \sum_{i=1}^2 \| D u_i \|_{H^2(\Sigma)}^2.
\end{equation}
Note that $ \mathcal{K}\lesssim \mathcal{E}_{N+2}$.

Finally, to do the a priori estimates, we shall assume that a solution is  given on the interval $[0,T]$, that $\mathcal{G}_{2N}(T)\le \delta$ for some sufficiently small $\delta$ (in particular, small enough for the estimate of Remark \ref{temporal_remark} to hold), and that $\eta$ satisfies the zero-average condition \eqref{zerot}, which allows us to use Poincar\'e's inequality on $\mathrm{T}^2$.

%%%%%%%%%%%%%%%%%%%%%%%%%%%%%%%%%%%%%%%%%%%%%%%
\subsubsection{Energy evolution in the geometric form}
%%%%%%%%%%%%%%%%%%%%%%%%%%%%%%%%%%%%%%%%%%%%%%%

To derive the energy evolution of the pure temporal derivatives we shall use the following linear geometric formulation. Suppose that $\eta, u$ (and hence $\mathcal{A}$, etc) are known, satisfying \eqref{nosurface}.  We then consider the linear equation  for the unknowns $\zeta, v, q$ given by
\begin{equation}\label{geoform}
\left\{\begin{array}{lll}
\rho\partial_t v-\rho W\partial_3v
+ \rho u\cdot \nabla_{\mathcal{A}}v + \diverge_{\mathcal{A}} S_{\mathcal{A}}(q,v) = F^1\quad&\text{in }\Omega
\\ \diverge_{\mathcal{A}}{v}=F^2&\text{in }\Omega
\\  S_{\mathcal{A}}(q_+,v_+) \mathcal{ N}_+=\rho_+g\zeta_+ \mathcal{N}_++F^3_+&\text{on }\Sigma_+
\\\llbracket v\rrbracket=0,\quad\llbracket S_{\mathcal{A}}(q,v)\rrbracket \mathcal{N}_-=\rj g\zeta_- \mathcal{N}_--F^3_-&\text{on }\Sigma_-
\\ \partial_t\zeta=v\cdot \mathcal{N}+F^4&\text{on }\Sigma
\\ v=0 &\text{on }\Sigma_b.
\end{array}\right.
\end{equation}

We have the following natural energy identity for the system \eqref{geoform}.

\begin{lemma}\label{geoid}
Suppose that $u,\eta$ are given solutions to \eqref{nosurface} and suppose that $\zeta, v, q$ solve \eqref{geoform}. Then
\begin{multline}\label{iden0}
\frac{d}{dt}\left(\frac{1}{2}\int_\Omega \rho J|v|^2 +\frac{1}{2}\int_{\Sigma_+}
\rho_+g|\zeta_+|^2+\frac{1}{2}\int_{\Sigma_-} -\rj g|\zeta_-|^2\right)+\frac{1}{2}\int_{\Omega}\mu
J|\mathbb{D}_\mathcal{A}v|^2 \\
=\int_\Omega J(v\cdot F^1+qF^2)+\int_{\Sigma_+}(-v_+\cdot F^3_+ +\rho_+g\zeta_+ F_+^4) +\int_{\Sigma_-}(-v\cdot F_-^3-\rj g\zeta_- F_-^4 ).
\end{multline}
\end{lemma}
\begin{proof}
The equality \eqref{iden0} may be derived as in Lemma 2.1 of \cite{GT_per} by taking the dot product of \eqref{geoform} with $J v$ and integrating by parts.
\end{proof}

In order to utilize \eqref{iden0} we apply the temporal differential operator $\partial_t^l$ to \eqref{nosurface}, the resulting equations are \eqref{geoform} for $\zeta = \partial_t^l \eta,$ $v = \partial_t^l u$ and $q = \partial_t^l p$, where
\begin{equation}\label{formF1}
\begin{split}
F^1_i &= \sum_{0<m< l}C_l^m \rho\partial_t^m W\partial_t^{l-m}\partial_3u_i + \sum_{0<m\le l} C_l^m\partial_t^{l-m}\partial_t \Theta^3\partial_t^mK\partial_3u_i \\
&-\sum_{0<m\le l}C_l^m(\rho\partial_t^m(u_j\mathcal{A}_{jk})\partial_t^{l-m} \partial_k u_i + \partial_t^m \mathcal{A}_{ik} \partial_t^{l-m} \partial_k p) \\
&+ \sum_{0<m\le l}C_l^m \mu \partial_t^m \mathcal{A}_{jk} \partial_t^{l-m} \partial_k (\mathcal{A}_{is} \partial_s u_j + \mathcal{A}_{js} \partial_su_i) \\
&+ \sum_{0<m< l} C_l^m \mathcal{A}_{jk} \partial_k (\partial_t^{m}\mathcal{A}_{is} \partial_t^{l-m} \partial_s u_j + \partial_t^{m} \mathcal{A}_{js} \partial_t^{l-m} \partial_s u_i) \\
&+ \partial_t^l \partial_t \Theta^3 K\partial_3 u_i + \mathcal{A}_{jk} \partial_k(\partial_t^{l} \mathcal{A}_{is} \partial_s u_j + \partial_t^{l} \mathcal{A}_{js} \partial_su_i),
\end{split}
\end{equation}
\begin{equation} \label{formF2}
F^2 = -\sum_{0<m<l}C_l^m\partial_t^m \mathcal{A}_{ij}\partial_t^{l-m}\partial_ju_i-\partial_t^l\mathcal{A}_{ij} \partial_ju_i,
\end{equation}
\begin{equation}\label{formF31}
 \begin{split}
  F^{3}_+ &= \sum_{0<m\le l}C_l^m\partial_t^m D\eta_+(\partial_t^{l-m}\eta_+-\partial_t^{l-m}  p_+) \\
  &-\sum_{0<m\le l}C_l^m \mu_+ \left(\partial_t^m(\mathcal{N}_{j,+} \mathcal{A}_{is,+}) \partial_t^{l-m}\partial_su_{j,+}+\partial_t^m(\mathcal{N}_{j,+} \mathcal{A}_{js,+})  \partial_t^{l-m}\partial_su_{i,+}\right),
 \end{split}
\end{equation}
\begin{equation}\label{formF32}
 \begin{split}
 F^{3}_- &= - \sum_{0<m\le l} C_l^m \partial_t^m D\eta_-(\partial_t^{l-m}\eta_- -\partial_t^{l-m} \llbracket p\rrbracket) \\
 &-\sum_{0<m\le l} C_l^m\left(\partial_t^m(\mathcal{N}_{j,-} \mathcal{A}_{is}) \llbracket \mu \partial_t^{l-m} \partial_s u_j \rrbracket + \partial_t^m(\mathcal{N}_{j,-} \mathcal{A}_{js})\llbracket \mu \partial_t^{l-m} \partial_s u_i \rrbracket\right),
 \end{split}
\end{equation}
\begin{equation}\label{formF4}
 F^4=\sum_{0<m\le l}C_l^m\partial_t^m D\eta\cdot \partial_t^{l-m} u,
\end{equation}
where in each of these equalities, the terms $C_l^m$ are positive constants that result from applying the Leibniz rule.

We present the estimates of these nonlinear perturbation terms $F^1,F^2,F^3$ and $F^4$ in the following lemma, at both $2N$ and $N+2$ levels.

\begin{lemma}
Let the $F^i$ be given as above.  Then the following estimates hold.
\begin{enumerate}
 \item For $0\le l\le 2N$, we have
\begin{equation} \label{Fes1}
\|F^1\|_0^2+\|\partial_t(JF^2)\|_0^2+\|F^3\|_0^2+\|F^4\|_0^2\lesssim \mathcal{E}_{2N} \mathcal{D}_{2N}.
\end{equation}

\item  For $0\le l\le N+2$, we have
\begin{equation} \label{Fes2}
\|F^1\|_0^2+\|\partial_t(JF^2)\|_0^2+\|F^3\|_0^2+\|F^4\|_0^2\lesssim \mathcal{E}_{2N} \mathcal{D}_{N+2},
\end{equation}
and also if $N\ge 3$, then there exists a $\theta>0$ so that
\begin{equation} \label{Fes3}
\|F^2\|_0^2 \lesssim \mathcal{E}_{2N}^\theta \mathcal{E}_{N+2}.
\end{equation}

\end{enumerate}
\end{lemma}

\begin{proof}
Since our perturbations $F^1,F^2,F^3,F^4$ have the same structure as those of \cite{GT_per}, these estimates
\eqref{Fes1}--\eqref{Fes3} are recorded in Theorems 4.1--4.2 of \cite{GT_per}.
\end{proof}

We now use these estimates to estimate the evolution of $\bar{\mathcal{E}}_{2N}^0$.

\begin{Proposition}\label{temp}
It holds that
\begin{equation}\label{geoes}
\bar{\mathcal{E}}^{0}_{2N}(t)+\int_0^t \bar{\mathcal{D}}^{0}_{2N} \lesssim \mathcal{E}_{2N}(0)
 + (\mathcal{E}_{2N}(t))^{ {3}/{2}}+\int_0^t \sqrt{\mathcal{E}_{2N}}  \mathcal{D}_{2N}.
\end{equation}
\end{Proposition}
\begin{proof}
The proof follows that of Proposition 4.3 of \cite{GT_per}.  We apply $\partial_t^l,\ l=0,1,\dots,2N$ to \eqref{nosurface}, then $v=\partial_t^lu,q=\partial_t^lp$ and $\zeta = \partial_t^l\eta$ solve \eqref{geoform}  where $F^1,F^2,F^3,F^4$ are given by \eqref{formF1}--\eqref{formF4}. Applying Lemma \ref{geoid}   to these functions  and then integrating in time for $0$ to $T$, we have
\begin{multline}\label{fiden0}
 \frac{1}{2}\int_\Omega \rho J|\partial_t^lu(t)|^2+\frac{1}{2}\int_{\Sigma_+}  \rho_+g|\partial_t^l\eta_+(t)|^2+\frac{1}{2}\int_{\Sigma_-} -\rj g|\partial_t^l\eta_-(t)|^2 +\frac{1}{2}\int_0^t\int_{\Omega}\mu J|\mathbb{D}_\mathcal{A}\partial_t^lu|^2
\\
=\frac{1}{2}\int_\Omega \rho J|\partial_t^lu(0)|^2+\frac{1}{2}\int_{\Sigma_+}  \rho_+g|\partial_t^l\eta_+(0)|^2+\frac{1}{2}\int_{\Sigma_-} -\rj g|\partial_t^l\eta_-(0)|^2
\\
+ \int_0^t\int_\Omega J(\partial_t^lu\cdot F^1+\partial_t^lpF^2)
\\
+ \int_0^t\int_{\Sigma_+}(-\partial_t^lu_+\cdot F^3_+ + \rho_+ g \partial_t^l \eta_+ F_+^4) + \int_0^t \int_{\Sigma_-} (-\partial_t^l u \cdot F_-^3-\rj g\partial_t^l\eta_- F_-^4 ).
\end{multline}

We now estimate the right hand side of \eqref{fiden0}. We first estimate the  three terms involves $F^1,F^3,F^4$. For the $F^1$ term, by \eqref{Fes1}, we have
\begin{equation}\label{tempor1}
\int_0^t \int_\Omega J \partial_t^l u \cdot F^1  \le  \int_0^t \|\partial_t^l u\|_0 \|J\|_{L^\infty} \|F^1\|_0  \lesssim \int_0^t \sqrt{{\mathcal{E}_{2N}}} \mathcal{D}_{2N}.
\end{equation}
For the $F^3$ and $F^4$ terms, by \eqref{Fes1} and the trace theorem,  we have
\begin{multline}\label{tempor2}
\int_0^t \int_{\Sigma_+} (-\partial_t^l u_+ \cdot F^3_+ + \rho_+ g \partial_t^l \eta_+ F_+^4) + \int_0^t \int_{\Sigma_-} (-\partial_t^l u \cdot F_-^3 - \rj  g \partial_t^l \eta_- F_-^4 )
\\
\le \int_0^t \|\partial_t^lu\|_{L^2(\Sigma)} \|F^3\|_{0} + \|\partial_t^l \eta \|_{0} \| F^4\|_{0}
\\
\lesssim \int_0^t (\|\partial_t^l u \|_{1} + \|\partial_t^l \eta \|_{0}) \sqrt{\mathcal{E}_{2N} \mathcal{D}_{2N}} \lesssim \int_0^t  \sqrt{\mathcal{E}_{2N} } \mathcal{D}_{2N}.
\end{multline}
Next, we estimate for the $F^2$ term. Notice that we can not control $\partial_t^{2N} p$ by $\mathcal{D}_{2N}$, hence we need to integrate by parts in time to find
\begin{equation}
\int_0^t\int_\Omega J \partial_t^lpF^2 = - \int_0^t\int_\Omega \partial_t^{l-1}p \partial_t ({JF^2})+
\int_\Omega (\partial_t^{l-1}p JF^2)(t) -\int_\Omega (\partial_t^{l-1}p JF^2)(0).
\end{equation}
Then by \eqref{Fes1}, we may estimate
\begin{equation}
-\int_0^t\int_\Omega \partial_t^{l-1}p\partial_t({JF^2}) \le \int_0^t  \|\partial_t^{l-1}p\|_0\|\partial_t({JF^2})\|_0 \lesssim \int_0^t\sqrt{{\mathcal{E}_{2N}}} \mathcal{D}_{2N}.
\end{equation}
Also, it is easy to deduce that
\begin{equation} \label{tempor3}
\int_\Omega (\partial_t^{l-1}p JF^2)(t) -\int_\Omega (\partial_t^{l-1}p JF^2)(0) \lesssim \mathcal{E}_{2N}(0)
 + (\mathcal{E}_{2N}(t))^{ {3}/{2}}.
\end{equation}
Hence, we have
\begin{equation}
\int_0^t\int_\Omega J \partial_t^l p F^2 \lesssim \mathcal{E}_{2N}(0)  + (\mathcal{E}_{2N}(t))^{ {3}/{2}} + \int_0^t \sqrt{{\mathcal{E}_{2N}}} \mathcal{D}_{2N}.
\end{equation}
Now by \eqref{tempor1}, \eqref{tempor2} and \eqref{tempor3},  we deduce from \eqref{fiden0} that
\begin{multline}\label{fiden11}
\frac{1}{2} \int_\Omega \rho J|\partial_t^lu(t)|^2 + \frac{1}{2} \int_{\Sigma_+}  \rho_+ g | \partial_t^l \eta_+(t) |^2 + \frac{1}{2} \int_{\Sigma_-} - \rj g |\partial_t^l \eta_-(t) |^2 + \frac{1}{2} \int_0^t \int_{\Omega} \mu J |\mathbb{D}_\mathcal{A} \partial_t^l u|^2
\\
\lesssim \mathcal{E}_{2N}(0) + (\mathcal{E}_{2N}(t))^{ {3}/{2}}+ \int_0^t\sqrt{{\mathcal{E}_{2N}}} \mathcal{D}_{2N}.
\end{multline}

Finally, it remains to estimate the left hand side of \eqref{fiden11}. For this we write
\begin{equation}\label{j1}
J|\mathbb{D}_\mathcal{A}\partial_t^l u |^2 = |\mathbb{D} \partial_t^lu|^2+(J-1)| \mathbb{D}
\partial_t^l u|^2+J(\mathbb{D}_\mathcal{A}  \partial_t^l u + \mathbb{D}  \partial_t^lu) :(\mathbb{D}_\mathcal{A} \partial_t^l u - \mathbb{D} \partial_t^lu).
\end{equation}
For the last term in the above, since
\begin{equation}
\mathbb{D}_\mathcal{A} \partial_t^l u\pm \mathbb{D} \partial_t^l u = (\mathcal{A}_{ik} \pm \delta_{ik} ) \partial_k \partial_t^l u^j + (\mathcal{A}_{jk} \pm \delta_{jk} ) \partial_k \partial_t^l u^i,
\end{equation}
we have
\begin{equation}\label{j2}
\int_\Omega J(\mathbb{D}_\mathcal{A} \partial_t^l u + \mathbb{D} \partial_t^l u) :(\mathbb{D}_\mathcal{A} \partial_t^l u - \mathbb{D} \partial_t^l u) \lesssim \sqrt{{\mathcal{E}_{2N}}} \mathcal{D}_{2N}.
\end{equation}
On the other hand, we easily have
\begin{equation} \label{j3}
\int_{\Omega}(J-1) | \mathbb{D} \partial_t^l u|^2 \lesssim \sqrt{{\mathcal{E}_{2N}}} \mathcal{D}_{2N}.
%\text{ and } \int_\Omega \rho (J-1)| \partial_t^l u |^2(t) \lesssim \sqrt{{\mathcal{E}_{2N}}} \mathcal{D}_{2N}.
\end{equation}
Hence, by \eqref{j1}, \eqref{j2} and \eqref{j3}, we obtain from \eqref{fiden11} that
\begin{multline}
\frac{1}{2}\int_\Omega \rho  J |\partial_t^lu(t)|^2+\frac{1}{2}\int_{\Sigma_+}  \rho_+g|\partial_t^l\eta_+(t)|^2 + \frac{1}{2}\int_{\Sigma_-} -\rj g|\partial_t^l\eta_-(t)|^2
+ \frac{1}{2}\int_0^t \int_{\Omega} \mu |\mathbb{D} \partial_t^l u|^2 \\
\lesssim\mathcal{E}_{2N}(0) + (\mathcal{E}_{2N}(t))^{{3}/{2}}+ \int_0^t\sqrt{{\mathcal{E}_{2N}}} \mathcal{D}_{2N}
\end{multline}
for $0\le l\le 2n$. Summing  the above over $0\le l\le 2N$ gives \eqref{geoes}.
\end{proof}

We now record a similar result at the $N+2$ level.

\begin{Proposition}\label{temp2}
Let $F^2$ be given by \eqref{formF2} with $l=N+2$. Then
 \begin{equation}\label{geoes2}
 \frac{d}{dt}\left(\bar{\mathcal{E}}^{0}_{N+2} - \int_\Omega J\partial_t^{N+1}p  F^2\right)
 + \bar{\mathcal{D}}^{0}_{ N+2} \lesssim  \sqrt{\mathcal{E}_{2N}}  \mathcal{D}_{N+2}.
\end{equation}
\end{Proposition}

\begin{proof}
Now we follow the proof of Proposition 4.4 of \cite{GT_per}.  We again apply $\partial_t^l,\ l=0,1,\dots,N+2$ to \eqref{nosurface}, then $v=\partial_t^lu,q=\partial_t^lp$ and $\zeta=\partial_t^l\eta$ solve \eqref{geoform}  with $F^1,F^2,F^3,F^4$ given by \eqref{formF1}--\eqref{formF4}. Applying  Lemma \ref{geoid} directly to these functions, we have
\begin{multline}\label{fiden10}
\frac{d}{dt} \left( \frac{1}{2} \int_\Omega \rho J| \partial_t^l u |^2 + \frac{1}{2} \int_{\Sigma_+} \rho_+ g |\partial_t^l \eta_+|^2 + \frac{1}{2} \int_{\Sigma_-} - \rj  g | \partial_t^l\eta_- |^2\right) + \frac{1}{2} \int_{\Omega} \mu J|\mathbb{D}_\mathcal{A}\partial_t^lu|^2 \\
=
\int_\Omega J(\partial_t^lu\cdot F^1+\partial_t^lpF^2)
 \\
+ \int_{\Sigma_+}(-\partial_t^lu_+\cdot F^3_+ +\rho_+g\partial_t^l\eta_+ F_+^4)
+ \int_{\Sigma_-}(-\partial_t^lu\cdot F_-^3-\rj g\partial_t^l\eta_- F_-^4 ).
\end{multline}

We now estimate the right hand side of \eqref{fiden10}. We first estimate the  three terms involves $F^1,F^3,F^4$. For the $F^1$ term, by \eqref{Fes2}, we have
\begin{equation}\label{tempor11}
 \int_\Omega J\partial_t^lu\cdot F^1 \le   \|\partial_t^lu\|_0\|J\|_{L^\infty}\|F^1\|_0  \lesssim
 \sqrt{{\mathcal{E}_{2N}}} \mathcal{D}_{N+2}.
\end{equation}
For the $F^3$ and $F^4$ terms, by \eqref{Fes2} and the trace theorem,  we have
\begin{multline}\label{tempor12}
 \int_{\Sigma_+}(-\partial_t^lu_+\cdot F^3_+ +\rho_+g\partial_t^l\eta_+ F_+^4) + \int_{\Sigma_-}(-\partial_t^lu\cdot F_-^3-\rj g\partial_t^l\eta_- F_-^4 )
\\
\le \|\partial_t^lu\|_{L^2(\Sigma)} \|F^3\|_{0}+\|\partial_t^l\eta\|_{0} \|  F^4\|_{0}
\\
\lesssim  (\|\partial_t^lu\|_{1}+ \|\partial_t^l\eta\|_{0})  \sqrt{\mathcal{E}_{2N}  \mathcal{D}_{ N+2} }
\lesssim    \sqrt{\mathcal{E}_{2N} }\mathcal{D}_{ N+2}.
\end{multline}
Next, we estimate for the $F^2$ term. Notice again that we can not control $\partial_t^{N+2} p$ by $\mathcal{D}_{N+2}$, hence we need to integrate by parts in time to find
\begin{equation}
 \int_\Omega J \partial_t^{N+2}pF^2 = - \int_\Omega \partial_t^{N+1}p\partial_t({JF^2})+
\frac{d}{dt}\int_\Omega \partial_t^{N+1}p JF^2 .
\end{equation}
Then by \eqref{Fes2},  we may estimate
\begin{equation}
 - \int_\Omega \partial_t^{N+1}p\partial_t({JF^2}) \le    \|\partial_t^{N+1}p\|_0\|\partial_t{JF^2}\|_0 \lesssim  \sqrt{{\mathcal{E}_{2N}}} \mathcal{D}_{ N+2}.
\end{equation}
Hence, we have
\begin{equation}\label{tempor13}
\int_\Omega J \partial_t^{N+2}pF^2\le   \frac{d}{dt}\int_\Omega \partial_t^{N+1}p JF^2 + C \sqrt{ {\mathcal{E}_{2N}} }  \mathcal{D}_{N+2}.
\end{equation}
Note that when $l\le N+1$, we do not need the integration by parts to know this term can be bounded by the right hand side of \eqref{geoes2}.  Now by \eqref{tempor11}, \eqref{tempor12} and \eqref{tempor13},  we deduce from \eqref{fiden10} that
\begin{multline}
 \frac{d}{dt}\left(\frac{1}{2}\int_\Omega \rho J|\partial_t^lu |^2-\int_\Omega \partial_t^{l-1}p JF^2+\frac{1}{2}\int_{\Sigma_+}  \rho_+g|\partial_t^l\eta_+ |^2+\frac{1}{2}\int_{\Sigma_-} -\rj g|\partial_t^l\eta_- |^2\right)
\\
+\frac{1}{2} \int_{\Omega}\mu J|\mathbb{D}_\mathcal{A}\partial_t^lu|^2\lesssim \sqrt{{\mathcal{E}_{2N}}} \mathcal{D}_{N+2}
\end{multline}
for $0\le l\le N+2$. Furthermore, we may argue as in Proposition \ref{temp} to replace $\frac{1}{2} \int_{\Omega}\mu
J|\mathbb{D}_\mathcal{A}\partial_t^lu|^2$ by $\frac{1}{2} \int_{\Omega}\mu  |\mathbb{D} \partial_t^l u|^2$ in the last equation.  Summing  the above over $0\le l\le N+2$ gives \eqref{geoes2}.
\end{proof}

%%%%%%%%%%%%%%%%%%%%%%%%%%%%%%%%%%%%%%%%%%%%%%%
\subsubsection{Energy evolution in the perturbed form}\label{sec_EEP}
%%%%%%%%%%%%%%%%%%%%%%%%%%%%%%%%%%%%%%%%%%%%%%%

In Section \ref{sec_EEP} we shall derive the localized energy evolution of the horizontal derivatives by using the following linear perturbed form for $\zeta,v,q$.
\begin{equation}\label{perform}
\left\{\begin{array}{lll}\rho\partial_t v-\mu\Delta v+\nabla q=\Phi^1\quad&\text{in }&\Omega
\\ \diverge{v}=\Phi^2&\text{in }&\Omega
\\\llbracket q_+I-\mu_+\mathbb{D}(v_+)\rrbracket e_3=\rho_+g\zeta_+ e_3 +\Phi_+^3 &\text{on }&\Sigma_+
\\\llbracket v\rrbracket=0,\quad \llbracket qI-\mu\mathbb{D}(v)\rrbracket e_3=\rj g\zeta_- e_3-\Phi_-^3 &\text{on }&\Sigma_-
\\ \partial_t\zeta-v_3=\Phi^4&\text{on }&\Sigma
\\ v=0 &\text{on }&\Sigma_b.
\end{array}\right.
\end{equation}

We have the following natural energy identity for the system \eqref{perform}.
\begin{lemma}\label{perle}
Suppose that $\zeta, v, q$ solve \eqref{perform}, then
\begin{multline}\label{perid}
\frac{d}{dt}\left(\frac{1}{2}\int_\Omega \rho |v|^2+\frac{1}{2}\int_{\Sigma_+}  \rho_+g|\zeta_+|^2+\frac{1}{2}\int_{\Sigma_-} -\rj g|\zeta_-|^2\right)+\frac{1}{2}\int_{\Omega}\mu|\mathbb{D} v|^2 \\
 =\int_\Omega v\cdot \Phi^1+q\Phi^2+\int_{\Sigma_+}-v_+\cdot \Phi_+^3+ \rho_+g\zeta_+  \Phi_+^4
+\int_{\Sigma_-}-v\cdot \Phi_-^3-\rj g\zeta_- \Phi_-^4.
\end{multline}
\end{lemma}
\begin{proof}
The identity \eqref{perle} may be derived as in Lemma 2.2 of \cite{GT_per} by taking the dot product with $v$ and integrating by parts.
%Multiplying the   equation   $\eqref{perform}_1$ by $v$, integrating by parts over $\Omega_\pm$ respectively %and then adding up to find
%\begin{align} \label{periden0}&\frac{d}{dt}\int_\Omega \frac{\rho|v|^2}{2}-\int_\Omega(qI-\mu\mathbb{D}v):%\nabla v + \int_{\Sigma_+} (q_+I-\mu_+\mathbb{D}v_+) e_3\cdot v_+-\int_{\Sigma_-}\llbracket %qI-\mu\mathbb{D}v)\rrbracket e_3\cdot v
%\nonumber\\&\quad= \int_\Omega v\cdot\Phi^1.\end{align} For the
%second term in the left hand side of \eqref{periden0} we use the
%equation $\eqref{perform}_2$ to compute
%\begin{equation}
%\label{periden1}-\int_\Omega(qI-\mu\mathbb{D}v):\nabla v=\int_\Omega
%-q\diverge{v}+\frac{ \mu|\mathbb{D}v|^2 }{2} =\int_\Omega
%-q\Phi^2+\frac{ \mu|\mathbb{D}v|^2 }{2}.\end{equation} For the third
%and fourth terms in \eqref{periden0} we use  the boundary conditions
%in \eqref{perform} to obtain
%\begin{align}\label{periden2}&\int_{\Sigma_+} (q_+I-\mu_+\mathbb{D}v_+) e_3\cdot v_+-\int_{\Sigma_-}\llbracket %qI-\mu\mathbb{D}v)\rrbracket e_3\cdot v
%\nonumber\\&\quad=\int_{\Sigma_+} \rho_+g\zeta_+v_{3,+}+v_+\cdot
%\Phi^3_++\int_{\Sigma_-} -\rj g\zeta_-v_{3 }+v \cdot \Phi^3_-
%\nonumber\\&\quad=\frac{d}{dt}\left(\frac{1}{2}\int_{\Sigma_+}
%\rho_+g|\zeta_+|^2+\frac{1}{2}\int_{\Sigma_-}-\rj g|\zeta_-|^2\right)
%\nonumber\\&\qquad+\int_{\Sigma_+}-\rho_+g\zeta_+\Phi^4_++v_+\cdot
%\Phi^3_+ +\int_{\Sigma_-} \rj g\zeta_-\Phi^4_-+v \cdot \Phi^3_-
%.\end{align} Plugging \eqref{periden1}--\eqref{periden2} into
%\eqref{periden0}, we obtain \eqref{perid}.
\end{proof}

In order to utilize \eqref{perid}, we want to apply the mixed time-horizontal differential operator $\partial^\alpha$ to \eqref{nosurface2}, with $\alpha\in \mathbb{N}^{1+2}$ so that $\alpha_0\le 2N-1$ and $|\alpha|\le 4N$. However, the lower boundary $\Sigma_b$ may not be flat, so we are not free to apply such derivatives to \eqref{nosurface2}. The idea then is to localize away from the lower boundary. For this, we will use the cutoff function $\chi$ defined at the beginning of the section, satisfying \eqref{chi_properties}.

Multiplying the equations \eqref{nosurface2} by $\chi$, which satisfies \eqref{chi_properties}, we find that $(\chi u,\chi p,\eta)$ solve  the problem
\begin{equation}\label{local1}
\left\{\begin{array}{lll}
\rho\partial_t (\chi u)-\mu\Delta (\chi u)+\nabla (\chi p)=\chi G^1+H^1\quad&\text{in }&\Omega
\\ \diverge(\chi u)=\chi G^2+H^2&\text{in }&\Omega
\\ (\chi p_+I-\mu_+\mathbb{D}(\chi u_+)) e_3= \rho_+g\eta_+ e_3+ G_+^3 &\text{on }&\Sigma_+
\\\llbracket \chi u\rrbracket=0,\quad \llbracket \chi pI-\mu\mathbb{D}(\chi u)\rrbracket e_3=\rj g\eta_- e_3-G_-^3 &\text{on }&\Sigma_-
\\ \partial_t\eta-\chi u_3=G^4&\text{on }&\Sigma
\\ \chi u=0 &\text{on }&\Sigma_b,
\end{array}\right.
\end{equation}
where
\begin{equation}\label{H1H2}
H^1 = \partial_3\chi(p e_3-2 \mu \partial_3 \chi \partial_3 u)-\mu \partial_3^2 \chi u ,\quad H^2 = \partial_3 \chi u_3.
\end{equation}
Since now $\chi u$ and $\chi p$ have the support  away from $\Sigma_b$, we are free to apply $\partial^\alpha$ ($\alpha \in \mathbb{N}^{1+2}$) to \eqref{local1} to find that $v = \chi\partial^\alpha u,q=\chi\partial^\alpha p,\zeta= \partial^\alpha\eta$ solve \eqref{perform} with $\Phi^1=\chi\partial^\alpha G^1+\partial^\alpha H^1,\Phi^2=\partial^\alpha G^2+\partial^\alpha H^2,\Phi^3=\partial^\alpha G^3,$ and $\Phi^4=\partial^\alpha G^4$.  To proceed, we present the estimates of $G^1,G^2,G^3,G^4$ in
the following lemma, at both $2N$ and $N+2$ levels.

\begin{lemma}\label{Gesle}
The following hold.
\begin{enumerate}
 \item There exists a $\theta>0$ so that
\begin{equation} \label{Ges1}
\|\bar{\nabla}^{4N-2}G^1\|_0^2+ \|\bar{\nabla}^{4N-2}G^2\|_0^2 +\| \bna^{4N-2}G^3\|_{1/2}^2
+ \| \bna^{4N-2} G^4 \|_{1/2}^2
\lesssim \mathcal{E}_{2N}^{1+\theta},
\end{equation}
\begin{multline} \label{Ges2}
\|\bar{\nabla}^{4N-2} G^1\|_0^2 + \|\bar{\nabla}^{4N-2} G^2\|_0^2 + \| \bna^{4N-2}G^3\|_{1/2}^2
+ \|\bna^{4N-2} G^4\|_{1/2}^2 \\
+ \|\bar{\nabla}^{4N-3} \partial_tG^1\|_0^2 + \|\bar{\nabla}^{4N-3} \partial_t G^2\|_0^2
+ \|\bna^{4N-3} \partial_t G^3\|_{1/2}^2 + \|\bna^{4N-3} \partial_tG^4\|_{1/2}^2
\\
\lesssim \mathcal{E}_{2N}^{\theta} \mathcal{D}_{2N},
\end{multline}
\begin{equation}\label{Ges3}
\| \nabla^{4N-1} G^1\|_0^2 + \|\nabla^{4N-1} G^2\|_0^2 +\| \na^{4N-1} G^3\|_{1/2}^2
+ \| \na^{4N-1} G^4\|_{1/2}^2
 \lesssim \mathcal{E}_{2N}^{\theta} \mathcal{D}_{2N} + \mathcal{K} \mathcal{F}_{2N}.
\end{equation}

\item There exists a $\theta>0$ so that
\begin{multline}\label{Ges4}
\|\bar{\nabla}^{2(N+2)-2} G^1\|_0^2 + \|\bar{\nabla}^{2(N+2)-2} G^2\|_0^2 + \|\bna^{2(N+2)-2} G^3\|_{1/2}^2
+ \|\bna^{2(N+2)-2} G^4\|_{1/2}^2
\\
\lesssim \mathcal{E}_{2N}^{ \theta}\mathcal{E}_{N+2},
\end{multline}
\begin{multline}\label{Ges5}
\|\bar{\nabla}^{2(N+2)-1} G^1\|_0^2+ \|\bar{\nabla}^{2(N+2)-1}G^2\|_0^2 + \|\bna^{2(N+2)-1} G^3\|_{1/2}^2
+ \|\bna^{2(N+2)-1}G^4\|_{1/2}^2
\\
+ \|\bna^{2(N+2)-2}\partial_t G^4\|_{1/2}^2 \lesssim \mathcal{E}_{2N}^{ \theta}\mathcal{D}_{N+2}.
\end{multline}

\end{enumerate}
\end{lemma}

\begin{proof}
Since our perturbations $G^1,G^2,G^3,G^4$ have the same structure as those of \cite{GT_per}, these estimates
\eqref{Ges1}--\eqref{Ges5} follow from Theorems 3.1--3.2 of \cite{GT_per}.
\end{proof}

With these estimates in hand, we may now derive an estimate for the evolution of $\bar{\mathcal{E}}^{+}_{2N}$.

\begin{Proposition}\label{hori}
For any $\varepsilon\in (0,1)$ there exists a constant $C(\varepsilon)>0$ so that
 \begin{equation} \label{peres}
\bar{\mathcal{E}}^{+}_{2N}(t)+\int_0^t \bar{\mathcal{D}}^{+}_{2N} \lesssim\bar{\mathcal{E}}^{+}_{2N}(0)+\int_0^t \mathcal{E}_{2N}^\theta  \mathcal{D}_{2N}+\sqrt{\mathcal{D}_{2N}\mathcal{K}\mathcal{F}_{2N}}
 + \varepsilon\mathcal{D}_{2N}+C(\varepsilon) \bar{\mathcal{D}}^{0}_{2N}.
\end{equation}
\end{Proposition}
\begin{proof}
Our proof is inspired by that of Proposition 5.5 of \cite{GT_per}.  Let $\alpha\in \mathbb{N}^{1+2}$ so that $\alpha_0\le 2N-1$ and $|\alpha|\le 4N$, applying Lemma \ref{perle}, we find
\begin{multline}\label{peres1}
\frac{d}{dt} \left(\frac{1}{2} \int_\Omega \rho |\partial^\alpha (\chi u)|^2+\frac{1}{2}\int_{\Sigma_+} \rho_+g|\partial^\alpha\eta_+|^2 + \frac{1}{2}\int_{\Sigma_-} - \rj g|\partial^\alpha\eta_- |^2 \right) + \frac{1}{2} \int_{\Omega} \mu |\mathbb{D} \partial^\alpha (\chi u)|^2 \\
= \int_\Omega \chi \partial^\alpha u\cdot(\chi\partial^\alpha G^1 + \partial^\alpha H^1)
+ \chi \partial^\alpha p(\partial^\alpha G^2+\partial^\alpha H^2)
\\
+\int_{\Sigma_+}-\partial^\alpha u_+\cdot \partial^\alpha G^3_++ \rho_+g\partial^\alpha\eta_+  \partial^\alpha
G^4_+ + \int_{\Sigma_-}-\partial^\alpha u\cdot \partial^\alpha G^3_- - \rj g\partial^\alpha\eta_-  \partial^\alpha G^4_-.
\end{multline}

We will estimate the right hand side of \eqref{peres1}. First we estimate the $G^1,G^2,G^3,G^4$ terms. Assume that $|\alpha|\le 4N-1$, then by the estimates \eqref{Ges2}--\eqref{Ges3} in Lemma \ref{Gesle} we have
\begin{multline}\label{m1}
\displaystyle\int_\Omega \chi \partial^\alpha u \cdot \chi\partial^\alpha G^1 + \chi \partial^\alpha p \partial^\alpha G^2  \lesssim \|\partial^\alpha u\|_0\|\partial^\alpha G^1\|_0+\|\partial^\alpha p\|_0\|\partial^\alpha G^2\|_0
\\
\lesssim \sqrt{\mathcal{D}_{2N}}\sqrt{\mathcal{E}_{2N}^{\theta}\mathcal{D}_{2N} + \mathcal{K}\mathcal{F}_{2N}}.
\end{multline}
Again by  \eqref{Ges2}--\eqref{Ges3}, together with the trace theorem, we have
\begin{multline} \label{m2}
\int_{\Sigma_+}- \partial^\alpha u_+\cdot \partial^\alpha G^3_++ \rho_+g\partial^\alpha\eta_+  \partial^\alpha G^4_+ + \int_{\Sigma_-}  - \partial^\alpha u\cdot \partial^\alpha G^3_- - \rj g\partial^\alpha\eta_-  \partial^\alpha G^4_- \\
\lesssim \|\partial^\alpha u\|_{L^2(\Sigma)} \|\partial^\alpha G^3 \|_0 + \|\partial^\alpha\eta\|_0\|\partial^\alpha G^4\|_0 \\
\lesssim \sqrt{\mathcal{D}_{2N}} \sqrt{\mathcal{E}_{2N}^{\theta}\mathcal{D}_{2N}+\mathcal{K} \mathcal{F}_{2N}}.
\end{multline}
Now we assume that $|\alpha|=4N$. Since $\alpha_0\le 2N-1,$ we have $\alpha_1+\alpha_2\ge 2$, then we can integrating by parts on the horizontal directions. We write $\partial^\alpha=\partial^\beta\partial^\gamma$ so that $|\gamma|=4N-1$. So by integrating by parts and using the estimates \eqref{Ges2}--\eqref{Ges3}, we obtain
\begin{multline}\label{m3}
\int_\Omega \chi \partial^\alpha u\cdot \chi\partial^\alpha G^1  = - \int_\Omega \chi \partial^{\alpha+\beta} u\cdot \chi\partial^\gamma G^1 \lesssim \|\partial^{\alpha+\beta} u\|_0\|\partial^\gamma G^1\|_0
\\
\lesssim \sqrt{\mathcal{D}_{2N}}\sqrt{\mathcal{E}_{2N}^{\theta}\mathcal{D}_{2N}+\mathcal{K}\mathcal{F}_{2N}}.\end{multline}
For the $G^2$ term, we do not need to integrate by parts:
\begin{equation}\label{m4}
\int_\Omega \chi \partial^\alpha p \partial^\alpha G^2 \lesssim \|\partial^\alpha p\|_0\| \partial^\gamma G^2\|_1  \lesssim \sqrt{\mathcal{D}_{2N}} \sqrt{\mathcal{E}_{2N}^{\theta} \mathcal{D}_{2N} + \mathcal{K} \mathcal{F}_{2N}}.
\end{equation}
For the $G^3$ term, we use the trace theorem to see that
\begin{multline} \label{m5}
\int_{\Sigma_+} - \partial^\alpha u_+\cdot \partial^\alpha G^3_+ + \int_{\Sigma_-}-\partial^\alpha u\cdot \partial^\alpha G^3_- \lesssim \left| \int_{\Sigma}\partial^{\alpha+\beta} u \cdot
\partial^\gamma G^3 \right|  \lesssim\|\partial^{\alpha+\beta} u \|_{H^{-1/2}(\Sigma)}\|\partial^\gamma G^3\|_{1/2}
\\
\lesssim \|\partial^{\alpha} u \|_{H^{1/2}(\Sigma)}\|\partial^\gamma G^3\|_{1/2}
\lesssim\|\partial^{\alpha} u \|_{1}\|\partial^\gamma G^3\|_{1/2}
\\
\lesssim \sqrt{\mathcal{D}_{2N}}\sqrt{\mathcal{E}_{2N}^{\theta}\mathcal{D}_{2N}+\mathcal{K}\mathcal{F}_{2N}}.
\end{multline}
For the $G^4$ term we split into two cases: $\alpha_0\ge 1$ and $\alpha_0=0$. In the former case, we have $\|\partial^{\alpha}\eta\|_{1/2}\le \sqrt{\mathcal{D}_{2N}}$, and hence
\begin{multline}\label{m6}
\int_{\Sigma_+} \rho_+g\partial^\alpha\eta_+  \partial^\alpha G^4_+
+\int_{\Sigma_-}  -\rj g\partial^\alpha\eta_-  \partial^\alpha
G^4_-
\\
\lesssim  \left|\int_{\Sigma} \partial^{\alpha+\beta}\eta  \partial^{\alpha-\beta} G^4\right|
\lesssim\| \partial^{\alpha+\beta}\eta\|_{-1/2} \|\partial^{\alpha-\beta} G^4\|_{1/2} \lesssim\|
\partial^{\alpha}\eta\|_{1/2}  \|\partial^{\alpha-\beta} G^4\|_{1/2}
\\
\lesssim \sqrt{\mathcal{D}_{2N}}\sqrt{\mathcal{E}_{2N}^{\theta}\mathcal{D}_{2N}+\mathcal{K}\mathcal{F}_{2N}}.
\end{multline}
In   the latter case, $\partial^\alpha$ only involves spatial derivatives, we may use Lemma 5.1 of \cite{GT_per} to bound
\begin{equation} \label{m7}
\int_{\Sigma_+} \rho_+g\partial^\alpha\eta_+  \partial^\alpha G^4_+ + \int_{\Sigma_-}  -\rj g\partial^\alpha\eta_-  \partial^\alpha G^4_-  \lesssim \sqrt{\mathcal{E}_{2N}}  \mathcal{D}_{2N}+\sqrt{\mathcal{D}_{2N}\mathcal{K}\mathcal{F}_{2N}}.
\end{equation}

Now we turn to estimate the $H^1,H^2$ terms. By the expression \eqref{H1H2} of $H^1,H^2$, we have
\begin{multline}
\int_\Omega \chi \partial^\alpha u\cdot \partial^\alpha H^1 + \chi \partial^\alpha p \partial^\alpha H^2 \lesssim \|\partial^\alpha u\|_0(\|\partial^\alpha p\|_0+\|\partial^\alpha u\|_1) + \|\partial^\alpha p\|_0\|\partial^\alpha u\|_0 \\
\lesssim \|\partial^\alpha u\|_0(\|\partial^\alpha p\|_0 + \|\partial^\alpha u\|_1)\lesssim \|\partial^{\alpha_0} u\|_{4N-2\alpha_0}\sqrt{\mathcal{D}_{2N}}.
\end{multline}
We use the standard Sobolev interpolation to obtain
\begin{equation}
\|\partial^{\alpha_0} u\|_{4N-2\alpha_0} \lesssim \|\partial^{\alpha_0} u\|_{0}^{\theta_1}\| \partial^{\alpha_0} u\|_{4N-2\alpha_0+1}^{1-\theta_1} \lesssim \bar{\mathcal{D}}_{2N}^{\theta_1/2} \mathcal{D}_{2N}^{(1-\theta_1)/2}.
\end{equation}
Hence, this together with Young's inequality implies
\begin{equation}\label{m8}
\int_\Omega \chi \partial^\alpha u\cdot \partial^\alpha H^1 + \chi \partial^\alpha p \partial^\alpha H^2 \lesssim (\bar{\mathcal{D}}_{2N}^0)^{\theta_1/2} \mathcal{D}_{2N}^{(1-\theta_1)/2}
\lesssim \varepsilon \mathcal{D}_{2N} + C(\varepsilon) \bar{\mathcal{D}}_{2N}^0.
\end{equation}

Consequently, in light of \eqref{m1}--\eqref{m7} and \eqref{m8}, we may integrate \eqref{peres1} from $0$ to $t$ and sum over such $\alpha$ to conclude \eqref{peres}.
\end{proof}

Now we record a similar estimate for the evolution of $\bar{\mathcal{E}}^{+}_{N+2}$.

\begin{Proposition}\label{hori2}
For any $\varepsilon\in (0,1)$ there exists a constant $C(\varepsilon)>0$ so that
\begin{equation}\label{peres2}
\frac{d}{dt} \bar{\mathcal{E}}^{+}_{N+2} + \bar{\mathcal{D}}^{+}_{N+2} \lesssim  \mathcal{E}_{2N}^\theta \mathcal{D}_{N+2} + \varepsilon \mathcal{D}_{N+2} + C(\varepsilon) \bar{\mathcal{D}}^{0}_{N+2}.
\end{equation}
\end{Proposition}
\begin{proof}
The proof is similar to Proposition \ref{hori}, except that we use  \eqref{Ges4} and \eqref{Ges5} in place of \eqref{Ges1}--\eqref{Ges3}.
\end{proof}

%%%%%%%%%%%%%%%%%%%%%%%%%%%%%%%%%%%%%%%%%%%%%%%
\subsubsection{Comparison results}
%%%%%%%%%%%%%%%%%%%%%%%%%%%%%%%%%%%%%%%%%%%%%%%

We now want to  show  that ${\mathcal{E}}_{n}$ is comparable to $ \bar{\mathcal{E}}_{n}$ and that ${\mathcal{D}} _{n}$ is comparable to $ \bar{\mathcal{D}}_{n}$, for both $n=2N$ and $n=N+2$.  We begin with the energy estimate.

\begin{theorem}\label{eth}
There exists a $\theta>0$ so that
\begin{equation}\label{e2n}
{\mathcal{E}}_{2N}\lesssim  \bar{\mathcal{E}}_{2N} + \mathcal{E}_{2N}^{1+\theta}
\end{equation}
and
\begin{equation}\label{en+2}
{\mathcal{E}}_{N+2}\lesssim  \bar{\mathcal{E}}_{N+2} + \mathcal{E}_{2N}^\theta \mathcal{E}_{N+2}.
\end{equation}
\end{theorem}

\begin{proof}
Here we modify the proof of Theorem 6.1 of \cite{GT_per}.   We first let $n$ denote either $2N$ or $N+2$ throughout the proof, and we compactly write
\begin{equation}\label{n1}
\mathcal{Z}_n=\sum_{j=0}^{n-1}\|\partial_t^{j} G^1\|_{4n-2j-2}^2  +\|\partial_t^{j} G^2\|_{4n-2j-1}^2+\|\partial_t^{j} G^3\|_{4n-2j-3/2}^2.
\end{equation}
Note that by the definitions of $\bar{\mathcal{E}}_{n}^0$ and $\bar{\mathcal{E}}_{n}^+$ as well as the estimate stated in Remark \ref{temporal_remark}, we have
\begin{equation} \label{n2}
\|\partial_t^nu\|_{0}^2+\sum_{j=0}^{n}\|\partial_t^j\eta\|_{2n-2j}^2\lesssim  \bar{\mathcal{E}}_{n}.
\end{equation}

Now we let $j=0,\dots,n-1$ and then apply $\partial_t^j$ to the equations in \eqref{nosurface2} to find
\begin{equation}\label{jellip}
\left\{
\begin{array}{lll}
-\mu\Delta \partial_t^j u+\nabla\partial_t^j p=-\rho\partial_t^{j+1} u+\partial_t^j G^1
\quad&\text{in }\Omega
\\ \diverge\partial_t^j u=\partial_t^j G^2&\text{in }\Omega
\\ \llbracket \partial_t^j p_+I-\mu_+\mathbb{D}(\partial_t^j u_+)\rrbracket e_3=\rho_+g\partial_t^j\eta_+ e_3-\partial_t^jG^3_+&\text{on }\Sigma_+
\\ \llbracket \partial_t^j u\rrbracket=0,
\ \llbracket\partial_t^j pI-\mu\mathbb{D}(\partial_t^j u)\rrbracket
e_3=\rj g\partial_t^j\eta_- e_3-\partial_t^jG_-^3&\text{on
}\Sigma_-
\\ \partial_t^j u_-=0 &\text{on }\Sigma_b.
\end{array}
\right.
\end{equation}
Applying the elliptic estimates of Theorem \ref{cStheorem} with $r=2n-2j\ge 2$ to the problem \eqref{jellip} and using \eqref{n1}--\eqref{n2}, we obtain
\begin{multline}\label{n3}
\|\partial_t^j u\|_{2n-2j}^2 + \|\partial_t^j  p\|_{2n-2j-1}^2
\\
\lesssim \|\partial_t^{j+1} u\|_{2n-2 j-2}^2 + \|\partial_t^{j} G^1\|_{2n-2j-2}^2 + \|\partial_t^{j} G^2 \|_{2n-2j-1}^2 + \|\partial_t^{j} \eta\|_{2n-2j-3/2}^2 + \|\partial_t^{j} G^3\|_{2n-2j-3/2}^2
\\
\lesssim \|\partial_t^{j+1} u\|_{2n-2(j+1)}^2+\bar{\mathcal{E}}_{n}+\mathcal{Z}_n.
\end{multline}

We claim that
\begin{equation}\label{claim}
\mathcal{E}_{n} \lesssim \bar{\mathcal{E}}_{n} + \mathcal{Z}_n.
\end{equation}
We prove the claim \eqref{claim} by a finite induction based on the estimate \eqref{n3}. For $j=n-1$, we obtain from \eqref{n3} and \eqref{n2} that
\begin{equation}
\|\partial_t^{n-1} u\|_{2}^2+\|\partial_t^{n-1} p\|_{1}^2 \lesssim\|\partial_t^{n} u\|_{0}^2
+\bar{\mathcal{E}}_{n}+\mathcal{Z}_n \lesssim\bar{\mathcal{E}}_{n} + \mathcal{Z}_n.
\end{equation}
Now suppose that the following holds for $1\le l<n$
\begin{equation}\label{n4}
\|\partial_t^{n-l} u\|_{2l}^2+\|\partial_t^{n-l}  p\|_{2l-1}^2 \lesssim\bar{\mathcal{E}}_{n}+\mathcal{Z}_n.
\end{equation}
We apply \eqref{n3} with $j=n-(l+1)$ and use the induction hypothesis \eqref{n4} to find
\begin{equation}
\|\partial_t^{n-(l+1)} u\|_{2(l+1)}^2+\|\partial_t^{2n-(l+1)}  p\|_{2(l+1)-1}^2
 \lesssim\|\partial_t^{n-l} u\|_{2l}^2+\bar{\mathcal{E}}_{n}+\mathcal{Z}_n
\lesssim\bar{\mathcal{E}}_{n}+\mathcal{Z}_n.
\end{equation}
Hence by  finite induction, the bound \eqref{n4} holds for all $l=1,\dots,n.$ Summing \eqref{n4} over  $l=1,\dots,n$ and changing the index, we then have
\begin{equation}\label{n5}
\sum_{j=0}^{n-1}\|\partial_t^j u\|_{2n-2j}^2  +\|\partial_t^j  p\|_{2n-2j-1}^2 \lesssim \bar{\mathcal{E}}_{n}+\mathcal{Z}_n.
\end{equation}
We then conclude the claim \eqref{claim} by summing \eqref{n2} and \eqref{n5}.

Finally, setting $n=2N$ in \eqref{claim}, and using \eqref{Ges1} of Lemma \ref{Gesle} to bound $\mathcal{Z}_{2N} \lesssim ({\mathcal{E}}_{2N})^{1+\theta}$, we obtain \eqref{e2n}; setting $n=N+2$ in \eqref{claim}, and using \eqref{Ges4} of Lemma \ref{Gesle} to bound $\mathcal{Z}_{N+2}\lesssim  {\mathcal{E}} _{2N}^ \theta {\mathcal{E}}_{N+2}$, we obtain \eqref{en+2}.
\end{proof}

Now we consider a similar estimate for the dissipation.

\begin{theorem}\label{dth}
There exists a $\theta>0$ so that
\begin{equation}\label{d2n}
{\mathcal{D}}_{2N}\lesssim  \bar{\mathcal{D}}_{2N} + \mathcal{K}{\mathcal{F}}_{2N} + \mathcal{E}_{2N}^\theta \mathcal{D}_{2N}
\end{equation}
and
\begin{equation}\label{dn+2}
{\mathcal{D}}_{N+2}\lesssim  \bar{\mathcal{D}}_{N+2}
+\mathcal{E}_{2N}^\theta \mathcal{D}_{N+2}.
\end{equation}
\end{theorem}

\begin{proof}
We again let $n$ denote either $2N$ or $N+2$ and compactly write
\begin{equation}\label{n11}
\mathcal{Z}_n= \| \bar{\nabla}^{2n-1} G^1\|_{0}^2+\| \bar{\nabla}^{2n-1} G^2\|_{1}^2
 +\| \bna^{2n-1} G^3\|_{1/2}^2+\| \bna^{2n-1} G^4\|_{1/2}^2+\| \bna^{2n-2} \partial_t G^4\|_{1/2}^2.
\end{equation}

First, we recall from the definition of  $\chi$ in \eqref{chi_properties} that $\chi=1$ on $\Omega_1:=\{-b_-/2 \le x_3 \le 3/2\}$, which contains $\Omega_+$ and $\Sigma_\pm$. Hence by the definition of $\bar{{\mathcal{D}}}_{n}^0$, $\bar{{\mathcal{D}}}_{n}^+$ and Korn's inequality, we obtain
\begin{equation} \label{n12}
\| \bna^{2n-1} u\|_{H^1(\Omega_1)}^2 + \| \na \bna^{2n-1}u\|_{H^1(\Omega_1)}^2
\le \| \bna^{2n-1}(\chi u)\|_{1}^2 + \| \na \bna^{2n-1}(\chi u)\|_{1}^2 \lesssim \bar{\mathcal{D}}_{n}^+
\end{equation}
and
\begin{equation} \label{n13}
\sum_{j=0}^n\|\partial_t^ju\|_{1}^2 \lesssim \bar{\mathcal{D}}_{n}^0.
\end{equation}
Summing \eqref{n12}--\eqref{n13},  we find that
\begin{equation} \label{n14}
\|\bna^{2n} u\|_{H^1(\Omega_1)}^2 \lesssim \bar{\mathcal{D}}_{n}.
\end{equation}

Notice that we have not yet derived an estimate of $\eta$ in terms of the dissipation, so we can not apply the two-phase elliptic estimates of Theorem \ref{cStheorem} as in Theorem \ref{eth}.  It is crucial to observe that  from \eqref{n14} we can get higher regularity estimates of $u$ on the boundaries $\Sigma=\Sigma_+\cup\Sigma_-$. Indeed, since $\Sigma_\pm$ are flat, we may use the definition of Sobolev norm on $\mathrm{T}^2$  and the trace
theorem to see from \eqref{n14} that
\begin{multline}\label{n21}
\|\partial_t^{j} u\|_{H^{2n-2j+1/2}(\Sigma)}^2 \lesssim \|\partial_t^{j}   u \|_{L^2(\Sigma)}^2
+\|\na^{2n-2j}\partial_t^{j} u \|_{H^{1/2}(\Sigma )}^2
\\
\lesssim \| \partial_t^{j}   u \|_{H^1(\Omega_1)}^2+\|\na^{2n-2j}\partial_t^{j}   u \|_{H^1(\Omega_1)}^2
 \lesssim \bar{\mathcal{D}}_{n}.
\end{multline}
This motivates us to use the one-phase elliptic estimates of Theorem \ref{cS1phaselemma2}.

Let $j=0,\dots,n-1$, and observe that $(\partial_t^j u_+,\partial_t^j p_+) $ solve the  problem
\begin{equation} \label{pro+1}
\left\{\begin{array}{lll}
-\mu_+\Delta \partial_t^j u_++\nabla\partial_t^j p_+=-\rho_+\partial_t^{j+1} u_++\partial_t^j G^1_+\quad&\text{in }&\Omega_+
\\ \diverge\partial_t^j u_+=\partial_t^j G^2_+&\text{in }&\Omega_+
\\ \partial_t^j u_+=\partial_t^j u_+&\text{on }&\Sigma,
\end{array}\right.
\end{equation}
and  $(\partial_t^j u_-,\partial_t^j p_-) $ solve the  problem
\begin{equation} \label{pro-1}
\left\{\begin{array}{lll}
-\mu_-\Delta \partial_t^j u_-+\nabla\partial_t^j p_-=-\rho_-\partial_t^{j+1} u_-+\partial_t^j G^1_-\quad&\text{in }&\Omega_-
\\ \diverge\partial_t^j u_-=\partial_t^j G^2_-&\text{in }&\Omega_-
\\ \partial_t^j u_-=\partial_t^j u_-&\text{on }&\Sigma_-,
\\ \partial_t^j u_-=0 &\text{on }&\Sigma_b.
\end{array}\right.
\end{equation}
We apply Theorem \ref{cS1phaselemma2} with $r=2n-2j+1$ to the problem \eqref{pro+1} for $u_+,\ p_+$ and to the  problem \eqref{pro-1} for $u_-,\ p_-$, respectively; using \eqref{n11}, \eqref{n14}, \eqref{n21} and summing up,  we find
\begin{multline}\label{n31}
\|\partial_t^j u\|_{2n-2j+1}^2+\| \nabla\partial_t^j p\|_{2n-2j-1}^2
\\
\lesssim\|\partial_t^{j+1} u\|_{2n-2j-1}^2+ \|\partial_t^j G^1\|_{2n-2j-1}^2+\|\partial_t^{j}
G^2\|_{2n-2j}^2+\|\partial_t^{j} u\|_{H^{2n-2j+1/2}(\Sigma)}^2
\\
\lesssim\|\partial_t^{j+1} u\|_{2n-2j-1}^2+\mathcal{Z}_n+  \bar{\mathcal{D}}_{n}.
\end{multline}

We now claim that
\begin{equation}\label{claim2}
\sum_{j=0}^{ n}\|\partial_t^ju\|_{2n-2j+1}^2+\sum_{j=0}^{ n-1}\|\partial_t^j\nabla p \|_{2n-2j-1} \lesssim \mathcal{Z}_n +  \bar{\mathcal{D}}_{n}.
\end{equation}
We prove the claim \eqref{claim2} by a finite induction  as in Theorem \ref{eth}. For $j=n-1$, by \eqref{n11} and \eqref{n31}, we obtain
\begin{equation}
\|\partial_t^{n-1} u\|_{3}^2+\| \nabla\partial_t^{n-1} p\|_{1}^2 \lesssim\|\partial_t^{n} u\|_{1}^2 +\mathcal{Z}_n + \bar{\mathcal{D}}_{n} \lesssim \mathcal{Z}_n+  \bar{{\mathcal{D}}}_{ n}.
\end{equation}
Now suppose that the following holds for $1\le l<n$:
\begin{equation}\label{n41}
\|\partial_t^{n-l} u\|_{ 2l+1}^2+\| \nabla\partial_t^{n-l} p\|_{2l-1}^2 \lesssim\mathcal{Z}_n+  \bar{{\mathcal{D}}}_{ n}.
\end{equation}
We apply \eqref{n31} with $j=n-(l+1)$ and use the induction hypothesis \eqref{n41} to find
\begin{equation}
\|\partial_t^{n-(l+1)} u\|_{2(l+1)+1}^2+\| \nabla\partial_t^{n-(l+1) }p\|_{2(l+1)-1}^2
\lesssim\|\partial_t^{n-l} u\|_{2l+1}^2+\mathcal{Z}_n + \bar{\mathcal{D}}_{n}
\lesssim\mathcal{Z}_n+  \bar{\mathcal{D}}_{n}.
\end{equation}
Hence the bound \eqref{n41} holds for all $l=1,\dots,n.$ We then conclude the claim \eqref{claim2} by summing this over $l=1,\dots, n$, adding \eqref{n13} and  changing the index.

Now that we have obtained \eqref{claim2}, we estimate the remaining parts in $\mathcal{D}_{n}$. We will turn to the boundary conditions in \eqref{nosurface2}. First we derive estimates for $\eta$. For the term $\dt^j \eta$ for $j\ge 2$ we use the boundary condition
\begin{equation}\label{n61}
\partial_t\eta=u_3+G^4\text{ on }\Sigma.
\end{equation}
Indeed, for $j=2,\dots,n+1$ we apply $\partial_t^{j-1}$ to \eqref{n61} to see, by \eqref{claim2} and \eqref{n11}, that
\begin{multline}\label{eta1}
\|\partial_t^j\eta\|_{2n-2j+5/2}^2 \lesssim \|\partial_t^{j-1}u_3\|_{H^{2n-2j+5/2}(\Sigma)}^2+\|\partial_t^{j-1}G^4\|_{2n-2j+5/2}^2
\\
\lesssim \|\partial_t^{j-1}u_3\|_{{2n-2(j-1)+1}}^2+\|\partial_t^{j-1}G^4\|_{2n-2(j-1)+1/2}^2
 \lesssim \mathcal{Z}_n+  \bar{\mathcal{D}}_{n}.
\end{multline}
For the term $\partial_t\eta$, we again use \eqref{n61}, \eqref{claim2} and \eqref{n11} to find
\begin{equation}\label{eta2}
\|\partial_t \eta\|_{2n-1/2}^2 \lesssim \| u_3\|_{H^{2n-1/2}(\Sigma)}^2+\|\partial_t^{j-1}G^4\|_{2n-1/2}^2
\lesssim \| u_3\|_{{2n }}^2+\| G^4\|_{2n-1/2}^2 \lesssim\mathcal{Z}_n+  \bar{\mathcal{D}}_{n}.
\end{equation}
For the remaining $\eta$ term, i.e. those without temporal derivatives, we use the boundary conditions
\begin{equation}\label{pb1}
 \rho g\eta_+= p_+- \mu_+\partial_3u_{3,+} -G_{3,+}^3\text{ on }\Sigma_+
\end{equation}
 and
\begin{equation} \label{pb2}
\rj g\eta_-=\llbracket p\rrbracket-\llbracket\mu\partial_3u_3\rrbracket +G_{3,-}^3 \text{ on } \Sigma_-.
\end{equation}
Notice that at this point we do not have any bound on $p$ on the boundary $\Sigma$, but we have bounded  $\nabla p$ in $\Omega$.  Applying $\partial_1$, $\partial_2$  to \eqref{pb1} and \eqref{pb2}, respectively,  by \eqref{claim2} and \eqref{n11}, we obtain
\begin{equation}\label{n51}
\begin{split}
\| \na \eta\|_{2n-3/2}^2 &\lesssim \|  \na p \|_{H^{2n-3/2}(\Sigma)}^2 + \|  \na \partial_3u_3 \|_{H^{2n-3/2}(\Sigma)}^2 +\|\na G^3\|_{ 2n-3/2 }^2  \\
& \lesssim \|\nabla p\|_{2n-1}^2 + \|u_3\|_{2n+1}^2 + \|G^3\|_{2n-1/2}^2 \lesssim \mathcal{Z}_n+  \bar{\mathcal{D}}_{n}.
\end{split}
\end{equation}
Since $\int_{\mathbb{T}^2}\eta=0$, we may then use Poincar\'e's inequality on $\Sigma_\pm$   to obtain from \eqref{n51} that
\begin{equation} \label{eta3}
\|\eta\|_{2n-1/2}^2\lesssim \|\eta\|_{0}^2 + \|\na \eta\|_{2n-3/2}^2
\lesssim \|\na \eta\|_{2n-3/2}^2 \lesssim \mathcal{Z}_n+  \bar{{\mathcal{D}}}_{ n}.
\end{equation}
Summing  \eqref{eta1}, \eqref{eta2}  and \eqref{eta3}, we complete the estimates for $\eta$:
\begin{equation} \label{eta0}
\|\eta\|_{2n-1/2}^2+\|\partial_t\eta\|_{2n-1/2}^2 + \sum_{j=2}^{n+1} \|\partial_t^j\eta\|_{2n-2j+5/2}\lesssim\mathcal{Z}_n+ \bar{\mathcal{D}}_{n}.
\end{equation}

It remains to bound $\|\partial_t^jp\|_0$. Applying $\partial_t^j,\ j=0,\dots,n-1$ to \eqref{pb1}--\eqref{pb2} and employing \eqref{claim2}, \eqref{eta0} and \eqref{n11}, we find
\begin{multline}\label{pp1}
\|\partial_t^j p_+\|_{L^2(\Sigma_+)}^2 + \|\llbracket \partial_t^j p\rrbracket\|_{L^2(\Sigma_-)}^2
\lesssim \|\partial_t^j \eta\|_{0}^2 + \|\partial_3\partial_t^j u_3 \|_{L^2(\Sigma)}^2 + \|\partial_t^j G^3\|_{0}^2
\\
\lesssim \|\partial_t^j \eta\|_{0}^2+\| \partial_t^j u_3 \|_{2}^2+\|\partial_t^j G^3\|_{0}^2 \lesssim \mathcal{Z}_n + \bar{\mathcal{D}}_{n}.
\end{multline}
By Poincar\'e's inequality on $\Omega_+$ (Lemma \ref{poincare}) and \eqref{claim2} and \eqref{pp1}, we have
\begin{multline}\label{pp2}
\|\partial_t^j p_+\|_{H^1(\Omega_+)}^2 = \|\partial_t^j p_+\|_{L^2(\Omega_+)}^2
+ \|\nabla \partial_t^j p_+\|_{L^2(\Omega_+)}^2
\lesssim \|\partial_t^j p_+ \|_{L^2(\Sigma_+)}^2
+  \|\nabla\partial_t^j p_+ \|_{L^2(\Omega_+)}^2
 \\
\lesssim \mathcal{Z}_n + \bar{\mathcal{D}}_{n}.
\end{multline}
On the other hand, by the trace theorem and \eqref{pp1}--\eqref{pp2}, we have
\begin{multline}\label{pp3}
\|\partial_t^j p_-\|_{L^2(\Sigma_-)}^2\le \|\partial_t^j p_+\|_{L^2(\Sigma_-)}^2+\|\llbracket \partial_t^j p\rrbracket\|_{L^2(\Sigma_-)}^2  \lesssim \|\partial_t^j p_+\|_{H^1(\Omega_+)}^2+\|\llbracket \partial_t^j p\rrbracket\|_{L^2(\Sigma_-)}^2
\\
\lesssim\mathcal{Z}_n+  \bar{{\mathcal{D}}}_{n},
\end{multline}
so again by Poincar\'e's inequality on $\Omega_-$ as well as \eqref{claim2} and \eqref{pp1}, we have
\begin{multline}\label{pp4}
\|\partial_t^j p_-\|_{H^1(\Omega_-)}^2 = \|\partial_t^j p_-\|_{L^2(\Omega_-)}^2 + \|\nabla \partial_t^j p_-\|_{L^2(\Omega_-)}^2 \lesssim \|\partial_t^j p_-\|_{L^2(\Sigma_-)}^2 + \|\nabla\partial_t^j p_-\|_{L^2(\Omega_-)}^2
\\
\lesssim \mathcal{Z}_n + \bar{\mathcal{D}}_{n}.
\end{multline}
In light of \eqref{pp2} and \eqref{pp4}, we may improve the estimate \eqref{claim2} to be
\begin{equation}\label{up0}
\sum_{j=0}^{ n} \|\partial_t^j u\|_{2n-2j+1}^2 + \sum_{j=0}^{ n-1} \|\partial_t^j p\|_{2n-2j} \lesssim \mathcal{Z}_n +  \bar{\mathcal{D}}_{n}.
\end{equation}

Consequently, summing \eqref{eta0} and \eqref{up0}, we obtain
\begin{equation}\label{upeta0}
{\mathcal{D}}_{n} \lesssim \mathcal{Z}_n + \bar{\mathcal{D}}_{n}.
\end{equation}
Setting $n=2N$ in \eqref{upeta0} and using \eqref{Ges2}--\eqref{Ges3} of Lemma \ref{Gesle} to estimate
$\mathcal{Z}_{2N}\lesssim ({\mathcal{E}}_{2N})^{\theta} {\mathcal{D}}_{2N} + \mathcal{K} {\mathcal{F}}_{2N}$,
we obtain \eqref{d2n}.  On the other hand, we may set $n=N+2$ in \eqref{upeta0} and use \eqref{Ges5} of Lemma \ref{Gesle} to bound $\mathcal{Z}_{N+2}\lesssim  {\mathcal{E}}_{2N}^ \theta {\mathcal{D}}_{N+2}$, from which we then obtain \eqref{dn+2}.
\end{proof}

%%%%%%%%%%%%%%%%%%%%%%%%%%%%%%%%%%%%%%%%%%%%%%%
\subsubsection{A priori estimates and proof of Theorem \ref{th0}}
%%%%%%%%%%%%%%%%%%%%%%%%%%%%%%%%%%%%%%%%%%%%%%%

Now that  Propositions \ref{temp}--\ref{temp2}, Propositions \ref{hori}--\ref{hori2} and Theorems \ref{eth}--\ref{dth} have been established, the rest  of the proof of Theorem \ref{th0} follows in the same way as in the proof of Theorem 1.3 in \cite{GT_per}. For completeness we shall sketch the main steps of the arguments but omit some details.

We first need to control ${\mathcal{F}}_{2N}$. This is achieved by the following proposition.

\begin{Proposition}\label{fglm}
There exists $ 0<\delta<1$ so that if $\mathcal{G}_{2N}(T)\le\delta$, then
\begin{equation}\label{Fg}
\sup_{0\le r\le t}\mathcal{F}_{2N}(r)\lesssim \mathcal{F}_{2N}(0)+ t\int_0^t\mathcal{D}_{2N}, \text{ for all }0\le t\le T.
\end{equation}
\end{Proposition}
\begin{proof}
Based on the transport estimate on the kinematic boundary condition, we may show as in Lemma 7.1 of \cite{GT_per} that
\begin{multline}\label{l0}
\sup_{0\le r\le t}\mathcal{F}_{2N}(r) \lesssim \exp\left(C\int_0^t\sqrt{\mathcal{K}(r)}dr\right)
\\
\times\left[ \mathcal{F}_{2N}(0)+t\int_0^t(1+\mathcal{E}_{2N}(r))\mathcal{D}_{2N}(r)dr  + \left(\int_0^t \sqrt{\mathcal{K}(r) \mathcal{F}_{2N}(r)} dr\right)^2 \right].
\end{multline}
It is easy to see that $\mathcal{K}\lesssim \mathcal{E}_{ N+2}$, and hence
\begin{equation}\label{l1}
\int_0^t\sqrt{\mathcal{K}(r)}dr\lesssim \int_0^t\sqrt{\mathcal{E}_{ N+2}(r)}dr
  \lesssim  \sqrt{\delta}\int_0^\infty \frac{1}{(1+r)^{2N-4}}dr\lesssim\sqrt{\delta}.
\end{equation}
Then by \eqref{l1}, we deduce from \eqref{l0} that
\begin{multline}
\sup_{0\le r\le t}\mathcal{F}_{2N}(r)  \lesssim \mathcal{F}_{2N}(0)+t\int_0^t\mathcal{D}_{2N}+\sup_{0\le r\le t}{{\mathcal{F}}_{2N}}(r)\left(
  \int_0^t\sqrt{ \mathcal{K}(r)}dr\right)^2
  \\
\lesssim \mathcal{F}_{2N}(0) + t\int_0^t\mathcal{D}_{2N}+\delta\sup_{0\le r\le t}{{\mathcal{F}}_{2N}}(r).
\end{multline}
By taking $\delta$ small enough, we get \eqref{Fg}.
\end{proof}

Now we show the boundedness of the high-order terms.

\begin{Proposition} \label{Dgle}
There exists $\delta>0$ so that if $\mathcal{G}_{2N}(T)\le\delta$, then
\begin{equation}\label{Dg}
\sup_{0\le r\le t}\mathcal{E}_{2N}(r)+\int_0^t\mathcal{D}_{2N}+\sup_{0\le r\le t}\frac{\mathcal{F}_{2N}(r)}{(1+r)}\lesssim \mathcal{E}_{2N}(0)+ \mathcal{F}_{2N}(0) \text{ for all }0\le t\le T.
\end{equation}
\end{Proposition}

\begin{proof}
Note first that since ${\mathcal{E}}_{2N}(t)\le {\mathcal{G}}_{2N}(T)\le \delta$, by taking $\delta$ small, we obtain from \eqref{e2n} of Theorem \ref{eth} and \eqref{d2n} of Theorem \ref{dth} that
\begin{equation}\label{q0}
{\mathcal{E}}_{2N}\lesssim\bar{\mathcal{E}}_{2N} \lesssim{\mathcal{E}}_{N+2},
\text{ and }
\bar{\mathcal{D}}_{2N} \lesssim {\mathcal{D}}_{2N} \lesssim \bar{\mathcal{D}}_{2N} + \mathcal{K} \mathcal{F}_{2N}.
\end{equation}

Now we multiply  \eqref{geoes} by a constant $1+\beta\ge 2$ and add this to \eqref{peres} to find
\begin{multline} \label{q1}
\bar{\mathcal{E}}^{0}_{2N}(t) + \bar{\mathcal{E}}^{+}_{2N}(t) + \int_0^t (1+\beta)\bar{\mathcal{D}}^{0}_{2N}+ \bar{\mathcal{D}}^{+}_{2N}
\\
\lesssim  (1+\beta) \mathcal{E}_{2N}(0)+\bar{\mathcal{E}}^{+}_{2N}(0) + (1+\beta) (\mathcal{E}_{2N}(t))^{ {3}/{2}}+ \int_0^t  (2+\beta) \mathcal{E}_{2N}^\theta \mathcal{D}_{2N}
\\
+ \int_0^t  \sqrt{\mathcal{D}_{2N}\mathcal{K}\mathcal{F}_{2N}}+\varepsilon\mathcal{D}_{2N} + C(\varepsilon)\bar{\mathcal{D}}^{0}_{2N}.
\end{multline}
Then by \eqref{q0} we may improve \eqref{q1} to be
\begin{multline} \label{q2}
\mathcal{E}_{2N}(t)+\int_0^t \mathcal{D}_{2N}+ \beta \bar{\mathcal{D}}^{0}_{2N} \lesssim  (1+\beta) \mathcal{E}_{2N}(0)  +  (1+\beta) (\mathcal{E}_{2N}(t))^{ 1+\theta}+ \int_0^t (2+ \beta) \mathcal{E}_{2N}^\theta \mathcal{D}_{2N}
\\
+ \int_0^t  \sqrt{\mathcal{D}_{2N}\mathcal{K}\mathcal{F}_{2N}}+\mathcal{K}\mathcal{F}_{2N} + \varepsilon \mathcal{D}_{2N} +C(\varepsilon)\bar{\mathcal{D}}^{0}_{2N}.
\end{multline}
Since $\mathcal{G}_{2N} \le \delta$, it is easy to use Proposition \ref{fglm} to verify that
\begin{equation}\label{k1}
\int_0^t   \mathcal{K}(r)\mathcal{F}_{2N}(r)dr\lesssim \delta \mathcal{F}_{2N}(0)+\delta\int_0^t\mathcal{D}_{2N}(r)dr,
\end{equation}
\begin{equation}\label{k2}
\int_0^t \sqrt{\mathcal{D}_{2N}(r)\mathcal{K}(r)\mathcal{F}_{2N}(r)}\lesssim   \mathcal{F}_{2N}(0)
+\sqrt{\delta}\int_0^t\mathcal{D}_{2N}(r)dr.
\end{equation}
We plug \eqref{k1}--\eqref{k2} into  \eqref{q2} and then take $\varepsilon$ sufficiently small first, $\beta$ sufficiently large second, and $\delta$ sufficiently small third; we may then conclude
\begin{equation}\label{k3}
\sup_{0\le r\le t}\mathcal{E}_{2N}(r) + \int_0^t\mathcal{D}_{2N} \lesssim \mathcal{E}_{2N}(0) + \mathcal{F}_{2N}(0).
\end{equation}
Then \eqref{Dg} follows from \eqref{Fg} and \eqref{k3}.
\end{proof}

It remains to show the decay estimates of $\mathcal{E}_{N+2}$.

\begin{Proposition} \label{decaylm}
There exists $\delta>0$ so that if $\mathcal{G}_{2N}(T)\le\delta$, then
\begin{equation}\label{n+2}
(1+t^{4N-8})\mathcal{E}_{N+2}(t)\lesssim \mathcal{E}_{2N}(0)+ \mathcal{F}_{2N}(0) \quad \text{for all }0\le t\le T.
\end{equation}
\end{Proposition}

\begin{proof}
Since ${\mathcal{E}}_{2N}(t)\le {\mathcal{G}}_{2N}(T)\le \delta$, by taking $\delta$ small, we obtain from \eqref{en+2} and \eqref{dn+2} that
\begin{equation}
{\mathcal{E}}_{N+2}\lesssim\bar{\mathcal{E}}_{N+2} \lesssim{\mathcal{E}}_{N+2},
\text{ and }
{\mathcal{D}}_{N+2}\lesssim\bar{\mathcal{D}}_{N+2} \lesssim{\mathcal{D}}_{N+2}.
\end{equation}
By these estimates and the smallness of $\delta$, we may deduce from Proposition \ref{temp2}, Proposition \ref{hori2} and \eqref{Fes3} that there exists an instantaneous energy, which is equivalent to ${\mathcal{E}}_{N+2}$, but for simplicity is still denoted by ${\mathcal{E}}_{N+2}$, such that
\begin{equation} \label{u1}
\frac{d}{dt} \mathcal{E}_{N+2} + C \mathcal{D}_{N+2}\le 0.
\end{equation}

On the other hand, based on the Sobolev interpolation inequality we can prove
\begin{equation} \label{intep}
{\mathcal{E}}_{N+2}\lesssim{\mathcal{D}}_{N+2}^\theta{\mathcal{E}}_{2N}^{1-\theta},\text{ where }\theta=\frac{4N-8}{4N-7}
\end{equation}
as in Proposition 7.5 of \cite{GT_per}. Now since by Proposition \ref{Dgle},
\begin{equation}
\sup_{0\le r\le t}\mathcal{E}_{2N}(r)\lesssim \mathcal{E}_{2N}(0)+ \mathcal{F}_{2N}(0):=\mathcal{M}_0,
\end{equation}
we obtain from  \eqref{intep} that
\begin{equation} \label{u2}
\mathcal{E}_{N+2}\lesssim\mathcal{M}_0^{1-\theta} \mathcal{D}_{N+2}^\theta.
\end{equation}
Hence by \eqref{u1} and \eqref{u2}, there exists some constant $C_1>0$ such that
\begin{equation}
\frac{d}{dt} \mathcal{E}_{N+2}+\frac{C_1}{\mathcal{M}_0^s} \mathcal{E}_{N+2}^{1+s}\le 0,\ \text{ where } s = \frac{1}{\theta}-1 = \frac{1}{4N-8}.
\end{equation}
Solving this differential inequality directly, we obtain
\begin{equation} \label{u3}
\mathcal{E}_{N+2}(t)\le \frac{\mathcal{M}_0}{(\mathcal{M}_0^s + s C_1( \mathcal{E}_{N+2}(0))^s t)^{1/s} } {\mathcal{E}}_{N+2}(0).
\end{equation}
Using that ${\mathcal{E}}_{N+2}(0)\lesssim\mathcal{M}_0 $ and the fact $1/s=4n-8>1$, we obtain from \eqref{u3} that
\begin{equation}
{\mathcal{E}}_{N+2}(t)\lesssim \frac{\mathcal{M}_0}{(1+sC_1t)^{1/s} }\lesssim \frac{\mathcal{M}_0}{(1+t^{1/s}) } = \frac{\mathcal{M}_0}{(1+t^{4N-8}) }.
\end{equation}
This implies \eqref{n+2}.
\end{proof}

Now we combine Propositions \ref{Dgle}--\ref{decaylm} to arrive at our ultimate a priori estimates for $\mathcal{G}_{2N}$.

\begin{theorem}\label{Ap}
There exists a universal $0 < \delta < 1$ so that if $
\mathcal{G}_{2N}(T) \le \delta$, then
\begin{equation}\label{Apriori}
 \mathcal{G}_{2N}(t) \le C_2(\mathcal{E}_{2N}(0) + \mathcal{F}_{2N}(0)) \text{ for all }0 \le t \le T.
\end{equation}
\end{theorem}
\begin{proof}
The conclusion follows directly from the definition of $\mathcal{G}_{2N}$ and Propositions  \ref{Dgle}--\ref{decaylm}.
\end{proof}

In order to combine the local existence result in Theorem \ref{Localth} with the a prior estimates in Theorem \ref{Ap}, we must be able to estimate $\mathcal{G}_{2N}$ in terms of the right hand side of
\eqref{localestimate}--\eqref{localestimate2}. This is achieved by the following proposition.

\begin{Proposition}\label{GG2N}
There exists a universal constant $C_3>0$ so that the following hold.  If $0 \le T$, then we have the estimate
\begin{equation}\label{ggg00}
 \mathcal{G}_{2N}(T) \le   \sup_{0 \le t \le T} \mathcal{E}_{2N}(t)  + \int_{0}^{T_2} \mathcal{D}_{2N}(t) dt
+  \sup_{0 \le t \le T} \mathcal{F}_{2N}(t)  + C_3  (1+T)^{4N-8} \sup_{0 \le t \le T} \mathcal{E}_{2N}(t).
\end{equation}
If $0 < T_1 \le T_2$, then we have
\begin{multline}\label{ggg0}
\mathcal{G}_{2N}(T_2) \le C_3 \mathcal{G}_{2N}(T_1) +  \sup_{T_1
\le t \le T_2} \mathcal{E}_{2N}(t)  + \int_{T_1}^{T_2}
\mathcal{D}_{2N}(t) dt
\\
 + \frac{1}{(1+T_1)} \sup_{T_1 \le t \le T_2} \mathcal{F}_{2N}(t) + C_3 (T_2-T_1)^2 (1+T_2)^{4N-8} \sup_{T_1 \le t \le T_2} \mathcal{E}_{2N}(t).
\end{multline}
\end{Proposition}

\begin{proof}
See Proposition 9.1 of \cite{GT_per}.
\end{proof}

Now we turn to the completion of the proof of our main theorem.

\begin{proof}[Proof of Theorem \ref{th0}]

Let $0 < \delta <1$ and $C_2>0$ be the constants in Theorem \ref{Ap}, $C_1>0$ be the constant in \eqref{localestimate2} and $C_3>0$ be the constant in Proposition \ref{GG2N}. By the local existence result, Theorem \ref{Localth}, for any $\varepsilon>0$ there exist $\delta_0(\varepsilon)<1$ and $0<T<1$ so that if $\kappa<\delta_0$, then there is a unique solution of \eqref{nosurface} on $[0,T]$ satisfying the estimates
\begin{equation}
  \sup_{0\le t\le T}\mathcal{E}_{2N}(t)
  +\int_0^T\mathcal{D}_{2N}(t)dt+\int_0^T\left(\|\rho\partial_t^{2N+1}u(t)\|_{-1}^2+\|\partial_t^{2N}p(t)\|_0^2\right)dt
\le \varepsilon
\end{equation}
and
\begin{equation}
\sup_{0\le t\le T}\mathcal{F}_{2N}(t) \le C_1 \mathcal{F}_{2N}(0) +\varepsilon.
\end{equation}
Hence if we choose $\varepsilon=\delta/(4+C_3 2^{4N-7})$ and then choose $\kappa<\delta/(2C_1)$, we may use the estimate \eqref{ggg00} of Proposition \ref{GG2N} to see that
\begin{equation}
\mathcal{G}_{2N}(T) \le C_1\kappa +  \varepsilon(2 + C_3 2^{4N-8}  ) < \delta.
\end{equation}

Now we define
\begin{equation}
\begin{split}
 T_*(\kappa) = \sup \{  T>0 \;\vert\;
&\text{for every choice of initial data satisfying the compatibility}
\\
 &\text{conditions and } \mathcal{E}_{2N}(0) +  \mathcal{F}_{2N}(0) <
\kappa, \text{ there exists a unique }
 \\
&\text{solution of } \eqref{nosurface} \text{ on }[0,T] \text{ satisfying }\mathcal{G}_{2N}(T) \le \delta
\}.
\end{split}
\end{equation}
By the above analysis, $T_*(\kappa)$ is well-defined and satisfies $T_*(\kappa)>0$ if $\kappa$ is small enough, i.e. there is a $\kappa_1>0$ so that $T_* : (0,\kappa_1] \rightarrow (0,\infty]$. It is easy to verify that $T_*$ is non-increasing on $(0,\kappa_1]$. Now we set
\begin{equation}\label{ggg-1}
\varepsilon = \frac{\delta}{3}  \min\left\{ \frac{1}{2},
\frac{1}{C_3} \right\}
\end{equation}
and then define $\kappa_0 \in (0,\kappa_1]$ by
\begin{equation}\label{ggg2}
 \kappa_0 = \min\left\{  \frac{\delta}{3 C_2(C_3 + 2 C_1)}, \frac{\delta_0( \varepsilon)}{C_2}, \kappa_1
 \right\}.
\end{equation}
We claim that $T_*(\kappa_0) = \infty$. Once the claim is established, the proof of the theorem is complete since then $T_*(\kappa) = \infty$ for all $0 < \kappa \le \kappa_0$.

We prove the claim by contradiction. Suppose that  $T_*(\kappa_0) < \infty$. By the definition of $T_*(\kappa_0)$, for any $0 < T_1 < T_*(\kappa_0)$ and for any choice of data satisfying the compatibility conditions and the bound $\mathcal{E}_{2N}(0) + \mathcal{F}_{2N}(0) < \kappa_0$, there exists a unique solution of \eqref{nosurface} on  $[0,T_1]$ satisfying $\mathcal{G}_{2N}(T_1) \le \delta$.  Then by Theorem \ref{Apriori}, we have
\begin{equation}
 \mathcal{G}_{2N}(T_1) \le C_2 (\mathcal{E}_{2N}(0) + \mathcal{F}_{2N}(0)) < C_2
 \kappa_0.\label{ggg1}
\end{equation}
In particular, \eqref{ggg1}, \eqref{ggg2} imply
\begin{equation}
 \mathcal{E}_{2N}(T_1) + \frac{ \mathcal{F}_{2N}(T_1)}{(1+T_1)} < C_2 \kappa_0 \le \delta_0(\varepsilon),
 \ \forall\,0 < T_1 < T_*(\kappa_0).\label{ggg3}
\end{equation}
We can view $(u(T_1),p(T_1),\eta(T_1))$ as initial data for a new problem; they satisfy the compatibility conditions as initial data since they are already solutions on $[0,T_1]$. Since $ \mathcal{E}_{2N}(T_1) < \delta_0(\varepsilon)$, we may use Theorem \ref{Localth} to extend the solution to  $[T_1, T_2]$, where $T_2$ is any time satisfying
\begin{equation}
 0 < T_2 - T_1 \le T_0 := C(\varepsilon) \min\{ 1,  \mathcal{F}_{2N}(T_1)^{-1}  \}.
\end{equation}
By \eqref{ggg3}, we have
\begin{equation}
 \bar{T} :=  C(\varepsilon) \min\left\{1, \frac{1}{\delta_0(\varepsilon) (1+T_*(\kappa_0))} \right\} \le T_0.
\end{equation}
Note that $\bar{T}$ depends on $\varepsilon$ and $T_*(\kappa_0)$ but does not depend on $T_1$. Let
\begin{equation}
 \gamma = \min\left\{  \bar{T}, T_*(\kappa_0),  \frac{1}{(1+2 T_*(\kappa_0))^{2N-4}}        \right\},
\end{equation}
and then choose $T_1 = T_*(\kappa_0)-\gamma/2$ and  $T_2 = T_*(\kappa_0) + \gamma/2$.  We have
\begin{equation}\label{ggg6}
 0 < T_1  < T_*(\kappa_0) < T_2  <  2 T_*(\kappa_0) \text{ and } 0< \gamma= T_2 - T_1 \le \bar{T} \le T_0.
\end{equation}
By the above argument, we have extended the solution to $[0,T_2]$ and on the extended time interval $[T_1,T_2]$ we have
\begin{equation}\label{ggg4}
 \sup_{T_1 \le t \le T_2} \mathcal{E}_{2N}(t) +  \int_{T_1}^{T_2} \mathcal{D}_{2N}(t) dt
+ \int_{T_1}^{T_2} \left( \|\rho\partial_t^{2N+1} u(t)\|_{-1}^2
+\|\partial_t^{2N} p(t)\|_{0}^2 \right) dt \le \varepsilon,
\end{equation}
and
\begin{equation}\label{ggg5}
 \sup_{T_1 \le t \le T_2} \mathcal{F}_{2N}(t) \le C_1 \mathcal{F}_{2N}(T_1) + \varepsilon.
\end{equation}

We now combine the estimates \eqref{ggg4}--\eqref{ggg5}, \eqref{ggg1}--\eqref{ggg3}, and \eqref{ggg0} of Proposition \ref{GG2N} with  the definitions \eqref{ggg-1}, \eqref{ggg2}, and \eqref{ggg6} to see that
\begin{equation}
 \begin{split}
 \mathcal{G}_{2N}(T_2) &< C_2 C_3 \kappa_0  + \varepsilon + \frac{C_1 C_2 \kappa_0 (1+T_1) + \varepsilon}{(1+T_1)}
+ \varepsilon C_3 (T_2-T_1)^2 (1+T_2)^{4N-8}\nonumber \\
&\le \kappa_0 C_2(C_3 +   C_1) + 2\varepsilon  + \varepsilon C_3
\gamma^2 (1+2 T_*(\kappa_0))^{4N-8}\nonumber \\
&\le  \frac{\delta}{3} + \frac{\delta}{3} + \frac{\delta}{3}
=\delta.
 \end{split}
\end{equation}
Hence $\mathcal{G}_{2N}(T_2) \le \delta$, which contradicts the definition of $T_*(\kappa_0)$. Therefore, we have $T_*(\kappa_0) = \infty$. This proves the claim and completes the proof of Theorem \ref{th0}.
\end{proof}

%%%%%%%%%%%%%%%%%%%%%%%%%%%%%%%%%%%%%%%%%%%%%%%
\section{Case with surface tension: Proof of Theorem \ref{surth}}\label{surfacetension}
%%%%%%%%%%%%%%%%%%%%%%%%%%%%%%%%%%%%%%%%%%%%%%%

In this section we prove Theorem \ref{surth} through a contraction mapping argument in an appropriate space.  We begin with an analysis of the linearized problem.  Then we consider the nonlinear problem and develop the details of the contraction mapping.

%%%%%%%%%%%%%%%%%%%%%%%%%%%%%%%%%%%%%%%%%%%%%%%
 \subsection{The linearized problem}\label{sec_surface_linear}
%%%%%%%%%%%%%%%%%%%%%%%%%%%%%%%%%%%%%%%%%%%%%%%

Throughout Section \ref{sec_surface_linear} we suppose that $f$ and $g$ are given and we consider the following linearized problem:
\begin{equation}\label{surfaceLP}
\left\{\begin{array}{lll}
\rho\partial_t u-\mu\Delta u+\nabla p=f\quad&\text{in }\Omega
\\ \diverge{u}=0&\text{in }\Omega
\\ ( p_+I-\mu_+\mathbb{D}(u_+)) e_3= (\rho_+g\eta_+  -\sigma_+\Delta_\ast\eta_+)e_3+g_+&\text{on }\Sigma_+
\\\Lbrack u\Rbrack=0,\quad \Lbrack pI-\mu\mathbb{D}(u)\Rbrack e_3=(\rj g\eta_- +\sigma_-\Delta_\ast\eta_-)e_3-g_-&\text{on }\Sigma_-
\\ \partial_t\eta=u_3 &\text{on }\Sigma
\\ u_-=0 &\text{on }\Sigma_b,
\end{array}\right.
\end{equation}
with the initial data $u(0)=u_0$ and $\eta(0)=\eta_0$.

The key observation required to produce a solution to \eqref{surfaceLP} is that the kinematic boundary condition $\partial_t\eta=u_3$  allows us to eliminate $\eta$ and  view \eqref{surfaceLP}  as a time-dependent linear problem for the pair $(u,p)$. More precisely, we first consider the following linear problem
 \begin{equation}\label{sLP0}
\left\{
\begin{array}{lll}
\rho\partial_t u-\mu\Delta u+\nabla p=f\quad&\text{in }\Omega
\\ \diverge{u}=0&\text{in }\Omega
\\ \left[ p_+I-\mu_+\mathbb{D}(u_+) +(-\rho_+g  +\sigma_+ \Delta_\ast ) \left(\int_0^tu_{3,+} \,ds+\eta_{0,+}\right)I  \right]e_3=  g_+&\text{on }\Sigma_+
\\  \Lbrack u\Rbrack=0&\text{on }\Sigma_-
 \\ \left[ \Lbrack pI-\mu \mathbb{D}(u )\Rbrack +( -\rj g -\sigma_- \Delta_\ast) \left( \int_0^t u_3 \,ds+\eta_{0,-}\right ) I \right]e_3=-g_- &\text{on }\Sigma_-
\\ u_-=0 &\text{on }\Sigma_b,
\end{array}
\right.
\end{equation}
with the initial data $u(0)=u_0$. After finding a solution $(u,p)$ to the problem \eqref{sLP0}, we define $\eta=\int_0^tu_3 \,ds+\eta_{0}$ on $\Sigma$; then the triple $(u,p,\eta)$ solves the problem \eqref{surfaceLP}.

We will first prove that the problem \eqref{sLP0} is solvable in the weak sense via a standard  Galerkin method. Then we explore  how to use elliptic estimates and a localization trick to improve the regularity of the weak solution in order  to show that the problem \eqref{sLP0} (and hence also \eqref{surfaceLP}) admits a unique strong solution.

%%%%%%%%%%%%%%%%%%%%%%%%%%%%%%%%%%%%%%%%%%%%%%%
\subsubsection{Preliminaries}
%%%%%%%%%%%%%%%%%%%%%%%%%%%%%%%%%%%%%%%%%%%%%%%

The following theorem will be crucial in our subsequent energy analysis.  The theorem also provides the value of the critical surface tension $\sigma_c$, given by \eqref{crit_st_def}, that appears in Theorem \ref{surth}.

\begin{theorem}\label{critical}
Suppose that one of the following is true:
\begin{itemize}
 \item $\rj \le 0$ and $\sigma_->0$, or
 \item $\rj >0$ and $\sigma_- > \sigma_c = \rj g \max\{L_1^2,L_2^2\}$.
\end{itemize}
Then the following hold for all $\eta$ satisfying  $\int_{\mathrm{T}^2}\eta=0$.
\begin{enumerate}

 \item There exists $C>0$ such that
\begin{equation}\label{criticals}
\sigma_-\|\nabla_\ast \eta\|_0^2-\rj g\|\eta\|_0^2\ge C \|\eta\|_1^2 .
\end{equation}

 \item If $\eta$ satisfies
\begin{equation}\label{tequa}
-\sigma_-\Delta_\ast\eta-\rj g\eta=\varphi\text{ on }\mathrm{T}^2,
\end{equation}
then we have for $r\ge 2$,
\begin{equation}\label{elliptt}
\|\eta\|_{r}\lesssim\|\varphi\|_{r-2}.
\end{equation}

\end{enumerate}
\end{theorem}

\begin{proof}
It is clear that $(1)$ and $(2)$ hold for the case  $\rj \le 0,\ \sigma_->0$. Now we suppose that $\rj >0,\
\sigma_->\sigma_c$.  Since $\int_{\mathrm{T}^2}\eta=0$, i.e. $\hat{\eta}(0)=0$, we may use the Parseval theorem to estimate
\begin{equation}
 \begin{split}
  \|\nabla_\ast\eta\|_0^2 & =
\sum_{n\in (L_1^{-1}\mathbb{Z}) \times (L_2^{-1}\mathbb{Z})
\backslash\{0\}}  |n|^2|\hat{\eta}(n)|^2 \\
& \ge  \min\{L_1^{-2},L_2^{-2}\} \sum_{n\in (L_1^{-1}\mathbb{Z}) \times (L_2^{-1}\mathbb{Z})
\backslash\{0\}} |\hat{\eta}(n)|^2 =  \frac{1}{\max\{L_1^2,L_2^2\}}\| \eta \|_0^2.
 \end{split}
\end{equation}
This implies that
\begin{equation}
\sigma_-\|\nabla_\ast \eta\|_0^2-\rj g\|\eta\|_0^2 \ge (\sigma_--\sigma_c)\|\nabla_\ast\eta\|_0^2.\end{equation}
This, the estimate $\sigma_- - \sigma_c>0$, and the Poincar\'e inequality then imply \eqref{criticals}.

To prove \eqref{elliptt}, we multiply \eqref{tequa} by $\eta$ and then integrate over $\mathrm{T}^2$ to see that
\begin{equation}
\sigma_-\|\nabla_\ast\eta\|_0^2-\rj g\|\eta\|_0^2\le\|\varphi\|_0\|\eta\|_0.
\end{equation}
Combining this with \eqref{criticals}, we find that
\begin{equation}\label{estimate1}
\|\eta\|_1 \lesssim\|\varphi\|_0.
\end{equation}
Now we rewrite \eqref{tequa} as
\begin{equation}\label{rrrr}
-\sigma_-\Delta_\ast\eta = \rj g\eta + \varphi \text{ on }\mathrm{T}^2.
\end{equation}
By the standard elliptic regularity theory and \eqref{estimate1}, we then have
\begin{equation}
\|\eta\|_2^2\lesssim \|\eta\|_0^2+\|\varphi\|_0^2\lesssim \|\varphi\|_0^2.
\end{equation}
With this estimate in hand, we may then employ a standard bootstrap argument on the equation \eqref{rrrr} to
obtain \eqref{elliptt}.
\end{proof}

We now define a pair of function spaces that we will use in our weak formulation.  We define
\begin{equation}
\mathcal{W}=\{v\in {}_0H^1(\Omega)\mid  v_3\in H^1(\Sigma)\}\text{\ \
 and\ \  } \mathcal{V}=\{v\in \mathcal{W}\mid \diverge{v}=0\}.
\end{equation}
We endow these spaces with the inner product
\begin{equation}
(u,v)_{\mathcal{W}}=(u,v)_{H^1(\Omega)}+(u_3,v_3)_{H^1(\Sigma)},
\end{equation}
and we denote the corresponding norm $\|u\|_{\mathcal{W}}^2 :=\|u\|_{H^1(\Omega)} ^2 + \|u_3\|_{H^1(\Sigma)}^2 $.

The following proposition allows us to introduce the pressure as a Lagrange multiplier.

\begin{Proposition}\label{lagrange}
If $\Lambda\in \mathcal{W}^\ast$ is such that $\Lambda(v)=0$,\ $\forall\,v\in \mathcal{V}$, then there exists a unique $p\in L^2(\Omega)$ so that
\begin{equation}\label{wdiv}
(p,\diverge{v})=\Lambda(v),\ \forall\,v\in \mathcal{W}.
\end{equation}
Also, $\|p\|_0\lesssim\|\Lambda\|_{\mathcal{W}^\ast}$.
\end{Proposition}

\begin{proof}
As in \cite{GT_lwp,SS}, we only need to prove that for any $p\in L^2(\Omega)$ there exists a $v\in \mathcal{W}$ so that $\diverge v=p$~in~$\Omega$. We construct $v$ as follows. First, let
\begin{equation}
C_1=-\frac{1}{|\mathrm{T}^2|}\left(\int_{\Omega_+}p_++\int_{\Omega_-}p_-\right)
\text{ and } C_2=-\frac{1}{|\mathrm{T}^2|}\int_{\Omega_-}p_-.
\end{equation}
Then we set $v^{(1)}=\nabla \phi$ with $\phi_+$ solving the problem
\begin{equation}
\left\{\begin{array}{lll}
-\Delta \phi_+=p_+ &\text{ in }\Omega_+
\\  \partial_3\phi_+=C_1& \text{ on
}\Sigma_+
\\  \partial_3\phi_+= C_2 & \text{ on
}\Sigma_-,
\end{array}\right.
\end{equation}
and $\phi_-$ solving the problem
\begin{equation}
\left\{\begin{array}{lll}
-\Delta \phi_-=p_- &\text{ in }\Omega_-
\\  \partial_3\phi_-=C_2& \text{ on
}\Sigma_-
\\ \nabla \phi_-\cdot\nu=0 & \text{ on
}\Sigma_b.
\end{array}\right.
\end{equation}
Our choice of $C_1$ and $C_2$ makes the two problems satisfy the necessary compatibility conditions, so the classical elliptic theory guarantees the existence of $\phi_\pm$, and we have
\begin{equation}\label{111es}
\|v^{(1)}\|_{\ddot{H}^1(\Omega)}=\|\nabla\phi\|_{\ddot{H}^1(\Omega)}
\lesssim\|p\|_{0}+|C_1|+|C_2|\lesssim\|p\|_{0}.
\end{equation}
Moreover,
\begin{equation}
\begin{cases}
\diverge{v^{(1)}}_\pm=p_\pm &\text{in }\Omega \\
v_+^{(1)}\cdot e_3=C_1  & \text{on }\Sigma_+ \\
v_+^{(1)}\cdot e_3=v_-^{(1)}\cdot e_3=C_2  &\text{on }\Sigma_- \\
v_-^{(1)}\cdot\nu = 0 &\text{on }\Sigma_b.
\end{cases}
\end{equation}

Next, we let $v^{(2)}$ be the function constructed in Lemma \ref{div}, but we restrict it to have support near $\Sigma_b$. Then we have $\diverge v^{(2)}=0$ in $\Omega$ and $v^{(2)}=-v^{(1)}$ on $\Sigma_b$.
Moreover,
\begin{equation}\label{222es}
\|v ^{(2)}\|_{1}\lesssim\|v^{(1)}_-\|_{H^{1/2}(\Sigma_b)}.
\end{equation}
Similarly, we can find $v^{(3)}$ so that its support is near and below $\Sigma_-$, $\diverge v^{(3)}=0$ in $\Omega$ and $v^{(3)}=-\Lbrack v^{(1)}\Rbrack$ on $\Sigma_-$. Moreover,
\begin{equation}\label{333es}
\|v ^{(3)}\|_{1}\lesssim\|v^{(1)}_+\|_{H^{1/2}(\Sigma_-)}+\|v^{(1)}_-\|_{H^{1/2}(\Sigma_-)}.
\end{equation}

Hence, setting $v =v^{(1)}+v^{(2)}+v^{(3)}$, we have $\diverge{v}=p$ in $\Omega$,  $v_3=C_1$ on $\Sigma_+$, $v_3=C_2$, $\Lbrack v\Rbrack=0$ on $\Sigma_-$ and  $v=0$ on $\Sigma_b$. In particular, combining this with
 \eqref{111es}--\eqref{333es} and the trace theorem, we obtain $v\in \mathcal{W}$ and
\begin{equation}
\|v\|_{\mathcal{W}}\lesssim \|p\|_0.
\end{equation}
With this $v$ in hand, we can prove the rest of the proposition as in  \cite{GT_lwp,SS}.
\end{proof}

%%%%%%%%%%%%%%%%%%%%%%%%%%%%%%%%%%%%%%%%%%%%%%%
 \subsubsection{Weak solution of \eqref{sLP0}}
%%%%%%%%%%%%%%%%%%%%%%%%%%%%%%%%%%%%%%%%%%%%%%%

Recall that in the formulation \eqref{sLP0} we have eliminated $\eta$ via
\begin{equation}\label{etadef}
\eta=\int_0^tu_3 \,ds+\eta_{0},\ \text{ on }\Sigma.
\end{equation}
That is, when we mention $\eta$ we always regard it as determined by $u$ through \eqref{etadef}.

%We define the time-dependent function spaces
%\begin{equation}
%\mathcal{W}_T=\{v\in L^2(0,T;{}_0H^1(\Omega))\mid \int_0^t  v_3\in L^\infty(0,T;H^1(\Sigma))\}.
%\end{equation}
%and
%\begin{equation}
%\mathcal{V}_T=\{v\in \mathcal{W}_T\mid \diverge{v}=0\}.
%\end{equation}
We will need the orthogonal decomposition $L^2(\Omega)=\mathcal{Y}\oplus\mathcal{Y}^\bot$, where
\begin{equation}
\mathcal{Y}^\bot:=\left\{\nabla\varphi\mid\varphi\in {}^0H^1(\Omega)\right\}.
\end{equation}

Motivated by the identity that results from formally multiplying \eqref{sLP0} by a smooth vector field $v$ with $v|_{\Sigma_b}=0$ and integrating by parts over $\Omega$, we define the weak solution of \eqref{sLP0} as follows.  Suppose that
\begin{equation}\label{data_weak_assump}
 f\in L^2(0,T; \mathcal{W}^\ast ) ,\ g\in L^2(0,T;H^{-1/2}(\Sigma)), u_0\in   \mathcal{Y} \text{ and } \eta_0\in H^1(\Sigma).
\end{equation}
We  say $(u,p)$ is a weak solution of \eqref{sLP0} if
\begin{equation}\label{sLP0ws}
\left\{\begin{array}{lll}u\in \mathcal{V}_T,\
 \rho\partial_t u\in L^2(0,T;\mathcal{W}^\ast),\ p\in L^2(0,T;L^2(\Omega));\\
\displaystyle\langle\rho\partial_tu
,v\rangle_\ast+(\frac{\mu}{2}\mathbb{D}u, \mathbb{D}v)-(p,\diverge{v})
 + \rho_+g(\eta_{+},v_{3,+} )_{+}
\\ \quad+\sigma_+(\nabla_\ast \eta_{+},\nabla_\ast v_{3,+})_{+}
-\rj g(\eta_{-},v_3)_{-}+\sigma_-(\nabla_\ast\eta_{-},\nabla_\ast
v_3)_{-}
 \\ \qquad=\langle f, v\rangle_\ast -\langle g, v\rangle_{-1/2},\quad \text{for every }v\in  \mathcal{W} \text{ and }a.e.\ t\in [0,T];
\\u(0)=u_0.
\end{array}\right.
\end{equation}
Here $\langle\cdot,\cdot\rangle_\ast$ denotes the dual pairing between $\mathcal{W}^\ast$ and $\mathcal{W}$, $\langle \cdot,\cdot \rangle_{-1/2}$ denotes the dual pairing between $H^{1/2}(\Sigma)$ and $H^{-1/2}(\Sigma)$, and $(\cdot,\cdot)_{\pm}$ is the $L^2$ inner product on $\Sigma_\pm$.

We could prove the existence and uniqueness of a weak solution to \eqref{sLP0} in the sense of \eqref{sLP0ws} using only the assumptions stated in \eqref{data_weak_assump}. However, our aim is to  construct solutions of \eqref{sLP0} with higher regularity, so we will skip this analysis and instead presently make  stronger assumptions on the data $f,g,u_0$ and $\eta_0$.   Indeed, we suppose
\begin{equation}\label{data1}
\left\{\begin{array}{lll}
f\in L^2(0,T; \ddot{H}^{1}(\Omega)),\ \partial_tf\in L^2(0,T;\mathcal{W}^\ast),
\\  g\in L^2(0,T; {H}^{3/2}(\Sigma)),\ \partial_tg\in L^2(0,T; {H}^{-1/2}(\Sigma))
\\ u_0\in  \mathcal{V}\cap\ddot{H}^2(\Omega),\ \eta_0\in  {H}^3(\Sigma) .
\end{array}\right.
\end{equation}
Note that \eqref{data1} implies that (see, for instance Lemmas 2.4 and A.2 of \cite{GT_lwp}) $f \in C([0,T]; L^2(\Omega)),$  $g\in C([0,T]; {H}^{1/2}(\Sigma))$, and in particular,
\begin{equation}\label{data2}
f(0)\in L^2(\Omega),\ g(0)\in H^{1/2}(\Sigma).
\end{equation}

We now record a result on the existence and uniqueness of a  weak solution under the stronger assumptions \eqref{data1}.

\begin{theorem}\label{weak_wp}
Suppose that $f,g,u_0,\eta_0$ satisfy \eqref{data1}--\eqref{data2}, and that $u_0,g(0)$ satisfy the compatibility conditions
\begin{equation}\label{compatibility11}
\Pi_\ast\left(g_+(0) + \mu_+ \mathbb{D}u_{0,+} e_3\right)=0 \text{ on }\Sigma_+,  \quad
\Pi_\ast \left(g_-(0)- \Lbrack \mu\mathbb{D}u_0 \Rbrack e_3\right) = 0 \text{ on }\Sigma_-,
\end{equation}
where $\Pi_\ast$ is the horizontal projection defined by $\Pi_\ast v=(v_1,v_2,0)$. Then there exists a unique weak solution $(u,p)$ to \eqref{sLP0} satisfying the estimates
\begin{multline}\label{weakest}
 \|u\|_{L^\infty L^2 }^2 + \|u\|_{L^2  H^1 }^2 + \|\partial_tu\|_{L^\infty L^2 }^2 + \|\partial_tu\|_{L^2 H^1 }^2 +\|u_3\|_{L^\infty H^1(\Sigma)}^2+\|\eta\|_{L^\infty  H^1 }^2 + \|\partial_t\eta\|_{L^\infty H^1 }^2
\\
\lesssim  \mathcal{Z}_0:=\|u_0\|_{2}^2+\|\eta_0\|_{3}^2+\|f(0)\|_{ 0}^2+\| g(0)\|_{1/2}^2
   \\  +\|f\|_{ L^2  {H}^{1} }^2+\|\partial_tf\|_{ L^2   \mathcal{W}^\ast }^2
+ \|g\|_{L^2 {H}^{3/2} }^2+\|\partial_tg\|_{L^2 {H}^{-1/2}}^2
\end{multline}
and
\begin{equation}\label{weakpes}
\|p\|_{L^2L^2}\lesssim (1+T)\mathcal{Z}_0.
\end{equation}
\end{theorem}

\begin{proof}
The divergence-free condition satisfied  by $u$ allows us to first solve a pressureless problem via a variational formulation and then to introduce the pressure as a Lagrange multiplier. For the pressureless problem we use the Galerkin method. We divide the proof into three steps.

{\it Step 1. The Galerkin scheme.} We will carry out the Galerkin scheme within the functional setting of the space  $\mathcal{V} \cap \ddot{H}^2(\Omega)$.  This space  is clearly separable, and we choose an orthogonal
basis $\{w^j\}_{j=1}^\infty$. For any integer $m\ge 1$ we define the finite dimensional space $X_m:={\rm span}\{w^1,\dots,w^m\}$ and write $\mathcal{P}^m$ for the orthogonal projection of $\mathcal{V}\cap \ddot{H}^2(\Omega) $ onto $X_m$.

For each $m\ge 1$  we define the approximate solution
\begin{equation}
u^m(t):=\sum_{j=1}^m d^m_j(t) w^j,
\end{equation}
where the coefficients $d^m_j(t)$ are chosen so that for $t\in[0, T]$ and any $\psi\in X_m$,
\begin{multline}\label{Galerkin}
( \rho\partial_tu^m ,\psi)+(\frac{\mu}{2}\mathbb{D}u^m, \mathbb{D}\psi)
+ \rho_+g(\eta^m_{ +}
,\psi_{3,+})_{+}+\sigma_+(\nabla_\ast\eta^m_+,\nabla_\ast
\psi_{3,+})_{+} -\rj g(\eta^m_-,\psi_3)_{-} \\
+\sigma_-(\nabla_\ast\eta^m_-,\nabla_\ast \psi_3)_{-} =( f, \psi)-(g_+ - \Pi_\ast(g_+(0)+\mu_+\mathbb{D}(\mathcal{P}^mu_0)_+e_3),\psi_+)_{+}
 \\  -( g_--\Pi_\ast(g_-(0)-\llbracket \mu\mathbb{D}(\mathcal{P}^mu_0)e_3\rrbracket ), \psi)_{-},
 \end{multline}
supplemented with the initial condition
\begin{equation}
u^m (0)=\mathcal{P}^mu_0.
\end{equation}
Note that the terms involving $\Pi_\ast$ in \eqref{Galerkin} have been added to compensate for the fact that $u^m(0)$ may not  satisfy the compatibility conditions \eqref{compatibility11}.

If we define $a^m_j(t)=\int_0^td^m_j(s)\,ds$, then we may readily deduce from  \eqref{Galerkin} an equivalent second-order linear ordinary differential system for $\{a^m_j(t)\}_{j=1}^m$, subject to the initial conditions $a^m_j(0)=0$ and $a^m_j{'}(0)=(u_0,w^j)_\Omega$.  In this system the forcing terms are continuous due to
the assumptions of the data, \eqref{data1}.   As such, the classical theory of ODEs guarantees the solvability of the ODE in the interval $[0,T]$, which in turn implies that $u_m$ (and $d_m^k$) satisfies \eqref{Galerkin} for  a.e. $0\le t\le T$.  Note that since $f,g$ satisfy \eqref{data1}, we may deduce from  \eqref{Galerkin}  that actually $d^m_j\in C^{1,1}([0,T])$, and hence is twice differentiable a.e. in $[0,T]$.

With $u^m$ constructed, we define $\eta^m$ by \eqref{etadef} with $u_3$ replaced by $u^m_3$.  Since $\diverge{u^m}=0$ and $u^m \vert_{\Sigma_b}=0$, an application of the divergence theorem shows that
\begin{equation}
 \int_{\Sigma_+ } u^m_+ = \int_{\Sigma_- } u^m_- = 0,
\end{equation}
which in turn implies that
\begin{equation}\label{approx_zero}
 \int_{\Sigma_+} \eta^m_+ = \int_{\Sigma_-} \eta^m_- = 0
\end{equation}
since $\eta_0$ satisfies the zero average condition \eqref{zero0}.

{\it Step 2. Energy estimates.} Taking $\psi=u^m\in X_m$ in \eqref{Galerkin}, we obtain
\begin{multline}\label{i11}
( \rho\partial_tu^m , u^m)+(\frac{\mu}{2}\mathbb{D}u^m, \mathbb{D}u^m)
+ \rho_+g(\eta^m_+,u_{3,+}^m)_{+}
+ \sigma_+(\nabla_\ast\eta^m_+,\nabla_\ast
u_{3,+}^m)_{+}-\rj g(\eta_-^m,u_3^m)_{-}
\\ +\sigma_-(\nabla_\ast\eta_-^m,\nabla_\ast u^m_3)_{-}=( f, u^m)
 -( g_+-\Pi_\ast(g_+(0)+\mu_+\mathbb{D} u_{ +}^m(0)e_3), u_{ +}^m)_{+}
 \\ -( g_--\Pi_\ast(g_-(0)-\llbracket \mu\mathbb{D} u^m(0)e_3\rrbracket ), u^m)_{-}.
\end{multline}
Applying  Korn's inequality to the second term and applying Cauchy's inequality and the trace theorem to the right-hand side, we deduce from  \eqref{i11} that
 \begin{multline}\label{i12}
\frac{d}{dt}\left(\|\sqrt{\rho}u^m \|_{L^2(\Omega)}^2+ \rho_+g\|\eta^m_+\|_{L^2(\Sigma_+)}^2
 +\sigma_+\|\nabla_\ast\eta^m_+\|_{L^2(\Sigma_+)}^2 \right.\\
\left. -\rj g\|\eta^m_-\|_{L^2(\Sigma_-)}^2+\sigma_-\|\nabla_\ast\eta^m_-\|_{L^2(\Sigma_-)}^2\right)
+ C \|u^m\|_{H^1(\Omega)}^2 \lesssim \|f\|_{L^2(\Omega)}^2+\|g\|_{L^2(\Sigma)}^2 \\
+ \|\Pi_\ast(g_+(0)+\mu_+\mathbb{D}
u_{m,+}(0)e_3)\|_{L^2(\Sigma_+)}^2
+\|\Pi_\ast(g_-(0)-\llbracket \mu\mathbb{D} u_m(0)e_3\rrbracket )\|_{L^2(\Sigma_-)}^2.
\end{multline}
Since $\sigma_->\sigma_c$ we know from Lemma \ref{critical} (which is applicable to $\eta^m$ due to \eqref{approx_zero}) that the expression involving $\eta$ under the time derivative is equivalent to $\| \eta \|_1^2$. Hence, upon integrating \eqref{i12} in time we find that
\begin{multline}\label{es1}
 \sup_{0\le t\le T}\left(\|u^m \|_{L^2(\Omega)}^2+\|\eta^m\|_{H^1(\Sigma)}^2\right) + \|u_m\|_{L^2H^1 }^2
\lesssim
\|u_0\|_{L^2(\Omega)}^2 + \|\eta_0\|_{H^1(\Sigma)}^2
+ \|f\|_{L^2L^2 }^2 + \|g\|_{L^2L^2}^2  \\
+ T\left(\|\Pi_\ast(g_+(0)+\mu_+\mathbb{D} u_{
+}^m(0)e_3)\|_{L^2(\Sigma_+)}^2 +\|\Pi_\ast(g_-(0)-\Lbrack \mu\mathbb{D}
u^m(0)e_3\Rbrack)\|_{L^2(\Sigma_-)}^2\right).
\end{multline}

Since $d^m_j$ is twice differentiable a.e., we can temporally differentiate \eqref{Galerkin} to
see that for any $\psi\in X_m$,
\begin{multline}\label{Galerkin2}
 ( \rho\partial_t^2u^m , \psi)+(\frac{\mu}{2}\mathbb{D}\partial_t u^m, \mathbb{D}\psi)
+ \rho_+g( u_{3,+}^m , \psi_{3,+} )_{+}+\sigma_+(\nabla_\ast
u_{3,+}^m ,\nabla_\ast \psi_{3,+} )_{+} \\
-\rj g( u^m_3 ,\psi_3)_{-} +\sigma_-(\nabla_\ast u^m_3 ,\nabla_\ast \psi_3)_{-}
=  \langle \partial_t f, \psi \rangle_\ast-\langle \partial_t g,
\psi \rangle_{-1/2}.
\end{multline}
Then, taking $\psi=\partial_t u^m$ in \eqref{Galerkin2}, we have
\begin{multline}\label{i21}
( \rho\partial_t^2u^m , \partial_t u^m)+(\frac{\mu}{2}\mathbb{D}\partial_t u^m, \mathbb{D}\partial_t u^m)
+ \rho_+g( u_{3,+}^m , \partial_t u_{3,+}^m)_{+} + \sigma_+(\nabla_\ast  u_{3,+}^m ,\nabla_\ast
\partial_t u_{3,+}^m)_{+} \\
-\rj g( u^m_3 ,\partial_tu^m_3)_{-} +\sigma_-(\nabla_\ast u^m_3 ,\nabla_\ast \partial_t u^m_3)_{-}
=\langle \partial_t f, \partial_t u^m\rangle_\ast-\langle \partial_t g, \partial_tu^m\rangle_{-1/2}.
\end{multline}
We bound the right-hand side of \eqref{i21} via
 \begin{multline}\label{bound}
 \langle \partial_t f, \partial_t u^m\rangle_\ast+\langle \partial_t g, \partial_tu^m\rangle_{-1/2}
 \le  \|\partial_t f\|_{\mathcal{W}^\ast}\|\partial_t u^m\|_{ H^{1}(\Omega)}+\|\partial_t g\|_{H^{-1/2}(\Sigma)}\|\partial_t u^m\|_{H^{1/2}(\Sigma)}  \\
 \le  C(\varepsilon)\left(\|\partial_t f\|_{\mathcal{W}^\ast}^2+\|\partial_t g\|_{H^{-1/2}(\Sigma)}^2\right)
+ \varepsilon\|\partial_t u^m\|_{H^{1}(\Omega)}^2.
\end{multline}
As above, we apply  Korn's inequality to the second term on the left-hand side of \eqref{i21}; then from \eqref{bound} and  \eqref{i21} we deduce that
\begin{multline}\label{i22}
 \frac{d}{dt}\left( \|\sqrt{\rho}\partial_tu^m \|_{L^2(\Omega)}^2
+ \rho_+g\| u_{3,+}^m \|_{L^2(\Sigma_+)}^2+\sigma_+\|\nabla_\ast
u_{3,+}^m\|_{L^2(\Sigma_+)}^2
 -\rj g\| u^m_3\|_{L^2(\Sigma_-)}^2\right. \\
\left. +\sigma_-\|\nabla_\ast u^m_3\|_{L^2(\Sigma_-)}^2\right) +\|\partial_t u^m\|_{H^1(\Omega)}^2\lesssim\|\partial_t f\|_{ \mathcal{W}^\ast}^2+\|\partial_t g\|_{H^{-1/2}(\Sigma)}^2.
\end{multline}

In order for the temporal integral of \eqref{i22} to be useful, we must be able to estimate the term $\|\partial_t u^m(0)\|_{L^2(\Omega)}$. For this, we evaluate \eqref{Galerkin} at $t=0$  to see that for any $\psi\in X_m$,
\begin{multline}\label{ut01}
( \rho\partial_tu^m (0), \psi)+(\frac{\mu}{2}\mathbb{D}u^m(0), \mathbb{D}\psi)
+ \rho_+ g( \eta_{0,+},\psi_3)_{+} + \sigma_+(\nabla_\ast \eta_{0,+} ,\nabla_\ast \psi_3)_{+} \\
- \rj g(\eta_{0,-},\psi_3)_{-} + \sigma_-(\nabla_\ast \eta_{0,-},\nabla_\ast \psi_3)_{-}
=( f(0), \psi) \\
-(g_+(0)-\Pi_\ast(g_+(0)+\mu_+\mathbb{D} u_{ +}^m(0)e_3), \psi_{+})_{+}
-(g_-(0)-\Pi_\ast(g_-(0)-\llbracket \mu\mathbb{D} u^m(0)e_3\rrbracket ), \psi_- )_{-}.
\end{multline}
We may integrate by parts in the second, fourth and sixth terms to rewrite \eqref{ut01} as
\begin{multline}\label{ut02}
( \rho\partial_tu^m (0), \psi)=( \mu\Delta u^m(0)+f(0),  \psi)  \\
- \rho_+g( \eta_{0,+},\psi_3)_{+} +\sigma_+(\Delta_\ast \eta_{0,+} ,  \psi_3)_{+}
+\rj g( \eta_{0,-},\psi_3)_{-} +\sigma_-(\Delta_\ast \eta_{0,-},\psi_3)_{-} \\
 -(\Pi_\ast^\perp(g_+(0)+\mu_+\mathbb{D}u^m_{+}(0) e_3), \psi_+)_{+} - (\Pi_\ast^\perp(g_-(0)-\llbracket \mu\mathbb{D} u^m(0)e_3\rrbracket ), \psi_-)_{-}.
\end{multline}
Now we note that, due to the appearance of the projection $\Pi_\ast^\perp$, the sum of the last two terms in \eqref{ut02} is equal to
\begin{equation}\label{iii3}
( (g_+(0)+\mu_+\mathbb{D}u^m_{+}(0) e_3)\cdot e_3, \psi_{3,+})_{+}  +  ((g_-(0)+\llbracket \mu \mathbb{D} u^m(0) e_3 \rrbracket )\cdot e_3, \psi_{3,-})_{-}.
\end{equation}
As such, only the third component of $\psi$ appears in the boundary integrals in \eqref{ut02}. Since $\Sigma_\pm$ are flat and $e_3$ is the normal vector, we may bound the right-hand side of \eqref{ut02} in terms of $\|\psi\|_{L^2(\Omega)}$ through the following well-known inequality (a consequence of trace theory and the divergence theorem)
\begin{equation}\label{divineq}
 \|v_3\|_{H^{-1/2}(\Sigma)}\lesssim\|v\|_{L^2(\Omega)}+\|\diverge{v}\|_{L^2(\Omega)}.
\end{equation}
Hence, setting $\psi=\partial_tu^m(0)$ in \eqref{ut02} and making use of \eqref{iii3}--\eqref{divineq}, we find that
\begin{multline}\label{ut03}
\|\partial_tu^m (0)\|_{L^2(\Omega)}^2 \\
\lesssim\left(\|u^m(0)\|_{\ddot{H}^2(\Omega)}+\|f(0)\|_{L^2(\Omega)}\right)\|\partial_tu^m (0)\|_{L^2(\Omega)}
+\|\eta_{0}\|_{{H}^{5/2}(\Sigma)}\|\partial_tu^m_3
(0)\|_{{H}^{-1/2}(\Sigma)} \\
+\|(g_+(0)+\mu_+\mathbb{D}u_{m,+}(0) e_3)\cdot e_3\|_{H^{1/2}(\Sigma_+)}\|\partial_tu_{3,+}^m(0)\|_{H^{-1/2}(\Sigma_+)} \\
+  \|(g_-(0)-\llbracket \mu\mathbb{D}u^m(0)e_3\rrbracket )\cdot e_3\|_{H^{1/2}(\Sigma_-)}\|\partial_t u^m_{3,-}
(0)\|_{H^{-1/2}(\Sigma_-)} \\
\lesssim \left(\|u_0\|_{2}+\|\eta_{0}\|_{5/2}+\|f(0)\|_{L^2(\Omega)}+\|g(0)\|_{H^{1/2}(\Sigma)}\right)\|\partial_tu^m (0)\|_{L^2(\Omega)},
\end{multline}
which in turn yields an estimate for $\|\partial_tu^m (0)\|_{L^2(\Omega)}^2$ by division.  Using the resulting estimate, we may integrate \eqref{i22} in time to deduce the bound
\begin{multline}\label{es2}
\sup_{0\le t\le T}\left( \|\partial_tu^m \|_{L^2(\Omega)}^2+\|u^m_3\|_{H^1(\Sigma)}^2 \right)
+\|\partial_t u^m\|_{L^2H^1}^2 \\
\lesssim\|u_0\|_{2}^2+\|\eta_{0}\|_{5/2}^2+\|f(0)\|_{L^2(\Omega)}^2+\|g(0)\|_{H^{1/2}(\Sigma)}^2+\|\partial_t
f\|_{L^2 \mathcal{W}^\ast }^2+\|\partial_t g\|_{L^2H^{-1/2} }^2.
\end{multline}

{\it Step 3. Passing to the limit.} Now we utilize the energy estimates \eqref{es1} and \eqref{es2} to pass to the limit as  $m\rightarrow\infty$. First, we can bound further the right-hand sides of \eqref{es1}, \eqref{es2} by
\begin{multline}
C(1+T)(\|u_0\|_{2}+ \|g(0)\|_{H^{1/2}(\Sigma)})+\|\eta_{0}\|_{5/2}+\|f(0)\|_{L^2(\Omega)}
\\
+\|f\|_{L^2L^2 }^2+\|g\|_{L^2L^2}^2+\|\partial_t f\|_{L^2 \mathcal{W}^\ast }^2+\|\partial_t g\|_{L^2H^{-1/2}}^2.
\end{multline}
This implies that the sequences $\{u^m\}, \{\partial_tu^m\}$ are uniformly bounded in $L^\infty(0,T;
L^2(\Omega))$ as well as in $L^2(0,T;H^1(\Omega))$ and $\{\int_0^t u^{m}_3 ds+\eta_{0}\}$, $\{ u^m_3\}$ are uniformly bounded in $L^\infty(0,T; H^1(\Sigma))$. Hence, up to the extraction of a subsequence, we have
\begin{equation}\label{convergence1}
\left\{\begin{array}{lll}
u^m\rightarrow u\ \text{ weakly-}\ast\text{ in } L^\infty(0,T; L^2(\Omega))\ \text{ and weakly in } L^2(0,T;H^1(\Omega)),\\
\partial_t u^m\rightarrow \partial_tu\ \text{ weakly-}\ast\text{ in }L^\infty(0,T; L^2(\Omega))\ \text{ and weakly in } L^2(0,T;H^1(\Omega)),\\
 \displaystyle\int_0^tu^m_3 \,ds+\eta_{0}\rightarrow \int_0^tu_3 \,ds+\eta_{0}\text{ weakly-}\ast \text{ in } L^\infty(0,T;  {H}^1(\Sigma))
 ,\\
 \displaystyle u^m_3 \rightarrow  u_3  \text{ weakly-}\ast \text{ in } L^\infty(0,T;  {H}^1(\Sigma)).\end{array}\right.
\end{equation}
To get rid of the last two terms in \eqref{i11} we use the fact that $u^m(0)\rightarrow u_0$ in $\mathcal{V} \cap \ddot{H}^2(\Omega)$ together with the fact that $u_0,g(0)$ satisfy the compatibility conditions \eqref{compatibility11} to see that
\begin{equation}\label{convergence2}
\|\Pi_\ast(g_+(0)+\mu_+\mathbb{D} u^m_{+}(0)e_3)\|_{H^{1/2}(\Sigma_+)} +\|\Pi_\ast(g_-(0)-\Lbrack \mu\mathbb{D}
u^m(0)e_3\Rbrack)\|_{H^{1/2}(\Sigma_-)}\rightarrow0.
\end{equation}

The convergences \eqref{convergence1}--\eqref{convergence2} allow us to pass to the limit in \eqref{Galerkin} to find that
\begin{multline}\label{weak1}
\displaystyle\langle\rho\partial_tu ,\psi\rangle_\ast+(\frac{\mu}{2}\mathbb{D}u, \mathbb{D}\psi)
 + \rho_+g(\eta_{+},\psi_{3,+} )_{+}+\sigma_+(\nabla_\ast \eta_{+},\nabla_\ast\psi_{3,+})_{+}
-\rj g(\eta_{-},\psi_3)_{-} \\
+\sigma_-(\nabla_\ast\eta_{-},\nabla_\ast \psi_3)_{-}
 =\langle f, \psi\rangle_\ast +\langle g,\psi\rangle_{-1/2} \text{ for all } \psi\in  \mathcal{V} \text{ and  a.e } t\in [0,T].
\end{multline}
We may also pass to the limit in \eqref{es1}--\eqref{es2} and employ weak lower-semicontinuity to find that
\begin{equation}\label{pll1}
\|u\|_{L^\infty L^2 }^2+\|u\|_{L^2  H^1 }^2+\|\partial_tu\|_{L^\infty L^2 }^2
+\|\partial_tu\|_{L^2 H^1 }^2 +\|u_3\|_{L^\infty
H^1(\Sigma)}^2+\|\eta\|_{L^\infty  H^1
}^2+\|\partial_t\eta\|_{L^\infty H^1 }^2\lesssim
\mathcal{Z}_0.
\end{equation}
That is, $u$ is a pressureless weak  solution of \eqref{sLP0}.

We can now introduce the pressure $p$.   Define the functional $\Lambda_t\in  \mathcal{W}^\ast$ so that $\Lambda_t(v)$ equals to the difference between the left- and right- hand sides of \eqref{weak1}, with $\psi$ replaced by $v\in \mathcal{W}$. Then $\Lambda_t\equiv0$ on $ \mathcal{V}$, so by Proposition \ref{lagrange} there exists a unique $p(t) \in L^2(\Omega)$ so that $(p(t),\diverge{v})=\Lambda_t(v)$ for all $v \in \mathcal{W}$. This yields that $(u,p)$ is a weak solution of \eqref{sLP0} in the sense of \eqref{sLP0ws} satisfying the estimate \eqref{weakest}, and by \eqref{sLP0ws} we have $p\in L^2(0,T; L^2(\Omega))$ and satisfies the estimate $\eqref{weakpes}$. The uniqueness of weak solutions follows directly from \eqref{weakest}--\eqref{weakpes}.
\end{proof}

%%%%%%%%%%%%%%%%%%%%%%%%%%%%%%%%%%%%%%%%%%%%%%%
\subsubsection{Strong solution of \eqref{surfaceLP}}
%%%%%%%%%%%%%%%%%%%%%%%%%%%%%%%%%%%%%%%%%%%%%%%

We now turn to the issue of higher regularity of the solutions constructed in Theorem \ref{weak_wp}.  In the case without surface tension we were able to directly apply the elliptic regularity results of Theorem \ref{cStheorem} in order to deduce higher regularity of the weak solutions.  However, in the case with surface tension, the appearance of the surface terms $\int_0^tu_3\,ds$  prevents us from directly using Theorem \ref{cStheorem}.  Our way around this obstacle is to use difference quotients to gain control of some horizontal derivatives, which suffice for higher regularity estimates on the interfaces $\Sigma_\pm$.  These estimates then allow us to employ the one phase elliptic regularity in each of $\Omega_\pm$.  The technical challenge in this procedure is that the lack of flatness in the lower boundary $\Sigma_b$ requires a localization argument like the one used in the proof of Theorem \ref{cStheorem}.

\begin{theorem}\label{surlinear}
Suppose that $f,g,u_0,\eta_0$ satisfy \eqref{data1}--\eqref{data2}, and that $u_0,g(0)$ satisfy the compatibility conditions \eqref{compatibility11}. Then there exists a unique strong solution
$(u,p,\eta)$ to \eqref{surfaceLP} so that
 \begin{equation}\label{regggg}
\begin{array}{lll}
u\in   C([0,T];   \mathcal{V}\cap\ddot{H}^2(\Omega))\cap L^2(0,T; \ddot{H}^3(\Omega)),\\\partial_t u\in C([0,T]; L^2 (\Omega))\cap L^2(0,T;H^1(\Omega)),\
\rho\partial_t^2 u\in L^2(0,T;\mathcal{W}^\ast),
\\ \nabla \partial_t u_+\in  L^2(0,T;H^{-1/2}(\Sigma_+)),\
\llbracket \mu\nabla \partial_t u\rrbracket \in
L^2(0,T;H^{-1/2}(\Sigma_-));
  \\  \eta\in  C([0,T];   {H}^3(\Sigma))\cap L^2(0,T; {H}^{7/2}(\Omega)),
  \\ \partial_t \eta\in  C([0,T];   {H}^{3/2}(\Sigma))\cap L^2(0,T; {H}^{5/2}(\Sigma)),
  \\\partial_t^2 \eta\in L^\infty(0,T; H^{-1/2}(\Sigma))\cap L^2(0,T;
  {H}^{1/2}(\Sigma)),
\\  p\in  C([0,T]; \ddot{H}^1 (\Omega))\cap L^2(0,T;
\ddot{H}^2(\Omega)),\ \partial_tp\in L^2(0,T;L^2(\Omega)),
\\ \partial_t p_+\in  L^2(0,T;H^{-1/2}(\Sigma_+)),\
\llbracket \partial_t p\rrbracket \in
L^2(0,T;H^{-1/2}(\Sigma_-)).
\end{array}
\end{equation}
The solution satisfies the estimate
\begin{multline}\label{strongest}
\|u\|_{L^\infty H^2 }^2 + \|u\|_{L^2  H^3 }^2 + \|\partial_tu\|_{L^\infty L^2 }^2
+\|\partial_tu\|_{L^2 H^1 }^2 \\
 +\|\rho\partial_t^2u\|_{L^2 \mathcal{W}^\ast }^2+ \|\nabla \partial_t u_+\|_{ L^2
H^{-1/2}(\Sigma_+) }^2+ \|\llbracket \mu \nabla \partial_t
u\rrbracket \|_{ L^2 H^{-1/2}(\Sigma_-) }^2 \\
+ \|\eta\|_{L^\infty H^3 }^2 +\|\eta\|_{L^2 H^{7/2} }^2  +\|\partial_t\eta\|_{L^\infty
H^{3/2} }^2 + \|\partial_t\eta\|_{L^2 H^{5/2} }^2
+  \|\partial_t^2\eta\|_{L^\infty H^{-1/2} }^2 + \|\partial_t^2\eta\|_{L^2 H^{1/2} }^2 \\
+ \|p\|_{ L^\infty H^1 }^2+ \|p\|_{ L^2 H^2 }^2+ \|\partial_tp\|_{ L^2 L^2 }^2 + \|\partial_tp_+\|_{ L^2 H^{-1/2}(\Sigma_+) }^2 + \|\llbracket  \partial_t p\rrbracket \|_{ L^2 H^{-1/2}(\Sigma_-) }^2 \\
\le C_4(\|u_0\|_{2}^2+\|\eta_0\|_{3}^2+\|f(0)\|_{ 0}^2+\| g(0)\|_{1/2}^2+\|f\|_{ L^\infty L^2 }^2+ \|g\|_{L^\infty {H}^{1/2} }^2 \\
+\|f\|_{ L^2 {H}^{1} }^2+\|\partial_tf\|_{ L^2  \mathcal{W}^\ast}^2 + \|g\|_{L^2 {H}^{3/2} }^2+\|\partial_tg\|_{L^2 {H}^{-1/2}}^2).
\end{multline}
The initial pressure, $p(0)\in \ddot{H}^1(\Omega)$, is determined in terms of $u_0,\eta_0,f(0),g(0)$ as the weak solution (in the sense of \eqref{poisson0}) to
\begin{equation}\label{pressure0}
\left\{\begin{array}{ll}
 \diverge \left(\rho^{-1}(\nabla  p(0)-f(0)) \right)=0 \ &\text{ in}\ \Omega
 \\p_+(0)= g_{+}^3(0) + 2\mu_+ \partial_3 u_{3,+}(0)+\rho_+g\eta_+(0)-\sigma_+\Delta_\ast\eta_+(0) &\text{ on }\Sigma_+
\\
 \Lbrack p(0)\Rbrack= -g_{-}^3(0) +2\Lbrack \mu  \partial_3 u_{3}(0)\Rbrack+\rj g\eta_-(0)+\sigma_-\Delta_\ast\eta_-(0)  \ &\text{ on }\Sigma_-
  \\ \Lbrack\rho^{-1}(\partial_3 p(0)-f^3(0))  \Rbrack  =\Lbrack \rho^{-1}\mu \Delta u_ 3(0)\Rbrack  \ &\text{ on }\Sigma_-
\\  \rho^{-1}(\nabla p_-(0)-f_-(0))\cdot\nu=\rho_-^{-1}\mu_-\Delta u_{ -}(0)\cdot\nu&\text{ on }\Sigma_b.\end{array}\right.
\end{equation}
Also, $\partial_t\eta(0)=u_3(0)$ on $\Sigma$ and
\begin{equation}
\partial_tu(0)=\rho^{-1}(\Delta u_0-\nabla p(0)+f(0))\in \mathcal{Y}.
\end{equation}
\end{theorem}

\begin{proof}
In order to introduce horizontal derivatives in the weak formulation \eqref{sLP0ws} we restrict to test functions that are localized away from $\Sigma_b$.  That is, we replace the $v$ in \eqref{sLP0ws} with $\chi v$, where $\chi$ is the cut-off function supported away from $\Sigma_b$ defined by \eqref{chi_properties} and $v\in \mathcal{W}$. After a simple computation, we find that $w:=\chi u$ and  $q:=\chi p$ satisfy
\begin{multline}\label{localweakdef}
 (\rho\partial_t w,v)+(\frac{\mu}{2}\mathbb{D}w, \mathbb{D}v)-(q,\diverge{v})
  + \rho_+g(\eta_{+},v_{3,+} )_{+} + \sigma_+(\nabla_\ast \eta_{+},\nabla_\ast v_{3,+})_{+}
\\
-\rj g(\eta_{-},v_3)_{-}+\sigma_-(\nabla_\ast\eta_{-},\nabla_\ast
v_3)_{-}  =( \tilde{f}, v) -( g, v)_{\Sigma},\ \forall v\in  \mathcal{W}
\end{multline}
with $\diverge w=  \partial_3 \chi  u_3$ and
\begin{equation}
\tilde{f}=\chi f+\partial_3 \chi  (  p e_3 - 2  \partial_3 u) - \partial_3^2 \chi   u.
\end{equation}

Since the support of $w$ is away from $\Sigma_b$ we have that $D_{-h}D_h w \in \mathcal{W}$, where $D_h$ is the standard horizontal difference quotient in any horizontal direction $h \in \mathbb{R}^2$.  Then we can let $v=D_{-h}D_h w$ in \eqref{localweakdef} to find that
 \begin{multline}\label{toes1}
 \frac{1}{2}\frac{d}{dt}\left(\|\sqrt{\rho}D_h w\|_{L^2(\Omega)}^2+  \rho_+g\|  D_h \eta_+\|_{L^2(\Sigma_+)}^2+ \sigma_+\|\nabla_\ast D_h \eta_+\|_{L^2(\Sigma_+)}^2 \right.
\\
\left. -\rj g\|D_h \eta_- \|_{L^2(\Sigma_-)}^2 +\sigma_-\|\nabla_\ast D_h \eta_- \|_{L^2(\Sigma_-)}^2 \right) + (\frac{\mu}{2}\mathbb{D}D_h w, \mathbb{D}D_h w)
  \\
=(q,\diverge D_{-h}D_h w)+( \tilde{f}, D_{-h}D_h w) +( g, D_{-h}D_h w)_{\Sigma}.
\end{multline}
By Korn's inequality, we have
\begin{equation}\label{j11}
(\frac{\mu}{2}\mathbb{D}D_h w, \mathbb{D}D_h w) \ge C\| D_h
w\|_{H^1(\Omega)}^2.
\end{equation}
We estimate the right hand side of \eqref{toes1} as follows,  recalling the definitions of $w,q$,
\begin{equation}\label{j12}
(q,\diverge D_{-h}D_h w)=( q, \partial_3 \chi D_{-h}D_h u_3)\le C(\varepsilon)\|p\|_{L^2(\Omega)}^2 +
\varepsilon \| D_h w\|_{H^1(\Omega)}^2,
\end{equation}
\begin{multline}\label{j13}
( \tilde{f}, D_{-h}D_h w) \le \|\tilde{f}\|_{L^2(\Omega)}\|D_{-h}D_h w\|_{L^2(\Omega)}
\\
\le C(\varepsilon)(\|f\|_{L^2(\Omega)}^2 + \|p\|_{L^2(\Omega)}^2 + \|u\|_{H^1(\Omega)}^2)
 + \varepsilon \| D_h w\|_{H^1(\Omega)}^2,
\end{multline}
\begin{multline}\label{j14}
( g, D_{-h}D_h w)_{\Sigma}\le \|g\|_{H^{1/2}(\Sigma)}\| D_{-h}D_h w\|_{H^{-1/2}(\Sigma)}
 \le C \|g\|_{H^{1/2}(\Sigma)}\|  D_h w\|_{H^{ 1/2}(\Sigma)}
 \\
  \le C(\varepsilon)\|g\|_{H^{1/2}(\Sigma)}^2+\varepsilon\|  D_h w\|_{H^{ 1}(\Omega)}^2.
\end{multline}
Note that in the second inequality in \eqref{j14} we have used the general estimate
\begin{equation}
 \| D_h \varphi \|_{H^{-1/2}(\mathrm{T}^2)} \lesssim \|\varphi\|_{H^{1/2}(\mathrm{T}^2)} \text{ for all } \varphi \in H^{1/2}(\mathrm{T}^2),
\end{equation}
which may be derived from the fact that for $h \in \mathbb{R}^2$,
\begin{equation}
 |\widehat{D_h \varphi}(n)|^2 = |\hat{\varphi}(n)|^2 \frac{|e^{i n \cdot h} - 1|^2}{|h|^2} \le
 C |n|^2 |\hat{\varphi}(n)|^2 \text{ for } n \in L_1^{-1} \mathbb{Z} \times L_2^{-1} \mathbb{Z},
\end{equation}
where $\hat{\cdot}$ denotes the Fourier transform on $\mathrm{T}^2 = (2\pi L_1 \mathbb{T})\times(2\pi L_2 \mathbb{T})$.  Now we integrate \eqref{toes1}  in time from $0$ to $t\in[0,T]$ and employ the estimates \eqref{j11}--\eqref{j14} and \eqref{weakest}--\eqref{weakpes} in addition to the usual relation between $D_h$ and $\nabla_\ast$ to obtain the bound
\begin{multline}\label{toes2}
 \sup_{0\le t\le T}\left\{\|\nabla_\ast w\|_{L^2(\Omega)}^2+ \|\nabla_\ast\eta\|_{H^1(\Sigma)}^2\right\}
+\|\nabla_\ast w\|_{L^2H^1 }^2 \\
\lesssim \|u_0\|_{1}^2+\|\eta_0\|_2^2+ \|f\|_{L^2 L^2}+\|g\|_{L^2H^{1/2} }+\|p\|_{L^2 L^2 }^2
+\|u\|_{L^2 H^1}^2 \lesssim \mathcal{Z}(1+T),
\end{multline}
where we have compactly written $\mathcal{Z}$ for the right hand side of \eqref{strongest} (containing \eqref{weakest}). Since $u=w$ on $\Sigma$ we may use  \eqref{weakest} and \eqref{toes2} to see that
\begin{equation}\label{toesb1}
\begin{split}
\|u\|_{L^2(0,T;H^{3/2}(\Sigma)) }^2
   &\lesssim \|u\|_{L^2(0,T;H^{1/2}(\Sigma)) }^2 +\|\nabla_\ast w\|_{L^2(0,T;H^{1/2}(\Sigma)) }^2 \\
    &\lesssim \|u\|_{L^2 H^1}^2 +\|\nabla_\ast w\|_{L^2 H^1 }^2 \lesssim\mathcal{Z}(1+T).
\end{split}
\end{equation}

At this point we cannot apply Theorem \ref{cStheorem} to improve any of our estimates since we do not yet have control of the boundary terms involving $\eta$.  To circumvent this problem we will employ one-phase elliptic regularity theory.  Note that for  a.e. $t\in[0,T]$ the pair $(u_+(t),p_+(t))$ is the unique weak solution  to the elliptic problem \eqref{cS1phase2} in $G=\Omega_+$, with $f$ replaced by $f(t)-\rho\partial_t u(t)$, $g=0$ and $\varphi=u_+(t)$ on $\Sigma$. By Lemma \ref{cS1phaselemma2}, \eqref{toesb1} and an integration in time from $0$ to $t$, we have
\begin{equation}\label{toesd1}
\begin{split}
\|u_+\|_{L^2(0,T;{H}^{2}(\Omega_+)) }^2 + \|\nabla p_+\|_{L^2(0,T;L^2(\Omega_+)) }^2
   &\lesssim \|f\|_{L^2 L^2}^2+\|\partial_t u\|_{L^2 L^2}^2+\|u\|_{L^2(0,T;H^{3/2}(\Sigma)) }^2
   \\& \lesssim\mathcal{Z}(1+T).
\end{split}
\end{equation}
A similar argument yields an estimate for $(u_-,p_-)$; then together with \eqref{weakest} we have
\begin{equation}\label{toesd2}
\|u\|_{L^2(0,T;\ddot{H}^{2}(\Omega)) }^2+\| p\|_{L^2(0,T;\ddot{H}^1(\Omega)) }^2 \lesssim\mathcal{Z}(1+T).\end{equation}

Now  we use the test function  $v=D_{-h}D_h D_{-h}D_h w$ in \eqref{localweakdef} to  find that
\begin{multline}\label{toes12}
 \frac{1}{2}\frac{d}{dt} \left\{\|\sqrt{\rho}D_{-h}D_h w\|_{L^2(\Omega)}^2+  \rho_+g\|  D_{-h}D_h \eta_+\|_{L^2(\Sigma_+)}^2+ \sigma_+\|\nabla_\ast D_{-h}D_h \eta_+\|_{L^2(\Sigma_+)}^2\right.
 \\
\left.-\rj  g \| D_{-h}D_h \eta_- \|_{L^2(\Sigma_-)}^2 + \sigma_- \|\nabla_\ast D_{-h} D_h \eta_-\|_{L^2 (\Sigma_-) }^2 \right\}+(\frac{\mu}{2}\mathbb{D}D_{-h}D_h  w, \mathbb{D}D_{-h}D_h  w)  \\
=(q,\diverge D_{-h}D_h D_{-h}D_h   w)+( \tilde{f}, D_{-h}D_h D_{-h}D_h  w) +( g, D_{-h}D_h D_{-h}D_h w)_{\Sigma}.
\end{multline}
%By Korn's inequality, we have
%    \begin{equation}\label{j21}
%(\frac{\mu}{2}\mathbb{D} D_{-h}D_h w, \mathbb{D}D_{-h}D_h w) \ge C\|
%D_{-h}D_h w\|_{H^1(\Omega)}^2.\end{equation} and we estimate the
%right hand side of \eqref{toes1} as follows
%  \begin{align}
%&\label{j22}(q,\diverge D_{-h}D_hD_{-h}D_h w)=(D_h q,\diverge
%D_hD_{-h}D_h w)=(D_h q, \partial_3 \chi  D_h D_{-h}D_h u_3)
%\nonumber\\&\qquad\qquad\qquad\qquad\quad\quad\ \ \le
%C(\varepsilon)\|p\|_{H^1(\Omega)}^2+\varepsilon \| D_{-h}D_h
%w\|_{H^1(\Omega)}^2,
%\\&\label{j23}( \tilde{f}, D_{-h}D_h D_{-h}D_h w) \le \|D_h \tilde{f}\|_{L^2(\Omega)}\|D_h D_{-h}D_h   w\|%_{L^2(\Omega)}
%\nonumber\\&\qquad\qquad\qquad\qquad \ \,\,\, \le
%C(\varepsilon)(\|f\|_{H^1(\Omega)}^2+\|p\|_{H^1(\Omega)}^2+\|u\|_{H^2(\Omega)}^2)+\varepsilon
%\| D_{-h}D_h w\|_{H^1(\Omega)}^2 ,\\ &\label{j24}( g, D_{-h}D_h
%D_{-h}D_h w)_{\Sigma}
%  \le \|D_h
%g\|_{H^{1/2}(\Sigma)}\|D_h D_{-h}D_h w\|_{H^{-1/2}(\Sigma)}
%  \nonumber\\&\qquad\qquad\qquad\qquad \ \,\,\,\,\,\, \le \|g\|_{H^{3/2}(\Sigma)}\|  D_{-h}D_h  w\|_{H^{ 1/2}%(\Sigma)}
%\nonumber\\&\qquad\qquad\qquad\qquad \ \,\,\,\,\,\,\le C(\varepsilon)\|g\|_{H^{3/2}(\Sigma)}^2+\varepsilon\|  %D_{-h}D_h  w\|_{H^{ 1}(\Omega)}^2.\end{align}
Arguing as above (deriving estimates of the form \eqref{j11}--\eqref{j14}), we may deduce from \eqref{toes12} that
\begin{multline}\label{toes3}
 \sup_{0\le t\le T}\left\{\|\nabla_\ast^2 w\|_{L^2(\Omega)}^2+ \|\nabla_\ast^2\eta\|_{H^1(\Sigma)}^2\right\} +\|\nabla_\ast^2 w\|_{L^2H^1 }^2 \\
\lesssim \|w_0\|_{2}^2+\|\eta_0\|_3^2+ \|f\|_{L^2 H^1}+\|g\|_{L^2H^{3/2} }+\|p\|_{L^2 H^1}^2+\|u\|_{L^2 H^2}^2 \lesssim\mathcal{Z}(1+T).
\end{multline}
Then by \eqref{weakest} and \eqref{toes3} we have that
\begin{equation}\label{toesb2}
\begin{split}
\|u\|_{L^2(0,T;H^{5/2}(\Sigma)) }^2
   &\lesssim \|u\|_{L^2(0,T;H^{1/2}(\Sigma)) }^2 +\|\nabla_\ast^2 w\|_{L^2(0,T;H^{1/2}(\Sigma)) }^2 \\
    &\lesssim \|u\|_{L^2 H^1}^2 +\|\nabla_\ast^2 w\|_{L^2 H^1 }^2 \lesssim\mathcal{Z}(1+T).
\end{split}
\end{equation}
and
\begin{equation}\label{toese134}
\|\eta\|_{L^\infty(0,T;{H}^{3}(\Sigma)) }^2 \lesssim\mathcal{Z}(1+T).
\end{equation}
As before, applying Lemma \ref{cS1phaselemma2} in $\Omega_\pm$ then yields the estimate
\begin{equation}\label{toesd3}
\|u\|_{L^2(0,T;\ddot{H}^{3}(\Omega)) }^2+\|  p\|_{L^2(0,T;\ddot{H}^2(\Omega)) }^2
  \lesssim\mathcal{Z}(1+T).
\end{equation}

The regularity we have now established allows us to deduce from \eqref{sLP0ws} that $(u,p,\eta)$ solve the problem \eqref{surfaceLP} in the strong sense.  In order to complete the estimates in a time-independent fashion  we now need to improve the estimate \eqref{toesd3}.  We replace the estimates \eqref{j12}--\eqref{j13} by
\begin{equation}\label{j122}
(q,\diverge D_{-h}D_h w)=(D_h q, \partial_3 \chi D_h u_3)\le C(\varepsilon) \| u\|_{H^1(\Omega)}^2 +
\varepsilon \|\nabla p\|_{L^2(\Omega)}^2,
\end{equation}
\begin{multline}\label{j132}
(\tilde{f}, D_{-h}D_h w) \le C (\| f \|_{L^2(\Omega)} + \| \partial_3 u \|_{L^2(\Omega)} ) \|D_{-h}D_h w\|_{L^2(\Omega)} + C \| D_h p \|_{L^2(\Omega)} \| D_h w \|_{L^2(\Omega)}
\\
\le C(\varepsilon)(\|f\|_{L^2(\Omega)}^2+\|u\|_{H^1(\Omega)}^2)+\varepsilon \| D_h w\|_{H^1(\Omega)}^2 + \varepsilon \|\nabla p\|_{L^2(\Omega)}^2.
\end{multline}
Hence, we can replace \eqref{toes2} by
\begin{multline}\label{toes22}
 \sup_{0\le t\le T}\left\{\|\nabla_\ast w\|_{L^2(\Omega)}^2+ \|\nabla_\ast\eta\|_{H^1(\Sigma)}^2\right\}
+\|\nabla_\ast w\|_{L^2H^1 }^2  \\
\lesssim \|u_0\|_{1}^2+\|\eta_0\|_2^2+ \|f\|_{L^2 L^2}+\|g\|_{L^2H^{1/2} } + \varepsilon\|\nabla p\|_{L^2 L^2 }^2 +\|u\|_{L^2 H^1}^2  \\
\lesssim\mathcal{Z} + \varepsilon\|\nabla p\|_{L^2 L^2 }^2.
\end{multline}
Furthermore, after applying the one-phase Stokes elliptic regularity theory, we can replace \eqref{toesd2} by
\begin{equation}
\|u\|_{L^2(0,T;\ddot{H}^{2}(\Omega)) }^2+\|\nabla p\|_{L^2(0,T;L^2(\Omega)) }^2 \lesssim \mathcal{Z} + \varepsilon \|\nabla p\|_{L^2 L^2 }^2.
\end{equation}
We then take $\varepsilon$ to be sufficiently small to absorb the terms on the right into the left; this allows us to improve the estimate \eqref{toesb2}:
\begin{equation}\label{toesd22}
\|u\|_{L^2(0,T;\ddot{H}^{2}(\Omega)) }^2+\|\nabla p\|_{L^2(0,T;L^2(\Omega)) }^2  \lesssim\mathcal{Z}.
\end{equation}
Similarly, we we can improve the estimate \eqref{toesd3} with
\begin{equation}\label{toesd32}
\|u\|_{L^2(0,T;\ddot{H}^{3}(\Omega)) }^2+\| \nabla p\|_{L^2(0,T;\ddot{H}^1(\Omega)) }^2 \lesssim\mathcal{Z}
\end{equation}
and improve  \eqref{toese134}  with (employing Poincar\'{e}'s inequality)
\begin{equation}\label{toese88}
\|\eta\|_{L^\infty(0,T;{H}^{3}(\Sigma)) }^2 \lesssim\mathcal{Z}.
\end{equation}

Now that we have \eqref{toesd32}, \eqref{toese88} and \eqref{weakest}, we estimate the remaining terms in the left hand side of \eqref{strongest}. First, for the time derivatives of $\eta$ we use the kinematic
boundary conditions and the trace theorem to obtain
\begin{equation}\label{toesd52}
 \|\partial_t\eta\|_{L^2(0,T;H^{5/2}(\Sigma))}^2 + \|\partial_t^2\eta\|_{L^2(0,T;H^{1/2}(\Sigma))}^2
  = \|u_3\|_{L^2(0,T;H^{5/2}(\Sigma))}^2 + \|\partial_t u_3\|_{L^2(0,T;H^{1/2}(\Sigma))}^2
\lesssim \mathcal{Z}.
\end{equation}
Second,  for $\eta$ itself we use the dynamic boundary conditions
\begin{equation} \label{ellip1}
-\sigma_+\Delta_\ast\eta_++\rho_+g\eta_+  = p_+-2\mu_+\partial_3u_{3,+}-g_+\text{ on }\Sigma_+
\end{equation}
and
\begin{equation}\label{ellip2}
 -\sigma_-\Delta_\ast\eta_- e_3-\rj g\eta_-= \llbracket -p +2\mu\partial_3u_3\rrbracket
-g_-\text{ on }\Sigma_-
\end{equation}
Note that we do not yet have an estimate of $p$ on $\Sigma$, but we do have the estimate of $\nabla p$. We may then apply  $\nabla_\ast$ to  \eqref{ellip1}--\eqref{ellip2} and use Theorem \ref{critical} and \eqref{toesd32} to see that
\begin{equation}\label{toesd62}
\| \nabla_\ast \eta\|_{L^2(0,T;H^{5/2}(\Sigma))}^2 \lesssim  \|\nabla_\ast p\|_{L^2(0,T;H^{1/2}(\Sigma))}^2 + \| \nabla_\ast \partial_3u_3\|_{L^2(0,T;H^{1/2}(\Sigma))}^2 + \|g\|_{L^2(0,T;H^{1/2}(\Sigma))}^2
        \lesssim\mathcal{Z}.
\end{equation}
Since $\int_{ \mathrm{T}^2}\eta=0$, we may use Poincar\'e's inequality to deduce from \eqref{toesd62} that
\begin{equation}\label{toesd72}
\| \eta \|_{L^2(0,T;H^{7/2}(\Sigma))}^2 \lesssim\mathcal{Z}.
\end{equation}
With the $\eta$ estimate \eqref{toesd72} in hand, we may derive an estimate of $p$ as in Theorem \ref{dth}.  Indeed, we  use \eqref{toesd72}, the boundary conditions \eqref{ellip1}--\eqref{ellip2},  Poincar\'e's inequality  and the trace theorem to see that
\begin{equation}\label{toesd82}
\|  p\|_{L^2(0,T;\ddot{H}^2(\Omega)) }^2  \lesssim\mathcal{Z}.
\end{equation}

To derive the $L^\infty$-in-time estimates in \eqref{strongest} we note that for a.e. $t\in[0,T]$,
$(u(t),p(t))$ solve the problem
\begin{equation}\label{surfaceelliptic}
\left\{\begin{array}{lll}
-\mu\Delta u+\nabla p=f-\rho\partial_t u\quad&\text{ in }&\Omega
\\ \diverge{u}=0&\text{ in }&\Omega
\\ ( p_+I-\mu_+\mathbb{D}(u_+)) e_3= (\rho_+g\eta_+  -\sigma_+\Delta_\ast\eta_+)e_3+g_+&\text{ on }&\Sigma_+
\\\llbracket u\rrbracket =0,\quad \llbracket pI-\mu\mathbb{D}(u)\rrbracket  e_3=(\rj g\eta_- +\sigma_-\Delta_\ast\eta_-)e_3-g_-&\text{ on }&\Sigma_-
\\ u_-=0 &\text{ on }&\Sigma_b,
\end{array}\right.
\end{equation}
We apply the two-phase elliptic theory of Theorem \ref{cStheorem} to \eqref{surfaceelliptic} and use \eqref{toese88} and \eqref{weakest} to see that
\begin{equation}\label{toesd992}
\|u\|_{L^\infty(0,T; \ddot{H}^{2}(\Omega)) }^2+\|  p\|_{L^\infty(0,T;\ddot{H}^1(\Omega))
}^2  \lesssim\|f\|_{L^\infty L^2}^2 + \|g\|_{L^\infty H^{1/2}}^2 + \|\partial_tu\|_{L^\infty L^2}^2
+ \|\eta\|_{L^\infty H^{5/2}}^2 \lesssim \mathcal{Z}.
\end{equation}
Then from trace theory we see that
\begin{equation}\label{toesd0123}
 \|\partial_t\eta \|_{L^\infty(0,T;H^{3/2}(\Sigma))}^2 = \|u_3\|_{L^\infty(0,T;H^{3/2}(\Sigma))}^2
  \lesssim\mathcal{Z}.
\end{equation}
Similarly, employing the estimate written in \eqref{divineq} and the fact that $ \diverge{\dt u}=0$, we have that
\begin{equation}
 \| \partial_t^2 \eta \|_{L^\infty(0,T;H^{-1/2}(\Sigma))}^2 = \|\dt u_3\|_{L^\infty(0,T;H^{-1/2}(\Sigma))}^2
 \lesssim \|\dt u\|_{L^\infty(0,T;L^{2}(\Omega))}^2 \lesssim \mathcal{Z}.
\end{equation}

Lastly, we derive the boundary estimates for $\partial_t u,\partial_tp$, beginning with $\dt u$.  First we control the horizontal derivatives via a trace estimate:
 \begin{equation}
\|\nabla_\ast\partial_t u\|_{L^2(0,T;H^{-1/2}(\Sigma))}^2
\lesssim \|\partial_t u\|_{L^2(0,T;H^{1/2}(\Sigma))}^2\lesssim \|\partial_t u\|_{L^2(0,T;H^{1}(\Omega))}^2
\lesssim\mathcal{Z}.
\end{equation}
Then we use the incompressibility condition to control $\partial_3 \dt u_3$:
 \begin{equation}
\begin{split}
 \|\partial_3\partial_t u\|_{L^2(0,T;H^{-1/2}(\Sigma))}^2&=\|\partial_1\partial_t u_1+\partial_2\partial_t u_2\|_{L^2(0,T;H^{-1/2}(\Sigma))}^2
           \\& \lesssim \|\nabla_\ast\partial_t u\|_{L^2(0,T;H^{-1/2}(\Sigma))}^2
        \lesssim\mathcal{Z}.
\end{split}
\end{equation}
Now to handle $\partial_3 u_i$ for $i=1,2$ we have to consider $\Sigma_\pm$ as separate cases and utilize the dynamic boundary conditions.  On $\Sigma_+$ we have that $g_+\cdot e_i = -\mu_+(\partial_3 u_i + \partial_i u_3)$, which allows us to estimate
\begin{equation}
 \| \partial_3 \partial_t u_i \|_{H^{-1/2}(\Sigma_+)} \lesssim \| \partial_i \partial_t u_3 \|_{H^{-1/2}(\Sigma_+)} + \|  \partial_t g \|_{H^{-1/2}(\Sigma_+)}  \lesssim \mathcal{Z}.
\end{equation}
Similarly, on $\Sigma_-$ we use the dynamic boundary condition to estimate
\begin{equation}
        \|\llbracket \mu \partial_3 \partial_t
 u_i\rrbracket\|_{L^2(0,T;H^{-1/2}(\Sigma))}^2=\|\dt g_-\cdot e_i-\llbracket \mu \partial_i \partial_t
 u_3\rrbracket\|_{L^2(0,T;H^{-1/2}(\Sigma))}^2
        \lesssim\mathcal{Z}.
\end{equation}
Combining these estimates then yields the bound
\begin{equation}\label{toesd994}
\|\nabla\partial_t u_+\|_{L^2(0,T;H^{-1/2}(\Sigma_+))}^2
+ \|\llbracket \mu\nabla\partial_t u\rrbracket \|_{L^2(0,T;H^{-1/2}(\Sigma_-))}^2  \lesssim\mathcal{Z}.
\end{equation}
Then we solve for $\dt p$ in the dynamic boundary conditions on $\Sigma_\pm$ to find that
\begin{equation}\label{toesd995}
\|\partial_t p_+ \|_{L^2(0,T;H^{-1/2}(\Sigma_+))}^2 + \|\llbracket\partial_t p \rrbracket \|_{L^2(0,T; H^{-1/2}(\Sigma_-))}^2  \lesssim\mathcal{Z}.
\end{equation}

To conclude we  sum over the estimates \eqref{toesd995}, \eqref{toesd994}, \eqref{toesd0123}, \eqref{toesd992},
\eqref{toesd82}, \eqref{toesd72}, \eqref{toesd52},\ \eqref{toese88}, \eqref{toesd32} and  \eqref{weakest}  to get \eqref{strongest}. Here the boundedness of the $L^2\mathcal{W}^*$-norm of $\rho\partial_t^2u$ and the $L^2L^2$-norm of $\partial_tp$ follows in the same way as in Theorem \ref{H2SS}. The proof of the time continuity and the determination of initial data $p(0),\partial_t\eta(0), \partial_tu(0)$  can also be carried out as in Theorem \ref{H2SS}.
\end{proof}

%%%%%%%%%%%%%%%%%%%%%%%%%%%%%%%%%%%%%%%%%%%%%%%
 \subsection{The nonlinear problem}
%%%%%%%%%%%%%%%%%%%%%%%%%%%%%%%%%%%%%%%%%%%%%%%

%%%%%%%%%%%%%%%%%%%%%%%%%%%%%%%%%%%%%%%%%%%%%%%
\subsubsection{Nonlinear estimates}
%%%%%%%%%%%%%%%%%%%%%%%%%%%%%%%%%%%%%%%%%%%%%%%

Now we turn to the nonlinear problem \eqref{surface} to complete the proof of Theorem \ref{surth}. Recall that the nonlinear term $f$ is given by \eqref{f} and $g_\pm$ is given by \eqref{g_+^i}, \eqref{g_+^3}, \eqref{g_-^i}, \eqref{g_-^3}. Also recall that in the case with surface tension, we define the energy $\mathcal{E}$ by \eqref{E} and the dissipation $\mathcal{D}$ by \eqref{D}.  We shall now present estimates of  $f$ and $g$  in terms of $\mathcal{E}$ and $\mathcal{D}$ in the following lemma.

\begin{lemma}\label{nonlineares}
There exists $\delta>0$ so that if $\mathcal{E}\le \delta$, then
\begin{equation}\label{no1}
   \|f\|_{L^2}^2+\|g\|_{H^{1/2}}^2\lesssim   \mathcal{E}^2,
\end{equation}
\begin{equation}\label{no2}
   \|f\|_{H^1}^2+\|g\|_{H^{3/2}}^2\lesssim   \mathcal{E}\mathcal{D},
\end{equation}
\begin{equation}\label{no3}
   \|\partial_tf\|_{\mathcal{W}^\ast}^2+\|\partial_tg\|_{H^{-1/2}}^2\lesssim   \mathcal{E}\mathcal{D}.
\end{equation}
\end{lemma}
\begin{proof}
The full expressions that define $f$ and $g$   are rather complicated, and a full analysis of each term would be tedious.  As such, we will identify only the principal terms in the expressions and provide details for only the most delicate estimates.  The other terms (lower order terms) may be handled through straightforward modifications of the arguments presented here.

The principal terms appearing in $f$ and $g$ may be identified by first examining the regularity of the terms appearing in $\mathcal{E}$ and $\mathcal{D}$, as defined by \eqref{E}--\eqref{D}, and by appealing to Lemma \ref{Poi} to compare the regularity of $\bar{\eta}$ to $\eta$.  We find that, roughly speaking, the regularity of $u$ (and its time derivatives) is same as $\partial_t\bar{\eta}$, one order higher than $p$ and at least one order lower than $\bar{\eta}$. Based on this, we identify the principal terms in $f$ and $g$ as
\begin{equation}\label{app1}
f\sim \nabla^3\bar{\eta}    u+\nabla^2\bar{\eta} \nabla  u+\nabla\bar{\eta}\, \mu\nabla^2 u
\end{equation}
and
\begin{equation}\label{app2}
g_+\sim \nabla_\ast\eta_+\nabla u_++\nabla_\ast^2\eta_+u_+,\quad
g_- \sim \nabla_\ast\eta_-\Lbrack\mu\nabla u\Rbrack+\nabla_\ast^2\eta_-u.
\end{equation}

We will prove only the most involved estimate, namely \eqref{no3}.  The other two follow estimates may be derived through somewhat easier arguments.  By Lemma \ref{Poi}--\ref{-1norm}, we have
\begin{equation}
\begin{split}
\|\partial_t f\|_{ \mathcal{W}^\ast} & \le \|\partial_t f\|_{\Hd} \lesssim
\|\partial_t \nabla^3\bar{\eta}\|_{0} \|   u\|_1+\|    \nabla^3\bar{\eta}\|_0\|\partial_t u\|_1
 + \|\partial_t \nabla^2\bar{\eta} \|_0 \|\nabla u\|_1
\\
&+ \| \nabla^2\bar{\eta}\|_1\| \partial_t\nabla u\|_0 +\|\partial_t \nabla\bar{\eta} \|_0\|\nabla^2 u\|_1
  +\|\nabla\bar{\eta} \|_{2}  \|\mu\partial_t\nabla^2 u\|_{\Hd}
\\
&\lesssim \|\partial_t  {\eta}\|_{5/2} \|   u\|_1+\| {\eta}\|_{5/2}\|\partial_t u\|_1
 +\|\partial_t  {\eta} \|_{3/2} \|\nabla u\|_1
\\
&+\|  {\eta}\|_{5/2}\| \partial_t  u\|_1 + \|\partial_t  {\eta} \|_{1/2}\|  u\|_3  + \| {\eta} \|_{5/2}\|\mu\partial_t\nabla^2 u\|_{\Hd}
 \\
&\lesssim \sqrt{\mathcal{E}\mathcal{D}}.
\end{split}
\end{equation}
Here, in the last inequality, we have used the  estimate
\begin{equation}\label{a_bound}
 \|\mu\partial_t \nabla^2 u\|_{\Hd} \ls \|\partial_t u\|_{1} + \|\partial_t \nabla u_+\|_{H^{-1/2}(\Sigma_+)}
+ \|\partial_t \Lbrack\mu \nabla u\Rbrack \|_{H^{-1/2}(\Sigma_-)}\ls \sqrt{\mathcal{D}}.
\end{equation}
Indeed,   by  H\"older's inequality and the trace theorem, we obtain that for any $\varphi\in \H(\Omega)$ and any $i,j=1,2,3$,
\begin{multline}
\langle\mu\partial_t \partial_i \partial_j u,\varphi\rangle_\ast = -\int_\Omega \mu \partial_t\partial_i u \cdot \partial_j \varphi \,dx
+\mu_+ \int_{\Sigma_+} \partial_t \partial_i u_+ \cdot \varphi  (e_3\cdot e_j) - \int_{\Sigma_-}\partial_t\Lbrack\mu\partial_i u \Rbrack \cdot \varphi (e_3\cdot e_j)
\\
\lesssim \norm{\partial_tu}_1\norm{\varphi}_1+\|\partial_t \nabla u_+\|_{H^{-1/2}(\Sigma_+)}\|\varphi\|_{H^{1/2}(\Sigma_+)}
+\|\partial_t\Lbrack\mu\nabla u\Rbrack\|_{H^{-1/2}(\Sigma_-)}\|\varphi\|_{H^{1/2}(\Sigma_-)}
\\
\lesssim \left(\norm{\partial_tu}_1+\|\partial_t \nabla u_+\|_{H^{-1/2}(\Sigma_+)}
+\|\partial_t\Lbrack\mu\nabla u\Rbrack\|_{H^{-1/2}(\Sigma_-)}\right)\norm{\varphi}_1.
\end{multline}
Taking the supremum over such $\varphi$ with $\|\varphi\|_1\le 1$, we get \eqref{a_bound}.

A similar application of these lemmas, together with trace estimates, implies that
\begin{equation}
\begin{split}
\|\partial_t g_-\|_{-1/2} &\lesssim
\|\partial_t\nabla_\ast\eta_-\|_0\|\Lbrack\mu\nabla u\Rbrack\|_{L^2(\Sigma)}
+ \|\nabla_\ast\eta_-\|_2 \|\Lbrack\mu\partial_t\nabla u\Rbrack\|_{H^{-1/2}(\Sigma)} \\
& + \|\partial_t\nabla_\ast^2\eta_-\|_0\|u\|_{L^2(\Sigma)}
+ \|\nabla_\ast^2\eta_-\|_0\|\partial_tu\|_{L^2(\Sigma)} \\
&\lesssim\|\partial_t\eta_-\|_1\| u \|_{2}+\| \eta_-\|_3\|\Lbrack\mu\partial_t\nabla
u\Rbrack\|_{H^{-1/2}(\Sigma)} +  \|\partial_t \eta_-\|_2\|u\|_{1}
+ \| \eta_-\|_2\|\partial_tu\|_{1} \\
&\lesssim \sqrt{\mathcal{E} \mathcal{D}},
\end{split}
\end{equation}
and similarly, $\|\partial_t g_+\|_{-1/2} \lesssim \sqrt{\mathcal{E}\mathcal{D}}$. Hence, \eqref{no3} follows.
\end{proof}

%%%%%%%%%%%%%%%%%%%%%%%%%%%%%%%%%%%%%%%%%%%%%%%
\subsubsection{Fixed-point argument}
%%%%%%%%%%%%%%%%%%%%%%%%%%%%%%%%%%%%%%%%%%%%%%%

We shall use the linear theory provided by Theorem \ref{surlinear} and the nonlinear estimates of Lemma \ref{nonlineares} to prove Theorem \ref{surth} by using the the contraction mapping principle.  Recall that $\mathcal{E}$ and $\mathcal{D}$ are defined by \eqref{E} and \eqref{D}.

\begin{proof}[Proof of Theorem \ref{surth}]

As a first step we consider the dependence of $f,g,\mathcal{E},\mathcal{D}$ on $u,p,\eta$ at $t=0$.  Let $u_0\in  {}_0H_\sigma^1(\Omega) \cap \ddot{H}^{2}(\Omega)$ and $\eta_0\in  {H}^{3}(\Sigma)$ be the given initial data, which satisfy the zero-average condition \eqref{zero0} and the compatibility conditions \eqref{compatibility1ge} with $g(0)=g(u_0,\eta_0)$.  We define the initial data $\partial_t \eta(0)=u_3(0)$ on $\Sigma$ and then define $p(0)$ and $f(0)=f(u_0, \partial_t\eta(0), p(0))$ according to the problem \eqref{pressure0}. It is routine to verify that if $\|u_0\|_2^2+\|\eta_0\|_3^2 \le \delta_0$  with $\delta_0>0$ sufficiently small, then we have
\begin{equation}\label{data_est}
\mathcal{E}(0)\lesssim\|u_0\|_2^2+\|\eta_0\|_3^2\text{ and } \|f(0)\|_{ 0}^2+\| g(0)\|_{1/2}^2\lesssim\|u_0\|_2^2+\|\eta_0\|_3^2.
\end{equation}

We now let $\delta>0$ and define the space $\mathfrak{X}$ as
\begin{equation}
\begin{split}
\mathfrak{X}= \Biggl\{(w,\pi,\zeta) \,\Bigg\vert\, & (1)\,  w(0)=u_0, \zeta (0)=\eta_0; \\
  & (2) \,  \sup_{t\ge 0} \left( \mathcal{E}(w,\pi,\zeta)(t)+ \int_0^t\mathcal{D}(w,\pi,\zeta )(s) ds \right)\le \delta.   \Biggr\}
\end{split}
\end{equation}
Here we write $\mathcal{E}(w,\pi,\zeta)$ and $\mathcal{D}(w,\pi,\zeta)$ to mean $\mathcal{E}$ and $\mathcal{D}$, as defined by \eqref{E} and \eqref{D}, with $(w,\pi,\zeta)$ substituted for $(u,p,\eta)$.  We view $\mathfrak{X}$ as a metric space with the metric induced by the expression in item $(2)$.  The value of $\delta$ will be chosen later, but we will always assume that it is  smaller than that appearing in Lemma \ref{nonlineares}.  Note that if $(w,\pi,\zeta) \in \mathfrak{X}$, then $(w,\pi,\zeta)$ have the same continuity properties listed for $(u,p,\eta)$ in \eqref{regggg}.

We now turn to the mapping that we will use in the contraction mapping argument.  Given $(w,\pi,\zeta)\in \mathfrak{X}$, we let  $f=f(w,\pi,\zeta)$ and $g=g(w,\zeta)$. Then by Lemma \ref{nonlineares} we have
\begin{multline} \label{fges}
\|f\|_{L^\infty L^2}^2+\| g\|_{L^\infty H^{1/2}}^2+\|f\|_{L^2H^1}^2+\|g\|_{L^2 H^{3/2}}^2+\|\partial_tf\|_{L^2 \mathcal{W}^\ast}^2+\|\partial_tg\|_{L^2H^{-1/2}}^2 \\
\le C_5   (\mathcal{E}(w,\pi,\zeta)(t))^2+ 2C_5\int_0^t
\mathcal{E}(w,\pi,\zeta)(s) \mathcal{D}(w,\pi,\zeta)(s) ds \le 3 C_5 \delta^2.
\end{multline}
Since $w(0)=u_0$ and $\zeta(0)=\eta_0$ we have that $u_0,\eta_0, g(0)$ satisfy the compatibility conditions
\eqref{compatibility11}.  Hence, by Theorem \ref{surlinear}, there exists a unique strong solution $(v,q,\zeta)$ to the linear problem \eqref{surfaceLP} (with the forcing terms replaced by these $f,g$) on the interval on $[0,\infty)$; the solution satisfies the regularity properties listed in \eqref{regggg} and achieves the initial data $(u_0,\eta_0)$.  Moreover, by \eqref{strongest} and \eqref{fges}  we have
\begin{equation}\label{mm111}
\mathcal{E}(v,q,\zeta)(t)+ \int_0^t\mathcal{D}(v,q,\zeta)(s) ds \le C_4(\|u_0\|_2^2+\|\eta_0\|_3^2+4C_5\delta^2)
\end{equation}
for all $t\ge 0$.  By the above procedure, given $(w,\pi,\zeta)\in \mathfrak{X}$, we can define a nonlinear operator $\mathcal{L}$ such that $(v,q,\zeta)=\mathcal{L}(w,\pi,\zeta)$ is the solution of \eqref{surfaceLP} described above.   If we restrict $\delta,\delta_0$ so that $\delta\le 1/(8C_4C_5)$ and $\delta_0\le \delta/(2
C_4)$, then in fact  $\mathcal{L}:\mathfrak{X}\rightarrow\mathfrak{X}$.  Indeed, we deduce from \eqref{mm111} that
\begin{equation}\label{mm2}
\mathcal{E}(v,q,\zeta)(t)+ \int_0^t\mathcal{D}(v,q,\zeta)(s)ds \le \delta
\end{equation}
for all $t\ge 0$, which implies that $(v,q,\zeta)= \mathcal{L}(w,\pi,\zeta)$  is an element of $\mathfrak{X}$.

Now, given $(w^1,\pi^1,\zeta^1),(w^2,\pi^2,\zeta^2)\in \mathfrak{X}$, we let 
\begin{equation}
(v^1,q^1,\zeta^1) = \mathcal{L}(w^1,\pi^1,\zeta^1) \text{ and } (v^2,q^2,\zeta^2)=\mathcal{L}(w^2,\pi^2,\zeta^2). 
\end{equation}
Then $(\bar{v},\bar{q},\bar{\zeta}) = (v^1-v^1, q^1-q^2, \zeta^1-\zeta^2)$ is a solution of the linear problem \eqref{surfaceLP} with initial data $\bar{v}(0)=0,  \bar{\zeta}(0)=0$ and the forcing terms $\bar{f} = f(w_1,\pi_1,\zeta_1) - f(w_2,\pi_2,\zeta_2)$ and $\bar{g}=g(w_1,\zeta_1)-g(w_2,\zeta_2)$. We remark that $\bar{g}(0)=0$ so that $\bar{v}(0), \bar{\zeta}(0)$ and $\bar{g}(0)$ satisfy the compatibility condition \eqref{compatibility11}. As in Lemma \ref{nonlineares}, we can deduce  that
\begin{multline} \label{fgesbar}
\|\bar{f}\|_{L^\infty L^2}^2+\| \bar{g}\|_{L^\infty H^{1/2}}^2+\|\bar{f}\|_{L^2H^1}^2
+\|\bar{g}\|_{L^2 H^{3/2}}^2+\|\partial_t\bar{f}\|_{L^2
\mathcal{W}^\ast}^2+\|\partial_t\bar{g}\|_{L^2H^{-1/2}}^2
\\
\le C_5  \sup_{t\ge 0} \Bigl[ ( \mathcal{E}(w_1,\pi_1,\zeta_1)(t)+
\mathcal{E}(w_2,\pi_2,\zeta_2)(t))
\mathcal{E}(\bar{w},\bar{\pi},\bar{\zeta})(t)
\\
+   \int_0^t ( \mathcal{E}(w_1,\pi_1,\zeta_1)(s)+ \mathcal{E}(w_2,\pi_2,\zeta_2)(s))\mathcal{D}(\bar{w},\bar{\pi},\bar{\zeta})(s) ds
\\
+   \int_0^t \mathcal{E}(\bar{w},\bar{\pi},\bar{\zeta})(s)( \mathcal{D}(w_1,\pi_1,\zeta_1)(s)+ \mathcal{D}(w_2,\pi_2,\zeta_2)(s)) ds \Bigr]
\\
\le 4 C_5\delta \sup_{t\ge 0} \left[ \mathcal{E}(\bar{w},\bar{\pi},\bar{\zeta})(t) +\int_0^t
\mathcal{D}(\bar{w},\bar{\pi},\bar{\zeta})(s)ds  \right].
\end{multline}
Hence, applying Theorem \ref{surlinear} again, we obtain
\begin{equation}\label{mm1222}
\sup_{t\ge 0} \left[ \mathcal{E}(\bar{v},\bar{q},\bar{\zeta})(t) +\int_0^t
 \mathcal{D}(\bar{v},\bar{q},\bar{\zeta})(s) ds \right] \le 4 C_4 C_5 \delta\sup_{t\ge0} \left[\mathcal{E}(\bar{w},\bar{\pi},\bar{\zeta})(t)  +\int_0^t  \mathcal{D}(\bar{w},\bar{\pi},\bar{\zeta})(s)ds  \right].\end{equation}
Since $4 C_4C_5\delta<1/2$, we find that $\mathcal{L}$ is a contraction mapping in $\mathfrak{X}$. Therefore, there is a unique fixed point $(u,p,\eta)$ of the operator $\mathcal{L}$ in $\mathfrak{X}$. By the definition of $\mathcal{L}$, $(u,p,\eta)$  is the unique solution of \eqref{surface}.

We have now obtained a solution $(u,p,\eta)$ that satisfies $\mathcal{E} \le \delta$.  We may then revisit the proofs of Theorem \ref{surlinear} and Lemma \ref{nonlineares} to find that there exists an  energy functional that is equivalent to $\mathcal{E}$ and a dissipation that is equivalent to $\mathcal{D}$ (for simplicity we
still denote them by $\mathcal{E}$ and $\mathcal{D}$)  so that
\begin{equation}\label{hhh}
\frac{d}{dt}\mathcal{E} + C\mathcal{D}\le 0
\end{equation}
for a universal constant $C>0$.  Since $\mathcal{E} \lesssim \mathcal{D}$, the bound \eqref{gles} and the decay estimate \eqref{decayes} follow directly by \eqref{hhh} and an application of Gronwall's inequality.
\end{proof}

\appendix

%%%%%%%%%%%%%%%%%%%%%%%%%%%%%%%%%%%%%%%%%%%%%%%
\section{Analytic tools}\label{section_appendix}
%%%%%%%%%%%%%%%%%%%%%%%%%%%%%%%%%%%%%%%%%%%%%%%

%%%%%%%%%%%%%%%%%%%%%%%%%%%%%%%%%%%%%%%%%%%%%%%
\subsection{Poisson extension}
%%%%%%%%%%%%%%%%%%%%%%%%%%%%%%%%%%%%%%%%%%%%%%%

We will now define the appropriate Poisson integrals that allow us to extend $\eta_\pm$, defined on the surfaces $\Sigma_\pm$, to functions defined on $\Omega$, with ``good'' boundedness.

Suppose that $\Sigma_+ = \mathrm{T}^2\times \{1\}$, where $\mathrm{T}^2:=(2\pi L_1 \mathbb{T}) \times (2\pi L_2 \mathbb{T})$. We define the Poisson integral in $\mathrm{T}^2 \times (-\infty,1)$ by
\begin{equation}\label{P-1def}
\mathcal{P}_{-,1}f(x) = \sum_{n \in   (L_1^{-1} \mathbb{Z}) \times
(L_2^{-1} \mathbb{Z}) }  \frac{e^{i n \cdot x'}}{2\pi \sqrt{L_1 L_2}} e^{|n|(x_3-1)} \hat{f}(n),
\end{equation}
where for $n \in (L_1^{-1} \mathbb{Z}) \times (L_2^{-1} \mathbb{Z})$ we have written
\begin{equation}
 \hat{f}(n) = \int_{\mathrm{T}^2} f(x')  \frac{e^{- i n \cdot x'}}{2\pi \sqrt{L_1 L_2}} dx'.
\end{equation}
Here ``$-$'' stands for extending downward and ``$1$'' stands for extending at $x_3=1$, etc. It is well-known that $\mathcal{P}_{-,1}:H^{s}(\Sigma_+) \rightarrow H^{s+1/2}(\mathrm{T}^2 \times (-\infty,1))$ is a bounded linear operator for $s>0$. However, if restricted to the domain $\Omega$, we can have the following improvements.

\begin{lemma}\label{Poi}
Let $\mathcal{P}_{-,1}f$ be the Poisson integral of a function $f$ that is either in $\dot{H}^{q}(\Sigma_+)$ or
$\dot{H}^{q-1/2}(\Sigma_+)$ for $q \in \mathbb{N}=\{0,1,2,\dots\}$, where we have written $\dot{H}^s(\Sigma_+)$ for the homogeneous Sobolev space of order $s$.  Then
\begin{equation}
 \|\nabla^q \mathcal{P}_{-,1}f \|_{0} \lesssim \|f\|_{\dot{H}^{q-1/2}(\mathrm{T}^2)}^2 \text{ and }
 \|\nabla^q \mathcal{P}_{-,1}f \|_{0} \lesssim \|f\|_{\dot{H}^{q}(\mathrm{T}^2)}^2.
\end{equation}
\end{lemma}
\begin{proof}
 See Lemma A.3 of \cite{GT_per}.
\end{proof}

We extend $\eta_+$ to be defined on $\Omega$ by
\begin{equation}\label{P+def}
\bar{\eta}_+(x',x_3)=\mathcal{P}_+\eta_+(x',x_3):=\mathcal{P}_{-,1}\eta_+(x',x_3),\text{ for } x_3\le 1.
\end{equation}
Then Lemma \ref{Poi} implies in particular that if $\eta_+\in H^{s-1/2}(\Sigma_+)$ for $s\ge 0$, then $\bar{\eta}_+\in H^{s}(\Omega)$.

Similarly, for $\Sigma_- = \mathrm{T}^2\times \{0\}$ we define the Poisson integral in $\mathrm{T}^2 \times (-\infty,0)$ by
\begin{equation}\label{P-0def}
\mathcal{P}_{-,0}f(x) = \sum_{n \in   (L_1^{-1} \mathbb{Z}) \times (L_2^{-1} \mathbb{Z}) }  \frac{e^{ i n \cdot x'}}{2\pi \sqrt{L_1 L_2}} e^{ |n|x_3} \hat{f}(n).
\end{equation}
It is clear that $\mathcal{P}_{-,0}$ has the  same regularity properties as $\mathcal{P}_{-,1}$. This allows us to extend $\eta_-$ to be defined on $\Omega_-$. However, we do not extend $\eta_-$ to the upper domain $\Omega_+$ by the reflection  since this will result in the discontinuity of the partial derivatives in $x_3$ of the extension. For our purposes,  we instead to do the extension through the following. Let $0<\lambda_0<\lambda_1<\cdots<\lambda_m<\infty$ for $m\in \mathbb{N}$ and define the $(m+1) \times (m+1)$ Vandermonde matrix $V(\lambda_0,\lambda_1,\dots,\lambda_m)$ by $V(\lambda_0,\lambda_1,\dots,\lambda_m)_{ij} = (-\lambda_j)^i$ for $i,j=0,\dotsc,m$.  It is well-known that the Vandermonde matrices are invertible, so we are free to let $\alpha=(\alpha_0,\alpha_1,\dots,\alpha_m)^T$ be the solution to
\begin{equation}\label{Veq}
V(\lambda_0,\lambda_1,\dots,\lambda_m)\,\alpha=q_m,
\end{equation}
$q_m=(1,1,\dots,1)^T$.  Now we define the specialized Poisson integral in $\mathrm{T}^2 \times (0,\infty)$
by
\begin{equation}\label{P+0def}
\mathcal{P}_{+,0}f(x) = \sum_{n \in   (L_1^{-1} \mathbb{Z}) \times
(L_2^{-1} \mathbb{Z}) }  \frac{e^{ i n \cdot x'}}{2\pi \sqrt{L_1 L_2}}  \sum_{j=0}^m\alpha_j
e^{- |n|\lambda_jx_3} \hat{f}(n).
\end{equation}
It is easy to check that, due to \eqref{Veq}, $\partial_3^l\mathcal{P}_{+,0}f(x',0)=
\partial_3^l\mathcal{P}_{-,0}f(x',0)$  for all $0\le l\le m$ and hence
\begin{equation}
\partial^\alpha\mathcal{P}_{+,0}f(x',0)=
\partial^\alpha\mathcal{P}_{-,0}f(x',0) \, \forall\, \alpha\in \mathbb{N}^3 \text{ with }0\le |\alpha|\le m.\end{equation}
These facts allow us to  extend $\eta_-$ to be defined on $\Omega$
by
\begin{equation}\bar{\eta}_-(x',x_3)=
\mathcal{P}_-\eta_-(x',x_3):=\left\{\begin{array}{lll}\mathcal{P}_{+,0}\eta_-(x',x_3),\quad
x_3> 0 \\
\mathcal{P}_{-,0}\eta_-(x',x_3),\quad x_3\le
0.\end{array}\right.\label{P-def}\end{equation} It is clear now that if $\eta_-\in H^{s-1/2}(\Sigma_-)$ for $ 0\le s\le m$, then $\bar{\eta}_-\in H^{s}(\Omega)$.  Since we will only work with $s$ lying in a finite interval, we may assume that $m$ is sufficiently large in \eqref{Veq} for $\bar{\eta}_- \in H^s(\Omega)$ for all $s$ in the interval.

%%%%%%%%%%%%%%%%%%%%%%%%%%%%%%%%%%%%%%%%%%%%%%%
\subsection{Some inequalities}
%%%%%%%%%%%%%%%%%%%%%%%%%%%%%%%%%%%%%%%%%%%%%%%

We will need some estimates of the product of functions in Sobolev spaces.

\begin{lemma}\label{sobolev}
Let $U$ denote a domain either of the form $\Omega_\pm$ or of the form $\Sigma_\pm$.
\begin{enumerate}
 \item Let $0\le r \le s_1 \le s_2$ be such that  $s_1 > n/2$.  Let $f\in H^{s_1}(U)$, $g\in H^{s_2}(U)$.  Then $fg \in H^r(U)$ and
\begin{equation}\label{i_s_p_01}
 \norm{fg}_{H^r} \lesssim \norm{f}_{H^{s_1}} \norm{g}_{H^{s_2}}.
\end{equation}

\item Let $0\le r \le s_1 \le s_2$ be such that  $s_2 >r+ n/2$.  Let $f\in H^{s_1}(U)$, $g\in H^{s_2}(U)$.  Then $fg \in H^r(U)$ and
\begin{equation}\label{i_s_p_02}
 \norm{fg}_{H^r} \lesssim \norm{f}_{H^{s_1}} \norm{g}_{H^{s_2}}.
\end{equation}

\item Let $0\le r \le s_1 \le s_2$ be such that  $s_2 >r+ n/2$. Let $f \in H^{-r}(\Sigma),$ $g \in H^{s_2}(\Sigma)$.  Then $fg \in H^{-s_1}(\Sigma)$ and
\begin{equation}\label{i_s_p_03}
 \norm{fg}_{-s_1} \ls \norm{f}_{-r} \norm{g}_{s_2}.
\end{equation}
\end{enumerate}
\end{lemma}
\begin{proof}
These results are standard and may be derived, for example, by use of the Fourier characterization of the $H^s$ spaces.
\end{proof}

The estimates in $(2)$ and $(3)$ in Lemma \ref{sobolev} above are not the best possible estimates of that form.  We now record one specific improvement that we will use for products in $\H(\Omega)$.

\begin{lemma}\label{products}
Suppose that $f \in H^{s}(\Omega)$ with $s > 3/2$ and $g \in \H(\Omega)$.  Then $f g\in \H(\Omega)$ with $\|fg\|_{1} \lesssim \|f\|_{s} \|g\|_{1}$.  Moreover, if $g \in (\H(\Omega))^\ast$ then $fg \in (\H(\Omega))^\ast$ via $\langle fg,\varphi\rangle_{\ast} := \langle g,f\varphi\rangle_{\ast}$.  In this case we have the estimate $\|fg\|_{ \Hd } \lesssim \|f\|_{s} \|g\|_{\Hd}$.
\end{lemma}
\begin{proof}
 For the first assertion we use the Sobolev and H\"older inequalities to estimate
\begin{multline}
 \| fg \|_{H^1}^2 = \| f g\|_{L^2}^2 + \|f \nabla g\|_{L^2}^2 +\|\nabla f g\|_{L^2}^2   \\
\le
\|f\|_{L^\infty}^2 \|  g\|_{L^2}^2 + \|f\|_{L^\infty}^2 \|\nabla g\|_{L^2}^2 +\|\nabla f\|_{L^3} \|g\|_{L^6}^2
\lesssim  \|f\|_{H^s}^2 \| g\|_{H^1}^2.
\end{multline}
Also, clearly $f g \vert_{\Sigma_b} = 0$.  The second assertion follows easily from the first and the duality.
\end{proof}

We will also use the following lemma.

\begin{lemma}\label{-1norm}
The following hold.
\begin{enumerate}
\item Let $f\in L^2(\Omega),\ g\in H^1(\Omega)$, then  $fg\in (\H(\Omega))^\ast$ and
\begin{equation}\label{i_s_p_06}
\|fg\|_{(\H(\Omega))^\ast}\lesssim \|f\|_0\|g\|_{1}.
\end{equation}

\item Let $f\in L^2(\Sigma),\ g\in H^{1/2}(\Sigma)$, then  $fg\in H^{-1/2}(\Sigma)$ and
\begin{equation}\label{i_s_p_07}
\|fg\|_{-1/2}\lesssim \|f\|_0\|g\|_{0}.
\end{equation}

\end{enumerate}
\end{lemma}
\begin{proof}
To prove \eqref{i_s_p_06} we choose any $\varphi\in \H(\Omega)$.  By the H\"older and Sobolev inequalities, we have that
\begin{equation}
\int_\Omega fg\varphi\,dx\le
\|f\|_{L^2(\Omega)}\|g\|_{L^3(\Omega)}\|\varphi\|_{L^6(\Omega)}\le
\|f\|_0\|g\|_1\|\varphi\|_1.
\end{equation}
Taking the supremum over such $\varphi$ with $\|\varphi\|_1\le 1$, we get \eqref{i_s_p_06}. The estimate \eqref{i_s_p_07} follows in a similar way by recalling that the Sobolev embedding shows that $H^{1/2}(\Sigma) \hookrightarrow L^4(\Sigma)$.
\end{proof}

In the following lemma we let $U$ denote a periodic domain of the form $\Omega_\pm$ with flat upper boundary $\Gamma_{u}$ and lower boundary $\Gamma_{l}$, which may not be flat.  We now record some Poincar\'e-type inequalities for such domains.

\begin{lemma}\label{poincare}
 The following hold.
\begin{enumerate}
 \item $\|f\|_{L^2(U)}^2 \lesssim \|f\|_{L^2(\Gamma_u)}^2 +  \|\partial_3f\|_{L^2(U)}^2$ for all $f \in H^1(U)$.

 \item $\|f\|_{L^2(\Gamma_u)} \lesssim \|\partial_3f\|_{L^2(U)}$ for $f \in H^1(U)$ so that $f =0$ on $\Gamma_l$.

 \item $\|f\|_{0} \lesssim \|f\|_{1}\lesssim \|\nabla f\|_{0}$ for all $f \in H^1(U)$ so that $f=0$ on $\Gamma_l$.
\end{enumerate}
\end{lemma}
\begin{proof}
See Appendix A.4 of \cite{GT_per}.
\end{proof}

We will need the following version of Korn's inequality.

\begin{lemma}\label{korn}
It holds that $\|u\|_{1} \lesssim \|\mathbb{D}u \|_{0}$ for all $u \in \H(\Omega)$.
\end{lemma}
\begin{proof}
 See Lemma 2.7 of \cite{B1}.
\end{proof}

%%%%%%%%%%%%%%%%%%%%%%%%%%%%%%%%%%%%%%%%%%%%%%%
\subsection{Two classical one-phase Stokes problems}
%%%%%%%%%%%%%%%%%%%%%%%%%%%%%%%%%%%%%%%%%%%%%%%

We now let $G$ denote a horizontal periodic slab (like $\Omega_\pm$) and write $\Gamma$ for its boundary (not necessarily flat),  consisting of two smooth pieces $\Gamma_1,\ \Gamma_2$.  We shall recall the classical  regularity theory for two Stokes problems in $G$.

We first state the following Stokes problem
\begin{equation}\label{cS1phase1}
\left\{\begin{array}{lll}-\mu \Delta u +\nabla p =f &\text{ in}\ G
  \\\diverge{u} =h  \ &\text{ in}\ G
  \\(pI-\mu\mathbb{D}(u))\nu=\psi  \ &\text{ on }\Gamma_1
\\ u=\varphi &\text{ on }\Gamma_2.
\end{array}\right.
\end{equation}
Here $\nu$ denotes the outward pointing unit normal on the boundary.

\begin{lemma}\label{cS1phaselemma1}
Let $r \ge 2$.  If $f\in H^{r-2}(G),\ h\in H^{r-1}(G),\ \psi\in H^{r-3/2}(\Gamma_1),\ \varphi\in H^{r-1/2}(\Gamma_2)$, then there exists unique $u\in H^r(G),\  p\in H^{r-1}(G)$  solving \eqref{cS1phase1}. Moreover,
\begin{equation}
\|u\|_{H^r(G)}+\|  p\|_{H^{r-1}(G)} \lesssim\|f\|_{H^{r-2}(G)}+\|h\|_{H^{r-1}(G)}
+ \norm{\psi}_{H^{r-3/2}(\Gamma_1)} + \|\varphi\|_{H^{r-1/2}(\Gamma_2)}.
\end{equation}
\end{lemma}
\begin{proof}
 See \cite{B1}.
\end{proof}

An essential step in our analysis is to bypass the regularity theory for \eqref{cS1phase1} and instead appeal to a similar Stokes problem with Dirichlet boundary conditions on both $\Gamma_1$ and $\Gamma_2$:
\begin{equation}\label{cS1phase2}
  \left\{\begin{array}{lll}
-\mu \Delta u +\nabla p =f  \quad &\text{in }G
  \\   \diverge{u}=h  \quad  &\text{in } G
\\ u=\varphi_1\quad &\text{on }\Gamma_1
\\ u=\varphi_2\quad &\text{on }\Gamma_2.
\end{array}\right.
\end{equation}

The following records the regularity theory for this problem.

\begin{lemma}\label{cS1phaselemma2}
Let $r\ge 2$.  Let $f\in H^{r-2}(G),\ h\in H^{r-1}(G),\ \varphi_1\in H^{r-1/2}(\Gamma_1), \ \varphi_2\in H^{r-1/2}(\Gamma_2)$ be given such that
\begin{equation}\label{cS1phasecomp}
\int_G h =\int_{\Gamma_1} \varphi_1\cdot\nu +\int_{\Gamma_2} \varphi_2\cdot\nu.
\end{equation}
Then there exists unique $u\in H^r(G),\  p\in H^{r-1}(G)$(up to constants) solving \eqref{cS1phase2}. Moreover,
\begin{equation}
\|u\|_{H^r(G)}+\|\nabla p\|_{H^{r-2}(G)} \lesssim \|f\|_{H^{r-2}(G)} + \|h\|_{H^{r-1}(G)} + \|\varphi_1\|_{H^{r-1/2}(\Gamma_1)} + \|\varphi_2\|_{H^{r-1/2}(\Gamma_2)}.
\end{equation}
\end{lemma}
\begin{proof}
 See \cite{L,T}.
\end{proof}

%%%%%%%%%%%%%%%%%%%%%%%%%%%%%%%%%%%%%%%%%%%%%%%
\subsection{Time-dependent functional setting}\label{time_dep_fnal}
%%%%%%%%%%%%%%%%%%%%%%%%%%%%%%%%%%%%%%%%%%%%%%%

Let us define
\begin{equation}\label{0H}
\begin{split}
 {}_0H^1(\Omega) &:= \{ u \in H^1(\Omega)\mid u\mid_{\Sigma_b}=0\}, \\
{}^0H^1(\Omega) &:= \{u\in H^1(\Omega) \mid u\mid_{\Sigma_+}=0\}, \text{ and } \\
 {}_0H_\sigma^1(\Omega)&:= \{ u \in  {}_0H^1(\Omega) \mid \diverge{u}=0 \},
\end{split}
\end{equation}
with the obvious restriction that the last space is for vector-valued functions only.

In order to study the $\mathcal{A}$--Stokes problems, we use the time-dependent functional setting introduced in \cite{GT_lwp}.  In this context we think of $\eta_\pm$ as given (with $J$, $\mathcal{A}$, etc defined by $\eta_\pm$ as in \eqref{ABJ_def}) and use it to construct some time-dependent function spaces.  We define a time-dependent inner-product on  $L^2=H^0$ by introducing
\begin{equation}
(u,v)_{\mathcal{H}^0(t)} := \int_\Omega  ( u \cdot v)  J(t)
\end{equation}
with corresponding norm $\|u\|_{\mathcal{H}^0(t)}:=\sqrt{(u,u)_{\mathcal{H}^0(t)}}$.  Then we write $\mathcal{H}^0(t) := \{ \|u\|_{\mathcal{H}^0(t)}< \infty \}$.  Similarly, we define a time-dependent inner-product on
${}_0H^1(\Omega)$ according to
\begin{equation}
(u,v)_{\mathcal{H}^1(t)}:=  \int_\Omega \mu\left(
\mathbb{D}_{\mathcal{A}(t)} u : \mathbb{D}_{\mathcal{A}(t)} v
\right) J(t),
\end{equation}
where for two $n\times n$ matrices $A,B$ we write $A:B = \sum_{i,j} A_{i,j}B_{i,j}$.  We define the corresponding norm by $\|u\|_{\mathcal{H}^1(t)}:=\sqrt{(u,u)_{\mathcal{H}^1(t)}}$.  Then we define
\begin{equation}
 \mathcal{H}^1(t):= \{ u \mid  \|u\|_{\mathcal{H}^1(t)}< \infty
 ,  u\mid_{\Sigma_b}=0\} \text{ and } \mathcal{X}(t) := \{ u \in
\mathcal{H}^1(t) \mid \diverge_{ \mathcal{A}(t)} u=0\}.
\end{equation}
We will also need the orthogonal decomposition $ \mathcal{H}^0(t) = \mathcal{Y}(t) \oplus  \mathcal{Y}(t)^\bot,$ where
\begin{equation}
\mathcal{Y}(t)^\bot := \{ \nabla_\mathcal{A} \varphi \mid \varphi
\in {}^0H^1(\Omega)\}.
\end{equation}

Finally, for $T>0$ and $k=0,1$, we define inner-products on $L^2H^k:=L^2([0,T];H^k(\Omega))$ by
\begin{equation}
 (u,v)_{ \mathcal{H}^k_T} = \int_0^T  (u(t),v(t))_{ \mathcal{H}^k}  dt.
\end{equation}
Write $\|u\|_{\mathcal{H}^k_T}^2$ for the corresponding norms and $\mathcal{H}^k_T$ for the corresponding spaces.  We define the subspace of $\diverge_\mathcal{A}$-free vector fields as
\begin{equation}
 \mathcal{X}_T := \{ u \in  \mathcal{H}^1_T \mid \diverge_\mathcal{A}{u} =0
\text{ for a.e. } t\in[0,T]\}.
\end{equation}

Under a smallness assumption on $\eta$, we can show that the spaces $\mathcal{H}^k(t)$ and $\mathcal{H}^k_T$ have the same topology as $H^k$ and $L^2 H^k$, respectively.

\begin{lemma}\label{equal}
There exists a universal $\epsilon_0 > 0$ so that if $\sup_{0\le t \le T} \|\eta(t)\|_{3} < \epsilon_0,$
then
\begin{equation}
\frac{1}{\sqrt{2}} \|u\|_{k} \le \|u\|_{\mathcal{H}^k(t)} \le
\sqrt{2} \|u\|_{k}
\end{equation}
for $k=0,1$ and for all $t \in [0,T]$.  As a consequence, for $k=0,1$,
\begin{equation}
\frac{1}{\sqrt{2}} \|u\|_{L^2 H^k} \le \|u\|_{\mathcal{H}^k_T }\le
\sqrt{2} \|u\|_{L^2 H^k}.
\end{equation}
\end{lemma}
\begin{proof}
See Lemma 2.1 of \cite{GT_lwp}.
\end{proof}

It is more important that ${}_0H_\sigma^1(\Omega)$ is related to the spaces $ \mathcal{X}(t)$. To this end, we define the matrix
\begin{equation}\label{M_def}
 M := M(t) = K \nabla \Theta =
\begin{pmatrix}
K & 0 & 0 \\
0 & K & 0 \\
AK & BK & 1
\end{pmatrix}
\end{equation}
where $A,B,K$ are defined by \eqref{ABJ_def}.  The matrix $M(t)$ induces a linear operator $\mathcal{M}_t: u \mapsto \mathcal{M}_t(u) = M(t) u$, the properties of which are recorded in the following.

\begin{Proposition}\label{M}
For each $t \in[0,T]$, $\mathcal{M}_t$ is a bounded linear isomorphism: from $H^k(\Omega)$ to $H^k(\Omega)$ for $k=0,1$; from $\ddot{H}^2(\Omega)$ to $\ddot{H}^2(\Omega)$; from $L^2(\Omega)$ to $\mathcal{H}^0(t)$;  from ${}_0H^1(\Omega)$ to $\mathcal{H}^1(t)$; and from ${}_0H_\sigma^1(\Omega)$ to $ \mathcal{X}(t)$. In each case
the norms of the operators $\mathcal{M}_t, \mathcal{M}_t^{-1}$ are bounded by a constant times $1 + \|\eta(t)\|_{9/2}$.

Moreover,  the mapping $\mathcal{M}$ given by $\mathcal{M}u(t) := \mathcal{M}_t u(t)$ is a bounded linear isomorphism: from $L^2([0,T];H^k(\Omega))$ to $L^2([0,T];H^k(\Omega))$ for $k=0,1$; from $L^2([0,T];\ddot{H}^2(\Omega))$ to $L^2([0,T];\ddot{H}^2(\Omega))$; from $L^2([0,T];L^2(\Omega))$ to $\mathcal{H}^0_T$;  from $L^2([0,T];{}_0H^1(\Omega))$ to $\mathcal{H}^1_T$; and from $L^2([0,T];{}_0H_\sigma^1(\Omega))$ to $\mathcal{X}_T$. In each case, the norms of the operators $\mathcal{M}$ and $\mathcal{M}^{-1}$ are bounded by a constant times
the sum $1+\sup_{0\le t \le T} \|\eta(t)\|_{9/2}$.
\end{Proposition}
\begin{proof}
See Proposition 2.5 of \cite{GT_lwp}.
\end{proof}

The following proposition allows us to introduce the pressure as a Lagrange multiplier.

\begin{Proposition}\label{Pressure}
If $\Lambda_t \in ( \mathcal{H}^1(t))^\ast$ is such that $\Lambda_t(v) = 0$ for all $v \in  \mathcal{X}(t)$, then there exists a unique $p(t) \in  \mathcal{H}^0(t)$ so that
\begin{equation}
(p(t), \diverge_{\mathcal{A}(t)} v)_{\mathcal{H}^0(t)} = \Lambda_t(v) \text{ for all } v\in  \mathcal{H}^1(t)
\end{equation}
and $\|p(t)\|_{ \mathcal{H}^0(t)} \lesssim (1+\|\eta(t)\|_{9/2}) \|\Lambda_t\|_{(\mathcal{H}^1(t))^\ast}$.

If $\Lambda \in ( \mathcal{H}^1_T)^\ast$ is such that $\Lambda(v) = 0 $ for all $v \in  \mathcal{X}_T$, then there exists a unique $p \in  \mathcal{H}^0_T$ so that
\begin{equation}
(p, \diverge_{\mathcal{A}} v)_{ \mathcal{H}^0_T} = \Lambda(v)
\text{ for all } v\in  \mathcal{H}^1_T
\end{equation}
and $\|p\|_{ \mathcal{H}^0_T} \lesssim (1+\sup_{0\le t\le T} \|\eta(t)\|_{9/2}) \|\Lambda\|_{( \mathcal{H}^1_T)^\ast}$.
\end{Proposition}
\begin{proof}
See Proposition 2.9 of \cite{GT_lwp}.
\end{proof}

%%%%%%%%%%%%%%%%%%%%%%%%%%%%%%%%%%%%%%%%%%%%%%%
\section*{Acknowledgements}
%%%%%%%%%%%%%%%%%%%%%%%%%%%%%%%%%%%%%%%%%%%%%%%

The authors are grateful to Yan Guo for suggesting this problem and for many helpful discussions. Y. J. Wang would like to express his gratitude for the hospitality of the Division of Applied Mathematics at Brown University and for the support of the China Scholarship Council during his visit to Brown University, where this work was initiated.

\end{document}